\PassOptionsToPackage{dvipsnames}{xcolor}
\PassOptionsToPackage{pagebackref,draft=false}{hyperref}
\documentclass{article}

\usepackage{geometry}
\usepackage{amsmath,amsthm,amsfonts}
\usepackage{colonequals}
\usepackage{mathtools}
\usepackage{tikz,tikz-cd}
\usetikzlibrary{backgrounds,arrows,shapes,shapes.geometric,shapes.misc}
\usetikzlibrary{decorations.pathreplacing}
\tikzstyle{tikzfig}=[baseline=-0.25em,scale=0.5]
\pgfkeys{/tikz/tikzit fill/.initial=0}
\pgfkeys{/tikz/tikzit draw/.initial=0}
\pgfkeys{/tikz/tikzit shape/.initial=0}
\pgfkeys{/tikz/tikzit category/.initial=0}
\pgfdeclarelayer{edgelayer}
\pgfdeclarelayer{nodelayer}
\pgfsetlayers{background,edgelayer,nodelayer,main}
\tikzstyle{none}=[inner sep=0mm]
\tikzstyle{every loop}=[]

\tikzstyle{morphism}=[fill=white, draw=black, shape=rectangle]
\tikzstyle{medium box}=[fill=white, draw=black, shape=rectangle, minimum width=0.7cm, minimum height=0.7cm]
\tikzstyle{large morphism}=[fill=white, draw=black, shape=rectangle, minimum width=12.9cm, minimum height=0.1cm]
\tikzstyle{bn}=[fill=black, draw=black, shape=circle, inner sep=1.5pt]
\tikzstyle{state}=[fill=white, draw=black, regular polygon, regular polygon sides=3, minimum width=0.8cm, shape border rotate=180, inner sep=0pt]
\tikzstyle{medium state}=[fill=white, draw=black, regular polygon, regular polygon sides=3, minimum width=1.3cm, inner sep=0pt, shape border rotate=180]
\tikzstyle{large state}=[fill=white, draw=black, regular polygon, regular polygon sides=3, minimum width=2.2cm, shape border rotate=180, inner sep=0pt]
\tikzstyle{wide state}=[fill=white, draw=black, shape=isosceles triangle, minimum width=0.8cm, shape border rotate=270, inner sep=1.4pt, minimum height=0.5cm, isosceles triangle apex angle=80]
\tikzstyle{wn}=[fill=white, draw=black, shape=circle, inner sep=1.5pt]
\tikzstyle{blue morphism}=[fill=white, draw={rgb,255: red,15; green,0; blue,150}, shape=rectangle, text={rgb,255: red,15; green,0; blue,150}, tikzit category=blue]
\tikzstyle{blue state}=[fill=white, draw={rgb,255: red,15; green,0; blue,150}, shape=circle, regular polygon, regular polygon sides=3, minimum width=0.8cm, shape border rotate=180, inner sep=0pt, text={rgb,255: red,15; green,0; blue,150}, tikzit category=blue]
\tikzstyle{blue node}=[fill={rgb,255: red,15; green,0; blue,150}, draw={rgb,255: red,15; green,0; blue,150}, shape=circle, tikzit category=blue, inner sep=1.5pt]
\tikzstyle{blue wide state}=[fill=white, draw={rgb,255: red,15; green,0; blue,150}, text={rgb,255: red,15; green,0; blue,150}, shape=isosceles triangle, minimum width=0.8cm, shape border rotate=270, inner sep=1.4pt, minimum height=0.5cm, isosceles triangle apex angle=80]
\tikzstyle{red node}=[fill={rgb,255: red,150; green,0; blue,2}, draw={rgb,255: red,150; green,0; blue,2}, shape=circle, inner sep=1.5pt]
\tikzstyle{Purple node}=[fill={rgb,255: red,120; green,0; blue,120}, draw={rgb,255: red,120; green,0; blue,120}, text={rgb,255: red,120; green,0; blue,120}, shape=circle, inner sep=1.5pt]
\tikzstyle{white morphism}=[fill=white, draw=white, shape=rectangle, tikzit draw={rgb,255: red,139; green,139; blue,139}]
\tikzstyle{leak morphism}=[fill=white, draw={rgb,255: red,120; green,0; blue,85}, shape=rectangle, text={rgb,255: red,120; green,0; blue,85}, tikzit category=leak]
\tikzstyle{leak}=[text={rgb,255: red,120; green,0; blue,85}, inner sep=0mm, tikzit fill=white, tikzit draw={rgb,255: red,191; green,191; blue,191}, tikzit category=leak]
\tikzstyle{leak node}=[fill={rgb,255: red,120; green,0; blue,85}, draw={rgb,255: red,120; green,0; blue,85}, shape=circle, inner sep=1.5pt, tikzit category=leak]
\tikzstyle{blue}=[text={rgb,255: red,15; green,0; blue,150}, tikzit draw={rgb,255: red,191; green,191; blue,191}, tikzit category=blue, tikzit fill=white, inner sep=0mm]
\tikzstyle{purple}=[text={rgb,255: red,150; green,0; blue,150}, inner sep=0mm, tikzit fill=white, tikzit draw={rgb,255: red,191; green,191; blue,191}]
\tikzstyle{yellow}=[text={rgb,255: red,120; green,112; blue,0}, tikzit draw={rgb,255: red,120; green,112; blue,0}, tikzit fill=white, inner sep=0mm]
\tikzstyle{orange}=[text={rgb,255: red,120; green,76; blue,0}, tikzit draw={rgb,255: red,120; green,76; blue,0}, tikzit fill=white, inner sep=0mm]
\tikzstyle{red}=[text={rgb,255: red,120; green,36; blue,0}, tikzit draw={rgb,255: red,120; green,36; blue,0}, tikzit fill=white, inner sep=0mm]

\tikzstyle{arrow}=[->]
\tikzstyle{dashed box}=[-, dashed]
\tikzstyle{blue line}=[-, draw={rgb,255: red,15; green,0; blue,150}, tikzit category=blue]
\tikzstyle{red arrow}=[-, draw={rgb,255: red,150; green,0; blue,2}, tikzit category=red]
\tikzstyle{purple line}=[draw={rgb,255: red,120; green,0; blue,120}, >=stealth, shorten <=2pt, shorten >=2pt, -]
\tikzstyle{protected purple line}=[draw={rgb,255: red,120; green,0; blue,120}, >=stealth, shorten <=2pt, shorten >=2pt, preaction={line width=1.8pt, white, draw}, -]
\tikzstyle{mapsto}=[{|->}]
\tikzstyle{double wire}=[-, double]
\tikzstyle{protected}=[-, preaction={line width=1.8pt,white,draw}]
\tikzstyle{leak arrow}=[-, line join=round, decorate, decoration={snake, segment length=4, amplitude=0.75, pre=curveto, post=curveto, pre length=1pt, post length=1pt}]
\tikzstyle{protected leak arrow}=[-, line join=round, decorate, decoration={snake, segment length=4, amplitude=0.75, pre=curveto, post=curveto, pre length=1pt, post length=1pt}, preaction={line width=1.8pt, white, draw}]
\tikzstyle{hollow arrow}=[-, very thin, white, preaction={line width=0.7pt,draw={rgb,255: red,120; green,0; blue,85}}, tikzit category=leak, tikzit draw={rgb,255: red,150; green,0; blue,120}]
\tikzstyle{protected hollow arrow}=[-, very thin, white, preaction={line width=0.7pt,draw={rgb,255: red,120; green,0; blue,85},preaction={line width=2.1pt,white,draw}}, tikzit category=leak, tikzit draw={rgb,255: red,150; green,0; blue,120}]
\tikzstyle{curly brace}=[-, decorate, decoration={brace,amplitude=5pt}]
\tikzstyle{yellow fill}=[-, fill={rgb,255: red,255; green,253; blue,195}, draw=none]
\tikzstyle{orange fill}=[-, draw={rgb,255: red,220; green,185; blue,116}, fill={rgb,255: red,255; green,235; blue,195}]
\tikzstyle{red fill}=[-, draw={rgb,255: red,220; green,163; blue,134}, fill={rgb,255: red,255; green,215; blue,195}]
\tikzstyle{dark red fill}=[-, fill={rgb,255: red,220; green,163; blue,134}, draw=none]

\usepackage{enumitem}
\setlist[enumerate]{label=(\roman*)}  
\setlist[enumerate,2]{label=(\alph*)}  
\usepackage{amssymb,xparse}
\usepackage{authblk}
\usepackage{bm}
\usepackage{centernot}
\usepackage{doi}
\usepackage{nicefrac}
\usepackage[numbers,sort]{natbib}
\usepackage[verbose=true]{microtype}

\allowdisplaybreaks

\definecolor{myurlcolor}{rgb}{0,0,0.3}
\definecolor{mycitecolor}{rgb}{0,0.3,0}
\definecolor{myrefcolor}{rgb}{0.3,0,0}
\usepackage[pagebackref,draft=false]{hyperref}
\hypersetup{colorlinks,
	linkcolor=myrefcolor,
	citecolor=mycitecolor,
urlcolor=myurlcolor}

\usepackage[noabbrev,capitalize]{cleveref}

\definecolor{fillcolor}{rgb}{0.75,0.75,0.75}

\newtheorem{theorem}{Theorem}[subsection]
\newtheorem*{theorem*}{Theorem}
\newtheorem{proposition}[theorem]{Proposition}
\newtheorem{lemma}[theorem]{Lemma}
\newtheorem{corollary}[theorem]{Corollary}
\newtheorem{definition}[theorem]{Definition}
\newtheorem{question}[theorem]{Question}

\theoremstyle{definition}
\newtheorem{example}[theorem]{Example}

\newtheorem{remark}[theorem]{Remark}

\numberwithin{equation}{section}

\makeatletter
\def\cref@thmoptarg[#1]#2#3#4{%
	\ifhmode\unskip\unskip\par\fi%
	\normalfont%
	\trivlist%
	\let\thmheadnl\relax%
	\let\thm@swap\@gobble%
	\thm@notefont{\fontseries\mddefault\upshape}%
	\thm@headpunct{.}%
	\thm@headsep 5\p@ plus\p@ minus\p@\relax%
	\thm@space@setup%
	#2%
	\@topsep \thm@preskip
	\@topsepadd \thm@postskip
	\def\@tempa{#3}\ifx\@empty\@tempa%
		\def\@tempa{\@oparg{\@begintheorem{#4}{}}[]}%
	\else%
		\refstepcounter[#1]{#3}%
		\@namedef{cref@#3@alias}{#1}%
		\def\@tempa{\@oparg{\@begintheorem{#4}{\csname the#3\endcsname}}[]}%
	\fi%
\@tempa}%
\makeatother

\let\originalleft\left
\let\originalright\right
\renewcommand{\left}{\mathopen{}\mathclose\bgroup\originalleft}
\renewcommand{\right}{\aftergroup\egroup\originalright}

\newcommand{\coloniff}{\;\ratio\Longleftrightarrow\;}
\newcommand{\N}{\mathbb{N}}
\newcommand{\Z}{\mathbb{Z}}
\newcommand{\R}{\mathbb{R}}
\newcommand{\eps}{\varepsilon}
\newcommand{\e}{\varepsilon}

\newcommand{\ph}{\mathord{\rule[-0.05em]{0.6em}{0.05em}}}
\newcommand{\sngltn}{\ast}

\newcommand{\cat}[1]{{\mathsf{#1}}} 
\newcommand{\op}{\mathrm{op}}
\newcommand{\set}{\mathsf{Set}}
\newcommand{\finset}{\mathsf{FinSet}}
\newcommand{\Kl}{\mathsf{Kl}}
\newcommand{\id}{\mathrm{id}}
\newcommand{\tensor}{\otimes}
\newcommand{\comp}{%
	\mathchoice{\,}{\,}{}{}
}

\newcommand{\cop}{\mathrm{copy}}
\newcommand{\discard}{\mathrm{del}}
\newcommand{\cC}{\mathsf{C}}
\newcommand{\cD}{\mathsf{D}}
\newcommand{\stat}{\mathsf{Stat}}
\renewcommand{\det}{\mathrm{det}}
\newcommand{\samp}{\mathsf{samp}}
\newcommand{\suppcomp}[1][\cC]{\mathsf{Supp}(#1)}

\newcommand{\acsim}{\approx}

\newcommand{\meas}{\mathsf{Meas}}
\newcommand{\borelmeas}{\mathsf{BorelMeas}}
\newcommand{\topcat}{\mathsf{Top}}
\newcommand{\chaus}{\mathsf{CHaus}}
\newcommand{\stoch}{\mathsf{Stoch}}
\newcommand{\finstoch}{\mathsf{FinStoch}}
\newcommand{\borelstoch}{\mathsf{BorelStoch}}
\newcommand{\topstoch}{\mathsf{TopStoch}}
\newcommand{\tychstoch}{\mathsf{TychStoch}}
\newcommand{\chausstoch}{\mathsf{CHausStoch}}
\newcommand{\finsetmulti}{\mathsf{FinSetMulti}}
\newcommand{\setmulti}{\mathsf{SetMulti}}
\newcommand{\cring}{\mathsf{CRing}_+^\op}
\newcommand{\rel}{\mathsf{Rel}}

\newcommand{\as}[1]{%
	\def\relstate{#1}%
	\ifx\relstate\empty
		\text{a.s.}%
	\else
		{#1\text{-a.s.}}%
	\fi
}
\newcommand{\ase}[1]{=_{#1}}

\newcommand{\suppincop}{\mathsf{si}}	
\newcommand{\suppinc}[1]{%
	\def\relstate{#1}%
	\ifx\relstate\empty
		\mathsf{si}%
	\else
		\suppincop_{#1}%
	\fi
}
\newcommand{\Supp}[1]{%
	\def\relstate{#1}%
	\ifx\relstate\empty
		S%
	\else
		S_{#1}%
	\fi
}
\newcommand{\suppfactor}[1]{\widehat{#1}}
\newcommand{\suppprojop}{\mathsf{sp}}
\newcommand{\suppproj}[1]{%
	\def\relstate{#1}%
	\ifx\relstate\empty
		\mathsf{sp}%
	\else
		\suppprojop_{#1}%
	\fi
}
\newcommand{\detp}[1]{\overline{#1}}
\newcommand{\pbehord}[1]{\mathord{\prescript{#1}{}\gg}}
\newcommand{\pbeh}[1]{\mathbin{\prescript{#1}{}\gg}}
\newcommand{\Atoms}[1]{%
	\def\relstate{#1}%
	\ifx\relstate\empty
		\mathcal{A}%
	\else
		\mathcal{A}_{#1}%
	\fi
}

\providecommand{\given}{}
\newcommand{\SetSymbol}[1][]{%
	\nonscript\;\,#1\vert
	\allowbreak
	\nonscript\;\,
	\mathopen{}
}
\DeclarePairedDelimiterX{\Set}[1]{\{}{\}}{%
	\renewcommand{\given}{\SetSymbol[\delimsize]}
	#1
}
\makeatletter
\let\oldSet\Set
\def\Set{\@ifstar{\oldSet}{\oldSet*}}
\makeatother

\newsavebox{\numbox}%
\newsavebox{\slashbox}%
\newsavebox{\denbox}%
\newlength{\slashlength}%
\newlength{\faktorscale}%
\DeclareDocumentCommand{\newfaktor}{m O{0.35} m O{-0.35}}{%
	\savebox{\numbox}{\ensuremath{#1}}%
	\savebox{\slashbox}{\ensuremath{\diagup}}%
	\savebox{\denbox}{\ensuremath{#3}}%
	\setlength{\faktorscale}{0.5\ht\numbox+0.5\ht\denbox}%
	\setlength{\slashlength}{2pt+0.8\faktorscale+#2\faktorscale-#4\faktorscale}%
	\raisebox{#2\ht\slashbox}{\usebox{\numbox}}%
	\mkern-2mu%
	\rotatebox{-30}{\rule[#4\ht\denbox]{0.4pt}{\slashlength}}%
	\mkern9mu%
	\hspace{-0.44\slashlength}%
	\raisebox{#4\ht\denbox}{\usebox{\denbox}}%
}
\DeclareDocumentCommand{\linefaktor}{m O{0.08} m O{-0.08}}{%
	\savebox{\numbox}{\ensuremath{#1}}%
	\savebox{\slashbox}{\ensuremath{\diagup}}%
	\savebox{\denbox}{\ensuremath{#3}}%
	\setlength{\faktorscale}{0.5\ht\numbox+0.5\ht\denbox}%
	\setlength{\slashlength}{0.2\faktorscale+0.8\baselineskip}%
	\raisebox{#2\ht\slashbox}{\usebox{\numbox}}%
	\mkern-1mu%
	\raisebox{-0.8pt}{%
		\rotatebox{-25}{\rule[#4\ht\denbox]{0.4pt}{\slashlength}}
	}%
	\mkern-1mu%
	\hspace{-0.25\slashlength}%
	\raisebox{#4\ht\denbox}{\usebox{\denbox}}%
}

\newcommand{\newterm}[1]{\textbf{#1}}

\title{Absolute continuity, supports and idempotent splitting\\ in categorical probability}

\author[1]{Tobias Fritz}
\author[1,2]{Tom{\'a}{\v{s}} Gonda}
\author[3]{Antonio Lorenzin}
\author[4]{\\ Paolo Perrone}
\author[5]{Dario Stein}

\affil[1]{Department of Mathematics, University of Innsbruck, Austria}
\affil[2]{Institute for Theoretical Physics, University of Innsbruck, Austria}
\affil[3]{Independent researcher, Trento, Italy}
\affil[4]{Department of Computer Science, University of Oxford, United Kingdom}
\affil[5]{iHub, Radboud University Nijmegen, The Netherlands}

\begin{document}

\maketitle

\begin{abstract}
	Markov categories have recently turned out to be a powerful high-level framework for probability and statistics.
	They accommodate purely categorical definitions of notions like conditional probability and almost sure equality, as well as proofs of fundamental results such as the Hewitt--Savage 0/1 Law, the de Finetti Theorem and the Ergodic Decomposition Theorem.

	In this work, we develop additional relevant notions from probability theory in the setting of Markov categories.
	This comprises improved versions of previously introduced definitions of absolute continuity and supports, as well as a detailed study of idempotents and idempotent splitting in Markov categories.
	Our main result on idempotent splitting is that every idempotent measurable Markov kernel between standard Borel spaces splits across another standard Borel space, and we derive this as an instance of a general categorical criterion for idempotent splitting in Markov categories.
\end{abstract}

\tableofcontents

\section{Introduction}

In recent years, there has been growing interest in categorical methods for probability theory and related fields.
Structures such as \emph{garbage-share monoidal categories}~\cite{gadducci1996} (also called \emph{copy-discard categories}~\cite{chojacobs2019strings}) and \emph{Markov categories}~\cite{fritz2019synthetic} have appeared in contexts as diverse as 
mathematical statistics~\cite{fritz2023representable},
graphical models~\cite{fritz2022dseparation},
ergodic theory~\cite{moss2022ergodic},
machine learning~\cite{zhang}, and
information theory~\cite{ours_entropy}.

The common theme in all these works is the aim to distil the most conceptual aspects of reasoning in the face of uncertainty, such as stochastic independence or the composition of Markov kernels, and to prove theorems in probability theory in terms of these prime principles.
The hope is to achieve both greater conceptual clarity and generality.
There are several aspects of probability theory that lend themselves very well to this point of view.
The idea of forming \emph{conditional distributions} and \emph{Bayesian inverses} can be expressed categorically and even graphically~\cite{chojacobs2019strings}, and proofs in probability theory involving conditioning can be simplified by resorting to this abstract definition.
This facilitated for example the categorical proofs of de Finetti's theorem~\cite{fritz2021definetti} and the Aldous--Hoover theorem~\cite{chen2025aldoushoover}, as well as a generalization of Blackwell's theorem on the comparison of statistical experiments to the case of nondiscrete random variables~\cite{fritz2023representable}.
Similarly, a categorical notion of \emph{determinism} has allowed us to formulate and prove certain \emph{zero-one laws}~\cite{fritzrischel2019zeroone}, and can be seen to be the main idea behind the concept of \emph{ergodicity}~\cite{moss2022ergodic}. 
Moreover, measuring the deviation from determinism can lead to the notion of \emph{entropy}~\cite{ours_entropy}.
Last but not least, the aforementioned idea of stochastic dependence and independence can be refined to a rich theory of conditional independence and Markov separation properties that reduces stochastic dependencies to topological path-connectedness in a graph~\cite{fritz2022dseparation}.

\paragraph{Supports.}
One of the conceptual aspects of probabilistic processes that we address in this work is the notion of \emph{support}.
Roughly speaking, one could define the support of a probability measure as the smallest set of probability 1, if it exists.
For example, the probability distribution $p$ on the set $\{a,b,c\}$ given by
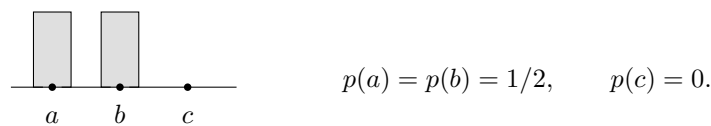
\begin{equation*}
	\begin{tikzpicture}[baseline,
		x={(0:1cm)},y={(90:1cm)},z={(0:1cm)}]
		\node (a) at (0.55,0,-2) {};
		\draw [fill=fillcolor!50] (a) -- ++(0.25,0,0) -- ++(0,1,0) -- ++(-0.5,0,0) -- ++(0,-1,0) -- (a);
		\node [circle,inner sep=1pt,fill=black,label=below:$a\strut$] at (a) {};
		\node (b) at (1.45,0,-2) {};
		\draw [fill=fillcolor!50] (b) -- ++(0.25,0,0) -- ++(0,1,0) -- ++(-0.5,0,0) -- ++(0,-1,0) -- (b);
		\node [circle,inner sep=1pt,fill=black,label=below:$b\strut$] at (b) {};
		\node (c) at (2.35,0,-2) {};
		\node [circle,inner sep=1pt,fill=black,label=below:$c\strut$] at (c) {};
		\draw (0,0,-2) -- (3,0,-2) ;
	\end{tikzpicture}
	\qquad\qquad
	p(a) = p(b) = 1/2, \qquad p(c) = 0 .
\end{equation*}
has as support the subset $\{a,b\}$.
For distributions on discrete spaces, this definition is appropriate, and the support consists exactly of the points with non-zero probability.

In the case of continuous sample spaces, points may have measure zero while still being in the support{\,\textemdash\,}this is the case, for example, for the normal distribution:
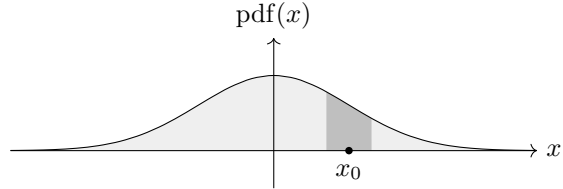
\[
	\begin{tikzpicture}
		\fill [fillcolor!25, domain=-3.5:3.5, variable=\x]
		(-3.5, 0)
		-- plot ({\x}, {exp(-(\x*\x)/2)})
		-- (3.5, 0)
		-- cycle;
		\fill [fillcolor, domain=0.7:1.3, variable=\x]
		(0.7, 0)
		-- plot ({\x}, {exp(-(\x*\x)/2)})
		-- (1.3, 0)
		-- cycle;
		\draw[->] (-3.5, 0) -- (3.5, 0) node[right] {$x$};
		\draw[->] (0, -0.5) -- (0, 1.5) node[above] {$\mathrm{pdf}(x)$};
		\draw[domain=-3.5:3.5, smooth, variable=\x] plot ({\x}, {exp(-(\x*\x)/2)});
		\node [circle,inner sep=1pt,fill=black,label=below:$x_0$] at (1,0) {};
	\end{tikzpicture}
\]
In this case, one can say that a point $x_0\in X$ lies in the support of a probability measure if every open neighborhood $U$ of $x_0$ has non-zero measure.
This is equivalent to saying that the support is the smallest \emph{closed} subset of full measure.
Such an approach, which we call the \emph{topological support} (\cref{ex:top_support}), works well for ($\tau$-smooth) probability measures on topological spaces~\cite{fritz2019support}.

On a general measurable space, there are no closed sets, and so this approach cannot work. 
Moreover, even in those measurable spaces which do come from a topological space, such as real numbers, using a topological notion in a measure-theoretic context is not going to be invariant under isomorphism. 
For example, we can equip the unit interval with the following two probability measures:
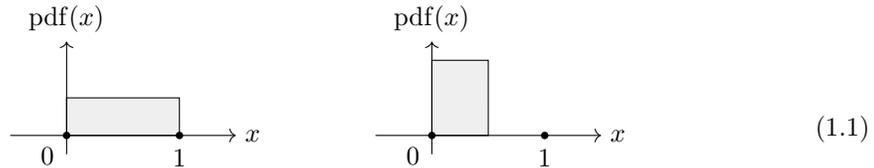
\begin{equation}
	\label{eq:two_measures}
	\begin{tikzpicture}[xscale=1.5,yscale=0.5,baseline]
		\draw [fill=fillcolor!25, domain=0:1, variable=\x]
		(0, 0)
		-- plot ({\x}, {1})
		-- (1, 0)
		-- cycle;
		\draw[->] (-0.5, 0) -- (1.5, 0) node[right] {$x$};
		\draw[->] (0, -0.5) -- (0, 2.5) node[above] {$\mathrm{pdf}(x)$};
		\node [circle,inner sep=1pt,fill=black,label=below left:$0$] at (0,0) {};
		\node [circle,inner sep=1pt,fill=black,label=below:$1$] at (1,0) {};
	\end{tikzpicture}
	\qquad\qquad
	\begin{tikzpicture}[xscale=1.5,yscale=0.5,baseline]
		\draw [fill=fillcolor!25, domain=0:0.5, variable=\x]
		(0, 0)
		-- plot ({\x}, {2})
		-- (0.5, 0)
		-- cycle;
		\draw[->] (-0.5, 0) -- (1.5, 0) node[right] {$x$};
		\draw[->] (0, -0.5) -- (0, 2.5) node[above] {$\mathrm{pdf}(x)$};
		\node [circle,inner sep=1pt,fill=black,label=below left:$0$] at (0,0) {};
		\node [circle,inner sep=1pt,fill=black,label=below:$1$] at (1,0) {};
	\end{tikzpicture}
\end{equation}
Topologically, it is clear that the first measure has full support, and the second measure is supported on $[0, \linefaktor{1}{2}]$.
However, as measure spaces, the two are isomorphic, as we show in \cref{thm:not_support}.
In particular, no invariant of measure spaces can tell the two situations apart.
Intuitively, the reason is that measurable functions alone cannot tell apart a region of zero measure which is outside the (topological) support, such as the half-interval $( \linefaktor{1}{2} ,1]$ on the right-hand side, from a region within the topological support, but which does not fill the space enough to have positive measure (see \Cref{not_support} for details). 
On this note, we may wonder whether such probability measures should even admit a support. 
In fact, the existence of a support is deeply connected to the notion of atomicity (\cref{cor:supp_iff_atomic,ex:lebesgue_nosupp}). 
In particular, non-atomic measures like the two described in \eqref{eq:two_measures} will never admit a support.

In this work, we provide a categorical perspective on supports.
Both the topological and measure-theoretic perspectives on supports are reflected in the general categorical notion.
It is not obvious how to define a notion of support in categorical terms, since abstracting away from the concrete category of measurable spaces and Markov kernels deprives one of the possibility to talk about sets of measure zero.
The way we deal with such situations abstractly is via the categorical notion of \emph{almost sure equality} \cite[Section~13]{fritz2019synthetic}, which we discuss in more detail in \cref{sec:as} of the present paper.
This concept reduces to the usual almost sure equality of functions and kernels when instantiated in the category $\borelstoch$ of standard Borel spaces and Markov kernels, and can be used conveniently in statements and proofs. 
For example, conditionals and Bayesian inverses, when they exist, are unique up to almost sure equality in any Markov category (\cref{lem:cond_unique}), and this is an immediate consequence of the definition.

What could then be a categorical definition of support?
In discrete probability theory, two functions on a probability space are almost surely equal if and only if they agree on the support of the probability measure. 
One can turn this idea into a categorical definition of support via the universal property of \emph{representing} almost sure equality in the sense of representable functors. 
This was suggested in \cite[Definition~13.20]{fritz2019synthetic}, and then used in \cite{dependent_bayesian_lenses} and \cite{braithwaite2023inference}.
Similar ideas were given in \cite{di2022monoidal} as well.

Here we expand on that idea, but we also give an alternative and weaker definition of support. 
The notion from the previous paragraph{\,\textemdash\,}which we rename to \emph{split support} in this paper{\,\textemdash\,}is a ``mapping-out'' universal property, i.e.\ it is about representing a \emph{covariant} functor to $\set$.
We can also define supports via a ``mapping-in'' universal property, representing a \emph{contravariant} functor to $\set$.
The latter notion is called a \emph{support} in the present article (\cref{def:support}), and we show that it is strictly weaker than that of a split support (\cref{prop:support_not_split}).
Concrete instantiations of the resulting theory are the following:
\begin{itemize}
	\item In the category $\finstoch$ of finite sets and stochastic matrices, all supports exist (\cref{ex:support_finstoch}), and the support of every morphism with non-empty domain is split (\cref{ex:non-empty_supp_split_fin}).
	\item In the category $\tychstoch$ of Tychonoff topological spaces and continuous Markov kernels from \cref{ex:continuous_kernels}, all supports exist and they correspond exactly to the traditional topological supports (\cref{ex:top_support}).
		Note that Tychonoff topological spaces include all metric spaces and all compact Hausdorff spaces.
		However, these supports do not split in general (\cref{prop:support_not_split}).

	\item For measure-theoretic probability theory, i.e.\ in $\borelstoch$, a probability measure has a support if and only if it is atomic, and similarly for Markov kernels (\cref{thm:borelstoch_supports}).
		In particular, the Lebesgue measure on the unit interval has no support in our sense. 
\end{itemize}
The second item in particular shows that our newly developed mapping-in universal property of support seems to be the right notion to use in general.
The third item is what one should expect in light of isomorphisms like the one of~\eqref{eq:two_measures} above.

Another weaker property that has a similar flavour is a condition that we call the \emph{equalizer principle} (\cref{sec:equalizer_principle}).
It states that if two deterministic morphisms agree $p$-almost surely, then $p$ factors across their equalizer. 
If supports exist, this simply means that the support of $p$ is included in the equalizer, but it makes sense and is useful even when $p$ has no support.
This property always holds in $\borelstoch$, and it is crucial for our results about idempotents in \cref{sec:idempotents}.

\paragraph{Absolute continuity.}
The mapping-in universal property of supports discussed above is based on the idea that a Markov kernel lands within the support of a measure if and only if it is absolutely continuous with respect to that measure. 
In light of this conceptualization, \cref{sec:abs_cont} is devoted to the notion of \emph{absolute continuity} (\cref{def:ac}). 
This definition improves upon our earlier definition given in~\cite[Definition~2.8]{fritz2023representable} by allowing for extra input wires. 
As we show, in the general categorical setting this new definition ensures the compatibility with tensor products (\cref{prop:ac_tensor}), which the previous definition does not (\cref{rem:old_not_ac_tensor}).
Moreover, this abstract absolute continuity reduces to the usual measure-theoretic notion when instantiated in $\stoch$ (\cref{ex:abs_cont_stoch}), and it corresponds to the inclusion of supports whenever they exist.

\paragraph{Idempotents.}

The other main theme of this paper is that of \emph{idempotents}.
While \cref{sec:supports} on supports makes use of concepts and results developed in \cref{sec:abs_cont}, \cref{sec:idempotents} on idempotents can be read independently of the other two for the most part.

To say that an endomorphism $e \colon X \to X$ is an idempotent means that it satisfies $e \comp e = e$.
Idempotent Markov kernels secretly come up in several technical aspects of probability theory and ergodic theory.
For example, conditional expectations can be thought of as idempotent kernels acting on bounded measurable functions by precomposition.
In the study of ergodic theorems, forming the long-term average of a process produces an idempotent kernel~\cite[p.~560]{blackwell1942idempotent}.

Given an idempotent $e \colon X \to X$, it is typically of interest to know whether $e$ \emph{splits}~\cite{borceux_idempotents}.
Such a splitting is given by another object $T$ together with morphisms
\[
	\iota \colon T \hookrightarrow X, \qquad \pi \colon X \twoheadrightarrow T
\]
satisfying
\[
	\pi \comp \iota = \id_T, \qquad
	\iota \comp \pi = e.
\]
When this is the case, we may think of $e$ as projecting via $\pi$ onto the subobject given \mbox{by $\iota$}.
Conversely, whenever we have morphisms $\iota$ and $\pi$ such that $\pi \comp \iota = \id_T$, then the other composition $\iota \comp \pi$ is an idempotent, which we think of as the projection onto $T$.
For example, split idempotents in this form arise from the split supports mentioned above: 
A split support is a subobject onto which one can project.

As it turns out, there are different ways in which an idempotent can interact with the Markov category structure (see \cref{fig:idempotents_venn} in \cref{sec:idempotents}), and this is related to the problem of splitting.
A particularly important role is played by \emph{balanced idempotents}, which are defined as satisfying a form of detailed balance (\cref{rem:detailed_balance}).
About these, we prove that:
\begin{itemize}
	\item Every split idempotent in a positive Markov category, such as $\borelstoch$, is balanced (\cref{sec:split}).
	\item Under suitable extra conditions on a positive Markov category, every idempotent is balanced and splits (\cref{sec:route_split}).
	\item The free completion under balanced idempotents of a Markov category is again a Markov category (\cref{sec:blackwell_envelope}).
\end{itemize}
The extra conditions for the second item are the aforementioned equalizer principle, and the fact that the Markov category is \emph{observationally representable}.
``Representable'' roughly means that we are in the Kleisli category of a probability monad (see \cite{fritz2023representable} for more details).
``Observational'' means that morphisms into a space of distributions $PX$ can be distinguished by taking repeated independent samples, as it is in ordinary probability (and as it reflects common practice in statistics and in all of science).
The mathematical details of this idea are outlined in \cref{sec:observational}.
Since $\borelstoch$ is positive and satisfies these extra conditions, we obtain our main result about measure-theoretic probability:

\begin{theorem*}
	Every idempotent Markov kernel between standard Borel spaces is balanced and splits.
\end{theorem*}

We prove this as \cref{cor:borelstoch_idempotents_split} and note that it strengthens a classical result of Blackwell on idempotent Markov kernels (\cref{thm:idemp_borel}).
As far as we know, this strengthening is new, and in particular no measure-theoretic proof exists to date.
We also note the analogous result that idempotents in the category of compact Hausdorff spaces and continuous Markov kernels are balanced and split (\cref{cor:chausstoch_idempotents_split}).

\subsection*{Outline}

We start in \cref{sec:abs_cont} with an extended notion of almost sure equality (\cref{def:as_eq}) that appears subsequently in our new definition of absolute continuity between morphisms in a Markov category (\cref{def:ac}). 
We then show that tensor product and post-composition respect absolute continuity (\cref{lem:ac_mon_comp} and \cref{prop:ac_tensor}). 

In \cref{sec:supports}, we define the support of a morphism $p$ via a mapping-in universal property as the universal object classifying morphisms absolutely continuous with respect to $p$. 
We study supports in this sense in some detail, showing in particular that this recovers the traditional notion in $\finstoch$ (\cref{ex:support_finstoch}), as well as in categories of continuous Markov kernels between topological spaces (\cref{ex:top_support}). 
We also show that a Markov category with all supports is causal if and only if its supports assemble into a functor (\cref{thm:supp_functorial}). 

In \cref{sec:support_completeness}, we establish which morphisms stand a chance to have a support at all (\cref{lem:suppdominance}).
We call morphisms satisfying this condition \emph{atomic} (\cref{def:atomic}).
For $\borelstoch$, we show that a Markov kernel is atomic in this sense if and only if it is atomic in the standard measure-theoretic sense.
However, not even all atomic morphisms have a support, as \cref{thm:borelstoch_supports} proves.
We also show how to extend every causal Markov category to a \emph{free regular support completion}{\,\textemdash\,}another causal Markov category in which every atomic morphism has a support (\cref{thm:supports_suppcompletion}).

We give a first application of supports in \cref{sec:rel}.
Under causality, supports and an additional assumption, we construct a functor from a Markov category to $\setmulti$, the category of sets and multivalued functions in \cref{sec:rel}.
Intuitively, the functor assigns to every morphism its ``possibilistic shadow'' in the form of the \emph{input-output relation}.
In $\finstoch$, every object is mapped to the set of deterministic states on it, and every stochastic matrix is sent to the relation of all input-output pairs with non-zero probability (\cref{ex:finstoch_relfunctor}).

As another application of supports, we show in \cref{sec:statistical_models} how they allow one to prove an equivalence between \emph{parametric} statistical models and \emph{non-parametric} statistical models. 
In particular, we define a category in which the objects, thought of as parametric statistical models, are morphisms in a suitable Markov category (\cref{def:statC}).
We then show an abstract version of the statement that the behavior of a statistical model $p$ can be identified with that of the collection of probability distributions that $p$ parametrizes, provided that $p$ has a support (\cref{prop:parametric_vs_unparametric}).

In \cref{sec:equalizer_principle}, we argue that even in the absence of supports, the weaker \emph{equalizer principle} can be used as a substitute in certain situations.
For example, it allows one to talk about the largest subobject on which two deterministic morphisms agree.
This principle holds in Markov categories with supports (\cref{rem:equalizer_principle}), but also in categories such as $\borelstoch$ (\cref{prop:borelstoch_eq}) in which generic morphisms do not have a support.
Equalizer principle also plays an important role in our study of idempotents.

We finally turn to \emph{split supports} in \cref{sec:split_supports}, which extends the notion of supports originally proposed in~\cite[Definition~13.20]{fritz2019synthetic}.
We introduce them as pre-supports for which the support inclusion has a left inverse (\cref{def:split_support}) and then describe their universal property in \cref{prop:split_support_universal}.
The support of every stochastic matrix is in fact a split support (\cref{ex:non-empty_supp_split_fin}), and likewise all supports that exist in $\borelstoch$ are split (\cref{ex:non-empty_supp_split_borel}).

We then turn to idempotents in \cref{sec:idempotents}.
First, we introduce several types of interaction between idempotents and the structure of Markov categories (\cref{def:idempotents}).
These give rise to \emph{static}, \emph{strong} and \emph{balanced} idempotents respectively.
We give examples of each kind and provide several reformulations of their definitions (\cref{prop:balancednew,lem:balanced_as_det}).
A simple but important observation is that in Markov categories that are \emph{balanced} (\cref{def:balanced_cat}), every idempotent is balanced (\cref{thm:balanced_idempotent}).
This holds in particular for any of the categories describing discrete or measure-theoretic probability (\cref{cor:balanced_cat}).

In \cref{sec:split}, we study split idempotents and give an illuminating characterization of the above types of idempotents in terms of their splittings (\cref{thm:split_props}).
We also present a number of examples of split idempotents.

The interplay between the two main themes of the paper, supports and idempotents, is explored in \cref{sec:proj_idempotent_support}.
In particular, we prove that split supports are equivalently splittings of certain static idempotents (\cref{thm:proj_idemp_supp}).

With the goal of showing that all idempotents of $\borelstoch$ split, we study sufficient conditions for the splitting of balanced idempotents in \cref{sec:route_split}.
Our main result is \cref{thm:balanced_split} that establishes conditions under which all (balanced) idempotents in a Markov category split.
It follows that every idempotent in $\borelstoch$ splits (\cref{cor:borelstoch_idempotents_split}).
This strengthens an existing result of Blackwell (\cref{thm:idemp_borel}).

Finally, in \cref{sec:blackwell_envelope}, we define the \emph{Blackwell envelope} as the balanced idempotent completion that mimics the Karoubi envelope from ordinary category theory.
We show that the Blackwell envelope of a Markov category is also a Markov category (\cref{prop:blackwellenv_markov}), while the Karoubi envelope need not be.

In \cref{appendix}, we provide relevant background material on Markov categories used in the rest of this work.
We start with basic definitions and examples in \cref{sec:basics}. After that, the sections \cref{sec:preliminaries_aseq} and \cref{sec:observational} present novel material.
In the former, we derive basic results on almost sure equality, which are used in \cref{sec:abs_cont}.
In the latter, we give some results on \emph{observationally representable} Markov categories.
In particular, every observationally representable Markov category is shown to be \as{}-compatibly representable (\cref{prop:reflect_aseq}).

In \cref{not_support}, we give an example of an isomorphism of measure spaces which does not preserve topological supports, making mathematically precise some of the ideas given in the introduction.

\subsection*{Acknowledgments}

We thank an anonymous referee for their careful reading and copious helpful suggestions,
Tommaso Russo for further suggestions and references on the proof of \cref{prop:support_not_split},
and Sean Moss for the pointers on observationality.
TF and AL acknowledge support from the Austrian Science Fund (FWF P 35992-N).
PP acknowledges support from the European Research Council (grant ``BLaSt -- A Better Language for Statistics'').
TG acknowledges support from the Austrian Science Fund (FWF) via the START Prize Y1261-N.
TF would also like to thank the Isaac Newton Institute for Mathematical Sciences
for the support and hospitality during the programme \emph{Causal inference: From theory to practice and back again} when work on this paper was undertaken. This work was supported by: EPSRC Grant Number EP/Z000580/1.

\section{Absolute continuity}
\label{sec:abs_cont}

The goal of this short section is to introduce and study absolute continuity in an improved version and in more detail than was done in our earlier work~\cite{fritz2023representable}, including the derivation of some properties relevant for later in the paper.
Throughout, all our considerations take place in a Markov category $\cC$ (for the basics of the theory, see \cref{sec:basics} and the references therein).

\subsection{An extended definition of almost sure equality}\label{sec:as}

In the context of Markov categories, almost sure equality has been defined in \cite[Definition 13.1]{fritz2019synthetic}.
In our improved notion of absolute continuity coming up as \cref{def:ac}, we use the following notion of almost sure equality extended to morphisms with an extra input $W$.

\begin{definition}\label{def:as_eq}
	Let $p\colon A \to X$ and $f, g\colon W \otimes X \to Y$. 
	We say that $f$ and $g$ are \newterm{$\as{\bm{p}}$ equal}, and write $f \ase{p} g$, if we have
	\begin{equation}\label{eq:as_eq_new}
		{%
			\tikzstyle{every picture}=[tikzfig]%
			\begin{tikzpicture}
				\begin{pgfonlayer}{nodelayer}
					\node [style=none] (251) at (-3.5, 0.75) {};
					\node [style=none] (252) at (-2, 1) {};
					\node [style=none] (253) at (-4, 2.25) {};
					\node [style=none] (254) at (-4, 2.75) {$Y$};
					\node [style=bn] (255) at (-2.75, 0) {};
					\node [style=none] (256) at (-2.75, -2.25) {};
					\node [style=none] (257) at (-2.75, -2.75) {$A$};
					\node [style=none] (258) at (0, 0) {$=$};
					\node [style=none] (259) at (-2, 2.75) {$X$};
					\node [style=morphism] (260) at (-4, 1.25) {$\,\;\; f \,\;\;$};
					\node [style=none] (261) at (-2, 2.25) {};
					\node [style=morphism] (262) at (-2.75, -1.25) {$p$};
					\node [style=none] (263) at (-4.5, -2.25) {};
					\node [style=none] (264) at (-4.5, -2.75) {$W$};
					\node [style=none] (265) at (-4.5, 1.25) {};
					\node [style=none] (266) at (-4.5, -0.75) {};
					\node [style=none] (267) at (3, 0.75) {};
					\node [style=none] (268) at (4.5, 1) {};
					\node [style=none] (269) at (2.5, 2.25) {};
					\node [style=none] (270) at (2.5, 2.75) {$Y$};
					\node [style=bn] (271) at (3.75, 0) {};
					\node [style=none] (272) at (3.75, -2.25) {};
					\node [style=none] (273) at (3.75, -2.75) {$A$};
					\node [style=none] (274) at (4.5, 2.75) {$X$};
					\node [style=morphism] (275) at (2.5, 1.25) {$\,\;\; g \,\;\;$};
					\node [style=none] (276) at (4.5, 2.25) {};
					\node [style=morphism] (277) at (3.75, -1.25) {$p$};
					\node [style=none] (278) at (2, -2.25) {};
					\node [style=none] (279) at (2, -2.75) {$W$};
					\node [style=none] (280) at (2, 1.25) {};
					\node [style=none] (281) at (2, -0.75) {};
					\node [style=none] (284) at (-3.5, 1.25) {};
					\node [style=none] (285) at (3, 1.25) {};
				\end{pgfonlayer}
				\begin{pgfonlayer}{edgelayer}
					\draw [in=-90, out=165] (255) to (251.center);
					\draw [in=-90, out=15] (255) to (252.center);
					\draw (260) to (253.center);
					\draw (252.center) to (261.center);
					\draw (256.center) to (262);
					\draw [style=protected] (263.center) to (266.center);
					\draw [style=protected, in=-90, out=90, looseness=1.25] (266.center) to (265.center);
					\draw [in=270, out=90] (262) to (255);
					\draw [in=-90, out=165] (271) to (267.center);
					\draw [in=-90, out=15] (271) to (268.center);
					\draw (275) to (269.center);
					\draw (268.center) to (276.center);
					\draw (272.center) to (277);
					\draw [style=protected] (278.center) to (281.center);
					\draw [style=protected, in=-90, out=90, looseness=1.25] (281.center) to (280.center);
					\draw (277) to (271);
					\draw (251.center) to (284.center);
					\draw (267.center) to (285.center);
				\end{pgfonlayer}
			\end{tikzpicture}
		}%
	\end{equation}
\end{definition}

\begin{remark}
	The original definition of almost sure equality~\cite[Definition 13.1]{fritz2019synthetic} differs from this one by not having the extra input $W$.
	Taking $W = I$ (the monoidal unit) in \cref{def:as_eq} recovers this original definition modulo precomposition with the unitor $I \otimes X \cong X$.\footnotemark
	\footnotetext{
		Our extended notion of almost sure equality can also be interpreted in terms of parametric Markov categories \cite[Section 2.2]{fritz2023representable}: 
		Condition \eqref{eq:as_eq_new} is equivalent to the original definition of almost sure equality  interpreted in the parametric Markov category $\cC_{W}$. 
	}
	Throughout this paper, we will make a small abuse of notation by leaving this unitor implicit and use the same notation $\ase{p}$ for both notions of almost sure equality.
\end{remark}

With this in mind, our new extended notion of almost sure equality can be interpreted in terms of the original one without the extra wire upon replacing $p$ by $\id_W \otimes p$:

\begin{lemma}\label{lem:as_eq_new}
	Let $p \colon A \to X$ and $f,g\colon W \otimes X \to Y$. 
	Then we have
	\begin{equation}
		f \ase{p} g  \qquad \iff \qquad  f  \ase{\id_W \otimes p } g	
	\end{equation}
\end{lemma}
Because of this, $f \ase{p} g$ can also be read as a short-hand notation for $f \ase{\id_W \otimes\, p} g$.
When the object $W$ is clear from context, we will not mention it explicitly.

\begin{proof}
	In string diagrams, the equality $f  \ase{\id_W \otimes p } g$ looks like
	\begin{equation}
		{%
			\tikzstyle{every picture}=[tikzfig]%
			\begin{tikzpicture}
				\begin{pgfonlayer}{nodelayer}
					\node [style=none] (112) at (0, 0) {$=$};
					\node [style=none] (160) at (-4.5, 1) {};
					\node [style=none] (161) at (-2.5, 1) {};
					\node [style=none] (162) at (-4, 2.25) {};
					\node [style=none] (163) at (-4, 2.75) {$Y$};
					\node [style=bn] (164) at (-3.5, -0.25) {};
					\node [style=none] (165) at (-3.5, -2.25) {};
					\node [style=none] (166) at (-2.5, -2.75) {$A$};
					\node [style=none] (167) at (-1.5, 2.75) {$X$};
					\node [style=morphism] (168) at (-4, 1.25) {$\;\; f \;\;$};
					\node [style=none] (169) at (-2.5, 2.25) {};
					\node [style=morphism] (170) at (-2.5, -1.25) {$p$};
					\node [style=bn] (171) at (-2.5, -0.25) {};
					\node [style=none] (172) at (-2.5, -2.25) {};
					\node [style=none] (173) at (-3.5, -2.75) {$W$};
					\node [style=none] (174) at (-1.5, 2.25) {};
					\node [style=none] (175) at (-2.5, 2.75) {$W$};
					\node [style=none] (176) at (-3.5, 1) {};
					\node [style=none] (177) at (-1.5, 1) {};
					\node [style=none] (178) at (1.5, 1) {};
					\node [style=none] (179) at (3.5, 1) {};
					\node [style=none] (180) at (2, 2.25) {};
					\node [style=none] (181) at (2, 2.75) {$Y$};
					\node [style=bn] (182) at (2.5, -0.25) {};
					\node [style=none] (183) at (2.5, -2.25) {};
					\node [style=none] (184) at (3.5, -2.75) {$A$};
					\node [style=none] (185) at (4.5, 2.75) {$X$};
					\node [style=morphism] (186) at (2, 1.25) {$\;\; g \;\;$};
					\node [style=none] (187) at (3.5, 2.25) {};
					\node [style=morphism] (188) at (3.5, -1.25) {$p$};
					\node [style=bn] (189) at (3.5, -0.25) {};
					\node [style=none] (190) at (3.5, -2.25) {};
					\node [style=none] (191) at (2.5, -2.75) {$W$};
					\node [style=none] (192) at (4.5, 2.25) {};
					\node [style=none] (193) at (3.5, 2.75) {$W$};
					\node [style=none] (194) at (2.5, 1) {};
					\node [style=none] (195) at (4.5, 1) {};
				\end{pgfonlayer}
				\begin{pgfonlayer}{edgelayer}
					\draw [in=-90, out=165] (164) to (160.center);
					\draw [in=-90, out=15] (164) to (161.center);
					\draw (168) to (162.center);
					\draw (161.center) to (169.center);
					\draw (174.center) to (177.center);
					\draw [in=15, out=-90] (177.center) to (171);
					\draw [style=protected, in=-90, out=165] (171) to (176.center);
					\draw [in=-90, out=165] (182) to (178.center);
					\draw [in=-90, out=15] (182) to (179.center);
					\draw (186) to (180.center);
					\draw (179.center) to (187.center);
					\draw (192.center) to (195.center);
					\draw [in=15, out=-90] (195.center) to (189);
					\draw [style=protected, in=-90, out=165] (189) to (194.center);
					\draw (172.center) to (170);
					\draw (165.center) to (164);
					\draw (170) to (171);
					\draw (190.center) to (188);
					\draw (188) to (189);
					\draw (183.center) to (182);
				\end{pgfonlayer}
			\end{tikzpicture}
		}%
	\end{equation}
	By marginalizing the middle output $W$, we conclude $f \ase{p} g$ in the sense of \cref{eq:as_eq_new}.
	For the converse, we merely copy the $W$ input in \cref{eq:as_eq_new} (i.e.\ pre-compose the equation with $\cop_W$) 
	and use standard string diagram manipulation
	to get $f  \ase{\id_W \otimes p } g$. 
\end{proof}

\begin{example}
	\label{ex:ase_cartesian}
	Suppose that $\cC$ is a cartesian monoidal category, or equivalently a Markov category in which all morphisms are deterministic.
	Then the universal property of the product $X \times Y$ shows that $f \ase{p} g$ is equivalent to $f \comp (\id_W \times p) = g \comp (\id_W \times p)$.	
\end{example}

\begin{example}
	\label{ex:ase_stoch}
	For the category $\stoch$, we have that Markov kernels $f,g \colon W\otimes X\to Y$ are $p$-almost surely equal for a probability measure $p \colon I \to X$ if and only if for every $w\in W$ and for $p$-almost all $x\in X$ we have $f(\ph|w,x)=g(\ph|w,x)$ (as measures on $Y$). 
	In other words, we recover standard almost sure equality in $X$ for every value of $W$.
\end{example}

The next result establishes the link between almost sure equality in Markov categories of topological spaces and the topological notion of support.\footnotemark{}
\footnotetext{See \cite[Section 5]{fritz2019support} for a treatment of the support as a morphism of monads.}%
It can serve as a motivation for several of the upcoming constructions in this manuscript.
As detailed in \cref{ex:continuous_kernels}, we consider $\topstoch$, the category of topological spaces and continuous $\tau$-smooth Markov kernels.
A morphism $p \colon A \to X$ in $\topstoch$ has a \newterm{support} $\Supp{p} \subseteq X$ defined as
\begin{equation}\label{eq:top_sup}
	\Supp{p} \coloneqq \bigcap \, \Set{ C \subseteq X \textrm{ closed} \given p(C|a) = 1 \textrm{ for all } a \in A }.
\end{equation}
The $\tau$-smoothness implies $p(\Supp{p}|a) = 1$ for all $a \in A$, and therefore $\Supp{p}$ can also be characterized as the smallest closed set of full measure under all $p(\ph|a)$.\footnotemark{}
\footnotetext{See e.g.\ \cite[Proposition 5.2]{fritz2019support}, where this is spelled out for $A = I$, i.e.\ for probability measures without an additional parameter.}%

In the following result, we restrict to Tychonoff spaces, which plays an important role in the proof.

\begin{proposition}
	\label{prop:ase_top}
	In the category $\tychstoch$ of Tychonoff spaces and continuous Markov kernels (see \cref{ex:continuous_kernels} for more details), almost sure equality is equality on the support: 
	For $p \colon A \to X$ and $f,g \colon W \otimes X \to Y$, we have
	\begin{equation}
		\label{eq:ase_top}
		f \ase{p} g \qquad \iff \qquad f(\ph|w,x) = g(\ph|w,x) \;\; \forall w \in W, \, x \in \Supp{p},
	\end{equation}
	where $\Supp{p}$ is given by \cref{eq:top_sup}.
\end{proposition}

\begin{proof}
	Two $\tau$-smooth probability measures on a topological space are equal if and only if they coincide on a basis of opens~\cite[Proposition~3.8]{fritz2019support}.
	In particular, we can test equality of morphisms on a product space by checking equality on all product opens.
	Therefore, the almost sure equality $f \ase{p} g$ can be explicitly written out as
	\begin{equation}
		\label{eq:ase_top_explicit}
		\int_{x \in U} f(V|w,x) \, p(dx|a) = \int_{x \in U} g(V|w,x) \, p(dx|a)
	\end{equation}
	for all $a \in A$ and $w \in W$ and all opens $U \subseteq X$ and $V \subseteq Y$.
	Suppose now that equality on the support in the sense of \cref{eq:ase_top} holds.
	Then also \cref{eq:ase_top_explicit} holds, since by $p(X \setminus \Supp{p} | a) = 0$ the integrands coincide almost everywhere.

	The converse direction is more difficult.
	Assuming \cref{eq:ase_top_explicit}, we want to show that $f(\ph|w,x) = g(\ph|w,x)$ for all $w \in W$ and $x \in \Supp{p}$.
	Since $w$ can be fixed throughout, let us suppress it from the notation to improve readability.
	Then since $Y$ is Tychonoff, the two probability measures $f(\ph|x)$ and $g(\ph|x)$ on $Y$ are equal if and only if the expectation values of all bounded continuous functions coincide (see for example~\cite[Proposition~4.7]{fritz2019support}).
	Therefore it is enough to show that for every $x \in \Supp{p}$ and every bounded continuous $h \colon Y \to \R$, we have
	\begin{equation}
		\label{eq:ase_top_cont}
		\int_{y \in Y} h(y) \, f(dy|x) = \int_{y \in Y} h(y) \, g(dy|x).
	\end{equation}
	What we know from \cref{eq:ase_top_explicit} is that
	\[
		\int_{x \in U} \int_{y \in Y} h(y) \, f(dy|x) \, p(dx|a) = \int_{x \in U} \int_{y \in Y} h(y) \, g(dy|x) \, p(dx|a),
	\]
	where we have switched the order of integration by Fubini's theorem.
	Since this holds for arbitrary open $U \subseteq X$, it likewise holds for every Borel set in place of $U$.
	The crucial point is now that the difference of the two integrands,
	\[
		D(x) \coloneqq \int_{y \in Y} h(y) \, f(dy|x) - \int_{y \in Y} h(y) \, g(dy|x),
	\]
	is a continuous function of $x$, since integrating against $h$ is itself a continuous function on the space of probability measures~\cite[Proposition~4.7]{fritz2019support}.
	For every $a \in A$, applying the previous integral equality to the open sets $D^{-1}((0,\infty))$ and $D^{-1}((-\infty,0))$ shows that both of these have $p(\ph|a)$-measure zero.
	Thus the closed set $D^{-1}(\{0\})$ has full measure with respect to $p(\ph|a)$ for every $a \in A$.
	By the definition of $\Supp{p}$, we conclude $\Supp{p} \subseteq D^{-1}(\{0\})$, and therefore \cref{eq:ase_top_cont} holds for all $x \in \Supp{p}$, as was to be shown.
\end{proof}

Some additional properties of almost sure equality used throughout the paper can also be found in \cref{sec:preliminaries_aseq}.

\subsection{Definition and examples of absolute continuity}

A definition of absolute continuity for morphisms in Markov categories was first given in~\cite[Definition~2.8]{fritz2023representable}.
Here we introduce and study an improved version of that definition.
One important improvement is that it is automatically compatible with $\otimes$, as we show in \cref{prop:ac_tensor}, while this fails for the earlier one (\cref{rem:old_not_ac_tensor}).
We also show that the two definitions coincide in the main examples of Markov categories.

\begin{definition}[Absolute continuity]\label{def:ac}
	Let $p \colon A \to X$ and $q \colon B \to X$ be morphisms in a Markov category $\cC$.
	We say that \newterm{$\bm{p}$ is absolutely continuous with respect to $\bm{q}$}, denoted\footnote{While the notation $p \ll q$ is arguably more commonly used in measure theory, we write $q \gg p$ in this section as it matches the direction of implication \eqref{eq:abs_cont}.}
	\begin{equation*}
		p \ll q \quad\text{ or }\quad q \gg p,
	\end{equation*}
	if for all objects $W$ and $Y$ and any two parallel morphisms $f, g \colon W \otimes X \to Y$, we have 
	\begin{equation}\label{eq:abs_cont}
		{%
			\tikzstyle{every picture}=[tikzfig]%
			\begin{tikzpicture}
				\begin{pgfonlayer}{nodelayer}
					\node [style=none] (13) at (-10, 0.75) {};
					\node [style=none] (14) at (-8.5, 1) {};
					\node [style=none] (19) at (-10.5, 2.25) {};
					\node [style=none] (20) at (-10.5, 2.75) {$Y$};
					\node [style=bn] (21) at (-9.25, 0) {};
					\node [style=none] (23) at (-9.25, -2.25) {};
					\node [style=none] (24) at (-9.25, -2.75) {$B$};
					\node [style=none] (25) at (-7, 0) {$=$};
					\node [style=none] (30) at (-8.5, 2.75) {$X$};
					\node [style=morphism] (33) at (-10.5, 1.25) {$\,\;\; f \,\;\;$};
					\node [style=none] (34) at (-8.5, 2.25) {};
					\node [style=morphism] (35) at (-9.25, -1.25) {$q$};
					\node [style=none] (75) at (-11, -2.25) {};
					\node [style=none] (79) at (-11, -2.75) {$W$};
					\node [style=none] (85) at (-11, 1.25) {};
					\node [style=none] (195) at (-11, -0.75) {};
					\node [style=none] (202) at (0, 0) {$\implies$};
					\node [style=none] (236) at (-4.5, 0.75) {};
					\node [style=none] (237) at (-3, 1) {};
					\node [style=none] (238) at (-5, 2.25) {};
					\node [style=none] (239) at (-5, 2.75) {$Y$};
					\node [style=bn] (240) at (-3.75, 0) {};
					\node [style=none] (241) at (-3.75, -2.25) {};
					\node [style=none] (242) at (-3.75, -2.75) {$B$};
					\node [style=none] (243) at (-3, 2.75) {$X$};
					\node [style=morphism] (244) at (-5, 1.25) {$\,\;\; g \,\;\;$};
					\node [style=none] (245) at (-3, 2.25) {};
					\node [style=morphism] (246) at (-3.75, -1.25) {$q$};
					\node [style=none] (247) at (-5.5, -2.25) {};
					\node [style=none] (248) at (-5.5, -2.75) {$W$};
					\node [style=none] (249) at (-5.5, 1.25) {};
					\node [style=none] (250) at (-5.5, -0.75) {};
					\node [style=none] (251) at (4, 0.75) {};
					\node [style=none] (252) at (5.5, 1) {};
					\node [style=none] (253) at (3.5, 2.25) {};
					\node [style=none] (254) at (3.5, 2.75) {$Y$};
					\node [style=bn] (255) at (4.75, 0) {};
					\node [style=none] (256) at (4.75, -2.25) {};
					\node [style=none] (257) at (4.75, -2.75) {$A$};
					\node [style=none] (258) at (7, 0) {$=$};
					\node [style=none] (259) at (5.5, 2.75) {$X$};
					\node [style=morphism] (260) at (3.5, 1.25) {$\,\;\; f \,\;\;$};
					\node [style=none] (261) at (5.5, 2.25) {};
					\node [style=morphism] (262) at (4.75, -1.25) {$p$};
					\node [style=none] (263) at (3, -2.25) {};
					\node [style=none] (264) at (3, -2.75) {$W$};
					\node [style=none] (265) at (3, 1.25) {};
					\node [style=none] (266) at (3, -0.75) {};
					\node [style=none] (267) at (9.5, 0.75) {};
					\node [style=none] (268) at (11, 1) {};
					\node [style=none] (269) at (9, 2.25) {};
					\node [style=none] (270) at (9, 2.75) {$Y$};
					\node [style=bn] (271) at (10.25, 0) {};
					\node [style=none] (272) at (10.25, -2.25) {};
					\node [style=none] (273) at (10.25, -2.75) {$A$};
					\node [style=none] (274) at (11, 2.75) {$X$};
					\node [style=morphism] (275) at (9, 1.25) {$\,\;\; g \,\;\;$};
					\node [style=none] (276) at (11, 2.25) {};
					\node [style=morphism] (277) at (10.25, -1.25) {$p$};
					\node [style=none] (278) at (8.5, -2.25) {};
					\node [style=none] (279) at (8.5, -2.75) {$W$};
					\node [style=none] (280) at (8.5, 1.25) {};
					\node [style=none] (281) at (8.5, -0.75) {};
					\node [style=none] (282) at (-10, 1.25) {};
					\node [style=none] (283) at (-4.5, 1.25) {};
					\node [style=none] (284) at (4, 1.25) {};
					\node [style=none] (285) at (9.5, 1.25) {};
				\end{pgfonlayer}
				\begin{pgfonlayer}{edgelayer}
					\draw [in=-90, out=165] (21) to (13.center);
					\draw [in=-90, out=15] (21) to (14.center);
					\draw (33) to (19.center);
					\draw (14.center) to (34.center);
					\draw (23.center) to (35);
					\draw [style=protected] (75.center) to (195.center);
					\draw [style=protected, in=-90, out=90, looseness=1.25] (195.center) to (85.center);
					\draw (35) to (21);
					\draw [in=-90, out=165] (240) to (236.center);
					\draw [in=-90, out=15] (240) to (237.center);
					\draw (244) to (238.center);
					\draw (237.center) to (245.center);
					\draw (241.center) to (246);
					\draw [style=protected] (247.center) to (250.center);
					\draw [style=protected, in=-90, out=90, looseness=1.25] (250.center) to (249.center);
					\draw (246) to (240);
					\draw [in=-90, out=165] (255) to (251.center);
					\draw [in=-90, out=15] (255) to (252.center);
					\draw (260) to (253.center);
					\draw (252.center) to (261.center);
					\draw (256.center) to (262);
					\draw [style=protected] (263.center) to (266.center);
					\draw [style=protected, in=-90, out=90, looseness=1.25] (266.center) to (265.center);
					\draw [in=270, out=90] (262) to (255);
					\draw [in=-90, out=165] (271) to (267.center);
					\draw [in=-90, out=15] (271) to (268.center);
					\draw (275) to (269.center);
					\draw (268.center) to (276.center);
					\draw (272.center) to (277);
					\draw [style=protected] (278.center) to (281.center);
					\draw [style=protected, in=-90, out=90, looseness=1.25] (281.center) to (280.center);
					\draw (277) to (271);
					\draw (13.center) to (282.center);
					\draw (236.center) to (283.center);
					\draw (251.center) to (284.center);
					\draw (267.center) to (285.center);
				\end{pgfonlayer}
			\end{tikzpicture}
		}%
	\end{equation}
\end{definition}

More concisely, $q \gg p$ says that $\as{q}$ equality implies $\as{p}$ equality.
Since the relation $\gg$ is clearly reflexive and transitive, it equips the class of morphisms having the same codomain $X$ with a preorder relation.

\begin{example}
	Since $f \ase{\id_X} g$ is equivalent to $f = g$, every morphism $p \colon A \to X$ satisfies $\id_X \gg p$.
	In other words, $\id_X$ is a greatest element with respect to $\gg$.
\end{example}

One may wonder whether it is enough to consider the case $W = I$ only in \cref{def:ac}.
That is, can one detect absolute continuity by merely comparing almost sure equality of morphisms without an extra input $W$?
This turns out to be related to the following notions.

\begin{definition}
	\label{def:separability}
	A Markov category $\cC$ is:
	\begin{enumerate}
		\item \newterm{locally state-separable} if for any parallel morphisms $f, g \colon W \otimes X \to Y$, we have
			\begin{equation}\label{eq:locally_state_separable}
				{%
					\tikzstyle{every picture}=[tikzfig]%
					\begin{tikzpicture}
						\begin{pgfonlayer}{nodelayer}
							\node [style=morphism] (0) at (-9.5, 0.5) {$\;\; f \;\;$};
							\node [style=none] (1) at (-9.5, 1.75) {};
							\node [style=none] (2) at (-9.5, 2.25) {$Y$};
							\node [style=none] (3) at (-10, 0.5) {};
							\node [style=none] (4) at (-9, 0.5) {};
							\node [style=none] (5) at (-9, -1.75) {};
							\node [style=none] (6) at (-7, 0) {$=$};
							\node [style=state] (7) at (-10, -1) {$w$};
							\node [style=none] (8) at (-9, -2.25) {$X$};
							\node [style=morphism] (9) at (-4.5, 0.5) {$\;\; g \;\;$};
							\node [style=none] (10) at (-4.5, 1.75) {};
							\node [style=none] (11) at (-4.5, 2.25) {$Y$};
							\node [style=none] (12) at (-5, 0.5) {};
							\node [style=none] (13) at (-4, 0.5) {};
							\node [style=none] (14) at (-4, -1.75) {};
							\node [style=state] (15) at (-5, -1) {$w$};
							\node [style=none] (16) at (-4, -2.25) {$X$};
							\node [style=none] (17) at (8.5, 0) {$f = g.$};
							\node [style=none] (18) at (0, 0) {};
							\node [style=none] (19) at (0, 0) {$\forall \; w \in \cC(I, W)$};
							\node [style=none] (20) at (4.5, 0) {$\implies$};
						\end{pgfonlayer}
						\begin{pgfonlayer}{edgelayer}
							\draw (5.center) to (4.center);
							\draw (3.center) to (7);
							\draw (1.center) to (0);
							\draw (14.center) to (13.center);
							\draw (12.center) to (15);
							\draw (10.center) to (9);
						\end{pgfonlayer}
					\end{tikzpicture}
				}%
			\end{equation}
		\item \newterm{point-separable} if for all parallel morphisms $f,g \colon X \to Y$, we have
			\begin{equation}\label{eq:point_separable}
				{%
					\tikzstyle{every picture}=[tikzfig]%
					\begin{tikzpicture}
						\begin{pgfonlayer}{nodelayer}
							\node [style=morphism] (0) at (-9, 0.25) {$f$};
							\node [style=none] (1) at (-9, 1.25) {};
							\node [style=none] (2) at (-9, 1.75) {$Y$};
							\node [style=none] (6) at (-7, 0) {$=$};
							\node [style=state] (7) at (-9, -1.25) {$a$};
							\node [style=none] (17) at (8.5, 0) {$f = g.$};
							\node [style=none] (18) at (0, 0) {};
							\node [style=none] (19) at (0, 0) {$\forall \; a \in \cC_\det(I, X)$};
							\node [style=none] (20) at (4.5, 0) {$\implies$};
							\node [style=morphism] (21) at (-5, 0.25) {$g$};
							\node [style=none] (22) at (-5, 1.25) {};
							\node [style=none] (23) at (-5, 1.75) {$Y$};
							\node [style=state] (24) at (-5, -1.25) {$a$};
							\node [style=none] (25) at (-10.75, 1) {\phantom{}};
						\end{pgfonlayer}
						\begin{pgfonlayer}{edgelayer}
							\draw (1.center) to (0);
							\draw (0) to (7);
							\draw (22.center) to (21);
							\draw (21) to (24);
						\end{pgfonlayer}
					\end{tikzpicture}
				}%
			\end{equation}
	\end{enumerate}
\end{definition}

For example, it is easy to see that all of the usual Markov categories of interest in categorical probability (like $\finstoch$, $\borelstoch$ and $\stoch$) are point-separable.
We introduce the point-separation condition mainly since it provides a convenient way to verify local state-separation:

\begin{lemma}
	If $\cC$ is point-separable, then it is locally state-separable.
\end{lemma}

\begin{proof}
	To show that point-separability implies local state-separability, assume that the antecedent of Implication \eqref{eq:locally_state_separable} holds.
	Then we have
	\begin{equation}
		{%
			\tikzstyle{every picture}=[tikzfig]%
			\begin{tikzpicture}
				\begin{pgfonlayer}{nodelayer}
					\node [style=morphism] (0) at (-2.75, 0.25) {$\quad f \quad$};
					\node [style=none] (1) at (-2.75, 1.5) {};
					\node [style=none] (2) at (-2.75, 2) {$Y$};
					\node [style=none] (3) at (-2, 0.25) {};
					\node [style=none] (4) at (-3.5, 0.25) {};
					\node [style=state] (5) at (-3.5, -1.25) {$w$};
					\node [style=none] (6) at (0, 0) {$=$};
					\node [style=state] (7) at (-2, -1.25) {$x$};
					\node [style=morphism] (8) at (2.75, 0.25) {$\quad g \quad$};
					\node [style=none] (9) at (2.75, 1.5) {};
					\node [style=none] (10) at (2.75, 2) {$Y$};
					\node [style=none] (11) at (3.5, 0.25) {};
					\node [style=none] (12) at (2, 0.25) {};
					\node [style=state] (13) at (2, -1.25) {$w$};
					\node [style=state] (14) at (3.5, -1.25) {$x$};
				\end{pgfonlayer}
				\begin{pgfonlayer}{edgelayer}
					\draw (5) to (4.center);
					\draw (3.center) to (7);
					\draw (1.center) to (0);
					\draw (13) to (12.center);
					\draw (11.center) to (14);
					\draw (9.center) to (8);
				\end{pgfonlayer}
			\end{tikzpicture}
		}%
	\end{equation}
	for all deterministic states $x \colon I \to X$ and $w \colon I \to W$.
	But since every deterministic state on $W \otimes X$ is of the form $w \otimes x$ (because $\cC_\det$ is cartesian monoidal), we obtain the desired $f = g$ from the assumed point-separability. 
\end{proof}

This gives a criterion for when absolute continuity can be detected without the extra \mbox{wire $W$}.

\begin{proposition}\label{prop:ac_is_completely}
	Consider a locally state-separable Markov category and morphisms ${p \colon A \to X}$ and $q \colon B \to X$.
	Then $q \gg p$ holds if and only if we have
	\begin{equation}\label{eq:ac_is_completely}
		{%
			\tikzstyle{every picture}=[tikzfig]%
			\begin{tikzpicture}
				\begin{pgfonlayer}{nodelayer}
					\node [style=none] (13) at (-11, 1) {};
					\node [style=none] (14) at (-9, 1) {};
					\node [style=none] (19) at (-11, 2.25) {};
					\node [style=none] (20) at (-11, 2.75) {$Y$};
					\node [style=bn] (21) at (-10, -0.25) {};
					\node [style=none] (23) at (-10, -2.25) {};
					\node [style=none] (24) at (-10, -2.75) {$A$};
					\node [style=none] (25) at (-7.25, 0) {$=$};
					\node [style=none] (30) at (-9, 2.75) {$X$};
					\node [style=morphism] (33) at (-11, 1.25) {$m$};
					\node [style=none] (34) at (-9, 2.25) {};
					\node [style=morphism] (35) at (-10, -1.25) {$q$};
					\node [style=none] (202) at (0, 0) {$\implies$};
					\node [style=none] (236) at (-5.5, 1) {};
					\node [style=none] (237) at (-3.5, 1) {};
					\node [style=none] (238) at (-5.5, 2.25) {};
					\node [style=none] (239) at (-5.5, 2.75) {$Y$};
					\node [style=bn] (240) at (-4.5, -0.25) {};
					\node [style=none] (241) at (-4.5, -2.25) {};
					\node [style=none] (242) at (-4.5, -2.75) {$A$};
					\node [style=none] (243) at (-3.5, 2.75) {$X$};
					\node [style=morphism] (244) at (-5.5, 1.25) {$k$};
					\node [style=none] (245) at (-3.5, 2.25) {};
					\node [style=morphism] (246) at (-4.5, -1.25) {$q$};
					\node [style=none] (247) at (3.5, 1) {};
					\node [style=none] (248) at (5.5, 1) {};
					\node [style=none] (249) at (3.5, 2.25) {};
					\node [style=none] (250) at (3.5, 2.75) {$Y$};
					\node [style=bn] (251) at (4.5, -0.25) {};
					\node [style=none] (252) at (4.5, -2.25) {};
					\node [style=none] (253) at (4.5, -2.75) {$B$};
					\node [style=none] (254) at (7.25, 0) {$=$};
					\node [style=none] (255) at (5.5, 2.75) {$X$};
					\node [style=morphism] (256) at (3.5, 1.25) {$m$};
					\node [style=none] (257) at (5.5, 2.25) {};
					\node [style=morphism] (258) at (4.5, -1.25) {$p$};
					\node [style=none] (259) at (9, 1) {};
					\node [style=none] (260) at (11, 1) {};
					\node [style=none] (261) at (9, 2.25) {};
					\node [style=none] (262) at (9, 2.75) {$Y$};
					\node [style=bn] (263) at (10, -0.25) {};
					\node [style=none] (264) at (10, -2.25) {};
					\node [style=none] (266) at (11, 2.75) {$X$};
					\node [style=morphism] (267) at (9, 1.25) {$k$};
					\node [style=none] (268) at (11, 2.25) {};
					\node [style=morphism] (269) at (10, -1.25) {$p$};
					\node [style=none] (270) at (10, -2.75) {$B$};
				\end{pgfonlayer}
				\begin{pgfonlayer}{edgelayer}
					\draw [in=-90, out=165] (21) to (13.center);
					\draw [in=-90, out=15] (21) to (14.center);
					\draw (33) to (19.center);
					\draw (14.center) to (34.center);
					\draw (23.center) to (35);
					\draw (35) to (21);
					\draw [in=-90, out=165] (240) to (236.center);
					\draw [in=-90, out=15] (240) to (237.center);
					\draw (244) to (238.center);
					\draw (237.center) to (245.center);
					\draw (241.center) to (246);
					\draw (246) to (240);
					\draw [in=-90, out=165] (251) to (247.center);
					\draw [in=-90, out=15] (251) to (248.center);
					\draw (256) to (249.center);
					\draw (248.center) to (257.center);
					\draw (252.center) to (258);
					\draw (258) to (251);
					\draw [in=-90, out=165] (263) to (259.center);
					\draw [in=-90, out=15] (263) to (260.center);
					\draw (267) to (261.center);
					\draw (260.center) to (268.center);
					\draw (264.center) to (269);
					\draw (269) to (263);
				\end{pgfonlayer}
			\end{tikzpicture}
		}%
	\end{equation}
	for all $m, k \colon X \to Y$.
\end{proposition}

\begin{proof}
	The ``only if'' direction is trivial, since each Implication \eqref{eq:ac_is_completely} is a special case of Implication \eqref{eq:abs_cont}.
	For the converse direction, assume that all Implications \eqref{eq:ac_is_completely} hold.
	Further note that, by local state-separability, the antecedent of Implication \eqref{eq:abs_cont}{\,\textemdash\,}the defining property of the relation $q \gg p${\,\textemdash\,}is equivalent to
	\begin{equation}\label{eq:completely_abs_cont_proof}
		{%
			\tikzstyle{every picture}=[tikzfig]%
			\begin{tikzpicture}
				\begin{pgfonlayer}{nodelayer}
					\node [style=none] (25) at (0, 0) {$=$};
					\node [style=none] (142) at (3.5, 1) {};
					\node [style=none] (143) at (5.5, 1) {};
					\node [style=none] (144) at (3, 2.25) {};
					\node [style=none] (145) at (3, 2.75) {$Y$};
					\node [style=bn] (146) at (4.5, -0.25) {};
					\node [style=none] (147) at (4.5, -2.25) {};
					\node [style=none] (148) at (4.5, -2.75) {$A$};
					\node [style=none] (149) at (5.5, 2.75) {$X$};
					\node [style=morphism] (150) at (3, 1.25) {$\;\; g \;\;$};
					\node [style=none] (151) at (5.5, 2.25) {};
					\node [style=morphism] (152) at (4.5, -1.25) {$q$};
					\node [style=none] (236) at (2.5, 1.25) {};
					\node [style=state] (237) at (2.5, -0.25) {$w$};
					\node [style=none] (238) at (-4, 1) {};
					\node [style=none] (239) at (-2, 1) {};
					\node [style=none] (240) at (-4.5, 2.25) {};
					\node [style=none] (241) at (-4.5, 2.75) {$Y$};
					\node [style=bn] (242) at (-3, -0.25) {};
					\node [style=none] (243) at (-3, -2.25) {};
					\node [style=none] (244) at (-3, -2.75) {$A$};
					\node [style=none] (245) at (-2, 2.75) {$X$};
					\node [style=morphism] (246) at (-4.5, 1.25) {$\;\; f \;\;$};
					\node [style=none] (247) at (-2, 2.25) {};
					\node [style=morphism] (248) at (-3, -1.25) {$q$};
					\node [style=none] (249) at (-5, 1.25) {};
					\node [style=state] (250) at (-5, -0.25) {$w$};
				\end{pgfonlayer}
				\begin{pgfonlayer}{edgelayer}
					\draw [in=-90, out=165] (146) to (142.center);
					\draw [in=-90, out=15] (146) to (143.center);
					\draw (150) to (144.center);
					\draw (143.center) to (151.center);
					\draw (147.center) to (152);
					\draw (152) to (146);
					\draw (237) to (236.center);
					\draw [in=-90, out=165] (242) to (238.center);
					\draw [in=-90, out=15] (242) to (239.center);
					\draw (246) to (240.center);
					\draw (239.center) to (247.center);
					\draw (243.center) to (248);
					\draw (248) to (242);
					\draw (250) to (249.center);
				\end{pgfonlayer}
			\end{tikzpicture}
		}%
	\end{equation}
	for all $w \colon I \to W$, and similarly for the consequent with $p$ instead of $q$.
	In this way, each Implication \eqref{eq:abs_cont} becomes of the form of the assumed \eqref{eq:ac_is_completely}, thus proving the result.
\end{proof}

\begin{example}\label{ex:abs_cont_stoch}
	For two Markov kernels $p \colon A \to X$ and $q \colon B \to X$ (i.e.\ morphisms in the Markov category $\stoch$), we can use \cref{prop:ac_is_completely} to characterize $q \gg p$.
	In particular, it says that for all measurable sets $S \in \Sigma_X$ we have
	\begin{equation}\label{eq:abs_cont_stoch}
		q(S|b) = 0 \quad \forall b \in B \qquad \implies \qquad   \; p(S|a) = 0 \quad \forall a \in A. 
	\end{equation}
	Indeed it is known that $m \ase{q} k$ holds if and only if for every $T \in \Sigma_Y$, the measurable functions
	\begin{equation}
		m(T|\ph), \, k(T|\ph) \: \colon \: X \to [0,1]
	\end{equation}
	are almost surely equal with respect to the probability measure $q(\ph|b)$ for every $b \in B$ \cite[Example 13.3]{fritz2019synthetic}, and this directly shows Implication \eqref{eq:abs_cont_stoch} by choosing the indicator functions. 
	Conversely, if Implication \eqref{eq:abs_cont_stoch} holds, then the description of almost sure equality expressed in terms of symmetric difference given in \cite[Example 13.3]{fritz2019synthetic} is sufficient to prove that $q \gg p$. 
	In conclusion, our categorical notion of absolute continuity is equivalent to the standard one.

	Let us consider two special cases.
	\begin{enumerate}
		\item For states $p, q \colon I \to X$, we recover the standard notion of absolute continuity of probability measures, where \eqref{eq:abs_cont_stoch} becomes
			\begin{equation}
				q(S) = 0  \quad \implies \quad  p(S) = 0.
			\end{equation}
		\item\label{it:borelstoch_det_abs_cont}
			Suppose that $p$ and $q$ are deterministic morphisms in $\borelstoch$, i.e.~measurable maps (which we also denote by $p$ and $q$ by abuse of notation) between standard Borel spaces.
			Then \eqref{eq:abs_cont_stoch} becomes
			\begin{equation}
				S \cap \mathrm{im}(q) = \emptyset  \quad \implies \quad  S \cap \mathrm{im}(p) = \emptyset,
			\end{equation}
			or equivalently simply $\mathrm{im}(p) \subseteq \mathrm{im}(q)$ since every singleton is measurable.
	\end{enumerate}
\end{example}

\begin{example}\label{ex:ac_setmulti}
	Recall the Markov category $\setmulti$ of multi-valued functions from \cref{ex:examples}, and consider two of its morphisms $p \colon A \to X$ and $q \colon B \to X$. 
	Then we have
	\begin{equation}
		q \gg p \qquad \iff \qquad \bigcup_{a \in A} p(a) \; \subseteq \; \bigcup_{b \in B} q(b),
	\end{equation}
	where $p(a)$ denotes the image of $a$ under $p$ and similarly for $q(b)$.
\end{example}

\begin{example}\label{ex:ac_cartesian}
	In the cartesian case (\cref{ex:ase_cartesian}), absolute continuity $q \gg p$ simply means 
	\begin{equation}
		f \comp (\id_W \times q) = g \comp (\id_W \times q) \qquad \implies \qquad f \comp (\id_W \times p) = g \comp (\id_W \times p).
	\end{equation}	
	For example, for $p = \id_X$ this condition states that $\id_W \times q$ must be an epimorphism for \mbox{every $W$}.
\end{example}

The next example shows that the notion of absolute continuity may depend heavily on the choice of morphisms of the category. 
In particular, restricting to continuous kernels can result in something quite different than allowing all measurable kernels.

\begin{example}
	\label{ex:ac_topological}
	Just as $\stoch$, it is easy to see that the Markov category $\tychstoch$ is also point-separable, and so \cref{prop:ac_is_completely} can be applied.
	By \cref{prop:ase_top} we have $q\gg p$ if and only if the topological support of $q$ contains the topological support of $p$.
	In other words, the absolute continuity preorder is the one induced by the inclusion order of topological supports.

	Note that this is in stark contrast to the measure-theoretic case: 
	Consider the space $Y$ given by the unit interval $[0,1]$ with its usual topology and Borel $\sigma$-algebra.
	Further consider the Dirac measure $\delta_{1/2}$ and the Lebesgue measure $\lambda$. 
	We can see both as morphisms of type $I\to Y$.
	In the Markov category $\stoch$ as well as in $\borelstoch$, we have the familiar facts $\delta_{1/2}\centernot\gg \lambda$ and $\delta_{1/2}\centernot\ll\lambda$.
	That is, in measure-theoretic terminology, the two measures are mutually singular.
	However, $\delta_{1/2}\ll\lambda$ holds in $\tychstoch$, since the topological support of $\delta_{1/2}$, the singleton $\{1/2\}$, is contained in the topological support of $\lambda$, which is the whole of $[0,1]$.

	For example, a pair of measurable maps $m,k \colon Y\to\R$ witnessing $\delta_{1/2}\centernot\ll\lambda$ in $\stoch$ is 
	\begin{equation}
		m(y) \coloneqq 0\quad \forall y, \qquad k(y) \coloneqq 
		\begin{cases}
			1 & \textrm{for } y = 1/2, \\
			0 & \textrm{for } y \ne 1/2,
		\end{cases}
	\end{equation}
	where we have $m \ase{\lambda} k$ but $m \not\ase{\delta_{1/2}} k$.
	The same argument does not work in $\tychstoch$ because $k$ is not continuous. 
	In fact by \cref{prop:ase_top}, any two continuous maps or kernels that are $\lambda$-almost surely equal must also be $\delta_{1/2}$-almost surely equal.
\end{example}

While local state-separability allows one to leave out the extra input $W$ in Implication \eqref{eq:abs_cont}, this is not true for all Markov categories, as we show in \cref{ex:old_vs_new_ac} below.
This is based on the following general criterion for when leaving out the extra input is consequential.

\begin{lemma}
	\label{lem:abs_cont_not_completely}
	Let $q \colon A \to X$ be an epimorphism and $W$ an object such that $W \otimes A$ is an initial object for which $\id_W \otimes q$ is not an epimorphism.
	Then we have
	\begin{equation}\label{eq:epi_as}
		m \ase{q} k  \quad \implies \quad  m \ase{\id_X} k,
	\end{equation}
	for all $m, k \colon X \to Y$, but $q \not\gg \id_X$.
\end{lemma}

\begin{proof}
	Implication \eqref{eq:epi_as} follows directly from the assumption that $q$ is an epimorphism.

	For the second statement, consider arbitrary morphisms $f, g \colon W \otimes X \to Y$.
	We trivially have $f \ase{q} g$ since $W \otimes A$ is initial.

	Since $\id_W \otimes q$ is not an epimorphism, there must exist two non-equal $f,g \colon W \otimes X \to Y$ for some $Y$ such that $f \comp (\id_W \otimes q) = g \comp (\id_W \otimes q)$.
	For these, Implication~\eqref{eq:abs_cont} with $p = \id_X$ is violated---where its antecedent holds again by the initiality of $W \otimes A$---and so we get $q \not\gg \id_X$.
\end{proof}

\begin{example}
	\label{ex:old_vs_new_ac}
	As in \cite[Example~2.3]{fritzrischel2019zeroone}, let $\cring$ be the opposite of the category of commutative rings and unital additive maps, considered as a symmetric monoidal category with respect to the usual tensor product.
	Using the ring multiplication for the comonoid structure on every object turns $\cring$ into a Markov category.

	Then the assumptions of \cref{lem:abs_cont_not_completely} can be satisfied if we choose $q \colon \R \hookleftarrow \Z^2$ to be any unital additive inclusion, where the arrow direction indicates the direction as a function.
	To this end, we can take $W = \Z_2$ to be the field with two elements.
	Then $\Z_2 \otimes \R$ is the zero ring while $\Z_2 \otimes \Z^2$ is isomorphic to $(\Z_2)^2$.
	In particular $\Z_2 \otimes \R$ is initial and $\id_{\Z_2} \otimes q \colon \Z_2 \otimes \R  \hookleftarrow \Z_2 \otimes \Z^2$ must be the unique zero map, so that it is not an epimorphism.
	Hence \cref{lem:abs_cont_not_completely} applies, and therefore Implication \eqref{eq:ac_is_completely} is not equivalent to $q \gg p$ in $\cring$.

	By \cref{prop:ac_is_completely}, we can also conclude that $\cring$ is not locally state-separable.
\end{example}

\subsection{Properties of absolute continuity}

\begin{lemma}\label{lem:factor_ac}
	For any $h \colon X \to Y$ and $p \colon A \to X$, we have
	\begin{equation}
		h \gg h\comp p.
	\end{equation}
\end{lemma}
In other words, if $h \colon X \to Y$ and $g \colon A \to Y$ are such that $g$ factors across $h$, then $h \gg g$ holds.
\begin{proof}
	To obtain $\ase{h \comp p}$ from $\ase{h}$, we simply pre-compose the relevant instances of \cref{eq:as_eq_new} with $p$.
\end{proof}

Since $\id_X \gg p$ is always true, the above \cref{lem:factor_ac} says that in this case, post-composition with $h$ preserves the absolute continuity relation.
General post-composition is also \mbox{$\gg$-monotone}, but only under the additional assumption of causality.
\begin{lemma}\label{lem:ac_mon_comp}
	If $\cC$ is a causal Markov category, then post-composition is $\gg$-monotone. 
	That is, for all $p \colon A \to X$, $q \colon B \to X$, and $h \colon X \to Y$, we have
	\begin{equation}
		q \gg p \quad \implies \quad h\comp q \gg h\comp p.
	\end{equation}
\end{lemma}

\begin{proof}
	We begin with the first statement.
	Assuming $f \ase{h \comp q} g$ amounts to
	\begin{equation}
		\label{eq:ac_mon_comp_proof1}
		{%
			\tikzstyle{every picture}=[tikzfig]%
			\begin{tikzpicture}
				\begin{pgfonlayer}{nodelayer}
					\node [style=none] (17) at (0, 0) {$=$};
					\node [style=none] (35) at (-4, 1.75) {};
					\node [style=none] (36) at (-2, 1.75) {};
					\node [style=none] (37) at (-4.5, 3) {};
					\node [style=none] (38) at (-4.5, 3.5) {$Z$};
					\node [style=bn] (39) at (-3, 0.5) {};
					\node [style=none] (40) at (-3, -3) {};
					\node [style=none] (41) at (-3, -3.5) {$B$};
					\node [style=none] (42) at (-2, 3.5) {$Y$};
					\node [style=morphism] (43) at (-4.5, 2) {$\;\; f \;\;$};
					\node [style=none] (44) at (-2, 3) {};
					\node [style=morphism] (45) at (-3, -2) {$q$};
					\node [style=none] (46) at (-5, -3) {};
					\node [style=none] (47) at (-5, -3.5) {$W$};
					\node [style=none] (48) at (-5, 1.75) {};
					\node [style=none] (49) at (-5, 0) {};
					\node [style=morphism] (51) at (-3, -0.5) {$h$};
					\node [style=none] (52) at (3, 1.75) {};
					\node [style=none] (53) at (5, 1.75) {};
					\node [style=none] (54) at (2.5, 3) {};
					\node [style=none] (55) at (2.5, 3.5) {$Z$};
					\node [style=bn] (56) at (4, 0.5) {};
					\node [style=none] (57) at (4, -3) {};
					\node [style=none] (58) at (4, -3.5) {$B$};
					\node [style=none] (59) at (5, 3.5) {$Y$};
					\node [style=morphism] (60) at (2.5, 2) {$\;\; g \;\;$};
					\node [style=none] (61) at (5, 3) {};
					\node [style=morphism] (62) at (4, -2) {$q$};
					\node [style=none] (63) at (2, -3) {};
					\node [style=none] (64) at (2, -3.5) {$W$};
					\node [style=none] (65) at (2, 1.75) {};
					\node [style=none] (66) at (2, 0) {};
					\node [style=morphism] (68) at (4, -0.5) {$h$};
				\end{pgfonlayer}
				\begin{pgfonlayer}{edgelayer}
					\draw [in=-90, out=165] (39) to (35.center);
					\draw [in=-90, out=15] (39) to (36.center);
					\draw (43) to (37.center);
					\draw (36.center) to (44.center);
					\draw (40.center) to (45);
					\draw [style=protected] (46.center) to (49.center);
					\draw [style=protected, in=-90, out=90, looseness=1.25] (49.center) to (48.center);
					\draw (51) to (39);
					\draw (45) to (51);
					\draw [in=-90, out=165] (56) to (52.center);
					\draw [in=-90, out=15] (56) to (53.center);
					\draw (60) to (54.center);
					\draw (53.center) to (61.center);
					\draw (57.center) to (62);
					\draw [style=protected] (63.center) to (66.center);
					\draw [style=protected, in=-90, out=90, looseness=1.25] (66.center) to (65.center);
					\draw (68) to (56);
					\draw (62) to (68);
				\end{pgfonlayer}
			\end{tikzpicture}
		}%
	\end{equation}
	An application of the causality axiom in the form of Implication \eqref{eq:causal_extrawire} then gives
	\begin{equation}\label{eq:ac_mon_comp_proof2}
		{%
			\tikzstyle{every picture}=[tikzfig]%
			\begin{tikzpicture}
				\begin{pgfonlayer}{nodelayer}
					\node [style=none] (17) at (0, 0) {$=$};
					\node [style=none] (42) at (-5.5, 2) {};
					\node [style=none] (43) at (-3.5, 2) {};
					\node [style=none] (44) at (-6, 3.25) {};
					\node [style=none] (45) at (-6, 3.75) {$Z$};
					\node [style=bn] (46) at (-4.5, 0.75) {};
					\node [style=none] (47) at (-3.25, -3.5) {};
					\node [style=none] (48) at (-3.25, -4) {$B$};
					\node [style=none] (49) at (-3.5, 3.75) {$Y$};
					\node [style=morphism] (50) at (-6, 2.25) {$\;\; f \;\;$};
					\node [style=none] (51) at (-3.5, 3.25) {};
					\node [style=morphism] (52) at (-3.25, -2.5) {$q$};
					\node [style=none] (53) at (-6.5, -0.25) {};
					\node [style=none] (54) at (-6.5, -4) {$W$};
					\node [style=none] (55) at (-6.5, 1.75) {};
					\node [style=none] (56) at (-6.5, 0.25) {};
					\node [style=none] (57) at (-6.5, 2) {};
					\node [style=morphism] (58) at (-4.5, -0.25) {$h$};
					\node [style=bn] (59) at (-3.25, -1.5) {};
					\node [style=none] (60) at (-2, 0) {};
					\node [style=none] (61) at (-2, 3.25) {};
					\node [style=none] (62) at (-2, 3.75) {$X$};
					\node [style=none] (63) at (-6.5, -3.5) {};
					\node [style=none] (64) at (-6.5, -1.75) {};
					\node [style=none] (65) at (-4.5, -0.5) {};
					\node [style=none] (66) at (3, 2) {};
					\node [style=none] (67) at (5, 2) {};
					\node [style=none] (68) at (2.5, 3.25) {};
					\node [style=bn] (70) at (4, 0.75) {};
					\node [style=none] (71) at (5.25, -3.5) {};
					\node [style=none] (72) at (5.25, -4) {$B$};
					\node [style=morphism] (74) at (2.5, 2.25) {$\;\; g \;\;$};
					\node [style=none] (75) at (5, 3.25) {};
					\node [style=morphism] (76) at (5.25, -2.5) {$q$};
					\node [style=none] (77) at (2, -0.25) {};
					\node [style=none] (78) at (2, -4) {$W$};
					\node [style=none] (79) at (2, 1.75) {};
					\node [style=none] (80) at (2, 0.25) {};
					\node [style=none] (81) at (2, 2) {};
					\node [style=morphism] (82) at (4, -0.25) {$h$};
					\node [style=bn] (83) at (5.25, -1.5) {};
					\node [style=none] (84) at (6.5, 0) {};
					\node [style=none] (85) at (6.5, 3.25) {};
					\node [style=none] (87) at (2, -3.5) {};
					\node [style=none] (88) at (2, -1.75) {};
					\node [style=none] (89) at (4, -0.5) {};
					\node [style=none] (90) at (2.5, 3.75) {$Z$};
					\node [style=none] (91) at (5, 3.75) {$Y$};
					\node [style=none] (92) at (6.5, 3.75) {$X$};
				\end{pgfonlayer}
				\begin{pgfonlayer}{edgelayer}
					\draw [in=-90, out=165] (46) to (42.center);
					\draw [in=-90, out=15] (46) to (43.center);
					\draw (50) to (44.center);
					\draw (43.center) to (51.center);
					\draw (47.center) to (52);
					\draw [style=protected] (53.center) to (56.center);
					\draw [style=protected, in=-90, out=90] (56.center) to (55.center);
					\draw [style=protected] (55.center) to (57.center);
					\draw (58) to (46);
					\draw (52) to (59);
					\draw [style=protected] (60.center) to (61.center);
					\draw [style=protected, in=-90, out=15] (59) to (60.center);
					\draw [style=protected, in=-90, out=90, looseness=0.75] (64.center) to (53.center);
					\draw [style=protected] (63.center) to (64.center);
					\draw [style=protected, in=-90, out=165] (59) to (65.center);
					\draw [in=-90, out=165] (70) to (66.center);
					\draw [in=-90, out=15] (70) to (67.center);
					\draw (74) to (68.center);
					\draw (67.center) to (75.center);
					\draw (71.center) to (76);
					\draw [style=protected] (77.center) to (80.center);
					\draw [style=protected, in=-90, out=90] (80.center) to (79.center);
					\draw [style=protected] (79.center) to (81.center);
					\draw (82) to (70);
					\draw (76) to (83);
					\draw [style=protected] (84.center) to (85.center);
					\draw [style=protected, in=-90, out=15] (83) to (84.center);
					\draw [style=protected, in=-90, out=90, looseness=0.75] (88.center) to (77.center);
					\draw [style=protected] (87.center) to (88.center);
					\draw [style=protected, in=-90, out=165] (83) to (89.center);
				\end{pgfonlayer}
			\end{tikzpicture}
		}%
	\end{equation}
	Using the assumption $q \gg p$ lets us replace $q$ by $p$ in this equation, and the required $f \ase{h \comp p} g$ then follows upon discarding the $X$ output.
\end{proof}

Under stronger assumptions on the morphism that we post-compose with, we can do away with the causality assumption and even obtain an equivalence.

\begin{lemma}\label{lem:ac_mon_comp_split_mono}
	In an arbitrary Markov category, consider morphisms $p \colon A \to T$ and ${q \colon B \to T}$. 
	Furthermore, let $\iota \colon T \to X$ be a deterministic split monomorphism. 
	Then we have
	\begin{equation}
		q \gg p \quad \iff \quad \iota \comp q \gg \iota \comp p.
	\end{equation}
\end{lemma}
\begin{proof}
	This is a direct consequence of \cref{lem:ase_detsplitmono}.
\end{proof}

Another property of such a left-invertible $\iota$ is that morphisms that are absolutely continuous with respect to it factor across $\iota$.
\begin{lemma}\label{lem:ac_mon_factor}
	Consider a deterministic split monomorphism $\iota \colon T \to X$ with left inverse given by $\pi \colon X \to T$, and an arbitrary morphism $p \colon A \to X$.
	Then we have
	\begin{equation}
		\iota \gg p  \quad \implies \quad  p = \iota \comp \pi \comp p.
	\end{equation}
\end{lemma}
\begin{proof}
	The assumptions imply that we have $\iota \comp \pi \ase{\iota} \id_X$.
	This can be shown by direct calculation, but it also follows from equivalence \eqref{eq:mp_iff_p} in \cref{lem:ase_detsplitmono} by taking $p = \id_X$, $f_2 = \iota \comp \pi$, and $g_2 = \id_X$.
	Recalling that relation $\iota \gg p$ means, by definition, that $\as{\iota}$ equality implies $\as{p}$ equality, we get $\iota \comp \pi \ase{p} \id_X$.
	The desired equation $\iota \comp \pi \comp p = p$ then follows by marginalization.
\end{proof}

\begin{remark}
	Unlike post-composition, pre-composition generally does not preserve absolute continuity, not even in $\finstoch$.
	For example, if we take $A = B = X = \{0,1\}$ with 
	\begin{equation}
		q \coloneqq \begin{pmatrix} 1 & 1/2 \\ 0 & 1/2 \end{pmatrix}  \qquad  p \coloneqq \begin{pmatrix} 1/2 & 1/2 \\ 1/2 & 1/2 \end{pmatrix},
	\end{equation}
	then $q \gg p$ holds.
	For the deterministic state $\delta_0 \colon I \to A$, however, we also have $q \comp \delta_0 \centernot\gg p \comp \delta_0$, since the latter has strictly larger support:
	\begin{equation}
		q \comp \delta_0 = \begin{pmatrix} 1 \\ 0 \end{pmatrix}  \qquad  p\comp \delta_0 = \begin{pmatrix} 1/2  \\ 1/2  \end{pmatrix}.
	\end{equation}
\end{remark}

However, our definition of absolute continuity interacts very well with the monoidal structure.

\begin{proposition}\label{prop:ac_tensor}
	The monoidal product $\otimes$ is $\gg$-monotone.
	That is, for any choice of morphisms ${q \colon B \to X}$, ${q' \colon B' \to X'}$, $p \colon A \to X$, and $p' \colon A' \to X'$ we have
	\begin{equation}\label{eq:ac_mon_tensor}
		q \gg p \; \text{ and } \; q' \gg p' \quad \implies \quad q \otimes q' \gg p \otimes p'.
	\end{equation}
\end{proposition}
\begin{proof}
	Consider generic morphisms $f, g \colon W \otimes X \otimes X' \to Y$ and assume that 
	\begin{equation}\label{eq:ac_tensor_1}
		{%
			\tikzstyle{every picture}=[tikzfig]%
			\begin{tikzpicture}
				\begin{pgfonlayer}{nodelayer}
					\node [style=none] (13) at (-3.75, 0.75) {};
					\node [style=none] (14) at (-1.75, 1) {};
					\node [style=none] (19) at (-4.5, 2.25) {};
					\node [style=none] (20) at (-4.5, 2.75) {$Y$};
					\node [style=bn] (21) at (-2.75, -0.25) {};
					\node [style=none] (23) at (-2.75, -2.25) {};
					\node [style=none] (24) at (-2.75, -2.75) {$B'$};
					\node [style=none] (25) at (0, 0) {$=$};
					\node [style=none] (30) at (-1.75, 2.75) {$X'$};
					\node [style=morphism] (33) at (-4.5, 1.25) {$\quad f \quad $};
					\node [style=none] (34) at (-1.75, 2.25) {};
					\node [style=morphism] (35) at (-2.75, -1.25) {$q'$};
					\node [style=none] (75) at (-5.25, -2.25) {};
					\node [style=none] (79) at (-5.25, -2.75) {$W$};
					\node [style=none] (85) at (-5.25, 1.25) {};
					\node [style=none] (195) at (-5.25, -0.75) {};
					\node [style=morphism] (251) at (-4, -1.25) {$q \vphantom{q'}$};
					\node [style=bn] (252) at (-4, -0.25) {};
					\node [style=none] (253) at (-2.75, 1) {};
					\node [style=none] (254) at (-2.75, 2.25) {};
					\node [style=none] (255) at (-2.75, 2.75) {$X$};
					\node [style=none] (256) at (-4.5, 0.75) {};
					\node [style=none] (257) at (-4, -2.25) {};
					\node [style=none] (258) at (-4, -2.75) {$B$};
					\node [style=none] (259) at (-3.75, 1.25) {};
					\node [style=none] (260) at (3.5, 0.75) {};
					\node [style=none] (261) at (5.5, 1) {};
					\node [style=none] (262) at (2.75, 2.25) {};
					\node [style=none] (263) at (2.75, 2.75) {$Y$};
					\node [style=bn] (264) at (4.5, -0.25) {};
					\node [style=none] (265) at (4.5, -2.25) {};
					\node [style=none] (266) at (4.5, -2.75) {$B'$};
					\node [style=none] (267) at (5.5, 2.75) {$X'$};
					\node [style=morphism] (268) at (2.75, 1.25) {$\quad g \quad $};
					\node [style=none] (269) at (5.5, 2.25) {};
					\node [style=morphism] (270) at (4.5, -1.25) {$q'$};
					\node [style=none] (271) at (2, -2.25) {};
					\node [style=none] (272) at (2, -2.75) {$W$};
					\node [style=none] (273) at (2, 1.25) {};
					\node [style=none] (274) at (2, -0.75) {};
					\node [style=morphism] (275) at (3.25, -1.25) {$q \vphantom{q'}$};
					\node [style=bn] (276) at (3.25, -0.25) {};
					\node [style=none] (277) at (4.5, 1) {};
					\node [style=none] (278) at (4.5, 2.25) {};
					\node [style=none] (279) at (4.5, 2.75) {$X$};
					\node [style=none] (280) at (2.75, 0.75) {};
					\node [style=none] (281) at (3.25, -2.25) {};
					\node [style=none] (282) at (3.25, -2.75) {$B$};
					\node [style=none] (283) at (3.5, 1.25) {};
				\end{pgfonlayer}
				\begin{pgfonlayer}{edgelayer}
					\draw [in=-90, out=165] (21) to (13.center);
					\draw [in=-90, out=15] (21) to (14.center);
					\draw (33) to (19.center);
					\draw (14.center) to (34.center);
					\draw (23.center) to (35);
					\draw [style=protected] (75.center) to (195.center);
					\draw [style=protected, in=-90, out=90, looseness=1.25] (195.center) to (85.center);
					\draw (35) to (21);
					\draw (253.center) to (254.center);
					\draw [in=-90, out=135] (252) to (256.center);
					\draw (251) to (252);
					\draw (257.center) to (251);
					\draw [style=protected, in=-90, out=15] (252) to (253.center);
					\draw [style=protected] (13.center) to (259.center);
					\draw [style=protected] (256.center) to (33);
					\draw [in=-90, out=165] (264) to (260.center);
					\draw [in=-90, out=15] (264) to (261.center);
					\draw (268) to (262.center);
					\draw (261.center) to (269.center);
					\draw (265.center) to (270);
					\draw [style=protected] (271.center) to (274.center);
					\draw [style=protected, in=-90, out=90, looseness=1.25] (274.center) to (273.center);
					\draw (270) to (264);
					\draw (277.center) to (278.center);
					\draw [in=-90, out=135] (276) to (280.center);
					\draw (275) to (276);
					\draw (281.center) to (275);
					\draw [style=protected, in=-90, out=15] (276) to (277.center);
					\draw [style=protected] (260.center) to (283.center);
					\draw [style=protected] (280.center) to (268);
				\end{pgfonlayer}
			\end{tikzpicture}
		}%
	\end{equation}
	holds.
	We can rearrange these string diagrams to get
	\begin{equation}\label{eq:ac_tensor_2}
		{%
			\tikzstyle{every picture}=[tikzfig]%
			\begin{tikzpicture}
				\begin{pgfonlayer}{nodelayer}
					\node [style=none] (13) at (-3.5, -0.75) {};
					\node [style=none] (14) at (-2, -0.75) {};
					\node [style=none] (19) at (-5, 3.25) {};
					\node [style=none] (20) at (-5, 3.75) {$Y$};
					\node [style=bn] (21) at (-2.75, -1.5) {};
					\node [style=none] (23) at (-2.75, -3.5) {};
					\node [style=none] (24) at (-2.75, -4) {$B'$};
					\node [style=none] (25) at (0, 0) {$=$};
					\node [style=none] (30) at (-2, 3.75) {$X'$};
					\node [style=morphism] (33) at (-5, 2.25) {$\quad f \quad $};
					\node [style=none] (34) at (-2, 3.25) {};
					\node [style=morphism] (35) at (-2.75, -2.5) {$q'$};
					\node [style=none] (75) at (-5.75, -3.5) {};
					\node [style=none] (79) at (-5.75, -4) {$W$};
					\node [style=none] (85) at (-5.75, 2.25) {};
					\node [style=none] (195) at (-5.75, 0.25) {};
					\node [style=morphism] (251) at (-4.5, -0.25) {$q \vphantom{q'}$};
					\node [style=bn] (252) at (-4.5, 0.75) {};
					\node [style=none] (253) at (-3.5, 1.75) {};
					\node [style=none] (254) at (-3.5, 3.25) {};
					\node [style=none] (255) at (-3.5, 3.75) {$X$};
					\node [style=none] (256) at (-5, 1.75) {};
					\node [style=none] (257) at (-4.5, -3.5) {};
					\node [style=none] (258) at (-4.5, -4) {$B$};
					\node [style=none] (284) at (-3.5, 0.5) {};
					\node [style=none] (285) at (-4.25, 1.75) {};
					\node [style=none] (286) at (-4.25, 2.25) {};
					\node [style=none] (287) at (-6.5, -1) {};
					\node [style=none] (288) at (-3, -1) {};
					\node [style=none] (289) at (-3, 3) {};
					\node [style=none] (290) at (-6.5, 3) {};
					\node [style=none] (291) at (4.75, -0.75) {};
					\node [style=none] (292) at (6.25, -0.75) {};
					\node [style=none] (293) at (3.25, 3.25) {};
					\node [style=none] (294) at (3.25, 3.75) {$Y$};
					\node [style=bn] (295) at (5.5, -1.5) {};
					\node [style=none] (296) at (5.5, -3.5) {};
					\node [style=none] (297) at (5.5, -4) {$B'$};
					\node [style=none] (298) at (6.25, 3.75) {$X'$};
					\node [style=morphism] (299) at (3.25, 2.25) {$\quad g \quad $};
					\node [style=none] (300) at (6.25, 3.25) {};
					\node [style=morphism] (301) at (5.5, -2.5) {$q'$};
					\node [style=none] (302) at (2.5, -3.5) {};
					\node [style=none] (303) at (2.5, -4) {$W$};
					\node [style=none] (304) at (2.5, 2.25) {};
					\node [style=none] (305) at (2.5, 0.25) {};
					\node [style=morphism] (306) at (3.75, -0.25) {$q \vphantom{q'}$};
					\node [style=bn] (307) at (3.75, 0.75) {};
					\node [style=none] (308) at (4.75, 1.75) {};
					\node [style=none] (309) at (4.75, 3.25) {};
					\node [style=none] (310) at (4.75, 3.75) {$X$};
					\node [style=none] (311) at (3.25, 1.75) {};
					\node [style=none] (312) at (3.75, -3.5) {};
					\node [style=none] (313) at (3.75, -4) {$B$};
					\node [style=none] (314) at (4.75, 0.5) {};
					\node [style=none] (315) at (4, 1.75) {};
					\node [style=none] (316) at (4, 2.25) {};
					\node [style=none] (317) at (1.75, -1) {};
					\node [style=none] (318) at (5.25, -1) {};
					\node [style=none] (319) at (5.25, 3) {};
					\node [style=none] (320) at (1.75, 3) {};
				\end{pgfonlayer}
				\begin{pgfonlayer}{edgelayer}
					\draw [style=dashed box] (288.center)
					to (289.center)
					to (290.center)
					to (287.center)
					to cycle;
					\draw [in=-90, out=15] (21) to (14.center);
					\draw (14.center) to (34.center);
					\draw (23.center) to (35);
					\draw [style=protected, in=-90, out=90, looseness=1.25] (195.center) to (85.center);
					\draw (35) to (21);
					\draw [in=-90, out=135] (252) to (256.center);
					\draw (251) to (252);
					\draw [style=protected] (256.center) to (33);
					\draw [style=protected] (13.center) to (284.center);
					\draw [style=protected] (285.center) to (286.center);
					\draw [in=-90, out=90] (284.center) to (285.center);
					\draw [style=protected, in=-90, out=15] (252) to (253.center);
					\draw [style=protected] (33) to (19.center);
					\draw [style=protected] (253.center) to (254.center);
					\draw [style=protected] (75.center) to (195.center);
					\draw [style=protected] (257.center) to (251);
					\draw [style=protected, in=-90, out=165] (21) to (13.center);
					\draw [style=dashed box] (318.center)
					to (319.center)
					to (320.center)
					to (317.center)
					to cycle;
					\draw [in=-90, out=15] (295) to (292.center);
					\draw (292.center) to (300.center);
					\draw (296.center) to (301);
					\draw [style=protected, in=-90, out=90, looseness=1.25] (305.center) to (304.center);
					\draw (301) to (295);
					\draw [in=-90, out=135] (307) to (311.center);
					\draw (306) to (307);
					\draw [style=protected] (311.center) to (299);
					\draw [style=protected] (291.center) to (314.center);
					\draw [style=protected] (315.center) to (316.center);
					\draw [in=-90, out=90] (314.center) to (315.center);
					\draw [style=protected, in=-90, out=15] (307) to (308.center);
					\draw [style=protected] (299) to (293.center);
					\draw [style=protected] (308.center) to (309.center);
					\draw [style=protected] (302.center) to (305.center);
					\draw [style=protected] (312.center) to (306);
					\draw [style=protected, in=-90, out=165] (295) to (291.center);
				\end{pgfonlayer}
			\end{tikzpicture}
		}%
	\end{equation}
	Using $q' \gg p'$, we can replace $q'$ with $p'$ in this equation and rearrange back to obtain 
	\begin{equation}\label{eq:ac_tensor_3}
		{%
			\tikzstyle{every picture}=[tikzfig]%
			\begin{tikzpicture}
				\begin{pgfonlayer}{nodelayer}
					\node [style=none] (13) at (-3.75, 0.75) {};
					\node [style=none] (14) at (-1.75, 1) {};
					\node [style=none] (19) at (-4.5, 2.25) {};
					\node [style=none] (20) at (-4.5, 2.75) {$Y$};
					\node [style=bn] (21) at (-2.75, -0.25) {};
					\node [style=none] (23) at (-2.75, -2.25) {};
					\node [style=none] (24) at (-2.75, -2.75) {$A'$};
					\node [style=none] (25) at (0, 0) {$=$};
					\node [style=none] (30) at (-1.75, 2.75) {$X'$};
					\node [style=morphism] (33) at (-4.5, 1.25) {$\quad f \quad $};
					\node [style=none] (34) at (-1.75, 2.25) {};
					\node [style=morphism] (35) at (-2.75, -1.25) {$p'$};
					\node [style=none] (75) at (-5.25, -2.25) {};
					\node [style=none] (79) at (-5.25, -2.75) {$W$};
					\node [style=none] (85) at (-5.25, 1.25) {};
					\node [style=none] (195) at (-5.25, -0.75) {};
					\node [style=morphism] (251) at (-4, -1.25) {$q \vphantom{p'}$};
					\node [style=bn] (252) at (-4, -0.25) {};
					\node [style=none] (253) at (-2.75, 1) {};
					\node [style=none] (254) at (-2.75, 2.25) {};
					\node [style=none] (255) at (-2.75, 2.75) {$X$};
					\node [style=none] (256) at (-4.5, 0.75) {};
					\node [style=none] (257) at (-4, -2.25) {};
					\node [style=none] (258) at (-4, -2.75) {$B$};
					\node [style=none] (259) at (-3.75, 1.25) {};
					\node [style=none] (260) at (3.5, 0.75) {};
					\node [style=none] (261) at (5.5, 1) {};
					\node [style=none] (262) at (2.75, 2.25) {};
					\node [style=none] (263) at (2.75, 2.75) {$Y$};
					\node [style=bn] (264) at (4.5, -0.25) {};
					\node [style=none] (265) at (4.5, -2.25) {};
					\node [style=none] (266) at (4.5, -2.75) {$A'$};
					\node [style=none] (267) at (5.5, 2.75) {$X'$};
					\node [style=morphism] (268) at (2.75, 1.25) {$\quad g \quad $};
					\node [style=none] (269) at (5.5, 2.25) {};
					\node [style=morphism] (270) at (4.5, -1.25) {$p'$};
					\node [style=none] (271) at (2, -2.25) {};
					\node [style=none] (272) at (2, -2.75) {$W$};
					\node [style=none] (273) at (2, 1.25) {};
					\node [style=none] (274) at (2, -0.75) {};
					\node [style=morphism] (275) at (3.25, -1.25) {$q \vphantom{p'}$};
					\node [style=bn] (276) at (3.25, -0.25) {};
					\node [style=none] (277) at (4.5, 1) {};
					\node [style=none] (278) at (4.5, 2.25) {};
					\node [style=none] (279) at (4.5, 2.75) {$X$};
					\node [style=none] (280) at (2.75, 0.75) {};
					\node [style=none] (281) at (3.25, -2.25) {};
					\node [style=none] (282) at (3.25, -2.75) {$B$};
					\node [style=none] (283) at (3.5, 1.25) {};
				\end{pgfonlayer}
				\begin{pgfonlayer}{edgelayer}
					\draw [in=-90, out=165] (21) to (13.center);
					\draw [in=-90, out=15] (21) to (14.center);
					\draw (33) to (19.center);
					\draw (14.center) to (34.center);
					\draw (23.center) to (35);
					\draw [style=protected] (75.center) to (195.center);
					\draw [style=protected, in=-90, out=90, looseness=1.25] (195.center) to (85.center);
					\draw (35) to (21);
					\draw (253.center) to (254.center);
					\draw [in=-90, out=135] (252) to (256.center);
					\draw (251) to (252);
					\draw (257.center) to (251);
					\draw [style=protected, in=-90, out=15] (252) to (253.center);
					\draw [style=protected] (13.center) to (259.center);
					\draw [style=protected] (256.center) to (33);
					\draw [in=-90, out=165] (264) to (260.center);
					\draw [in=-90, out=15] (264) to (261.center);
					\draw (268) to (262.center);
					\draw (261.center) to (269.center);
					\draw (265.center) to (270);
					\draw [style=protected] (271.center) to (274.center);
					\draw [style=protected, in=-90, out=90, looseness=1.25] (274.center) to (273.center);
					\draw (270) to (264);
					\draw (277.center) to (278.center);
					\draw [in=-90, out=135] (276) to (280.center);
					\draw (275) to (276);
					\draw (281.center) to (275);
					\draw [style=protected, in=-90, out=15] (276) to (277.center);
					\draw [style=protected] (260.center) to (283.center);
					\draw [style=protected] (280.center) to (268);
				\end{pgfonlayer}
			\end{tikzpicture}
		}%
	\end{equation}
	By swapping wires, we can similarly use $q \gg p$ to replace $q$ by $p$ and get the desired conclusion of the implication that defines $q \otimes q' \gg p \otimes p'$.
\end{proof}

\begin{remark}
	\label{rem:old_not_ac_tensor}
	\cref{ex:old_vs_new_ac} shows that our old definition of absolute continuity without the extra input $W$~\cite[Definition~2.8]{fritz2023representable} does not satisfy Implication \eqref{eq:ac_mon_tensor}: with that definition, we would get $q \gg \id_X$ and $q \otimes \id_W \not\gg \id_X \otimes \id_W$.
\end{remark}

\subsection{Absolute bicontinuity}

The following symmetrized version of absolute continuity comes up frequently in the paper, so it is useful to introduce notation for the resulting equivalence relation.

\begin{definition}
	\label{def:abs_bicont}
	We say that two morphisms $p \colon A \to X$ and $q \colon B \to X$ are \newterm{absolutely bicontinuous}, and write $p \acsim q$, if both
	\begin{equation}
		p \gg q \quad \text{ and } \quad q \gg p
	\end{equation}
	are satisfied.
	In this case, we also say that $p$ \newterm{is absolutely bicontinuous with respect to} $q$.
\end{definition}

As a direct consequence of \cref{prop:ac_tensor}, we obtain the fact that the monoidal product is compatible with absolute bicontinuity.
\begin{corollary}\label{cor:abc_tensor}
	For morphisms labelled as in \cref{prop:ac_tensor}, we have
	\begin{equation}
		p \acsim q \; \text{ and } \; p' \acsim q' \quad \implies \quad p \otimes p' \acsim q \otimes q'.
	\end{equation}
	In particular, for any object $W$, we have $p \acsim q \implies \id_W \otimes p \acsim \id_W \otimes q$.
\end{corollary}

The following small observation will be useful later on.

\begin{lemma}
	\label{lem:epi_ac_max}
	If $\pi \colon X \to T$ is such that $\id_W \otimes \pi$ is an epimorphism for every $W$, then we have $\pi \acsim \id_T$.
	More generally, $f \comp \pi \acsim f$ holds for any $f \colon T \to Y$.
\end{lemma}
For example, this applies if $\pi$ is a split epimorphism.
\begin{proof}
	The absolute continuity direction $f \gg f\comp \pi$ is an instance of \cref{lem:factor_ac}, while $f \comp \pi \gg f$ is straightforward from the definition of absolute continuity and the assumption that $\id_W \otimes \pi$ is an epimorphism for every $W$.
\end{proof}

\section{Supports}
\label{sec:supports}

The support of a probability measure on a finite set or discrete measurable space is a classical concept.
Here, we investigate an abstract categorical generalization of this notion in terms of a universal property.
Besides the intrinsic interest and utility of having a universal property, one of the features of this development is that it clarifies the significance of topology in the context of supports:
A continuous Markov kernel between Polish spaces has a support in $\tychstoch${\,\textemdash\,}which turns out to be a Markov category in which \emph{every} morphism has a support{\,\textemdash\,}but the same Markov kernel considered as a morphism in the purely measure-theoretic context of $\borelstoch$ will typically not have a support.\footnotemark{}
\footnotetext{See \cref{ex:top_support,thm:borelstoch_supports}.}%

We start in \cref{sec:supports_def} by giving our new definition of a support and developing some of its basic properties.
We make frequent use of the characterization of supports from \cref{lem:supp_crit}.
Using it, \cref{ex:top_support} shows that all morphisms in $\tychstoch$ have a support, and these are given by the usual topological supports.
This demonstrates that our definition of support is a natural generalization of the classical notion.
In \cref{ex:support_setmulti}, we prove that $\setmulti$ has supports as well.
\cref{thm:supp_functorial} shows that the functoriality of supports is closely related to the causality axiom \cite[Definition 11.31]{fritz2019synthetic}.

\cref{sec:support_completeness} is devoted to a more detailed study of the question: 
``Which morphisms can be expected to have a support?''
This results in the definition of an \emph{atomic} morphism in \cref{def:atomic}, which is intimately related to the traditional notion of an atomic measure.
We use these insights to characterize which morphisms in $\borelstoch$ have a support in \cref{thm:borelstoch_supports}.
In \cref{prop:suppcomp}, we construct a \emph{free regular support completion} for every causal Markov category by freely adjoining a support for every atomic morphism.
\cref{thm:suppcomp_universal} establishes a 2-categorical universal property for this construction.

The following \cref{sec:rel} on the input-output relations functor and \cref{sec:statistical_models} on categories of statistical models provide applications of absolute continuity and supports. 
These parts can safely be skipped without impeding understanding of the rest of the paper.

The lack of supports in categories like $\borelstoch$ can be seen as an (expected) shortcoming.
We partially remedy this in \cref{sec:equalizer_principle} by introducing the \emph{equalizer principle}.
This is an axiom for a Markov category $\cC$ that can serve as a substitute for supports to some extent.
\cref{rem:equalizer_principle} shows that the equalizer principle is implied by the existence of supports together with the existence of equalizers in the deterministic subcategory $\cC_\det$.
We use this principle in \cref{sec:idempotents} as a crucial assumption in our main result on the splitting of idempotents.
We also introduce the \emph{relative equalizer principle} in \cref{def:relative_equalizer}, which is a stronger version of the equalizer principle that still holds in $\borelstoch$. 
We do not otherwise use it, but merely note that, together with an additional assumption,  it implies causality.
This results in a statement (\cref{prop:equalizer_causal}) that is vaguely reminiscent of the earlier \cref{thm:supp_functorial} where causality was derived from the functoriality of supports.

In the final \cref{sec:split_supports}, we introduce a stronger notion of supports that is more in line with our original definition of supports given in \cite[Definition 13.20]{fritz2019synthetic}.
In order to distinguish it from \cref{sec:supports_def}, we call them \emph{split supports}.
Per \cref{def:split_support}, the idea behind these is that not only is the support of a morphism included in its codomain, but this inclusion also has a \emph{splitting} which projects back onto the support.
\Cref{prop:split_support_universal} shows that split supports also enjoy a \emph{second} universal property in addition to the mapping-in universal property of plain supports, namely the \emph{mapping-out} universal property considered originally in \cite[Definition 13.20]{fritz2019synthetic}.
As we show in \cref{rem:split_supp_idemp} and \cref{sec:idempotents}, they are also closely related to idempotents.
Unfortunately, not all supports are split supports: 
\Cref{prop:support_not_split} gives an example of a probability measure on a compact Hausdorff space whose support is not split.
In summary, the notion of split support results in a theory that is more appealing than that of plain supports, but at the same time has more limited applicability.

\subsection{Definition and basic properties of supports}
\label{sec:supports_def}

A first definition of support for morphisms in Markov categories was given in \cite[Definition 13.20]{fritz2019synthetic}.
Here we introduce a more intuitive and more broadly applicable definition, while a close variant of the previous notion will be developed in \cref{sec:split_supports}.
Let us start with a motivating example.

\begin{example}\label{ex:first}
	In $\finstoch$, let $X$ be the three element set $\{a,b,c\}$, and consider the probability distribution $p$ on $X$ from the introduction, which is
	
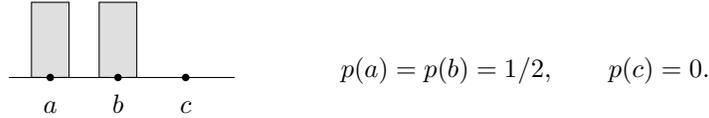
\[
		\begin{tikzpicture}[baseline,
			x={(0:1cm)},y={(90:1cm)},z={(0:1cm)}]
			\node (a) at (0.55,0,-2) {};
			\draw [fill=fillcolor!50] (a) -- ++(0.25,0,0) -- ++(0,1,0) -- ++(-0.5,0,0) -- ++(0,-1,0) -- (a);
			\node [circle,inner sep=1pt,fill=black,label=below:$a\strut$] at (a) {};
			\node (b) at (1.45,0,-2) {};
			\draw [fill=fillcolor!50] (b) -- ++(0.25,0,0) -- ++(0,1,0) -- ++(-0.5,0,0) -- ++(0,-1,0) -- (b);
			\node [circle,inner sep=1pt,fill=black,label=below:$b\strut$] at (b) {};
			\node (c) at (2.35,0,-2) {};
			\node [circle,inner sep=1pt,fill=black,label=below:$c\strut$] at (c) {};
			\draw (0,0,-2) -- (3,0,-2) ;
		\end{tikzpicture}
		\qquad\qquad
		p(a) = p(b) = 1/2, \qquad p(c) = 0 .
	\]
	The support of $p$ in the standard sense is the subset $\{a,b\} \subseteq X$.
	This subset is characterized by the universal property that a stochastic matrix $f \colon Z \to X$ factors across the inclusion $\{a,b\} \hookrightarrow X$ if and only if $f$ is absolutely continuous with respect to $p$.
	Indeed by \cref{ex:abs_cont_stoch}, the condition $f \ll p$ is equivalent to $f(c|z) = 0$ for all $z \in Z$, and this in turn is equivalent to $f$ factoring across $\{a,b\} \hookrightarrow X$.
\end{example}

In a general Markov category $\cC$, we can define the support of a morphism by means of a universal property generalizing the observation from this example.
For any morphism $p \colon A \to X$ in $\cC$, consider the functor $\cC(\ph,X)_{\ll p} \colon \cC^\op \to \set$ defined as the subfunctor of $\cC(\ph,X)$ on all morphisms absolutely continuous with respect to $p$, i.e.\ as
\begin{equation}\label{eq:supp_def}
\cC(Z, X)_{\ll p} \coloneqq \Set*[\big]{ f \colon Z \to X \given f \ll p }.
\end{equation}
Intuitively, a support amounts to a representation of this functor.
We also require such a representation to be compatible with the restriction to deterministic morphisms and with tensor products, as per the following definition.

\begin{definition}
	\label{def:support}
	Let $p \colon A \to X$ be a morphism in a Markov category $\cC$.
	\begin{enumerate}
		\item A \newterm{pre-support} for $p$ consists of an object $\Supp{}$ together with a \newterm{pre-support inclusion} $\suppinc{} \colon \Supp{} \to X$,
			which is a morphism such that post-composition with it gives bijections
			\begin{equation}
				\label{eq:supp_correspondence}
				\cC(\ph, S) \cong \cC(\ph, X)_{\ll \, p},
			\end{equation}
			which additionally restrict to bijections between deterministic morphisms on both sides.

		\item A \newterm{support} for $p$ is an object $\Supp{}$ together with a \newterm{support inclusion} $\suppinc{} \colon \Supp{} \to X$ such that, for every object $W\in \cC$, the object $W \otimes S$ is a pre-support for $\id_W \otimes p$ with pre-support inclusion given by $\id_W \otimes \suppinc{}$.
	\end{enumerate}
	We say that
	$\cC$ \newterm{has supports} if there exists a support for every morphism $p$ in $\cC$.
\end{definition}

The preservation by tensoring requirement facilitates a more elegant theory of supports than without.
The natural appearance of such conditions is a recurrent theme for Markov categories, see e.g.\ the definition of Kolmogorov products~\cite{fritzrischel2019zeroone}.
In both contexts, the preservation by tensoring is not the primary condition, but crucial as an auxiliary requirement which tends to hold in examples and ensure desirable properties.
Another recurrent theme is the requirement that such bijections should restrict to bijections between deterministic morphisms on both sides.
Since the support inclusion $\suppinc{}$ is the right-hand side counterpart of $\id_S$ in~\eqref{eq:supp_correspondence}, it follows in particular that $\suppinc{}$ is deterministic.

Applying Bijection~\eqref{eq:supp_correspondence} to $p$, which is an element of the right-hand side by the reflexivity inequality $p \ll p$, yields a unique $\suppfactor{p} \colon A \to S$ that gives us the factorization
\begin{equation}\label{eq:supp_inc_factor}
	p = \suppinc{} \comp \suppfactor{p},
\end{equation}
and we reserve the hat notation for this ``restriction of $p$ to its support''.
We typically denote a support $S$ and the support inclusion $\suppinc{}$ of a morphism $p$ without reference to $p$.
However, whenever confusion could arise, we indicate the $p$ by a subscript and write $\Supp{p}$ and $\suppinc{p}$ instead.

Let us emphasize again that \cref{def:support} differs from our earlier one \cite[Definition 13.20]{fritz2019synthetic}. 
There, the mapping-out property used to define the notion of a support in Markov categories is instead similar to the stronger notion of split supports studied in \cref{sec:split_supports}. 
A non-trivial example of a morphism for which the two definitions are inequivalent is provided in \cref{prop:support_not_split}.

For working with supports in practice, it is often helpful to use the following reformulation of the definition.

\begin{lemma}
	\label{lem:supp_crit}
	A morphism $\suppinc{} \colon \Supp{} \to X$ is a support inclusion for a morphism $p \colon A \to X$ if and only if it satisfies the following conditions:
	\begin{enumerate}
		\item\label{it:suppinc_det_mono} $\id_W \otimes \suppinc{}$ is a deterministic monomorphism for every $W \in \cC$.
		\item\label{it:suppinc_ac} $\suppinc{} \ll p$.
		\item\label{it:suppinc_fact} For every $Z, W \in \cC$ and $f \colon Z \to W \otimes X$, we have
			\begin{equation}\label{eq:supp_crit}
				f \ll \id_W \otimes p \quad \implies \quad f \text{ factors across } \id_W \otimes \suppinc{}.
			\end{equation}
	\end{enumerate}
\end{lemma}

\begin{proof}
	First let us assume that $\suppinc{}$ is a support inclusion for $p$.
	Note that we can also think of it as the result of applying Bijection \eqref{eq:supp_correspondence} to $\id_S$, which implies \cref{it:suppinc_ac}.

	Since this bijection restricts to deterministic morphisms (on both sides), the support inclusion must be itself deterministic.
	To show that it is a monomorphism, consider $f, g : W \to S$ such that $\suppinc{} \comp f = \suppinc{} \comp g$ holds.
	By the correspondence from \eqref{eq:supp_correspondence}, we have $\suppinc{} \comp f \ll p$, which implies $f = g$ by applying the bijection right-to-left. 
	Similarly, both properties follow for $\id_W \otimes \suppinc{}$ which is assumed to be a pre-support inclusion of $\id_W \otimes p$.
	Thus, \cref{it:suppinc_det_mono} follows.

	Finally, if $f \colon Z \to W \otimes X$ satisfies $f \ll \id_W \otimes p$, then the pre-support bijection \eqref{eq:supp_correspondence} for $\id_W \otimes p$ gives a unique $\bar{f} \colon Z \to W \otimes \Supp{}$ such that
	\[
		f = (\id_W \otimes \suppinc{}) \comp \bar{f}.
	\]

	Let us turn to the converse implication by first showing that $\suppinc{}$ is a pre-support inclusion for $p$.
	Since $\suppinc{}$ is a monomorphism, composing with it embeds $\cC(Z, \Supp{})$ into $\cC(Z, X)$ injectively.
	All morphisms in the image of this inclusion factor across $\suppinc{}$, so that applying \cref{lem:factor_ac} shows that they are all absolutely continuous with respect to $\suppinc{}$.
	As $\suppinc{} \ll p$ holds by assumption, these morphisms are all absolutely continuous with respect to $p$ as well and thus we can restrict the codomain of the injection to the subset $\cC(Z, X)_{\ll p}$ of $\cC(Z, X)$.
	Moreover, this embedding is surjective by the assumed condition \eqref{eq:supp_crit}.
	To show that it restricts to a bijection between deterministic morphisms on both sides, we need to prove that if a composite $\suppinc{} \comp g$ is deterministic, then so is $g$.
	This can be proven as follows
	\begin{equation}
		{%
			\tikzstyle{every picture}=[tikzfig]%
			\begin{tikzpicture}
				\begin{pgfonlayer}{nodelayer}
					\node [style=none] (0) at (-11.5, 0) {$=$};
					\node [style=morphism] (1) at (-15, 1.5) {$\suppinc{}$};
					\node [style=morphism] (2) at (-13, 1.5) {$\suppinc{}$};
					\node [style=morphism] (3) at (-15, 0) {$g$};
					\node [style=morphism] (4) at (-13, 0) {$g$};
					\node [style=bn] (5) at (-14, -1.5) {};
					\node [style=none] (6) at (-15, 2.5) {};
					\node [style=none] (7) at (-15, 3) {$X$};
					\node [style=none] (8) at (-13, 2.5) {};
					\node [style=none] (9) at (-13, 3) {$X$};
					\node [style=none] (10) at (-14, -2.25) {};
					\node [style=none] (11) at (-14, -2.75) {$Z$};
					\node [style=bn] (16) at (-9, 1.25) {};
					\node [style=none] (17) at (-10, 2.5) {};
					\node [style=none] (18) at (-10, 3) {$X$};
					\node [style=none] (19) at (-8, 2.5) {};
					\node [style=none] (20) at (-8, 3) {$X$};
					\node [style=none] (21) at (-9, -2.25) {};
					\node [style=none] (22) at (-9, -2.75) {$Z$};
					\node [style=morphism] (23) at (-9, 0.25) {$\suppinc{}$};
					\node [style=morphism] (24) at (-9, -1.25) {$g$};
					\node [style=none] (25) at (-10, 2.5) {};
					\node [style=none] (26) at (-8, 2.5) {};
					\node [style=none] (27) at (-6.5, 0) {$=$};
					\node [style=bn] (28) at (-4, -0.25) {};
					\node [style=none] (29) at (-5, 2.5) {};
					\node [style=none] (30) at (-5, 3) {$X$};
					\node [style=none] (31) at (-3, 2.5) {};
					\node [style=none] (32) at (-3, 3) {$X$};
					\node [style=none] (33) at (-4, -2.25) {};
					\node [style=none] (34) at (-4, -2.75) {$Z$};
					\node [style=morphism] (35) at (-5, 1.5) {$\suppinc{}$};
					\node [style=morphism] (36) at (-4, -1.25) {$g$};
					\node [style=none] (37) at (-5, 1) {};
					\node [style=none] (38) at (-3, 1) {};
					\node [style=morphism] (39) at (-3, 1.5) {$\suppinc{}$};
					\node [style=none] (40) at (0, 0) {$\implies$};
					\node [style=morphism] (41) at (3, 1.25) {$g$};
					\node [style=morphism] (42) at (5, 1.25) {$g$};
					\node [style=bn] (43) at (4, -0.75) {};
					\node [style=none] (44) at (4, -2) {};
					\node [style=none] (45) at (4, -2.5) {$Z$};
					\node [style=bn] (46) at (9, 0.5) {};
					\node [style=none] (47) at (9, -2) {};
					\node [style=none] (48) at (9, -2.5) {$Z$};
					\node [style=morphism] (49) at (9, -0.75) {$g$};
					\node [style=none] (50) at (8, 1.75) {};
					\node [style=none] (51) at (10, 1.75) {};
					\node [style=none] (52) at (3, 2.25) {};
					\node [style=none] (53) at (3, 2.75) {$S$};
					\node [style=none] (54) at (5, 2.25) {};
					\node [style=none] (55) at (5, 2.75) {$S$};
					\node [style=none] (56) at (3, 0.5) {};
					\node [style=none] (57) at (5, 0.5) {};
					\node [style=none] (58) at (8, 2.25) {};
					\node [style=none] (59) at (10, 2.25) {};
					\node [style=none] (60) at (8, 2.75) {$S$};
					\node [style=none] (61) at (10, 2.75) {$S$};
					\node [style=none] (62) at (6.5, 0) {$=$};
					\node [style=none] (63) at (-15, -0.25) {};
					\node [style=none] (64) at (-13, -0.25) {};
				\end{pgfonlayer}
				\begin{pgfonlayer}{edgelayer}
					\draw (10.center) to (5);
					\draw (3) to (1);
					\draw (1) to (6.center);
					\draw (4) to (2);
					\draw (2) to (8.center);
					\draw (21.center) to (16);
					\draw [in=-90, out=165] (16) to (25.center);
					\draw [in=-90, out=15] (16) to (26.center);
					\draw (26.center) to (19.center);
					\draw (25.center) to (17.center);
					\draw (33.center) to (28);
					\draw [in=-90, out=165] (28) to (37.center);
					\draw [in=-90, out=15] (28) to (38.center);
					\draw (38.center) to (31.center);
					\draw (37.center) to (29.center);
					\draw (44.center) to (43);
					\draw (47.center) to (46);
					\draw [in=-90, out=165] (46) to (50.center);
					\draw [in=-90, out=15] (46) to (51.center);
					\draw (41) to (52.center);
					\draw [in=-90, out=165] (43) to (56.center);
					\draw [in=-90, out=15] (43) to (57.center);
					\draw (57.center) to (42);
					\draw (42) to (54.center);
					\draw (56.center) to (41);
					\draw (51.center) to (59.center);
					\draw (50.center) to (58.center);
					\draw [in=-90, out=165] (5) to (63.center);
					\draw [in=-90, out=15] (5) to (64.center);
				\end{pgfonlayer}
			\end{tikzpicture}
		}%
	\end{equation}
	where we use the facts that $\suppinc{}$ is deterministic (to get the second equality) and that $\suppinc{} \otimes \suppinc{}$ is a monomorphism (to get the implication).
	The latter statement holds by virtue of $\id_S \otimes \suppinc{}$ being a monomorphism.

	Finally, the arguments of the previous paragraph apply likewise to $\id_W \otimes \suppinc{}$ in place of $\suppinc{}$, since all three conditions are stable under tensoring with arbitrary $W$ (the second by \cref{prop:ac_tensor}).
\end{proof}

Let us now use the above characterization to provide examples of supports in three Markov categories: $\finstoch$ describing discrete probability theory over finite sample spaces, $\tychstoch$ in which Markov kernels have to be continuous, and $\setmulti$ giving a model of possibilistic indeterminism.
In each of them, the determinism property for a purported support inclusion follows by the characterization of deterministic morphisms in each category \cite{fritz2019synthetic}.
The monomorphism property follows by definition since all are inclusions.
Thus, we only provide details on how to show $\suppinc{} \ll p$ as well as the factorization condition.
Verifying both of these properties can be done thanks to the characterization of absolute continuity in each category.

\begin{example}
	\label{ex:support_finstoch}
	In $\finstoch$, every morphism $p \colon A \to X$ has a support, namely the set
	\begin{equation}
	\Supp{} = \Set*[\big]{ x \in X  \given  \exists \, a \in A  \,:\,  p(x | a) > 0 } ,
\end{equation}
with $\suppinc{}$ given by the inclusion map into $X$.
Using \cref{ex:abs_cont_stoch}, we indeed have $\suppinc{} \ll p$.
An arbitrary $f \colon Z \to W \otimes X$ factors across $\id_W \otimes \suppinc{}$ if and only if we have $f(w,x|z) = 0$ for all $w \in W$ and $x \in X \setminus \Supp{}$ and $z \in Z$.
This indeed holds if $f \ll \id_W \otimes p$ by the same characterization of absolute continuity.
\end{example}

\begin{example}
	\label{ex:top_support}
	In the category $\tychstoch$, the support is given by the topological support, i.e.\ the smallest closed set of full measure.
	In more detail, let $p \colon A\to X$ be any morphism and define the set $\Supp{} \subseteq X$ as in \cref{eq:top_sup} as
	\begin{equation}\label{eq:top_sup_2}
		\Supp{} \coloneqq \bigcap \, \Set{ C \subseteq X \textrm{ closed} \given p(C|a) = 1 \textrm{ for all } a \in A }.
	\end{equation}
	Consider the deterministic inclusion $\suppinc{} \colon \Supp{} \to X$, and let us apply \cref{lem:supp_crit} to show that it is a support inclusion.
	\begin{itemize}
		\item We have $\suppinc{} \ll p$, and even $\suppinc{}\acsim p$ (compare \cref{lem:inc_sim_p} below), as both $\suppinc{}$-almost sure equality and $p$-almost sure equality are equivalent to equality of kernels restricted to $S$ by \cref{prop:ase_top}.

		\item Let now $f \colon Z\to W \otimes X$ be any morphism absolutely continuous with respect to $\id_W \otimes p$. 
			By \cref{prop:ase_top} again, this means that for all $z\in Z$ we have $f(W \times (X\setminus \Supp{})|z)=0$.
			That is, $f$ can be considered as a morphism of type $Z \to W \otimes \Supp{}$, and this provides the desired factorization.
	\end{itemize}
	By \cref{lem:supp_crit}, the topological support of $p$ thus coincides with the support of $p$ in our sense.
\end{example}

\begin{example}\label{ex:support_setmulti}
	In $\setmulti$, every morphism $p \colon A \to X$ has a support.
	In particular, the support inclusion $\suppinc{}$ is the inclusion of the set
	\begin{equation}
		\Supp{} \coloneqq \bigcup_{a \in A} p(a)
	\end{equation}
	into $X$, where $p(a)$ is the image of $a$ under $p$.
	The property $\suppinc{} \ll p$ and the factorization condition hold by \cref{ex:ac_setmulti}.
\end{example}

Let us return to the general theory for the moment.

\begin{lemma}
	\label{lem:inc_sim_p}
	For every morphism $p \colon A \to X$ with a pre-support, we have
	\begin{equation}
		\suppinc{} \acsim p.
	\end{equation}
	where $\acsim$ denotes the absolute bicontinuity relation (\cref{def:abs_bicont}).
\end{lemma}

\begin{proof}
	The relation $\suppinc{} \gg p$ follows directly from the factorization in \cref{eq:supp_inc_factor} in view of \cref{lem:factor_ac}.
	The opposite one, $\suppinc{} \ll p$, holds because $\suppinc{}$ is defined as an element of the set $\cC(\Supp{}, X)_{\ll p}$ from the right-hand side of Bijection \eqref{eq:supp_correspondence}. 
\end{proof}

This fact yields the following characteristic connection between supports and almost sure equality, which may elucidate the definition of supports further.

\begin{lemma}[A.s.\ equality is restriction to the support]
	\label{lem:supp_asfaithful}
	Let $p$ have a pre-support with pre-support inclusion $\suppinc{}$.
	Then we have 
	\begin{equation}\label{eq:supp_asfaithful}
		{%
			\tikzstyle{every picture}=[tikzfig]%
			\begin{tikzpicture}
				\begin{pgfonlayer}{nodelayer}
					\node [style=none] (202) at (0, 0) {$\iff$};
					\node [style=none] (251) at (-10, 0.75) {};
					\node [style=none] (252) at (-8.5, 1) {};
					\node [style=none] (253) at (-10.5, 2.25) {};
					\node [style=none] (254) at (-10.5, 2.75) {$Y$};
					\node [style=bn] (255) at (-9.25, 0) {};
					\node [style=none] (256) at (-9.25, -2) {};
					\node [style=none] (257) at (-9.25, -2.5) {$A$};
					\node [style=none] (258) at (-7, 0) {$=$};
					\node [style=none] (259) at (-8.5, 2.75) {$X$};
					\node [style=morphism] (260) at (-10.5, 1.25) {$\,\;\; f \,\;\;$};
					\node [style=none] (261) at (-8.5, 2.25) {};
					\node [style=morphism] (262) at (-9.25, -1) {$p$};
					\node [style=none] (263) at (-11, -2) {};
					\node [style=none] (264) at (-11, -2.5) {$W$};
					\node [style=none] (265) at (-11, 1.25) {};
					\node [style=none] (266) at (-11, -0.5) {};
					\node [style=none] (267) at (-4.5, 0.75) {};
					\node [style=none] (268) at (-3, 1) {};
					\node [style=none] (269) at (-5, 2.25) {};
					\node [style=none] (270) at (-5, 2.75) {$Y$};
					\node [style=bn] (271) at (-3.75, 0) {};
					\node [style=none] (272) at (-3.75, -2) {};
					\node [style=none] (273) at (-3.75, -2.5) {$A$};
					\node [style=none] (274) at (-3, 2.75) {$X$};
					\node [style=morphism] (275) at (-5, 1.25) {$\,\;\; g \,\;\;$};
					\node [style=none] (276) at (-3, 2.25) {};
					\node [style=morphism] (277) at (-3.75, -1) {$p$};
					\node [style=none] (278) at (-5.5, -2) {};
					\node [style=none] (279) at (-5.5, -2.5) {$W$};
					\node [style=none] (280) at (-5.5, 1.25) {};
					\node [style=none] (281) at (-5.5, -0.5) {};
					\node [style=none] (284) at (-10, 1.25) {};
					\node [style=none] (285) at (-4.5, 1.25) {};
					\node [style=none] (288) at (3.5, 2.25) {};
					\node [style=none] (289) at (3.5, 2.75) {$Y$};
					\node [style=none] (290) at (4, -2.5) {$\Supp{}$};
					\node [style=morphism] (292) at (3.5, 1) {$\,\;\; f \,\;\;$};
					\node [style=none] (294) at (3, -2) {};
					\node [style=none] (295) at (3, -2.5) {$W$};
					\node [style=none] (296) at (3, 1) {};
					\node [style=none] (298) at (9, 2.25) {};
					\node [style=none] (299) at (9, 2.75) {$Y$};
					\node [style=none] (300) at (9.5, -2.5) {$\Supp{}$};
					\node [style=morphism] (301) at (9, 1) {$\,\;\; g \,\;\;$};
					\node [style=none] (302) at (8.5, -2) {};
					\node [style=none] (303) at (8.5, -2.5) {$W$};
					\node [style=none] (304) at (8.5, 1) {};
					\node [style=none] (305) at (4, 1) {};
					\node [style=none] (306) at (9.5, 1) {};
					\node [style=morphism] (307) at (4, -0.75) {$\suppinc{}$};
					\node [style=morphism] (308) at (9.5, -0.75) {$\suppinc{}$};
					\node [style=none] (309) at (4, -2) {};
					\node [style=none] (310) at (9.5, -2) {};
					\node [style=none] (311) at (6.25, 0) {$=$};
					\node [style=none] (312) at (-11.75, -2.5) {};
				\end{pgfonlayer}
				\begin{pgfonlayer}{edgelayer}
					\draw [in=-90, out=165] (255) to (251.center);
					\draw [in=-90, out=15] (255) to (252.center);
					\draw (260) to (253.center);
					\draw (252.center) to (261.center);
					\draw (256.center) to (262);
					\draw [style=protected] (263.center) to (266.center);
					\draw [style=protected, in=-90, out=90, looseness=1.25] (266.center) to (265.center);
					\draw [in=270, out=90] (262) to (255);
					\draw [in=-90, out=165] (271) to (267.center);
					\draw [in=-90, out=15] (271) to (268.center);
					\draw (275) to (269.center);
					\draw (268.center) to (276.center);
					\draw (272.center) to (277);
					\draw [style=protected] (278.center) to (281.center);
					\draw [style=protected, in=-90, out=90, looseness=1.25] (281.center) to (280.center);
					\draw (277) to (271);
					\draw (251.center) to (284.center);
					\draw (267.center) to (285.center);
					\draw (292) to (288.center);
					\draw (301) to (298.center);
					\draw (309.center) to (307);
					\draw (307) to (305.center);
					\draw (294.center) to (296.center);
					\draw (302.center) to (304.center);
					\draw (308) to (306.center);
					\draw (310.center) to (308);
				\end{pgfonlayer}
			\end{tikzpicture}
		}%
	\end{equation}
	for all $f, g \colon W \otimes X \to Y$ with arbitrary $W$ and $Y$.
\end{lemma}

\begin{proof}
	The left-hand side of Equivalence \eqref{eq:supp_asfaithful} is just $f \ase{p} g$.
	If this holds, then we also get $f \ase{\suppinc{}} g$ by \cref{lem:inc_sim_p}, and the right-hand side of Equivalence \eqref{eq:supp_asfaithful} follows by marginalization.
	Conversely, if the right-hand side of \eqref{eq:supp_asfaithful} holds, then we have $f \ase{\suppinc{}} g$ by the fact that $\suppinc{}$ is deterministic via \cref{prop:ase_props} \ref{it:ase_deterministic}.
	We then conclude by another application of \cref{lem:inc_sim_p}.		
\end{proof}

\begin{proposition}[Absolute bicontinuity and supports]
	\label{prop:abs_bicont_same_supp}
	Let $p \colon A \to X$ and $q \colon B \to X$ have the same codomain.
	Then:
	\begin{enumerate}
		\item \label{it:same_support} If $p$ and $q$ have pre-supports with the same pre-support inclusion, then we have $p \acsim q$.
		\item \label{it:abs_bicont} If $p \acsim q$ holds and $p$ has a (pre-)support, then $q$ has the same (pre-)support.
	\end{enumerate}
\end{proposition}

\begin{proof}\hfill 
	\begin{enumerate}
		\item Since $\suppinc{p} = \suppinc{q}$, we just use \cref{lem:inc_sim_p}.
		\item The condition $p \acsim q$ entails 
			\begin{equation}
				\cC(\ph, X)_{\ll \, p} = \cC(\ph, X)_{\ll \, q}
			\end{equation}
			and so the result for pre-supports follows by the definition of pre-supports.
			If the pre-support of $p$ is actually a support, then we can run the same argument because ${\id_W \otimes p \acsim \id_W \otimes q}$ is ensured by \cref{cor:abc_tensor}.
			\qedhere
	\end{enumerate}
\end{proof}

\begin{example}
	\label{ex:supp_split_epi}
	If $p \colon A \to X$ is a split epimorphism, then $p \acsim \id_X$ holds by \cref{lem:epi_ac_max}, and similarly upon tensoring by any identity.
	Therefore the support inclusion of any such split epimorphism is just $\id_X$.
\end{example}

\begin{proposition}[Supports of composite morphism]\label{prop:support_pull}
	Consider an arbitrary morphism ${p \colon A \to T}$ and a deterministic split monomorphism $\iota \colon T \to X$ with left inverse $\pi \colon X \to T$.
	Then we have the following:
	\begin{enumerate}
		\item \label{it:supp_split_epi_1} If $p$ has a support, then so does $\iota \comp p$, and its support inclusion is
			\begin{equation}
				\suppinc{\iota\comp p} = \iota \comp \suppinc{p}.
			\end{equation}
		\item \label{it:supp_split_epi_2} If $\iota \comp p$ has a support, then so does $p$, and its support inclusion is
			\begin{equation}
				\suppinc{p} = \pi \comp \suppinc{\iota\comp p}.
			\end{equation}
	\end{enumerate}
\end{proposition}
\begin{proof}
	We employ \cref{lem:supp_crit} in both directions.
	\begin{enumerate}
		\item We show that the composite $\iota \comp \suppinc{p}$ is indeed a support inclusion $\suppinc{\iota\comp p}$.
			First, $\iota \comp \suppinc{p}$ is a deterministic monomorphism as a composite of two deterministic monomorphisms.
			The required inequality $\iota \comp \suppinc{p} \ll \iota\comp p$ holds as a consequence of $\suppinc{p} \ll p$ and \cref{lem:ac_mon_comp_split_mono}.
			Second, if $f \colon Z \to X$ is absolutely continuous with respect to $\iota\comp p$, then we obtain $\iota\comp \pi \ase{f} \id_X$ from $\iota\comp \pi \ase{\iota\comp p} \id_X$.
			This implies $\iota\comp \pi\comp f = f$.
			By \cref{lem:ac_mon_comp_split_mono}, we therefore obtain $\pi\comp f \ll p$ from the assumed $f \ll \iota\comp p$.
			The assumption that $\suppinc{p}$ is a support inclusion implies now that $\pi\comp f$ factors across $\suppinc{p}$.
			But then of course also $f = \iota\comp \pi\comp f$ factors across $\iota \comp \suppinc{p}$.

			Finally, since similarly $\id_W \otimes \iota$ is a deterministic split monomorphism with left inverse $\id_W \otimes \pi$ for every $W$, the same arguments apply to $\id_W \otimes \iota \comp \suppinc{p}$ in place of $\iota \comp \suppinc{p}$, and we are done.
		\item We show that the composite $\pi \comp \suppinc{\iota\comp p}$ is indeed a support inclusion $\suppinc{p}$.
			First, we get $\iota\comp \pi \comp \suppinc{\iota\comp p} = \suppinc{\iota\comp p}$ from $\suppinc{\iota \comp p} \ll \iota \comp p \ll \iota$ and $\iota \comp \pi \ase{\iota} \id_X$.
			Therefore our putative support inclusion for $p$ satisfies $\iota \comp \suppinc{p} = \suppinc{\iota\comp p}$.
			In particular, $\suppinc{p}$ is a deterministic monomorphism.\footnote{To see that it is deterministic, note that it is enough to prove the determinism equation composed with $\iota \otimes \iota$, and use the assumption that both $\iota$ and $\iota \comp \suppinc{p}$ are deterministic.}
			The required inequality $\suppinc{p} \ll p$ holds as a consequence of $\iota \comp \suppinc{p} \ll \iota\comp p$ and \cref{lem:ac_mon_comp_split_mono} again.

			Next, let $f \colon Z \to T$ be absolutely continuous with respect to $p$.
			Then $\iota \comp f \ll \iota\comp p$ by \cref{lem:ac_mon_comp_split_mono}, and hence $\iota\comp f$ factors across $\suppinc{\iota\comp p} = \iota \comp \suppinc{p}$.
			The claim now follows as all assumptions are again stable under tensoring with $\id_W$ for any object $W$.
			\qedhere
	\end{enumerate}
\end{proof}

\begin{example}[Copying does not change the support]\label{ex:support_poc}
	As a special case of \cref{prop:support_pull}, consider $X = T \otimes T$ and let
	\begin{equation}
		q \coloneqq \cop_T {}\comp p
	\end{equation}
	be the output-copy version of a morphism $p \colon A \to T$ with support $\Supp{p}$.
	Since $\cop_T$ is a deterministic split monomorphism with left inverse $\id_T \otimes \discard_T$, we can apply \cref{prop:support_pull} \ref{it:supp_split_epi_1} to obtain that $q$ has 
	\begin{equation}\label{eq:support_poc}
		{%
			\tikzstyle{every picture}=[tikzfig]%
			\begin{tikzpicture}
				\begin{pgfonlayer}{nodelayer}
					\node [style=bn] (9) at (0, 0) {};
					\node [style=none] (10) at (0, -2) {};
					\node [style=none] (13) at (-1, 1.25) {};
					\node [style=none] (14) at (1, 1.25) {};
					\node [style=morphism] (15) at (0, -1) {$\suppinc{}$};
					\node [style=none] (16) at (-1, 1.25) {};
					\node [style=none] (17) at (-1, 1.75) {$T$};
					\node [style=none] (18) at (1, 1.75) {$T$};
					\node [style=none] (19) at (1, 1.25) {};
					\node [style=none] (20) at (0, -2.5) {$\Supp{p}$};
				\end{pgfonlayer}
				\begin{pgfonlayer}{edgelayer}
					\draw (10.center) to (9);
					\draw [in=-90, out=165] (9) to (13.center);
					\draw [in=-90, out=15] (9) to (14.center);
					\draw (13.center) to (16.center);
					\draw (14.center) to (19.center);
				\end{pgfonlayer}
			\end{tikzpicture}
		}%
	\end{equation}
	as the support inclusion, where $\suppinc{}$ is the support inclusion of $p$.
\end{example}

In light of \cref{lem:supp_crit}, it seems like an interesting question to ask which deterministic monomorphisms precisely can arise as support inclusions.
While we do not know the answer, at least the following is an obvious consequence of \cref{lem:supp_crit}, which also arises as a special case of \cref{prop:support_pull} by taking $p$ to be the identity morphism.

\begin{corollary}
	\label{cor:split_mono_supp}
	Every deterministic split monomorphism is its own support inclusion.
\end{corollary}

\begin{example}
	If an object $X$ has a distribution object $PX$ \cite{fritz2023representable}, then we have a deterministic split monomorphism $\delta_X \colon X \to PX$ with left inverse given by the sampling map ${\samp_X \colon PX \to X}$.
	Thus, $\delta_X$ is its own support inclusion.
	On the other hand, the support inclusion of $\samp_X$ is given by $\id_X$ thanks to \cref{ex:supp_split_epi}.
\end{example}

In general, not every deterministic monomorphism is a support inclusion; \cref{prop:coregular_supports} will provide counterexamples.
However, the following holds.
\begin{lemma}\label{lem:factoring_suppincs}
	Let $\cC$ be a causal Markov category.
	Consider two arbitrary morphisms $p\colon A \to X$ and $q\colon B \to X$ having supports, and let $\iota\colon \Supp{p}\to \Supp{q}$ be a morphism satisfying
	\begin{equation}\label{eq:factoring_suppincs}
		\suppinc{q}\comp \iota = \suppinc{p}.
	\end{equation}
	Then $\iota$ is its own support inclusion: $\iota=\suppinc{\iota}$.
\end{lemma}
\begin{proof}
	It follows from a direct check that $\iota$ is a deterministic monomorphism. 
	It then suffices to show that $\iota$ is a pre-support, as we can then reapply the proof to $\id \otimes \iota$.
	We want to show that $f \ll \iota$ implies that $f$ factors across $\iota$. 
	By  \cref{lem:ac_mon_comp}, $\suppinc{q} \comp f \ll \suppinc{q} \comp \iota = \suppinc{p}$. 
	Then, there exists $\hat{f}$ such that $\suppinc{q}\comp f = \suppinc{p} \comp \hat{f} = \suppinc{q} \comp \iota \comp \hat{f}$.
	Since $\suppinc{q}$ is a monomorphism, we conclude that $f = \iota \comp \hat{f}$.
	If $f$ is deterministic, then $\hat{f}$ is also deterministic, since $\iota$ is a deterministic monomorphism.
	Hence, the statement is shown.
\end{proof}

\begin{corollary}
	\label{cor:support_inclusion_self}
	Every support inclusion is its own support inclusion.
\end{corollary}

\begin{proof}
	From \cref{lem:factoring_suppincs} by taking $q = \suppinc{q} = \id_X$.
\end{proof}

Although properties of cartesian monoidal categories viewed as Markov categories are usually trivial, the existence of supports is a non-trivial requirement. 

\begin{proposition}
	\label{prop:coregular_supports}
	Let $\cC$ be a coregular category\footnote{A category $\cC$ is \emph{coregular} if $\cC^\op$ is a regular category in the standard sense, meaning that $\cC$ has all finite colimits, every morphism factors into an epimorphism followed by a regular monomorphism, and this factorization is stable under pushout.} with products that preserve regular monomorphisms.
	Then every morphism $p \colon A \to X$ has a support determined by the image factorization
	\begin{equation}
		\begin{tikzcd}
			A \ar[rr, bend left, "p"] \ar[r, two heads] & \Supp{p} \ar[r, hook, "\suppinc{p}"'] & X
		\end{tikzcd}
	\end{equation}
	where $\suppinc{}$ is a regular monomorphism.
\end{proposition}

\begin{proof}
	We verify the conditions of \cref{lem:supp_crit}.
	The assumption that products of regular monomorphisms are again regular monomorphisms implies in particular that $\id_W \times \suppinc{}$ is a monomorphism again.

	The regular monomorphism $\suppinc{p}$ is the equalizer of the cokernel pair $X \rightrightarrows X +_A X$.
	Combining this with \cref{ex:ac_cartesian} shows that $p \ll q$ is equivalent to $\id_W \times p$ factoring across $\suppinc{\id_W \times q}$ for all $W$.
	Since $\suppinc{\id_W \times p} = \id_W \times \suppinc{p}$ by the stability of regular monomorphisms under products, 
	we conclude that $p \ll q$ holds if and only if $p$ factors across $\suppinc{q}$.
	In particular, we have $\suppinc{p} \ll p$.

	Finally, let $f \colon Z \to W \times X$ be any morphism with $f \ll \id_W \times p$.
	Then $f$ factors across $\suppinc{\id_W \times p} = \id_W \times \suppinc{p}$, as was to be shown.
\end{proof}

Another question is whether supports interact well with the monoidal structure.
This is indeed the case, as we now prove.
\begin{theorem}\label{thm:supp_multiplicativity}
	Let $p \colon A \to X$ and $q\colon B \to Y$ be two morphisms in a Markov category $\cC$. If $p$, $q$ and $p\otimes q$ have supports, then we have $\Supp{p\otimes q}\cong \Supp{p} \otimes \Supp{q}$, and $\suppinc{p} \otimes \suppinc{q}$ is a support inclusion for $p\otimes q$.
\end{theorem}
\begin{proof}
	To show the statement, we construct an explicit morphism $f \colon \Supp{p}\otimes \Supp{q} \to \Supp{p\otimes q}$ and its inverse, which will respect the support inclusions by construction.

	By \cref{lem:inc_sim_p}, we have $p \acsim \suppinc{p}$ and $q \acsim \suppinc{q}$. By \cref{prop:ac_tensor}, we conclude that $p \otimes q \acsim \suppinc{p} \otimes \suppinc{q}$.
	In particular, by the factorization property \eqref{eq:supp_crit} of the support $\Supp{p\otimes q}$ of $p\otimes q$, we obtain a morphism $f\colon \Supp{p}\otimes \Supp{q} \to \Supp{p\otimes q}$ satisfying 
	\begin{equation}\label{eq:f_supp_pq}
		{%
			\tikzstyle{every picture}=[tikzfig]%
			\begin{tikzpicture}
				\begin{pgfonlayer}{nodelayer}
					\node [style=none] (409) at (-2.75, 0.5) {};
					\node [style=none] (410) at (-3.25, 1.75) {};
					\node [style=none] (411) at (-2.75, 0) {};
					\node [style=none] (412) at (-2.75, -0.75) {};
					\node [style=none] (413) at (-3.25, 0.75) {};
					\node [style=none] (414) at (-3.25, -0.75) {};
					\node [style=none] (415) at (-2.25, -0.75) {};
					\node [style=none] (416) at (-3.25, -1.75) {};
					\node [style=none] (417) at (-2.25, -1.75) {};
					\node [style=none] (418) at (-2.25, 1.75) {};
					\node [style=none] (419) at (-2.25, 0.75) {};
					\node [style=none] (420) at (2.25, 1.75) {};
					\node [style=none] (421) at (2.25, -1.25) {};
					\node [style=none] (422) at (2.25, -1.25) {};
					\node [style=none] (423) at (4, -1.25) {};
					\node [style=none] (424) at (2.25, -1.75) {};
					\node [style=none] (425) at (4, -1.75) {};
					\node [style=none] (426) at (0, 0) {$=$};
					\node [style=none] (427) at (4, 1.75) {};
					\node [style=none] (428) at (4, -1.25) {};
					\node [style=morphism] (429) at (2.25, 0) {$\suppinc{p}$};
					\node [style=morphism] (430) at (4, 0) {$\suppinc{q}$};
					\node [style=morphism] (431) at (-2.75, -0.75) {$\,\;\; f \,\;\;$};
					\node [style=morphism] (432) at (-2.75, 0.75) {$\,\suppinc{p\otimes q}\,$};
					\node [style=none] (433) at (-3.25, 2.25) {$X$};
					\node [style=none] (434) at (-2.25, 2.25) {$Y$};
					\node [style=none] (435) at (-3.25, -2.25) {$\Supp{p}$};
					\node [style=none] (436) at (-2.25, -2.25) {$\Supp{q}$};
					\node [style=none] (437) at (2.25, -2.25) {$\Supp{p}$};
					\node [style=none] (438) at (4, -2.25) {$\Supp{q}$};
					\node [style=none] (439) at (2.25, 2.25) {$X$};
					\node [style=none] (440) at (4, 2.25) {$Y$};
				\end{pgfonlayer}
				\begin{pgfonlayer}{edgelayer}
					\draw (411.center) to (409.center);
					\draw (413.center) to (410.center);
					\draw (411.center) to (412.center);
					\draw (414.center) to (416.center);
					\draw (415.center) to (417.center);
					\draw (419.center) to (418.center);
					\draw (421.center) to (420.center);
					\draw (422.center) to (424.center);
					\draw (423.center) to (425.center);
					\draw (428.center) to (427.center);
				\end{pgfonlayer}
			\end{tikzpicture}
		}%
	\end{equation} 
	To construct its inverse, note that we have $\suppinc{p\otimes q} \acsim p \otimes q \ll p\otimes \id_Y \acsim \suppinc{p}\otimes \id_Y$ by \cref{lem:inc_sim_p,prop:ac_tensor}. 
	We thus obtain $g\colon \Supp{p\otimes q} \to \Supp{p} \otimes Y$ from the factorization property of the support of $p$. 
	Symmetrically, we can define $h\colon \Supp{p\otimes q} \to X\otimes \Supp{q}$, and these morphisms satisfy 
	\begin{equation}\label{eq:gh_supp_pq}
		{%
			\tikzstyle{every picture}=[tikzfig]%
			\begin{tikzpicture}
				\begin{pgfonlayer}{nodelayer}
					\node [style=none] (237) at (8.5, -1) {};
					\node [style=none] (238) at (-3, 1.75) {};
					\node [style=none] (244) at (-3, -0.75) {};
					\node [style=morphism] (247) at (-2.5, -0.75) {$\,\;\;g\,\;\;$};
					\node [style=none] (248) at (-2, -0.75) {};
					\node [style=none] (249) at (-2, 1.75) {};
					\node [style=none] (250) at (9, 1.75) {};
					\node [style=none] (252) at (9, -0.75) {};
					\node [style=morphism] (253) at (8.5, -0.75) {$\,\;\;h\,\;\;$};
					\node [style=none] (254) at (8, -0.75) {};
					\node [style=none] (255) at (8, 1.75) {};
					\node [style=morphism] (281) at (-3, 0.75) {$\suppinc{p}$};
					\node [style=morphism] (282) at (9, 0.75) {$\suppinc{q}$};
					\node [style=none] (304) at (0, 0) {$=$};
					\node [style=none] (305) at (3, -0.25) {};
					\node [style=none] (307) at (2.5, 1.75) {};
					\node [style=none] (308) at (3, -1.75) {};
					\node [style=none] (310) at (2.5, 0) {};
					\node [style=none] (325) at (3.5, 1.75) {};
					\node [style=none] (326) at (3.5, 0) {};
					\node [style=morphism] (349) at (3, 0) {$\,\suppinc{p\otimes q}\,$};
					\node [style=none] (350) at (8.5, -1.75) {};
					\node [style=none] (351) at (-2.5, -1) {};
					\node [style=none] (352) at (-2.5, -1.75) {};
					\node [style=none] (353) at (6, 0) {$=$};
					\node [style=none] (354) at (-3, 2.25) {$X$};
					\node [style=none] (355) at (8, 2.25) {$X$};
					\node [style=none] (356) at (2.5, 2.25) {$X$};
					\node [style=none] (357) at (-2, 2.25) {$Y$};
					\node [style=none] (358) at (3.5, 2.25) {$Y$};
					\node [style=none] (359) at (9, 2.25) {$Y$};
					\node [style=none] (360) at (-2.5, -2.25) {$\Supp{p\otimes q}$};
					\node [style=none] (361) at (8.5, -2.25) {$\Supp{p\otimes q}$};
					\node [style=none] (362) at (3, -2.25) {$\Supp{p\otimes q}$};
				\end{pgfonlayer}
				\begin{pgfonlayer}{edgelayer}
					\draw (244.center) to (238.center);
					\draw (249.center) to (248.center);
					\draw (252.center) to (250.center);
					\draw (255.center) to (254.center);
					\draw (308.center) to (305.center);
					\draw (310.center) to (307.center);
					\draw (326.center) to (325.center);
					\draw (237.center) to (350.center);
					\draw (351.center) to (352.center);
				\end{pgfonlayer}
			\end{tikzpicture}
		}%
	\end{equation}
	We now claim that the inverse of $f$ is given by 
	\begin{equation}\label{eq:supp_multiplicativity3}
		{%
			\tikzstyle{every picture}=[tikzfig]%
			\begin{tikzpicture}
				\begin{pgfonlayer}{nodelayer}
					\node [style=none] (236) at (-1.25, 0.5) {};
					\node [style=none] (237) at (1.25, 0.5) {};
					\node [style=none] (238) at (-1.75, 2.25) {};
					\node [style=none] (239) at (-1.75, 2.75) {$\Supp{p}$};
					\node [style=bn] (240) at (0, -0.5) {};
					\node [style=none] (241) at (0, -1.5) {};
					\node [style=none] (242) at (0, -2) {$\Supp{p\otimes q}$};
					\node [style=none] (244) at (-1.75, 0.75) {};
					\node [style=none] (248) at (-0.75, 0.75) {};
					\node [style=bn] (249) at (-0.75, 1.75) {};
					\node [style=none] (250) at (1.75, 2.25) {};
					\node [style=none] (251) at (1.75, 2.75) {$\Supp{q}$};
					\node [style=none] (252) at (1.75, 0.75) {};
					\node [style=none] (254) at (0.75, 0.75) {};
					\node [style=bn] (255) at (0.75, 1.75) {};
					\node [style=morphism] (256) at (-1.25, 0.75) {$\,\;\;g\,\;\;$};
					\node [style=morphism] (257) at (1.25, 0.75) {$\,\;\;h\,\;\;$};
				\end{pgfonlayer}
				\begin{pgfonlayer}{edgelayer}
					\draw [in=-90, out=165] (240) to (236.center);
					\draw [in=-90, out=15] (240) to (237.center);
					\draw (244.center) to (238.center);
					\draw (240) to (241.center);
					\draw (249) to (248.center);
					\draw (252.center) to (250.center);
					\draw (255) to (254.center);
				\end{pgfonlayer}
			\end{tikzpicture}
		}%
	\end{equation}
	This follows by a direct computation as follows.
	Using \cref{eq:f_supp_pq,eq:gh_supp_pq} above and the fact that $\suppinc{p}\otimes \suppinc{q}$ and $\suppinc{p\otimes q}$ are monomorphisms\footnote{Monicity of the former follows from the fact that both $\suppinc{p}\otimes \id_Y$ and $\id_{\Supp{p}}\otimes \suppinc{q}$ are monomorphisms.}:
	On the one hand, we have 
	\begin{equation}
		{%
			\tikzstyle{every picture}=[tikzfig]%
			\begin{tikzpicture}
				\begin{pgfonlayer}{nodelayer}
					\node [style=none] (236) at (-5, 0.25) {};
					\node [style=none] (237) at (-2.5, 0.25) {};
					\node [style=none] (238) at (-5.5, 3.25) {};
					\node [style=bn] (240) at (-3.75, -0.75) {};
					\node [style=none] (241) at (-3.75, -2) {};
					\node [style=none] (244) at (-5.5, 0.5) {};
					\node [style=morphism] (247) at (-5, 0.5) {$\,\;\;g\,\;\;$};
					\node [style=none] (248) at (-4.5, 0.5) {};
					\node [style=bn] (249) at (-4.5, 1.5) {};
					\node [style=none] (250) at (-2, 3.25) {};
					\node [style=none] (252) at (-2, 0.5) {};
					\node [style=morphism] (253) at (-2.5, 0.5) {$\,\;\;h\,\;\;$};
					\node [style=none] (254) at (-3, 0.5) {};
					\node [style=bn] (255) at (-3, 1.5) {};
					\node [style=none] (257) at (-4.25, -2) {};
					\node [style=none] (259) at (-3.25, -2) {};
					\node [style=none] (260) at (-4.25, -3.25) {};
					\node [style=none] (261) at (-3.25, -3.25) {};
					\node [style=morphism] (281) at (-5.5, 2) {$\suppinc{p}$};
					\node [style=morphism] (282) at (-2, 2) {$\suppinc{q}$};
					\node [style=none] (283) at (2.5, 0.5) {};
					\node [style=none] (284) at (5, 0.5) {};
					\node [style=none] (285) at (2, 3.25) {};
					\node [style=bn] (286) at (3.75, -0.5) {};
					\node [style=none] (287) at (3.75, -1.75) {};
					\node [style=none] (288) at (2, 1) {};
					\node [style=morphism] (289) at (2.5, 1) {$\,\suppinc{p\otimes q}\,$};
					\node [style=none] (290) at (3, 1) {};
					\node [style=bn] (291) at (3, 2.25) {};
					\node [style=none] (292) at (5.5, 3.25) {};
					\node [style=none] (293) at (5.5, 1) {};
					\node [style=none] (295) at (4.5, 1) {};
					\node [style=bn] (296) at (4.5, 2.25) {};
					\node [style=none] (297) at (3.25, -1.75) {};
					\node [style=none] (299) at (4.25, -1.75) {};
					\node [style=none] (300) at (3.25, -3.25) {};
					\node [style=none] (301) at (4.25, -3.25) {};
					\node [style=none] (304) at (0, 0) {$=$};
					\node [style=none] (305) at (10.75, 0.75) {};
					\node [style=none] (307) at (10.25, 3.25) {};
					\node [style=none] (308) at (10.75, 0.25) {};
					\node [style=none] (309) at (10.75, -1) {};
					\node [style=none] (310) at (10.25, 1) {};
					\node [style=none] (319) at (10.25, -1) {};
					\node [style=none] (321) at (11.25, -1) {};
					\node [style=none] (322) at (10.25, -3.25) {};
					\node [style=none] (323) at (11.25, -3.25) {};
					\node [style=none] (324) at (7.5, 0) {$=$};
					\node [style=none] (325) at (11.25, 3.25) {};
					\node [style=none] (326) at (11.25, 1) {};
					\node [style=none] (328) at (15.25, 3.25) {};
					\node [style=none] (331) at (15.25, -1.25) {};
					\node [style=none] (333) at (15.25, -1.25) {};
					\node [style=none] (335) at (17, -1.25) {};
					\node [style=none] (336) at (15.25, -3.25) {};
					\node [style=none] (337) at (17, -3.25) {};
					\node [style=none] (338) at (13.25, 0) {$=$};
					\node [style=none] (339) at (17, 3.25) {};
					\node [style=none] (340) at (17, -1.25) {};
					\node [style=morphism] (341) at (15.25, 0) {$\suppinc{p}$};
					\node [style=morphism] (342) at (17, 0) {$\suppinc{q}$};
					\node [style=morphism] (344) at (5, 1) {$\,\suppinc{p\otimes q}\,$};
					\node [style=morphism] (346) at (-3.75, -2) {$\,\;\; f \,\;\;$};
					\node [style=morphism] (347) at (3.75, -1.75) {$\,\;\; f \,\;\;$};
					\node [style=morphism] (348) at (10.75, -1) {$\,\;\; f \,\;\;$};
					\node [style=morphism] (349) at (10.75, 1) {$\,\suppinc{p\otimes q}\,$};
					\node [style=none] (350) at (-4.25, -3.75) {$\Supp{p}$};
					\node [style=none] (351) at (-3.25, -3.75) {$\Supp{q}$};
					\node [style=none] (352) at (-5.5, 3.75) {$X$};
					\node [style=none] (353) at (-2, 3.75) {$Y$};
					\node [style=none] (354) at (10.25, -3.75) {$\Supp{p}$};
					\node [style=none] (355) at (11.25, -3.75) {$\Supp{q}$};
					\node [style=none] (356) at (15.25, -3.75) {$\Supp{p}$};
					\node [style=none] (357) at (17, -3.75) {$\Supp{q}$};
					\node [style=none] (358) at (3.25, -3.75) {$\Supp{p}$};
					\node [style=none] (359) at (4.25, -3.75) {$\Supp{q}$};
					\node [style=none] (360) at (2, 3.75) {$X$};
					\node [style=none] (361) at (5.5, 3.75) {$Y$};
					\node [style=none] (362) at (10.25, 3.75) {$X$};
					\node [style=none] (363) at (17, 3.75) {$Y$};
					\node [style=none] (364) at (15.25, 3.75) {$X$};
					\node [style=none] (365) at (11.25, 3.75) {$Y$};
				\end{pgfonlayer}
				\begin{pgfonlayer}{edgelayer}
					\draw [in=-90, out=165] (240) to (236.center);
					\draw [in=-90, out=15] (240) to (237.center);
					\draw (244.center) to (238.center);
					\draw (240) to (241.center);
					\draw (249) to (248.center);
					\draw (252.center) to (250.center);
					\draw (255) to (254.center);
					\draw (257.center) to (260.center);
					\draw (259.center) to (261.center);
					\draw [in=-90, out=165] (286) to (283.center);
					\draw [in=-90, out=15] (286) to (284.center);
					\draw (288.center) to (285.center);
					\draw (286) to (287.center);
					\draw (291) to (290.center);
					\draw (293.center) to (292.center);
					\draw (296) to (295.center);
					\draw (297.center) to (300.center);
					\draw (299.center) to (301.center);
					\draw (308.center) to (305.center);
					\draw (310.center) to (307.center);
					\draw (308.center) to (309.center);
					\draw (319.center) to (322.center);
					\draw (321.center) to (323.center);
					\draw (326.center) to (325.center);
					\draw (331.center) to (328.center);
					\draw (333.center) to (336.center);
					\draw (335.center) to (337.center);
					\draw (340.center) to (339.center);
					\draw (283.center) to (289);
					\draw (284.center) to (344);
				\end{pgfonlayer}
			\end{tikzpicture}
		}%
	\end{equation}
	where the first equality uses \eqref{eq:gh_supp_pq}, the second one follows from the fact that $\suppinc{p\otimes q}$ is deterministic, and the last one is just \eqref{eq:f_supp_pq}.
	Since $\suppinc{p}\otimes \suppinc{q}$ is a monomorphism, we can cancel it to obtain that the morphism from \eqref{eq:supp_multiplicativity3} is a left inverse of $f$.

	On the other hand, we also have
	\begin{equation}
		{%
			\tikzstyle{every picture}=[tikzfig]%
			\begin{tikzpicture}
				\begin{pgfonlayer}{nodelayer}
					\node [style=none] (236) at (-5, -1.75) {};
					\node [style=none] (237) at (-2.5, -1.75) {};
					\node [style=none] (238) at (-5.5, 0.5) {};
					\node [style=bn] (240) at (-3.75, -2.75) {};
					\node [style=none] (241) at (-3.75, -3.5) {};
					\node [style=none] (244) at (-5.5, -1.5) {};
					\node [style=morphism] (247) at (-5, -1.5) {$\,\;\;g\,\;\;$};
					\node [style=none] (248) at (-4.5, -1.5) {};
					\node [style=bn] (249) at (-4.5, -0.5) {};
					\node [style=none] (250) at (-2, 0.5) {};
					\node [style=none] (252) at (-2, -1.5) {};
					\node [style=morphism] (253) at (-2.5, -1.5) {$\,\;\;h\,\;\;$};
					\node [style=none] (254) at (-3, -1.5) {};
					\node [style=bn] (255) at (-3, -0.5) {};
					\node [style=none] (304) at (0, 0) {$=$};
					\node [style=morphism] (346) at (-3.75, 0.5) {$\quad\quad\,\;\; f \,\;\;\quad \quad$};
					\node [style=none] (350) at (-3.75, 2) {};
					\node [style=none] (351) at (-4.5, 3.5) {};
					\node [style=none] (352) at (-3.75, 1.5) {};
					\node [style=none] (353) at (-3.75, 0.5) {};
					\node [style=none] (354) at (-4.5, 2.25) {};
					\node [style=none] (357) at (-3, 3.5) {};
					\node [style=none] (358) at (-3, 2.25) {};
					\node [style=morphism] (360) at (-3.75, 2.25) {$\;\suppinc{p\otimes q}\;$};
					\node [style=none] (361) at (2.5, -0.5) {};
					\node [style=none] (362) at (5, -0.5) {};
					\node [style=none] (363) at (2, 3.5) {};
					\node [style=bn] (364) at (3.75, -1.75) {};
					\node [style=none] (365) at (3.75, -3.5) {};
					\node [style=none] (366) at (2, 0) {};
					\node [style=morphism] (367) at (2.5, 0) {$\,\;\;g\,\;\;$};
					\node [style=none] (368) at (3, 0) {};
					\node [style=bn] (369) at (3, 1) {};
					\node [style=none] (370) at (5.5, 3.5) {};
					\node [style=none] (371) at (5.5, 0) {};
					\node [style=morphism] (372) at (5, 0) {$\,\;\;h\,\;\;$};
					\node [style=none] (373) at (4.5, 0) {};
					\node [style=bn] (374) at (4.5, 1) {};
					\node [style=morphism] (377) at (2, 2) {$\suppinc{p}$};
					\node [style=morphism] (378) at (5.5, 2) {$\suppinc{q}$};
					\node [style=none] (379) at (8, 0) {$=$};
					\node [style=none] (380) at (10.5, -0.25) {};
					\node [style=none] (381) at (13, -0.25) {};
					\node [style=none] (382) at (10, 3.5) {};
					\node [style=bn] (383) at (11.75, -1.5) {};
					\node [style=none] (384) at (11.75, -3.5) {};
					\node [style=none] (385) at (10, 0.5) {};
					\node [style=morphism] (386) at (10.5, 0.5) {$\,\suppinc{p\otimes q}\,$};
					\node [style=none] (387) at (11, 0.5) {};
					\node [style=bn] (388) at (11, 2) {};
					\node [style=none] (389) at (13.5, 3.5) {};
					\node [style=none] (390) at (13.5, 0.5) {};
					\node [style=none] (391) at (12.5, 0.5) {};
					\node [style=bn] (392) at (12.5, 2) {};
					\node [style=none] (395) at (18.5, -0.25) {};
					\node [style=none] (396) at (17.75, 3.5) {};
					\node [style=none] (397) at (18.5, -0.75) {};
					\node [style=none] (398) at (18.5, -3.5) {};
					\node [style=none] (399) at (17.75, 0) {};
					\node [style=none] (402) at (15.5, 0) {$=$};
					\node [style=none] (403) at (19.25, 3.5) {};
					\node [style=none] (404) at (19.25, 0) {};
					\node [style=morphism] (405) at (13, 0.5) {$\,\suppinc{p\otimes q}\,$};
					\node [style=morphism] (408) at (18.5, 0) {$\;\,\suppinc{p\otimes q}\;\,$};
					\node [style=none] (409) at (-4.5, 4) {$X$};
					\node [style=none] (410) at (-3, 4) {$Y$};
					\node [style=none] (411) at (2, 4) {$X$};
					\node [style=none] (412) at (5.5, 4) {$Y$};
					\node [style=none] (413) at (10, 4) {$X$};
					\node [style=none] (414) at (19.25, 4) {$Y$};
					\node [style=none] (415) at (17.75, 4) {$X$};
					\node [style=none] (416) at (13.5, 4) {$Y$};
					\node [style=none] (417) at (-3.75, -4) {$\Supp{p\otimes q}$};
					\node [style=none] (418) at (3.75, -4) {$\Supp{p\otimes q}$};
					\node [style=none] (419) at (11.75, -4) {$\Supp{p\otimes q}$};
					\node [style=none] (420) at (18.5, -4) {$\Supp{p\otimes q}$};
				\end{pgfonlayer}
				\begin{pgfonlayer}{edgelayer}
					\draw [in=-90, out=165] (240) to (236.center);
					\draw [in=-90, out=15] (240) to (237.center);
					\draw [in=-90, out=90] (244.center) to (238.center);
					\draw (240) to (241.center);
					\draw (249) to (248.center);
					\draw [in=-90, out=90] (252.center) to (250.center);
					\draw (255) to (254.center);
					\draw (352.center) to (350.center);
					\draw (354.center) to (351.center);
					\draw (352.center) to (353.center);
					\draw (358.center) to (357.center);
					\draw [in=-90, out=165] (364) to (361.center);
					\draw [in=-90, out=15] (364) to (362.center);
					\draw (366.center) to (363.center);
					\draw (364) to (365.center);
					\draw (369) to (368.center);
					\draw (371.center) to (370.center);
					\draw (374) to (373.center);
					\draw [in=-90, out=165] (383) to (380.center);
					\draw [in=-90, out=15] (383) to (381.center);
					\draw (385.center) to (382.center);
					\draw (383) to (384.center);
					\draw (388) to (387.center);
					\draw (390.center) to (389.center);
					\draw (392) to (391.center);
					\draw (397.center) to (395.center);
					\draw (399.center) to (396.center);
					\draw (397.center) to (398.center);
					\draw (404.center) to (403.center);
					\draw (361.center) to (367);
					\draw (362.center) to (372);
					\draw (380.center) to (386);
					\draw (381.center) to (405);
				\end{pgfonlayer}
			\end{tikzpicture}
		}%
	\end{equation}
	by using the same ingredients as above.
	Moreover, $\suppinc{p\otimes q}$ is a monomorphism, and we conclude that $f$ is an isomorphism between $\Supp{p\otimes q}$ and $\Supp{p} \otimes \Supp{q}$.
	It then follows from~\eqref{eq:f_supp_pq} that $\suppinc{p} \otimes \suppinc{q}$ is a support inclusion for $p \otimes q$ as well.
\end{proof}

We finally turn to the functoriality of supports.
Surprisingly, this turns out to be closely related to the causality axiom.

\begin{theorem}[Functoriality of supports]
	\label{thm:supp_functorial}
	For a Markov category $\cC$ in which every morphism has a support, the following are equivalent:
	\begin{enumerate}
		\item \label{it:causal} $\cC$ is causal.
		\item \label{it:supp_functorial} Supports are functorial: 
			Every commutative square 
			\begin{equation}
				\label{eq:support_func_square}
				\begin{tikzcd}
					A \ar[r, "p"] \ar[d, "f"'] & X \ar[d, "g"] \\
					B \ar[r, "q"'] & Y
				\end{tikzcd}
			\end{equation}
			factors into
			\begin{equation}
				\label{eq:supp_mapping}
				\begin{tikzcd}
					A \ar[r, "\suppfactor{p}"] \ar[d, "f"'] & \Supp{p} \ar[d, dashed] \ar[r, "\suppinc{p}"] & X \ar[d, "g"] \\
					B \ar[r, "\suppfactor{q}"'] & \Supp{q} \ar[r, "\suppinc{q}"'] & Y
				\end{tikzcd}
			\end{equation}
			for a (necessarily unique) dashed arrow $\Supp{p} \to \Supp{q}$.
	\end{enumerate}
\end{theorem}

Note that it is enough for the right square in \cref{eq:supp_mapping} to commute, since then the left square commutes automatically by commutativity of the composite square and the fact that $\suppinc{q}$ is a monomorphism.

\begin{proof}
	\ref{it:causal} $\Rightarrow$ \ref{it:supp_functorial}:
	We construct the dashed arrow from the universal property of $\Supp{q}$ by showing that $g \comp \suppinc{p} \ll q$ holds.
	Using \cref{lem:inc_sim_p} gives $\suppinc{p} \acsim p$, and applying \cref{lem:ac_mon_comp} then yields
	\begin{equation}
		g \comp \suppinc{p} \acsim g \comp p.
	\end{equation}
	Moreover, the assumed commutative diagram and \cref{lem:factor_ac} imply
	\begin{equation}
		g\comp p = q\comp f \ll q,
	\end{equation}
	as was to be shown.
	Therefore, we obtain a unique morphism of type $\Supp{p} \to \Supp{q}$ which makes the right square commute.
	The left square then commutes automatically as noted above.

	\ref{it:supp_functorial} $\Rightarrow$ \ref{it:causal}:
	We will show that the causality axiom holds in the version of \cref{lem:as_shift}. 
	So let $p \colon A \to X$, $g \colon X \to Y$, and $h_1, h_2 \colon Y \to Z$ be such that $h_1 \ase{g\comp p} h_2$.
	Then the assumed functoriality of supports applied to the square
	\begin{equation}
		\begin{tikzcd}
			A \ar[r, "p"] \ar[d, "\id_A"'] & X \ar[d, "g"] \\
			A \ar[r, "g\comp p"'] & Y
		\end{tikzcd}	
	\end{equation}
	gives us a commutative diagram
	\begin{equation}
		\begin{tikzcd}
			\Supp{p} \ar[r, hookrightarrow, "\suppinc{p}"] \ar[d, dashed] & X \ar[d, "g"] \\
			\Supp{g \comp p} \ar[r, hookrightarrow, "\suppinc{g\comp p}"'] & Y
		\end{tikzcd}
	\end{equation}
	By \cref{lem:supp_asfaithful}, the assumption $h_1 \ase{g\comp p} h_2$ is equivalent to ${h_1 \comp \suppinc{g \comp p}} = h_2 \comp \suppinc{g\comp p}$.
	But then the above diagram implies that we also have $h_1 \comp g \comp \suppinc{p} = h_2 \comp g \comp \suppinc{p}$.
	Thus we obtain $h_1\comp g \ase{p} h_2\comp g$ by \cref{lem:supp_asfaithful} again, which completes the proof.
\end{proof}

The following example shows that supports in $\tychstoch$ are functorial.
By \cref{thm:supp_functorial}, this proves that $\tychstoch$ is a causal Markov category. 
\begin{example}\label{ex:tychstoch_supp}
	$\tychstoch$ has supports by \cref{ex:top_support}.
	To show that they are indeed functorial, let us consider the topological support of the composite $g\comp p$ of two stochastic maps $p \colon A \to X$ and $g \colon X \to Y$. 
	We claim that the support object is given by\footnotemark{}
	\footnotetext{This equation can also be seen as a manifestation of the fact that the topological support is a morphism of monads \cite{fritz2019support}.}%
	\begin{equation}
		\Supp{g \comp p} = \overline{\bigcup_{x \in \Supp{p}} \Supp{g(\ph|x)}}.
	\end{equation}
	Indeed, if we tentatively write $\Supp{g \comp p}$ as shorthand for the right-hand side, then we have
	\begin{equation}
		\int_{x \in X} g(\Supp{g \comp p}|x) \, p(dx|a)
		\ge \int_{x \in \Supp{p}} g(\Supp{g \comp p}|x) \, p(dx|a)
		= \int_{x \in \Supp{p}} p(dx|a)
		= p(\Supp{p}| a)
		= 1,
	\end{equation}
	where the first equality holds by $\Supp{g \comp p} \supseteq \Supp{g(\ph|x)}$ for all $x \in \Supp{p}$.
	Conversely, if $C \subseteq Y$ is a closed set of full measure, meaning 
	\begin{equation}
		\int_{x \in X} g(C|x) \, p(dx|a) = 1,
	\end{equation}
	then we need to show $\Supp{g \comp p} \subseteq C$, or equivalently $\Supp{g(\ph|x)} \subseteq C$ for all $x \in \Supp{p}$.
	But this in turn amounts to showing $g(C|x) = 1$ for all $x \in \Supp{p}$.
	By assumption, we know that the integrand $g(C|x)$ is $p(\ph|a)$-almost everywhere equal to $1$.
	But since $g(C|\ph)$ is an upper semi-continuous function on $X$, the set where it is equal to $1$ is closed, and hence contains $\Supp{p}$.
	This shows the required $g(C|x) = 1$ for all $x \in \Supp{p}$.
	Hence the above $\Supp{g \comp p}$ is indeed the topological support of $g\comp p$.

	Equipped with this description of the topological support of a composite morphism, we can show that supports in $\tychstoch$ are functorial.
	To show the desired factorization, we need to prove $g(\Supp{q}|x) = 1$ for all $x \in \Supp{p}$ given a commutative square as in \eqref{eq:support_func_square}.
	This is an immediate consequence of
	\begin{equation}
		\Supp{q} \supseteq \Supp{q\comp f} = \Supp{g\comp p} \supseteq \Supp{g(\ph|x)} \qquad \forall x  \in \Supp{p},
	\end{equation}
	where the first containment also follows by the formula for supports of a composite derived above, but now applied to the composite $q \comp f$.

	In conclusion, supports in $\tychstoch$ are functorial, and \cref{thm:supp_functorial} implies that $\tychstoch$ is causal.
\end{example}

\subsection{Support completeness and the free regular support completion}\label{sec:support_completeness}

As we have seen, the Markov categories $\finstoch$, $\tychstoch$ and $\setmulti$ all have supports.
Can we expect supports in measure-theoretic probability?
How common are supports in general?
In this subsection, we consider these questions and show that supports in measure-theoretic probability (in $\borelstoch$ specifically) are quite rare.
First, we start with a necessary condition for any morphism to have a support.

\begin{lemma}\label{lem:suppdominance}
	If a morphism $p \colon A \to X$ has a support, then it satisfies $\cop \comp p \ll p\otimes p$.
\end{lemma}
\begin{proof}
	Whenever $p$ has a support, $p \acsim \suppinc{}$ holds by \cref{lem:inc_sim_p}. 
	Therefore, we have
	\begin{equation}
		\cop \comp p \acsim \cop \comp \suppinc{} = (\suppinc{} \otimes \suppinc{})\comp \cop \ll \suppinc{} \otimes \suppinc{} \acsim p \otimes p,
	\end{equation}
	where we use first
	\begin{itemize}[noitemsep]
		\item \cref{lem:ac_mon_comp_split_mono} and the fact that $\cop$ is a deterministic split monomorphism, then
		\item the fact that $\suppinc{}$ is deterministic, followed by
		\item \cref{lem:factor_ac} with $h$ and $p$ from the lemma given by $\suppinc{} \otimes \suppinc{}$ and $\cop$ respectively, and finally
		\item \cref{prop:ac_tensor} in the last step.
			\qedhere
	\end{itemize}
\end{proof}

This property of morphisms with a support is important enough to give it its own name, the choice of which we justify in \cref{ex:atomic}.
\begin{definition}
	\label{def:atomic}
	A morphism $p$ in a Markov category is \newterm{atomic} if it satisfies ${\cop \comp p \ll p \otimes p}$.
\end{definition}

\begin{example}\label{ex:atomic}
	The Lebesgue measure $\lambda \colon I \to [0,1]$ in $\borelstoch$ is not atomic, since $\cop \comp \lambda$ is the Lebesgue measure on the diagonal of the unit square $[0,1]^2$, while $\lambda \otimes \lambda$ is the Lebesgue measure on all of $[0,1]^2$, and the former is not absolutely continuous with respect to the latter (by \cref{ex:abs_cont_stoch}).
	Hence by \cref{lem:suppdominance}, the Lebesgue measure on $[0,1]$ does not have a support.

	The same argument works generally, and shows that a state $I \to X$ in $\borelstoch$ is atomic if and only if it is atomic in the usual sense as a probability measure on $X$ (see the forthcoming \cref{thm:borelstoch_supports}).
\end{example}

\begin{remark}
	\label{rem:det_atomic}
	If $p \acsim q$ holds for any deterministic morphism $q$, then $p$ is an atomic morphism.
	The proof is the same as that of \cref{lem:suppdominance}, but with $q$ in place of $\suppinc{}$.
\end{remark}

On the other hand, since $\finstoch$, $\tychstoch$ and $\setmulti$ have all supports, it follows a posteriori that every morphism in these categories is atomic.
Since typically not every morphism is atomic, the following more restricted existence of supports is the most that one can expect in general.

\begin{definition}
	A Markov category $\cC$ is \newterm{support complete} if every atomic morphism has a support.
\end{definition}

We now turn to a more detailed investigation of supports in measure-theoretic probability, which in particular justifies the term ``atomic'' further.
For a generic morphism $p \colon A \to X$ in $\borelstoch$, let us say that an element $x \in X$ is an \newterm{atom} if there is $a \in A$ such that the conditional probability $p(\{x\} | a)$ is non-zero.
We owe the following auxiliary result to a MathOverflow user.\footnotemark{}
\footnotetext{See \href{https://mathoverflow.net/questions/448485/atoms-for-markov-kernels/448496}{https://mathoverflow.net/questions/448485/atoms-for-markov-kernels/448496} for the relevant answer by Packo.}

\begin{lemma}[Packo]
	\label{lem:atoms_univ}
	For every morphism $p$ in $\borelstoch$, its set of atoms
	\begin{equation}
	\Atoms{} \coloneqq \Set*[\big]{x \in X   \given   \exists \, a \in A  \text{ such that }  p \bigl(\{x\} \big| a \bigr) > 0 }
\end{equation}
is universally measurable.
\end{lemma}

\begin{proof}
	The statement is trivial if $X$ is at most countable, since then every subset is measurable.
	By Kuratowski's theorem, we can assume we have $X = \R$ without loss of generality.

	The function
	\begin{equation}
		\begin{split}
			A \times \R & \longrightarrow [0,1]	\\
			(a, x) & \longmapsto p \bigl( [0,x] \big| a \bigr)
		\end{split}
	\end{equation}
	is measurable in the first argument and right-continuous in the second, which implies that it is jointly measurable.\footnotemark{}
	\footnotetext{This is a standard observation in stochastic process theory, see e.g.\ \cite[Proposition 1.13]{karatzas1991brownian} applied to the constant filtration.}%
	Similarly, also the function
	\begin{equation}
		\begin{split}
			A \times \R & \longrightarrow [0,1]	\\
			(a, x) & \longmapsto p \bigl( [0,x) \big| a \bigr)
		\end{split}
	\end{equation}
	is jointly measurable.
	Their difference is precisely
	\begin{equation}\label{eq:kernel_point}
		\begin{split}
			A \times \R & \longrightarrow [0,1]	\\
			(a, x) & \longmapsto p \bigl( \{x\} \big| a \bigr)
		\end{split}
	\end{equation}
	which is thus jointly measurable as well.
	The set of atoms $\Atoms{}$ is then the projection to $\R$ of the preimage of the set $(0,1]$ under the function from \eqref{eq:kernel_point}.
	As a projection of a measurable set, the set of atoms is therefore universally measurable.
\end{proof}

\Cref{lem:atoms_univ} allows us to consider the conditional probability $p(\Atoms{} | a)$ for every $a \in A$.
Generalizing the notion of an atomic measure, we may then call a Markov kernel $p$ \emph{atomic} if it satisfies $p(\Atoms{}|a) = 1$ for all $a \in A$.
The following theorem establishes that this is consistent with the notion of atomicity from \cref{def:atomic}.

\begin{theorem}
	\label{thm:borelstoch_supports}
	For a morphism $p \colon A \to X$ in $\borelstoch$, consider the following conditions:
	\begin{enumerate}
		\item\label{it:p_supported} $p$ has a support.
		\item\label{it:p_smallest} There is a smallest measurable set $\Supp{} \in \Sigma_X$ with $p(\Supp{}|a) = 1$ for all $a \in A$.
		\item\label{it:p_copydom} $p$ is atomic in the sense of \cref{def:atomic}.
		\item\label{it:p_atomic} $p(\Atoms{} | a) = 1$ for all $a \in A$.
	\end{enumerate}
	Then we have:\footnote{For the equivalence of \ref{it:p_copydom} and \ref{it:p_atomic} in the case of states, see also \cite[Theorem 4.6]{ours_entropy}.}
	\[
		\ref{it:p_supported} \Longleftrightarrow \ref{it:p_smallest} \,\Longrightarrow\, \ref{it:p_copydom} \Longleftrightarrow \ref{it:p_atomic}.
	\]
	The converse implication $\ref{it:p_copydom} \Longrightarrow \ref{it:p_smallest}$ holds if $A$ is at most countable, but not in general.
\end{theorem}

\begin{proof}
	If $p$ satisfies property \ref{it:p_supported}, then we can take $\Supp{} \in \Sigma_X$ to be the image of the support inclusion.
	This set is measurable by \cite[Corollary 15.2]{kechris}.
	Applying the universal property of the support of $p$ with respect to the inclusion of any other measurable subset establishes the minimality in \ref{it:p_smallest}.

	Conversely, suppose that \ref{it:p_smallest} holds, and take $\suppinc{} \colon \Supp{} \to X$ to be the inclusion map.
	Clearly every $\id_W \otimes \suppinc{}$ is a deterministic monomorphism and we have $\suppinc{} \ll p$.
	By \cref{lem:supp_crit}, it remains to be shown that if $f \colon Z \to W \otimes X$ satisfies $f \ll \id_W \otimes p$, then it factors across $\id_W \otimes \suppinc{}$.
	Indeed $f \ll \id_W \otimes p$ gives $f(W \times (X \setminus \Supp{})|z) = 0$ for all $z \in Z$, and hence $f$ factors across $\id_W \otimes \suppinc{}$ as required.
	Therefore $p$ has a support.

	The implication from \ref{it:p_supported} to \ref{it:p_copydom} is exactly \cref{lem:suppdominance} applied to $\borelstoch$.

	Next, assume that \ref{it:p_copydom} holds.
	Following the proof of \cite[Theorem 4.6]{ours_entropy}, we write ${\Delta \subseteq X \times X}$ for the diagonal and $\Atoms{}$ for the set of atoms of $p$ as in \cref{lem:atoms_univ}.
	Let us argue that the conditional probability 
	\begin{equation}
		(p \otimes p) \bigl( \Delta \setminus (\Atoms{} \times \Atoms{}) \big| a_1, a_2 \bigr) 
	\end{equation}
	of $\Delta \setminus (\Atoms{} \times \Atoms{})$ given any $a_1, a_2 \in A$ vanishes.
	Indeed, this follows upon decomposing every $p(\ph|a)$ into its atomic and diffuse parts, noting that the atomic part vanishes outside of $\Atoms{}$, and that the product of two diffuse finite measures vanishes on the diagonal.\footnotemark{}
	\footnotetext{This can be shown by decomposing $X$ into pieces of mass below $\eps$ for both measures.}%
	By \cref{eq:abs_cont_stoch} and the assumed property \ref{it:p_copydom}, we therefore obtain
	\begin{equation}
		(\cop \comp  p) \bigl(\Delta \setminus (\Atoms{} \times \Atoms{}) \big| a \bigr) = 0,
	\end{equation}
	or equivalently $p(X \setminus \Atoms{} | a) = 0$, for all $a$, which implies \ref{it:p_atomic}.

	Conversely, suppose that \ref{it:p_atomic} holds, so that we have $p(\Atoms{}|a) = 1$ for all $a \in A$.
	To show $\cop \comp p \ll p \otimes p$, suppose that we have a measurable set $T \in \Sigma_{X \times X}$ satisfying $(p \otimes p)(T | a_1, a_2) = 0$ for all $a_1, a_2 \in A$.
	This implies that $T$ is disjoint from $\Atoms{} \times \Atoms{}$, because if $(x,y) \in T$ with $x,y \in \Atoms{}$, then there are $a_x,a_y \in A$ with $p(\{x\}|a_x) > 0$ and $p(\{y\}|a_y)>0$, contradicting $(p \otimes p)(T | a_x, a_y)=0$.
	In particular, the intersection of $T$ with the diagonal does not contain any atom, and hence $(\cop \comp  p)(T|a) = 0$ holds for all $a \in A$, which gives the desired property \ref{it:p_copydom}.

	For the final statement, we note that there is an implication from \ref{it:p_atomic} to \ref{it:p_smallest} if the set of atoms is measurable.
	In such a case, we can take $\Supp{}$ from \ref{it:p_smallest} to be the set $\Atoms{}$ itself.
	This is guaranteed to be the case if $A$ is countable, since then $\Atoms{}$ is countable as well and hence trivially measurable.

	To see that this implication does not hold for general $A$, consider any measurable function $\R \to \R$ with non-measurable image. 
	Considered as a deterministic morphism $p \colon \R \to \R$ in $\borelstoch$, it is clearly atomic with the set of atoms $\Atoms{}$ being its image.
	If $p$ had a support, then by $\suppinc{} \acsim p$ and \cref{ex:abs_cont_stoch} \ref{it:borelstoch_det_abs_cont}, such support would have to be given by the inclusion map of a non-measurable subset, which is not possible in $\borelstoch$ \cite[Corollary 15.2]{kechris}.
\end{proof}

\begin{corollary}[Supports of measurable maps]
	\label{cor:support_det_borelstoch}
	A deterministic morphism $p$ in $\borelstoch$ has a support if and only if its image (as a measurable map) is measurable, and in this case the support is given by the image.
\end{corollary}

\begin{proof}
	If $p$ has measurable image, then this image serves as the set $\Supp{}$ in \cref{thm:borelstoch_supports} \ref{it:p_smallest}.
	Conversely, suppose that $p$ has a support.
	Then we have $\suppinc{} \acsim p$ by \cref{lem:inc_sim_p}, which implies $\mathrm{im}(\suppinc{}) = \mathrm{im}(p)$ by \cref{ex:abs_cont_stoch} \ref{it:borelstoch_det_abs_cont}.
	In particular $\mathrm{im}(p)$ is measurable, since $\suppinc{}$ is an injective measurable map and therefore has a measurable image \cite[Corollary 15.2]{kechris}.
\end{proof}

The final statement of \cref{thm:borelstoch_supports}{\,\textemdash\,}or in other words the existence of measurable maps between standard Borel spaces with non-measurable image{\,\textemdash\,}has the following unfortunate consequence.

\begin{corollary}
	$\borelstoch$ is not support complete.
\end{corollary}

In order to address this shortcoming, one may wonder whether it is possible to construct a support complete Markov category that contains $\borelstoch$ as a full Markov subcategory.
This is true, and it is an instance of a more general construction which we call the \emph{free regular support completion}.
This terminology is by analogy with the free regular completion of a category with finite limits~\cite{hu1996free_regular}; the word ``regular'' also reflects \cref{prop:coregular_supports}. 
In light of \cref{thm:suppcomp_universal}, it would be an interesting problem to find an explicit description of the free regular support completion of $\borelstoch$.

\begin{definition}
	Let $\cC$ be a causal Markov category. 
	Its \newterm{free regular support completion} $\suppcomp$ is the causal Markov category described as follows:
	\begin{enumerate}
		\item Objects are pairs $(X,p)$, where $X \in \cC$ is an object and $p$ is any atomic morphism landing in $X$.

		\item Morphisms $(X,p) \to (Y,q)$ are equivalence classes of morphisms $f\colon X \to Y$ (relative to $\as{p}$ equality) satisfying $f \comp p \ll q$.

		\item The Markov category structure is induced from $\cC$.
	\end{enumerate}
\end{definition}

Let us first demonstrate that this construction actually defines a causal Markov category.

\begin{proposition}
	\label{prop:suppcomp}
	$\suppcomp$ is a well-defined causal Markov category, and $\cC$ is naturally identified with the full Markov subcategory whose objects are of the form $(X,\id_X)$.
\end{proposition}
\begin{proof}
	We first show that composition in $\suppcomp$ is well-defined.
	Consider $[f] \colon (X,p) \to (Y,q)$ and $[g] \colon (Y,q) \to (Z,r)$ to be morphisms represented by $f \colon X \to Y$ and $g \colon Y \to Z$.
	To show that $g\comp f$ represents a morphism $(X,p) \to (Z,r)$, note that we have
	\begin{equation}
		(g\comp f) \comp p = g \comp (f \comp p) \ll g \comp q \ll r,
	\end{equation}
	where the second step is by \cref{lem:ac_mon_comp}.
	We also need to show that this composition respects the equivalence classes. 
	To this end, consider a morphism $f' \colon X \to Y$ satisfying $f' \ase{p} f$.
	Then, by \cref{prop:ase_props} \ref{it:ase_post_comp}, we also have $g\comp f \ase{p} g\comp f'$, which amounts to the desired $[g\comp f] = [g\comp f']$.
	Similarly, if $g' \colon Y \to Z$ satisfies $g' \ase{q} g$, then we obtain $g \ase{f \comp p} g'$ by $f \comp p \ll q$.
	Hence the required $g\comp f \ase{p} g'\comp f$ follows by \cref{lem:as_shift}. 

	The tensor product of objects is well-defined because atomic morphisms are closed under tensor products.
	Indeed, if $p \colon A \to X$ and $q \colon B \to Y$ are atomic, then \cref{prop:ac_tensor} gives
	\begin{equation}
		(\cop_X \comp p) \otimes (\cop_Y \comp q) \ll (p \otimes p) \otimes (q \otimes q),
	\end{equation}
	which is precisely $\cop_{X \otimes Y} \comp (p \otimes q) \ll (p \otimes q) \otimes (p \otimes q)$ by the multiplicativity of copying.
	It is clear that the symmetric monoidal structure morphisms of $\cC$ also belong to $\suppcomp$, where $(I, \id_I)$ is the monoidal unit of $\suppcomp$.
	Moreover, the tensor product of two morphisms respects the absolute continuity condition: 
	For $[f] \colon (X,p) \to (Y,q)$ and $[f'] \colon (X',p') \to (Y',q')$, we have
	\begin{equation}
		\label{eq:suppcomp_tensor_ac}
		(f \otimes f') (p \otimes p') = (f \comp p) \otimes (f' \comp p') \ll q \otimes q'
	\end{equation}
	by \cref{prop:ac_tensor} and the assumed $f\comp p \ll q$ and $f'\comp p' \ll q'$, as was to be shown.
	This tensor product is well-defined on equivalence classes by \cref{prop:ase_props}~\ref{it:ase_tensor}.
	Hence by what has been proven so far, $\suppcomp$ is a symmetric monoidal category.
	It is obvious that it is again semicartesian.

	Concerning the Markov category structure, we define the copy morphisms as
	\begin{equation}\label{eq:copy_suppcomp}
		\cop_{(X,p)} \coloneqq [\cop_X] \colon (X,p) \to (X,p) \otimes (X,p).
	\end{equation}
	We emphasize that such a $\cop_{(X,p)}$ is indeed a morphism in $\suppcomp$ because $p$ is atomic. 
	Direct checks show that these morphisms indeed equip $\suppcomp$ with a Markov category structure.

	From this description, it is immediate that we can identify $\cC$ with the full subcategory whose objects are of the form $(X,\id_X)$, and the inclusion functor is a strict Markov functor by construction.

	It remains to be shown that $\suppcomp$ is causal.
	This holds as a special case of the following implication in $\cC$:
	\begin{equation}
		{%
			\tikzstyle{every picture}=[tikzfig]%
			\begin{tikzpicture}
				\begin{pgfonlayer}{nodelayer}
					\node [style=none] (0) at (-8.25, 2.75) {};
					\node [style=none] (1) at (-6.25, 0) {$\ase{p}$};
					\node [style=none] (2) at (-9, -3.25) {$X$};
					\node [style=morphism] (3) at (-9, -1.75) {$f$};
					\node [style=none] (4) at (-9.75, 2.75) {};
					\node [style=none] (5) at (-8.25, 3.25) {$Z$};
					\node [style=none] (6) at (-9, -2.75) {};
					\node [style=none] (7) at (-9.75, 1.5) {};
					\node [style=none] (8) at (-8.25, 1.5) {};
					\node [style=bn] (9) at (-9, 0.5) {};
					\node [style=morphism] (10) at (-9, -0.5) {$g$};
					\node [style=morphism] (11) at (-9.75, 1.75) {$h_1$};
					\node [style=none] (12) at (-9.75, 3.25) {$A$};
					\node [style=none] (13) at (-3.75, -2.75) {};
					\node [style=none] (14) at (-3.75, -3.25) {$X$};
					\node [style=none] (15) at (-3, 2.75) {};
					\node [style=none] (16) at (-3, 1.5) {};
					\node [style=none] (17) at (-4.5, 1.5) {};
					\node [style=bn] (18) at (-3.75, 0.5) {};
					\node [style=none] (19) at (-4.5, 2.75) {};
					\node [style=morphism] (20) at (-3.75, -0.5) {$g$};
					\node [style=morphism] (21) at (-4.5, 1.75) {$h_2$};
					\node [style=none] (22) at (-4.5, 3.25) {$A$};
					\node [style=morphism] (23) at (-3.75, -1.75) {$f$};
					\node [style=none] (24) at (-3, 3.25) {$Z$};
					\node [style=none] (26) at (8.75, 0) {$\ase{p}$};
					\node [style=none] (64) at (0, 0) {$\implies$};
					\node [style=none] (65) at (5.75, 3) {};
					\node [style=none] (67) at (5, -3.5) {$X$};
					\node [style=morphism] (68) at (5, -2) {$f$};
					\node [style=none] (69) at (4.25, 3) {};
					\node [style=none] (70) at (5.75, 3.5) {$Z$};
					\node [style=none] (71) at (4.25, 2) {};
					\node [style=none] (72) at (5.75, 2) {};
					\node [style=bn] (73) at (5, 1) {};
					\node [style=morphism] (74) at (5, 0) {$g$};
					\node [style=morphism] (75) at (4.25, 2) {$h_1$};
					\node [style=none] (76) at (4.25, 3.5) {$A$};
					\node [style=none] (77) at (5, -3) {};
					\node [style=bn] (78) at (5, -1) {};
					\node [style=none] (79) at (6.75, 0.5) {};
					\node [style=none] (80) at (6.75, 3) {};
					\node [style=none] (81) at (6.75, 3.5) {$Y$};
					\node [style=none] (82) at (12.25, 3) {};
					\node [style=none] (83) at (11.5, -3.5) {$X$};
					\node [style=morphism] (84) at (11.5, -2) {$f$};
					\node [style=none] (85) at (10.75, 3) {};
					\node [style=none] (86) at (12.25, 3.5) {$Z$};
					\node [style=none] (87) at (10.75, 2) {};
					\node [style=none] (88) at (12.25, 2) {};
					\node [style=bn] (89) at (11.5, 1) {};
					\node [style=morphism] (90) at (11.5, 0) {$g$};
					\node [style=morphism] (91) at (10.75, 2) {$h_2$};
					\node [style=none] (92) at (10.75, 3.5) {$A$};
					\node [style=none] (93) at (11.5, -3) {};
					\node [style=bn] (94) at (11.5, -1) {};
					\node [style=none] (95) at (13.25, 0.5) {};
					\node [style=none] (96) at (13.25, 3) {};
					\node [style=none] (97) at (13.25, 3.5) {$Y$};
				\end{pgfonlayer}
				\begin{pgfonlayer}{edgelayer}
					\draw [style=morphism] (6.center) to (9);
					\draw [style=morphism, in=-90, out=15] (9) to (8.center);
					\draw [style=morphism, in=-90, out=165] (9) to (7.center);
					\draw [style=morphism] (8.center) to (0.center);
					\draw [style=morphism] (7.center) to (4.center);
					\draw [style=morphism] (16.center) to (15.center);
					\draw [style=morphism, in=-90, out=165] (18) to (17.center);
					\draw [style=morphism, in=-90, out=15] (18) to (16.center);
					\draw [style=morphism] (17.center) to (19.center);
					\draw [style=morphism] (13.center) to (18);
					\draw [style=morphism, in=-90, out=15] (73) to (72.center);
					\draw [style=morphism, in=-90, out=165] (73) to (71.center);
					\draw [style=morphism] (72.center) to (65.center);
					\draw [style=morphism] (71.center) to (69.center);
					\draw [style=morphism] (77.center) to (78);
					\draw [style=morphism, in=-90, out=0] (78) to (79.center);
					\draw (79.center) to (80.center);
					\draw (78) to (74);
					\draw (74) to (73);
					\draw [style=morphism, in=-90, out=15] (89) to (88.center);
					\draw [style=morphism, in=-90, out=165] (89) to (87.center);
					\draw [style=morphism] (88.center) to (82.center);
					\draw [style=morphism] (87.center) to (85.center);
					\draw [style=morphism] (93.center) to (94);
					\draw [style=morphism, in=-90, out=0] (94) to (95.center);
					\draw (95.center) to (96.center);
					\draw (94) to (90);
					\draw (90) to (89);
				\end{pgfonlayer}
			\end{tikzpicture}
		}%
	\end{equation}
	This is exactly \emph{relative causality} in the sense of~\cite[Corollary 2.28]{fritz2022dilations}, which is shown there to be equivalent to causality.
\end{proof}

Before showing that $\suppcomp$ is support complete, we give an explicit description of absolute continuity in this category.
\begin{lemma}\label{lem:abscont_suppcompletion}
	Let $[f]\colon (X,p)\to (Y,q)$ and $[g] \colon (Z,r) \to (Y,q)$ be two generic morphisms in $\suppcomp$. 
	Then $[f] \ll [g]$ holds in $\suppcomp$ if and only if we have $f\comp p \ll g\comp r$ in $\cC$.
\end{lemma}
\begin{proof}
	The absolute continuity relation $[f] \ll [g]$ in $\suppcomp$ reads in $\cC$ as
	\begin{equation}\label{eq:supp_comp_abs_cont}
		{%
			\tikzstyle{every picture}=[tikzfig]%
			\begin{tikzpicture}
				\begin{pgfonlayer}{nodelayer}
					\node [style=none] (13) at (-10.5, 0.75) {};
					\node [style=none] (14) at (-9, 1) {};
					\node [style=none] (19) at (-11, 2.25) {};
					\node [style=bn] (21) at (-9.75, 0) {};
					\node [style=none] (23) at (-9.75, -2.25) {};
					\node [style=none] (25) at (-7.25, 0) {$\ase{r}$};
					\node [style=morphism] (33) at (-11, 1.25) {$\,\;\; h \,\;\;$};
					\node [style=none] (34) at (-9, 2.25) {};
					\node [style=morphism] (35) at (-9.75, -1.25) {$g$};
					\node [style=none] (75) at (-11.5, -2.25) {};
					\node [style=none] (85) at (-11.5, 1.25) {};
					\node [style=none] (195) at (-11.5, -0.75) {};
					\node [style=none] (202) at (0, 0) {$\implies$};
					\node [style=none] (236) at (-4.5, 0.75) {};
					\node [style=none] (237) at (-3, 1) {};
					\node [style=none] (238) at (-5, 2.25) {};
					\node [style=bn] (240) at (-3.75, 0) {};
					\node [style=none] (241) at (-3.75, -2.25) {};
					\node [style=morphism] (244) at (-5, 1.25) {$\,\;\; k \,\;\;$};
					\node [style=none] (245) at (-3, 2.25) {};
					\node [style=morphism] (246) at (-3.75, -1.25) {$g$};
					\node [style=none] (247) at (-5.5, -2.25) {};
					\node [style=none] (249) at (-5.5, 1.25) {};
					\node [style=none] (250) at (-5.5, -0.75) {};
					\node [style=none] (251) at (4, 0.75) {};
					\node [style=none] (252) at (5.5, 1) {};
					\node [style=none] (253) at (3.5, 2.25) {};
					\node [style=bn] (255) at (4.75, 0) {};
					\node [style=none] (256) at (4.75, -2.25) {};
					\node [style=none] (258) at (7.25, 0) {$\ase{ p}$};
					\node [style=morphism] (260) at (3.5, 1.25) {$\,\;\; h \,\;\;$};
					\node [style=none] (261) at (5.5, 2.25) {};
					\node [style=morphism] (262) at (4.75, -1.25) {$f$};
					\node [style=none] (263) at (3, -2.25) {};
					\node [style=none] (265) at (3, 1.25) {};
					\node [style=none] (266) at (3, -0.75) {};
					\node [style=none] (267) at (10, 0.75) {};
					\node [style=none] (268) at (11.5, 1) {};
					\node [style=none] (269) at (9.5, 2.25) {};
					\node [style=bn] (271) at (10.75, 0) {};
					\node [style=none] (272) at (10.75, -2.25) {};
					\node [style=morphism] (275) at (9.5, 1.25) {$\,\;\; k \,\;\;$};
					\node [style=none] (276) at (11.5, 2.25) {};
					\node [style=morphism] (277) at (10.75, -1.25) {$f$};
					\node [style=none] (278) at (9, -2.25) {};
					\node [style=none] (280) at (9, 1.25) {};
					\node [style=none] (281) at (9, -0.75) {};
					\node [style=none] (282) at (-10.5, 1.25) {};
					\node [style=none] (283) at (-4.5, 1.25) {};
					\node [style=none] (284) at (4, 1.25) {};
					\node [style=none] (285) at (10, 1.25) {};
					\node [style=none] (286) at (-11.5, -2.75) {$W$};
					\node [style=none] (287) at (-9.75, -2.75) {$Z$};
					\node [style=none] (288) at (-11, 2.75) {$A$};
					\node [style=none] (289) at (-9, 2.75) {$Y$};
					\node [style=none] (290) at (-5.5, -2.75) {$W$};
					\node [style=none] (291) at (-3.75, -2.75) {$Z$};
					\node [style=none] (292) at (-5, 2.75) {$A$};
					\node [style=none] (293) at (-3, 2.75) {$Y$};
					\node [style=none] (294) at (3, -2.75) {$W$};
					\node [style=none] (295) at (4.75, -2.75) {$X$};
					\node [style=none] (296) at (3.5, 2.75) {$A$};
					\node [style=none] (297) at (5.5, 2.75) {$Y$};
					\node [style=none] (298) at (9, -2.75) {$W$};
					\node [style=none] (299) at (10.75, -2.75) {$X$};
					\node [style=none] (300) at (9.5, 2.75) {$A$};
					\node [style=none] (301) at (11.5, 2.75) {$Y$};
				\end{pgfonlayer}
				\begin{pgfonlayer}{edgelayer}
					\draw [in=-90, out=165] (21) to (13.center);
					\draw [in=-90, out=15] (21) to (14.center);
					\draw (33) to (19.center);
					\draw (14.center) to (34.center);
					\draw (23.center) to (35);
					\draw [style=protected] (75.center) to (195.center);
					\draw [style=protected, in=-90, out=90, looseness=1.25] (195.center) to (85.center);
					\draw (35) to (21);
					\draw [in=-90, out=165] (240) to (236.center);
					\draw [in=-90, out=15] (240) to (237.center);
					\draw (244) to (238.center);
					\draw (237.center) to (245.center);
					\draw (241.center) to (246);
					\draw [style=protected] (247.center) to (250.center);
					\draw [style=protected, in=-90, out=90, looseness=1.25] (250.center) to (249.center);
					\draw (246) to (240);
					\draw [in=-90, out=165] (255) to (251.center);
					\draw [in=-90, out=15] (255) to (252.center);
					\draw (260) to (253.center);
					\draw (252.center) to (261.center);
					\draw (256.center) to (262);
					\draw [style=protected] (263.center) to (266.center);
					\draw [style=protected, in=-90, out=90, looseness=1.25] (266.center) to (265.center);
					\draw [in=270, out=90] (262) to (255);
					\draw [in=-90, out=165] (271) to (267.center);
					\draw [in=-90, out=15] (271) to (268.center);
					\draw (275) to (269.center);
					\draw (268.center) to (276.center);
					\draw (272.center) to (277);
					\draw [style=protected] (278.center) to (281.center);
					\draw [style=protected, in=-90, out=90, looseness=1.25] (281.center) to (280.center);
					\draw (277) to (271);
					\draw (13.center) to (282.center);
					\draw (236.center) to (283.center);
					\draw (251.center) to (284.center);
					\draw (267.center) to (285.center);
				\end{pgfonlayer}
			\end{tikzpicture}
		}%
	\end{equation}
	where the $h$ and $k$ both have domain $(W,\id) \otimes (Y,q)$ in $\suppcomp$.
	Indeed we can set the left input to be $(W,\id)$ instead of a generic object $(W,s)$ without loss of generality.
	This is because, upon absorbing $s$ into the test morphism $h$ and $k$, 
	Property \eqref{eq:supp_comp_abs_cont} implies the same property with $(W,\id)$ replaced by $(W,s)$ and the reverse implication is clear.

	Using the causality of $\cC$, the antecedent of Implication \eqref{eq:supp_comp_abs_cont} can be equivalently written as
	\begin{equation}
		{%
			\tikzstyle{every picture}=[tikzfig]%
			\begin{tikzpicture}
				\begin{pgfonlayer}{nodelayer}
					\node [style=none] (332) at (-3.25, 1.5) {};
					\node [style=none] (333) at (-1.75, 1.75) {};
					\node [style=none] (334) at (-3.75, 3) {};
					\node [style=bn] (335) at (-2.5, 0.75) {};
					\node [style=none] (336) at (-2.5, -2.75) {};
					\node [style=none] (337) at (0, 0) {$=$};
					\node [style=morphism] (338) at (-3.75, 2) {$\,\;\; h \,\;\;$};
					\node [style=none] (339) at (-1.75, 3) {};
					\node [style=morphism] (340) at (-2.5, -0.25) {$g$};
					\node [style=none] (341) at (-4.25, -2.75) {};
					\node [style=none] (342) at (-4.25, 2) {};
					\node [style=none] (343) at (-4.25, 0) {};
					\node [style=none] (344) at (2.75, 1.5) {};
					\node [style=none] (345) at (4.25, 1.75) {};
					\node [style=none] (346) at (2.25, 3) {};
					\node [style=bn] (347) at (3.5, 0.75) {};
					\node [style=none] (348) at (3.5, -2.75) {};
					\node [style=morphism] (349) at (2.25, 2) {$\,\;\; k \,\;\;$};
					\node [style=none] (350) at (4.25, 3) {};
					\node [style=morphism] (351) at (3.5, -0.25) {$g$};
					\node [style=none] (352) at (1.75, -2.75) {};
					\node [style=none] (353) at (1.75, 2) {};
					\node [style=none] (354) at (1.75, 0) {};
					\node [style=none] (355) at (-3.25, 2) {};
					\node [style=none] (356) at (2.75, 2) {};
					\node [style=morphism] (357) at (-2.5, -1.75) {$r$};
					\node [style=morphism] (358) at (3.5, -1.75) {$r$};
					\node [style=none] (359) at (-4.25, -3.25) {$W$};
					\node [style=none] (360) at (-2.5, -3.25) {$Z$};
					\node [style=none] (361) at (-3.75, 3.5) {$A$};
					\node [style=none] (362) at (-1.75, 3.5) {$Y$};
					\node [style=none] (363) at (1.75, -3.25) {$W$};
					\node [style=none] (364) at (3.5, -3.25) {$Z$};
					\node [style=none] (365) at (2.25, 3.5) {$A$};
					\node [style=none] (366) at (4.25, 3.5) {$Y$};
				\end{pgfonlayer}
				\begin{pgfonlayer}{edgelayer}
					\draw [in=-90, out=165] (335) to (332.center);
					\draw [in=-90, out=15] (335) to (333.center);
					\draw (338) to (334.center);
					\draw (333.center) to (339.center);
					\draw (336.center) to (340);
					\draw [style=protected] (341.center) to (343.center);
					\draw [style=protected, in=-90, out=90, looseness=1.25] (343.center) to (342.center);
					\draw (340) to (335);
					\draw [in=-90, out=165] (347) to (344.center);
					\draw [in=-90, out=15] (347) to (345.center);
					\draw (349) to (346.center);
					\draw (345.center) to (350.center);
					\draw (348.center) to (351);
					\draw [style=protected] (352.center) to (354.center);
					\draw [style=protected, in=-90, out=90, looseness=1.25] (354.center) to (353.center);
					\draw (351) to (347);
					\draw (332.center) to (355.center);
					\draw (344.center) to (356.center);
				\end{pgfonlayer}
			\end{tikzpicture}
		}%
	\end{equation}
	where we use Implication \eqref{eq:causal_extrawire}.
	Since the same is true for the consequent of Implication \eqref{eq:supp_comp_abs_cont} with $f$ and $p$ in place of $g$ and $r$, we conclude that $f\comp p \ll g\comp r$ is equivalent to Implication \eqref{eq:supp_comp_abs_cont}.
\end{proof}

\begin{lemma}
	\label{lem:suppcomp_complete}
	A morphism $[f]\colon (X,p) \to (Y,q)$ in $\suppcomp$ is atomic if and only if $f \comp p$ is atomic in $\cC$.
	In this case, $(Y,f\comp p)$ is a support of $[f]$ with support inclusion
	\begin{equation}
		\label{eq:suppinc_suppcomp}
		[\id_Y] \colon (Y,f\comp p) \to (Y,q).
	\end{equation}
\end{lemma}
\begin{proof}
	The first claim is straightforward by \cref{lem:abscont_suppcompletion}.
	Combining this result with \cref{def:atomic} and the definition of copying in $\suppcomp$ via \eqref{eq:copy_suppcomp} shows that atomicity of $[f]$ in $\suppcomp$ can be expressed as 
	\begin{equation}
		\cop \comp f \comp p \ll (f \comp p) \otimes (f \comp p)
	\end{equation}
	in $\cC$, which says exactly that $f \comp p$ is atomic.
	Thus under the assumption that $[f]$ is atomic, $(Y, f\comp p)$ is itself an object of $\suppcomp$.

	For the second claim, we use the conditions of \cref{lem:supp_crit}.
	Since $\id_Y$ is a deterministic monomorphism, it is easy to verify that the proposed support inclusion $\suppinc{}$ given by $[\id_Y] \colon (Y, f \comp p) \to (Y, q)$ is a deterministic monomorphism as well, and likewise for every $\id_{(W,r)} \otimes \suppinc{}$.\footnotemark{}
	\footnotetext{The reader may wonder whether $\suppinc{}$ is also an epimorphism. 
	From the definition, one can check that this is true if and only if $f\comp p \acsim q$ holds.}%

	By \cref{lem:abscont_suppcompletion}, the condition $\suppinc{} \ll [f]$ from \cref{lem:supp_crit} is immediate because it corresponds to $f\comp p \ll f\comp p$ in $\cC$.

	Finally, to show Implication \eqref{eq:supp_crit}, consider a morphism $[g] \colon (Z,r) \to (W,s) \otimes (Y,q)$ satisfying $[g] \ll \id \otimes [f]$.
	Then we have $g\comp r \ll s \otimes f\comp p$ in $\cC$, so that $g$ defines a morphism $(Z,r) \to (W,s) \otimes (Y,f\comp p)$ in $\suppcomp$, which provides the desired factorization of $[g]$ across $\id_{(W, s)} \otimes \suppinc{}$.
\end{proof}

We will show that the free regular support completion is not only support complete, but is so in the following stronger sense.

\begin{definition}
	\label{def:suppreg}
	A Markov category $\cC$ is \newterm{regular support complete} if it is support complete and for all atomic $h \colon X \to Y$ and arbitrary $f : A \to \Supp{h}$ and $g : B \to \Supp{h}$, we have
	\begin{equation}\label{eq:regularity}
		\suppinc{h} \comp f \ll \suppinc{h} \comp g \quad \implies \quad f \ll g.
	\end{equation}
\end{definition}

One might call this kind of support a ``regular support'', but we will not dwell on this concept any further in this paper.

\begin{remark}
	Our terminology lines up with the fact that coregular categories with the extra assumption of \cref{prop:coregular_supports} are regular support complete: by the arguments given there, $\suppinc{h} \comp f \ll \suppinc{h} \comp g$ means that $\suppinc{h} \comp f$ factors across $\suppinc{\suppinc{h} \comp g} = \suppinc{h} \comp \suppinc{g}$, where the latter equation holds because regular monomorphisms are closed under composition.
	But then the desired factorization of $f$ across $\suppinc{g}$ follows simply because $\suppinc{h}$ is a monomorphism.
\end{remark}

Here is now our main result on the free regular support completion.

\begin{theorem}
	\label{thm:supports_suppcompletion}
	For every causal Markov category $\cC$:
	\begin{enumerate}
		\item The free regular support completion $\suppcomp$ is regular support complete.
		\item Any atomic morphism $p \colon A \to X$ in $\cC$ has support $(X,p)$ in the larger category $\suppcomp$.
	\end{enumerate}
\end{theorem}

\begin{proof}
	\begin{enumerate}
		\item $\suppcomp$ is support complete by \cref{lem:suppcomp_complete}. 
			For regularity, suppose that we are given morphisms $[h] \colon (X,p) \to (Y,q)$ and $[f] \colon (Z,r) \to (Y, h \comp p)$ and $[g] \colon (W,s) \to (Y, h \comp p)$, where we identify $\Supp{[h]} = (Y, h \comp p)$ by \cref{lem:suppcomp_complete}.
			Then $\suppinc{[h]} \comp [f] \ll \suppinc{[h]} \comp [g]$ is manifestly equivalent to $[f] \ll [g]$, since composing with $\suppinc{[h]}$ just amounts to replacing $h \comp p$ by $q$, and thanks to \cref{lem:abscont_suppcompletion} we know that this plays no role in the absolute continuity relation.
		\item In light of \cref{lem:abscont_suppcompletion}, $[p] \colon (A,\id_A) \to (X,\id_X)$ is still atomic in $\suppcomp$.
			Hence this follows from \cref{lem:suppcomp_complete} as well.
			\qedhere
	\end{enumerate}
\end{proof}

\begin{corollary}\label{cor:supp_iff_atomic}
	Let $\cC$ be a causal Markov category. A morphism admits a support in some Markov category containing $\cC$ as a Markov subcategory if and only if it is atomic.
\end{corollary}
\begin{proof}
	If $p$ is atomic, then \cref{thm:supports_suppcompletion} shows that $p$ has a support in $\suppcomp$. 
	Conversely, if we assume that $p$ is not atomic, then $\cop \comp p \ll p \otimes p$ fails already in $\cC$, i.e.\ there exist morphisms $f$ and $g$ satisfying both $f \ase{p \otimes p} g$ and $f \not\ase{\cop\, \comp p} g$.
	This remains the case in any Markov category containing $\cC$ as a Markov subcategory.
	The contrapositive of \cref{lem:suppdominance} says that such $p$ cannot have a support. 
\end{proof}

\begin{example}\label{ex:lebesgue_nosupp}
	By applying \cref{cor:supp_iff_atomic} to $\borelstoch$, we conclude that the Lebesgue measure (as well as any other non-atomic morphism) will never admit a support, even if we enlarge our category.
\end{example}

We conclude the discussion of the free regular support completion by proving a universal property for it.
To this aim, we introduce the following properties of functors.
\begin{definition}
	Let $F\colon \cC \to \cD$ be a Markov functor between Markov categories (see \cref{def:markov_functor}). We say that $F$ is 
	\begin{itemize} 
		\item \newterm{ac-monotone} if $f\ll g$ in $\cC$ implies $F(f)\ll F(g)$ in $\cD$.
		\item \newterm{support-preserving} if for every morphism $p$ in $\cC$ which has a support, also $F(p)$ in $\cD$ has a support with support inclusion $F(\suppinc{p})$.
	\end{itemize}
\end{definition}

In particular, every ac-monotone functor maps atomic morphisms to atomic morphisms.
Also the inclusion functor $\cC \to \suppcomp$ is ac-monotone by \cref{lem:abscont_suppcompletion}.

\begin{remark}
	However, the inclusion $\cC \to \suppcomp$ is not support-preserving in general, which we can establish as follows.
	Consider a causal Markov category $\cC$ and a morphism $p\colon A \to X$ in $\cC$ that admits a support $S$ (with support inclusion $\suppinc{}$) and is thus atomic by \cref{lem:suppdominance}.
	The inclusion of $p$'s support in $\cC$ is $(\Supp{},\id_{\Supp{}})$, while the support of $[p]$ in $\suppcomp$ is $(X,p)$.

	If the inclusion functor preserved supports, this would entail that $p$ has a \emph{split support} in the sense of the upcoming \cref{sec:split_supports}, which is not the case in general (\cref{prop:support_not_split}).
	Indeed, the support inclusion $\suppinc{}$ yields a morphism $[\suppinc{}]\colon (\Supp{},\id_{\Supp{}})\to (X,p)$ because $\suppinc{}\ll p$. 
	Moreover, under the assumption that $\cC \to \suppcomp$ is support-preserving, this morphism must be an isomorphism because supports are unique up to isomorphism. 
	Therefore, there exists $[\pi]\colon (X,p)\to (\Supp{},\id_{\Supp{}})$ such that $[\pi]\comp[\suppinc{}]=\id_{(\Supp{},\id_{\Supp{}})}$. 
	Translating this back to $\cC$, we have $\pi\comp \suppinc{}=\id_{\Supp{}}$, thus proving that $\suppinc{}$ is a split support.

	This is in contradiction with the example from \cref{prop:support_not_split}, since $\tychstoch$ is indeed causal by \cref{ex:tychstoch_supp}.
	This suggests that the inclusion functor of $\tychstoch$ into its free regular support completion may not be an equivalence of categories, although $\tychstoch$ already has all supports.

	These phenomena underline again why we speak of the \emph{regular} part of the free regular support completion: 
	By adding new supports, $\suppcomp$ can break the existing ones.
	This is analogous to the free regular completion of a category with finite limits, where the analogous inclusion functor can break existing image factorizations.\footnote{This is clear from how image factorizations in the regular completion are constructed~\cite[p.~27]{johnstone2002elephant}.}
\end{remark}

\begin{theorem}[Universal property of the free regular support completion]
	\label{thm:suppcomp_universal}
	Let $\cC$ and $\cD$ be causal Markov categories, and further assume $\cD$ to be regular support complete.
	Then composition with the inclusion functor $\cC \to \suppcomp$ induces an equivalence of categories between:
	\begin{itemize}
		\item The category of ac-monotone Markov functors $\cC \to \cD$ and monoidal natural transformations, and
		\item The category of support-preserving ac-monotone Markov functors $\suppcomp \to \cD$ and monoidal natural transformations.
	\end{itemize}
\end{theorem}
Informally speaking, $\suppcomp$ is hence the free regular support complete Markov category generated by $\cC$.
\begin{proof}
	By \cref{thm:supports_suppcompletion}, every object $(X, p)$ in $\suppcomp$, where $p \colon A \to X$ is an atomic morphism in $\cC$, is the support of the morphism $[p] \colon (A, \id_A) \to (X, \id_X)$ in $\suppcomp$ with support inclusion given by
	\begin{equation}\label{eq:UP_suppinc}
		\iota_{(X,p)} \coloneqq [\id_X] \colon (X, p) \to (X, \id_X).
	\end{equation}
	Then for any support-preserving Markov functor $\tilde{F} \colon \suppcomp \to \cD$, the support inclusion of $\tilde{F}([p])$ is given by the image of the support inclusion of $[p]$, i.e.\ we have
	\begin{equation}\label{eq:suppinc_F}
		\suppinc{\tilde{F}([p])} = \tilde{F}(\iota_{(X,p)}).
	\end{equation}
	In particular, $\suppinc{\tilde{F}([p])} \colon \tilde{F}(X,p) \to \tilde{F}(X,\id_X)$ is a deterministic monomorphism.

	Let now $\tilde{F}, \tilde{G} \colon \suppcomp \to \cD$ be two support-preserving ac-monotone Markov functors and $F, G \colon \cC \to \cD$ their restrictions to $\cC$.
	Then we prove that whiskering by $\cC \to \suppcomp$ induces a bijection between monoidal natural transformations $\tilde{F} \to \tilde{G}$ and monoidal natural transformations $F \to G$.
	The injectivity holds because the component of a natural transformation at any object $(X,p)$ is determined by its component at $(X,\id_X)$ thanks to the naturality on $\iota_{(X,p)}$ and monicity of $\suppinc{\tilde{G}([p])}$.

	For the surjectivity, suppose that we are given a monoidal natural transformation $\alpha \colon F \to G$.
	Then we define a natural transformation $\tilde{\alpha} \colon \tilde{F} \to \tilde{G}$ by setting its component at an object $(X,p)$ to be the unique morphism making the diagram
	\begin{equation}
		\begin{tikzcd}[row sep=large, column sep=large]
			\tilde{F}(X,p) \ar[r, "{\suppinc{\tilde{F}([p])}}"] \ar[d, "\tilde{\alpha}_{(X,p)}" description] & \tilde{F}(X,\id_X) \ar[d, "\alpha_X"] \\
			\tilde{G}(X,p) \ar[r, "{\suppinc{\tilde{G}([p])}}"'] & \tilde{G}(X,\id_X)
		\end{tikzcd}
	\end{equation}
	commute, whose uniqueness and existence follow from the functoriality of supports (\cref{thm:supp_functorial}),\footnote{Strictly speaking, \cref{thm:supp_functorial} assumes supports to exist for all morphisms, but the same proof as given there shows the relevant functoriality for all supports that exist.}
	since we have the commutative diagram 
	\begin{equation}
		\begin{tikzcd}[row sep=large, column sep=large]
			\tilde{F}(A,\id_A) \ar[r, "{\tilde{F}([p])}"] \ar[d, "\alpha_A"'] & \tilde{F}(X,\id_X) \ar[d, "\alpha_X"] \\
			\tilde{G}(A,\id_A) \ar[r, "{\tilde{G}([p])}"'] & \tilde{G}(X,\id_X)
		\end{tikzcd}
	\end{equation}
	by the naturality of $\alpha$ (and using the inclusion $\cC \to \suppcomp$ implicitly).

	Let us now argue that the thus defined $\tilde{\alpha}$ is indeed natural and monoidal.
	Concerning naturality, consider a morphism $[f]\colon (X,p)\to (Y,q)$ and note that commutativity of the central square in the diagram
	\begin{equation}
		\begin{tikzcd}[row sep=large, column sep=large]
			F(X)\ar[rrr, "F(f)"]\ar[ddd,"\alpha_X" description] &&& F(Y)\ar[ddd,"\alpha_Y" description] \\
									    & \tilde{F} (X,p) \ar[r, "{\tilde{F}([f])}"]\ar[d,"\tilde{\alpha}_{(X,p)}" description] \ar[lu,"\suppinc{\tilde{F}([p])}" description] 
									    & \tilde{F} (Y,q) \ar[d,"\tilde{\alpha}_{(Y,q)}" description]\ar[ru,"\suppinc{\tilde{F}([q])}" description] & \\
									    &  \tilde{G} (X,p) \ar[r, "{\tilde{G}([f])}"]\ar[ld,"\suppinc{\tilde{G}([p])}" description] 
									    & \tilde{G} (Y,q) \ar[rd,"\suppinc{\tilde{G}([q])}" description] & \\
			G(X) \ar[rrr, "G(f)"] &&& G(Y)
		\end{tikzcd}
	\end{equation}
	follows from the commutativity of the outer diagram (because $\alpha \colon F\to G$ is natural) and that of the trapezoids (the left and right trapezoids commute by definition of $\tilde{\alpha}_{(X,p)}$ and $\tilde{\alpha}_{(Y,q)}$, while the upper and lower ones do by functoriality), since $\suppinc{\tilde{G}([q])}$ is a monomorphism.
	Similarly, monoidality of the extension is checked using the diagram 
	\begin{equation*}
		\begin{tikzcd}[row sep=large, column sep=large]
			F(X)\otimes F(Y)\ar[rrr,"\phi"]\ar[ddd,"\alpha_X\otimes \alpha_Y" description] &&& F(X\otimes Y) \ar[ddd,"\alpha_{X\otimes Y}" description] \\
												       & \tilde{F} (X,p)\otimes \tilde{F} (Y,q) \ar[r,"\tilde{\phi}"]\ar[d,"\tilde{\alpha}_{(X,p)} \otimes \tilde{\alpha}_{(Y,q)}" description]\ar[lu,"\suppinc{\tilde{F}([p])} \otimes \suppinc{\tilde{F}([q])}" description] & \tilde{F} (X\otimes Y,p \otimes q) \ar[d,"\tilde{\alpha}_{(X,p) \otimes (Y,q)}" description]\ar[ru,"\suppinc{\tilde{F}([p \otimes q])}" description] & \\
												       &  \tilde{G} (X,p) \otimes \tilde{G} (Y,q) \ar[r,"\tilde{\psi}"]\ar[ld,"\suppinc{\tilde{G}([p])}\otimes \suppinc{\tilde{G} ([q])}" description] & \tilde{G} (X\otimes Y,p \otimes q) \ar[rd,"\suppinc{\tilde{G}([p\otimes q])}" description] & \\
			G(X)\otimes G(Y) \ar[rrr,"\psi"] &&& G(X\otimes Y)
		\end{tikzcd}
	\end{equation*}
	where $\phi$, $\tilde{\phi}$, $\psi$, and $\tilde{\psi}$ denote the components of the strongators of $F$, $\tilde{F}$, $G$, and $\tilde{G}$, respectively.

	Here, the outer square commutes by monoidality of $\alpha$, the left and right trapezoids by definition of $\tilde{\alpha}$ and the upper and lower trapezoids by \eqref{eq:suppinc_F}, so that we have 
	\[
		\phi \comp \bigl(\tilde{F}(\iota_{(X,p)})\otimes \tilde{F}(\iota_{(Y,q)}) \bigr) = \tilde{F}(\iota_{(X,p)}\otimes \iota_{(Y,q)}) \comp \tilde{\phi} = \tilde{F}(\iota_{(X \otimes Y, p \otimes q)})\comp \tilde{\phi},
	\]
	and likewise for $\psi$ and $\tilde{\psi}$, where the second equality is by \eqref{eq:UP_suppinc}.
	Therefore also the inner square commutes because $\suppinc{\tilde{G}([p\otimes q])}$ is a monomorphism, and hence $\tilde{\alpha}$ is compatible with the strongators.
	The compatibility with unitors is more straightforward and omitted here.
	In conclusion, our extended $\tilde{\alpha}$ is a monoidal natural transformation, and we have thus shown that the functor under consideration is fully faithful.

	We now turn to proving essential surjectivity.
	We assume that choices of supports have been made in $\cD$, and are such that the support of each identity morphism is the underlying object itself with support inclusion given by identity.
	Given an ac-monotone Markov functor $F\colon \cC \to \cD$, we define a functor $\tilde{F}\colon \suppcomp \to \cD$ as follows.
	On objects, we set
	\begin{equation}
		\tilde{F}(X,p) \coloneqq \Supp{F(p)}.
	\end{equation}
	For a morphism $[f] \colon (X,p) \to (Y,q)$ in $\suppcomp$, we take $\tilde{F}([f])\colon \Supp{F(p)} \to \Supp{F(q)}$ to be the unique morphism such that
	\begin{equation}\label{eq:supp_comp_functor}
		F(f) \comp \suppinc{F(p)} = \suppinc{F(q)} \comp \tilde{F}([f]).
	\end{equation}
	The existence of such $\tilde{F}([f])$ follows by \cref{lem:supp_crit}~\ref{it:suppinc_fact}, since we have
	\begin{equation}
		F(f) \comp \suppinc{F(p)} \ll F(f) \comp F(p) \ll F(q),
	\end{equation}
	where the first relation is by \cref{lem:ac_mon_comp,lem:inc_sim_p} and the second one by functoriality and ac-monotonicity of $F$.
	Furthermore, $\tilde{F}([f])$ is well-defined with respect to the equivalence relation $\ase{p}$ because the left-hand side of \cref{eq:supp_comp_functor} is invariant on its equivalence classes thanks to \cref{lem:supp_asfaithful}.
	Functoriality of $\tilde{F}\colon \suppcomp \to \cD$ is a straightforward check based on the monicity of support inclusions, and $\tilde{F}$ strictly extends $F$ by the innocuous assumption on supports of identity morphisms.

	We next show that $\tilde{F}$ is a strong monoidal functor.
	By definition, $\tilde{F}(X\otimes Y,p\otimes q) = \Supp{F(p\otimes q)}$.
	Since $F$ is strong monoidal with strongators $\phi$, we have a commutative diagram
	\begin{equation}
		\begin{tikzcd}[column sep=large]
			F(A)\otimes F(B)\ar[r,"F(p)\otimes F(q)"] \ar[d,"\phi"] & F(X)\otimes F(Y) \ar[d,"\phi"] \\
			F(A\otimes B) \ar[r,"F(p\otimes q)"] & F(X\otimes Y)
		\end{tikzcd}
	\end{equation}
	Therefore the universal property of supports gives us the first isomorphism in
	\begin{equation}
		\label{eq:strongator_tildeF}
		\begin{tikzcd}[column sep=large]
			\Supp{F(p\otimes q)} \ar[r, "\cong"] & \Supp{F(p)\otimes F(q)} \ar[r, "\cong"] & \Supp{F(p)} \otimes \Supp{F(q)},
		\end{tikzcd}
	\end{equation}
	where the second isomorphism is by \cref{thm:supp_multiplicativity}.
	We claim that the composite isomorphisms are coherence isomorphisms for $\tilde{F}$.
	Indeed to prove naturality on $f_i \colon (X_i,p_i)\to (Y_i,q_i)$ for $i=1,2$, we use the commuting diagram
	\[
		\begin{tikzcd}[column sep=large]
			\Supp{F(p_1)}\otimes \Supp{F(p_2)} \ar[r,"\cong"]\ar[d,"\suppinc{}\otimes \suppinc{}"] & \Supp{F(p_1 \otimes p_2)} \ar[r,"\tilde{F}(f_1 \otimes f_2)"]\ar[d,"\suppinc{}"] & \Supp{F(q_1 \otimes q_2)} \ar[r, "\cong"]\ar[d,"\suppinc{}"] & \Supp{F(q_1)}\otimes \Supp{F(q_2)} \ar[d,"\suppinc{}\otimes \suppinc{}"]\\
			F(X_1)\otimes F(X_2) \ar[rrr, bend right=10, "F(f_1)\otimes F(f_2)" below]\ar[r,"\phi"] & F(X_1\otimes X_2) \ar[r,"F(f_1 \otimes f_2)"] & F(Y_1 \otimes Y_2) \ar[r,"\phi^{-1}"] & F(Y_1)\otimes F(Y_2)
		\end{tikzcd}
	\] 
	From this we infer that the upper row composes to $\tilde{F}(f_1)\otimes \tilde{F}(f_2)$ because this morphism makes the outer diagram commute and $\suppinc{}\otimes \suppinc{}$ is a monomorphism.
	Similarly, one checks all the other properties for $\tilde{F}$ to be a strong symmetric monoidal functor.

	To see that $\tilde{F}$ is actually a Markov functor, we note that the top horizontal row in the commutative diagram
	\begin{equation}
		\begin{tikzcd}
			\Supp{F(p)} \ar[r,"\tilde{F}(\cop)"]\ar[d,"\suppinc{}"] &\Supp{F(p\otimes p)} \ar[r,"\cong"]\ar[d,"\suppinc{}"] & \Supp{F(p)}\otimes \Supp{F(p)} \ar[d,"\suppinc{}\otimes \suppinc{}"]\\
			F(X)\ar[r,"F(\cop)"]\ar[rr,bend right=20,"\cop" below] & F(X\otimes X) \ar[r,"\phi^{-1}"] & F(X) \otimes F(X)
		\end{tikzcd}
	\end{equation}
	is necessarily the copy morphism because $\suppinc{}$ is deterministic and $\suppinc{}\otimes \suppinc{}$ is a monomorphism.

	For ac-monotonicity, suppose now that we have $[f] \colon (X,p) \to (Y,q)$ and $[g] \colon (Z,r) \to (Y,q)$ with $[f] \ll [g]$.
	By \cref{lem:abscont_suppcompletion}, this means that $f\comp p \ll g\comp r$ holds in $\cC$.
	Since $F$ is ac-monotone, we obtain $F(f) \comp F(p) \ll F(g) \comp F(r)$, which gives
	\begin{equation}
		F(f) \comp \suppinc{F(p)} \acsim F(f) \comp F(p) \ll F(g) \comp F(r) \acsim F(g) \comp \suppinc{F(r)}
	\end{equation}
	via a combination of \cref{lem:ac_mon_comp,lem:inc_sim_p}.
	\Cref{eq:supp_comp_functor} lets us then deduce
	\begin{equation}
		\suppinc{F(q)} \comp \tilde{F}([f]) \ll \suppinc{F(q)} \comp \tilde{F}([g]).
	\end{equation}
	By regularity (Implication \eqref{eq:regularity}), this yields the desired $\tilde{F}([f]) \ll \tilde{F}([g])$.

	Finally, we need to show that $\tilde{F}$ is support-preserving. 
	To this end we consider a generic morphism in $\suppcomp$ with a support, i.e.\ an atomic $[g]\colon (Z,r)\to (Y,q)$, where $r$ is an atomic morphism of type $A \to Z$ in $\cC$. 
	By \cref{lem:suppcomp_complete}, we know that its support is $(Y,g\comp r)$ and support inclusion is $[\id_Y] \colon (Y,g\comp r) \to (Y,q)$.

	Since $[g]$ factors across $(Y,g\comp r)$, by functoriality of $\tilde{F}$ we get that $\tilde{F}([g])$ factors across $\Supp{F(g \comp r)}$ via
	\begin{equation}
		\tilde{F}([\id_Y]) \colon \Supp{F(g\comp r)}\to \Supp{F(q)}.
	\end{equation}
	By \cref{lem:factoring_suppincs} and the definition of $\tilde{F}$ via \cref{eq:supp_comp_functor}, $\tilde{F}([\id_Y])$ has support $\Supp{F(g \comp r)}$ and is its own support inclusion.

	By \cref{lem:inc_sim_p}, we have that $[g]$ is absolutely bicontinuous with respect to its own support inclusion $[\id_Y]$.
	Since we have already established that $\tilde{F}$ is ac-monotone, we get $\tilde{F}([g]) \acsim \tilde{F}([\id_Y])$.
	Using \cref{prop:abs_bicont_same_supp}~\ref{it:abs_bicont}, we obtain that $\tilde{F}([g])$ has support ${\Supp{F(g \comp r)} = \tilde{F}(Y,g \comp r)}$ with support inclusion given by $\tilde{F}([\id_Y])$.
	This means that $\tilde{F}$ is indeed support-preserving.
\end{proof}

\subsection{The input-output relation functor}
\label{sec:rel}

In this and the next subsections, we present several applications of supports and absolute continuity.
These are not otherwise used in the rest of the paper.

For a morphism $p \colon A \to X$ in $\finstoch$, it is sometimes of interest to consider the set of all pairs $(a, x) \in A \times X$ satisfying $p(x|a) > 0$.
This set is a relation from $A$ to $X$ which we call the \emph{input-output relation}, since it assigns to each input $a$ the set of possible outputs $x$.
Mapping every stochastic matrix $p$ to its input-output relation defines an identity-on-objects functor $\finstoch \to \setmulti$.
As we show below, one can associate an input-output relation to morphisms in any suitably well-behaved Markov category $\cC$, and this results in a Markov functor $\cC \to \setmulti$.
Whenever this is possible, one can say that every morphism has a consistent ``possibilistic shadow'' \cite{gonda2023framework}. 

Concretely, we map each object $A$ of $\cC$ to the set of its deterministic states 
\begin{equation}\label{eq:beh_object}
	\detp{A} \coloneqq \cC_\det(I,A),
\end{equation}
which one may also think of as ``points'' of $A$.
Further, we map a morphism $p \colon A \to X$ to the \newterm{input-output relation} $\pbeh{p}$ given by
\begin{equation}\label{eq:beh_morphism}
\pbeh{p} \coloneqq \Set*[\big]{ (a, x) \in \detp{A} \times \detp{X}  \given  p \comp a \gg x }.
\end{equation}
In other words, this relation is defined by
\begin{equation}
	a \pbeh{p} x  \quad \coloniff \quad p\comp a \gg x,
\end{equation}
and we think of this as stating that the output $x$ is possible for $p$ under input $a$.
In order to ensure the functoriality of the assignment $p \mapsto \pbeh{p}$, we need an additional assumption.

\begin{definition}\label{def:point_liftings}
	A Markov category $\cC$ has \newterm{point liftings} if for all $p \colon A \to X$ and all $x \in \detp{X}$ satisfying $p \gg x$, there exists an $a \in \detp{A}$ satisfying $p\comp a \gg x$.
\end{definition}

Intuitively, this says that every output that can occur under $p$ must be able to arise from \emph{some} input.

\begin{remark}
	\Cref{def:point_liftings} can be easier to check if supports exist, which is going to be the case in the assumptions of \cref{thm:Rel_embedding} below.
	For in this case, \cref{lem:supp_crit} implies that for every $a \in \detp{A}$, we have a canonical bijection
	\begin{equation}\label{eq:beh_morphism_supp}
		\Set{ x \in \detp{X} \given a \pbeh{p} x } \cong \detp{\Supp{p \comp a}}
	\end{equation}
	between points absolutely continuous with respect to $p \comp a$ on the left and deterministic states ``within the support of $p\comp a$'' on the right.
	The point lifting condition then says that each point $x \in \detp{X}$ in the support of $p$ is also in the support of the image of some point $a \in \detp{A}$ \mbox{under $p$}.
\end{remark}

\begin{example}
	$\finstoch$ has point liftings because the support of a stochastic matrix is the union of its ``pointwise'' supports (i.e.\ supports of its columns).
\end{example}

\begin{example}
	$\borelstoch$ also has point liftings.
	Indeed, by the characterization of absolute continuity of \cref{ex:abs_cont_stoch} used in the contrapositive, $p \gg x$ means that if $x$ is an element of a measurable set $S \in \Sigma_X$, then $p(S|a) > 0$ holds for some $a \in A$.
	Taking $S = \{x\}$ shows that there is $a \in A$ such that $x$ is an atom of the probability measure $p(\ph|a)$.
	This is why we also have the required $p \comp a \gg x$.
\end{example}

For an object $A \in \cC$, we have $\discard_A \gg \id_I$ as soon as $\cC(I,A) \neq \emptyset$, since then one can witness the relevant instance of the absolute continuity implication~\eqref{eq:abs_cont} by pre-composing with any state on $A$.
In this case, the point lifting property implies that $A$ also must have a deterministic state.

\begin{example}
	Let $G$ be a finite group, and let $\finstoch^G$ be the Markov category of finite $G$-sets together with $G$-equivariant Markov kernels.
	Then $\finstoch^G$ does not have point liftings, unless $G$ is the trivial group.
	Indeed the set $A \coloneqq G$ itself with its canonical $G$-action has a state given by the uniform distribution, but it has no deterministic state unless $G$ is trivial.
\end{example}

\begin{theorem}\label{thm:Rel_embedding}
	Let $\cC$ be a causal Markov category with supports and point liftings.
	Then there is an \newterm{input-output relation functor} $\Upsilon \colon \cC \to \setmulti$ given on objects and morphisms by
	\[
		\Upsilon(A) \coloneqq \detp{A}, \qquad \Upsilon(p) \coloneqq \pbehord{p},
	\]
	as defined in \eqref{eq:beh_object} and \eqref{eq:beh_morphism}.
	Moreover, $\Upsilon$ is a Markov functor.
\end{theorem}
\begin{proof}
	We show first that $\Upsilon(p)$ is actually a morphism in $\setmulti$ for every $p \colon A \to X$.
	This amounts to proving that for every $a \in \detp{A}$, there is $x \in \detp{X}$ with $a \pbeh{p} x$, or equivalently $p \comp a \gg x$.
	Indeed the support object $\Supp{p \comp a}$ has a state given by the support factorization $\suppfactor{p \comp a}$, and by what was noted above this implies that $\Supp{p \comp a}$ also has a deterministic state.
	Composing this state with the support inclusion $\suppinc{p \comp a}$ produces $x \in \detp{X}$.
	The desired $p \comp a \gg x$ then follows by $p \comp a \acsim \suppinc{p \comp a}$ and \cref{lem:factor_ac}.

	Next, we show that $\Upsilon$ is a functor.
	Proving that $\Upsilon$ preserves identities amounts to showing that if $a \gg b$ for $a,b \in \detp{A}$, then $a = b$.
	Since every deterministic state is its own support by \cref{cor:split_mono_supp}, the assumption $a \gg b$ implies that $b$ factors across $a$, and is thus equal to $a$.

	Consider now arbitrary morphisms $p \colon A \to X$ and $f \colon X \to Y$.
	We need to show that for every deterministic $a \colon I \to A$ and $y \colon I \to Y$, we have
	\begin{equation}
		a \pbeh{f\comp p} y \quad \iff \quad a \mathbin{\left( \pbeh{f} \circ \pbeh{p} \right)} y,
	\end{equation}
	where $\pbeh{f} \circ \pbeh{p}$ is the composite relation.
	This amounts to showing 
	\begin{equation}\label{eq:shadow_split_iff}
		f \comp p\comp a \gg y   \quad \iff \quad  \exists \, x \in \detp{X} \; \colon \;  p \comp a \gg x \text{ and } f\comp x \gg y.
	\end{equation}
	We consider both implications.
	\begin{itemize}
		\item[$(\Leftarrow)$:] Using the causality assumption and \cref{lem:ac_mon_comp}, we have 
			\begin{equation}
				p\comp a \gg x  \quad \implies \quad  f\comp p\comp a \gg f\comp x.
			\end{equation}
			Together with the transitivity of $\gg$ and the assumed $f\comp x \gg y$, this gives the desired $f\comp p\comp a \gg y$.

		\item[$(\Rightarrow)$:] We now assume that $f\comp p\comp a \gg y$ holds.
			Since $\suppinc{p\comp a} \gg p\comp a$ holds by \cref{lem:inc_sim_p}, we also have $f \comp \suppinc{p\comp a} \gg f\comp p\comp a$ by \cref{lem:ac_mon_comp}, and therefore
			\begin{equation}
				\label{eq:rel_fun_proof}
				f \comp \suppinc{p\comp a} \gg y
			\end{equation}
			holds by the transitivity of $\gg$.
			Hence, by point liftings, there is a deterministic $\tilde{x} \colon I \to \Supp{p\comp a}$ that satisfies
			\begin{equation}
				f \comp \suppinc{p\comp a} \comp \tilde{x} \gg y.
			\end{equation}
			If we define $x \coloneqq \suppinc{p\comp a} \comp \tilde{x}$, then $f x \gg y$ is necessarily true.
			The required $p\comp a \gg x$ then holds by $p\comp a \gg \suppinc{p \comp a} \gg \suppinc{p\comp a} \comp \tilde{x}$, where the first relation is by \cref{lem:inc_sim_p} again and the second by \cref{lem:factor_ac}.
			Since $x$ is deterministic, this concludes the proof of the forward implication \mbox{in \eqref{eq:shadow_split_iff}}. 
	\end{itemize}

	Having shown that $\Upsilon$ is a functor, let us provide the details of its strong monoidal structure.
	First, the monoidal structure morphisms are the relations given by the graphs of the functions
	\begin{equation}
		\begin{split}
			\psi_0 \: \colon \: I & \longrightarrow \detp{I} \\
			\sngltn & \longmapsto \id_I
		\end{split}
	\end{equation}
	where the monoidal unit $I \in \setmulti$ is a singleton set with element $\sngltn$, and
	\begin{equation}
		\begin{split}
			\psi_{A,B} \: \colon \: \detp{A} \otimes \detp{B}  &\longrightarrow \detp{A \otimes B} \\
			(a,b) &\longmapsto a \otimes b.
		\end{split}
	\end{equation}
	The only non-trivial condition to check in order for these to equip $\Upsilon$ with the structure of a strong symmetric monoidal functor is the naturality of $\psi$.
	This amounts to showing that, for all morphisms $p \colon A \to X$ and $q \colon B \to Y$ and all deterministic states
	\[
		a \in \detp{A}, \qquad b \in \detp{B}, \qquad x \in \detp{X}, \qquad y \in \detp{Y},
	\]
	we have
	\begin{equation}
		(a \otimes b) \pbeh{p \otimes q} (x \otimes y) \quad \iff \quad a \pbeh{p} x \: \land \: b \pbeh{q} y,
	\end{equation}
	or equivalently
	\begin{equation}
		p \comp a \otimes q\comp b \gg x \otimes y  \quad \iff \quad  p\comp a \gg x \: \land \: q\comp b \gg y.
	\end{equation}
	The right-to-left implication is clear by \cref{prop:ac_tensor}.
	In the other direction, we can post-compose with marginalization maps in order to obtain the claim from \cref{lem:ac_mon_comp}.

	To conclude the proof that $\Upsilon$ is a Markov functor, it remains to be shown that we have $\Upsilon(\cop_A) = \psi_{A,A} \comp \cop_{\Upsilon(A)}$ for every $A \in \cC$.
	This amounts to proving that for all $a, b, c \in \detp{A}$, we have
	\begin{equation}
		\cop_A {}\comp a \gg b \otimes c  \quad \iff \quad  a = b = c.
	\end{equation}
	Using $\cop_A {}\comp a = a \otimes a$, this follows by the same arguments as in the previous paragraph.
\end{proof}

\begin{example}
	\label{ex:finstoch_relfunctor}
	In $\finstoch$, the Markov functor $\Upsilon \colon \finstoch \to \setmulti$ is precisely the input-output relation functor discussed at the beginning of this subsection.
\end{example}

\begin{example}
	$\setmulti$ also satisfies the assumptions of \cref{thm:Rel_embedding}, and the associated input-output relation functor $\Upsilon \colon \setmulti \to \setmulti$ is the identity functor.
\end{example}

\begin{example}
	One may wonder how relevant the assumption of supports in \cref{thm:Rel_embedding} really is, given that its statement does not refer to supports in any way.
	To see that merely assuming point liftings is not sufficient, consider $\borelstoch$ and let $\lambda \colon I \to [0,1]$ be the Lebesgue measure, or similarly any other non-atomic measure on a standard Borel space.
	Then ${\Upsilon(\lambda) = \pbeh{\lambda}}$ is the empty relation since $\lambda$ has no atoms.
	We could then still hope to get a functor to $\rel$ rather than to $\setmulti$.
	If functoriality held, then also $\Upsilon(\discard_{[0,1]} {}\comp \lambda)$ would have to be the empty relation.
	But this is clearly equal to $\Upsilon(\discard_I) = \Upsilon(\id_I) = \id_{\{ \sngltn \}}$, which is not the empty relation.
	Therefore functoriality does not hold.

	Incidentally, looking at the proof around \eqref{eq:rel_fun_proof} shows that the proof would go through if $\lambda$ had a support.
	We therefore obtain another proof of the fact that $\lambda$ does not have a support, as in \cref{thm:borelstoch_supports}.
\end{example}

In \cite[Definition~IV.4]{stein2021conditioning}, the property of a Markov category with ``precise supports'' was given in the context of conditioning.
Note that this notion refers only to the absolute continuity relation $\gg$ rather than to supports in the sense of this paper.\footnotemark{}
\footnotetext{Also, the definition of absolute continuity used in \cite{stein2021conditioning} was still the ``old'' one without the extra input wire in \cref{def:ac}.}%
More concretely, a Markov category $\cC$ is said to have \newterm{precise supports} if, for all $x \in \detp{X}$ and $y \in \detp{Y}$ and arbitrary morphisms $p \colon I \to X$ and $f \colon X \to Y$, the relation
\begin{equation}\label{eq:precise_support_1}
	{%
		\tikzstyle{every picture}=[tikzfig]%
		\begin{tikzpicture}
			\begin{pgfonlayer}{nodelayer}
				\node [style=none] (0) at (0, 0) {$\gg$};
				\node [style=bn] (1) at (-3.25, -0.25) {};
				\node [style=morphism] (2) at (-2.25, 1) {$f$};
				\node [style=none] (3) at (-2.25, 1) {};
				\node [style=none] (4) at (-4.25, 1) {};
				\node [style=none] (5) at (-2.25, 2) {};
				\node [style=none] (6) at (-2.25, 2.5) {$Y$};
				\node [style=none] (7) at (-4.25, 2) {};
				\node [style=none] (8) at (-4.25, 2.5) {$X$};
				\node [style=state] (9) at (-3.25, -1.25) {$p$};
				\node [style=state] (10) at (2.25, -0.5) {$x$};
				\node [style=none] (11) at (2.25, 1) {};
				\node [style=none] (12) at (2.25, 1.5) {$X$};
				\node [style=state] (13) at (4, -0.5) {$y$};
				\node [style=none] (14) at (4, 1) {};
				\node [style=none] (15) at (4, 1.5) {$Y$};
			\end{pgfonlayer}
			\begin{pgfonlayer}{edgelayer}
				\draw [in=-90, out=15] (1) to (3.center);
				\draw [in=-90, out=165] (1) to (4.center);
				\draw (4.center) to (7.center);
				\draw (9) to (1);
				\draw (2) to (5.center);
				\draw (10) to (11.center);
				\draw (13) to (14.center);
			\end{pgfonlayer}
		\end{tikzpicture}
	}%
\end{equation}
is equivalent to 
\begin{equation}\label{eq:precise_support_2}
	{%
		\tikzstyle{every picture}=[tikzfig]%
		\begin{tikzpicture}
			\begin{pgfonlayer}{nodelayer}
				\node [style=none] (0) at (-5.5, 0) {$\gg$};
				\node [style=none] (1) at (-7.5, 1) {};
				\node [style=none] (2) at (-7.5, 1.5) {$X$};
				\node [style=state] (3) at (-7.5, -0.5) {$p$};
				\node [style=state] (4) at (-3.5, -0.5) {$x$};
				\node [style=none] (5) at (-3.5, 1) {};
				\node [style=none] (6) at (-3.5, 1.5) {$X$};
				\node [style=none] (7) at (0, 0) {and};
				\node [style=none] (8) at (5.5, 0) {$\gg$};
				\node [style=none] (9) at (3.5, 1.25) {};
				\node [style=none] (10) at (3.5, 1.75) {$Y$};
				\node [style=state] (11) at (3.5, -1) {$x$};
				\node [style=state] (12) at (7.5, -0.5) {$y$};
				\node [style=none] (13) at (7.5, 1) {};
				\node [style=none] (14) at (7.5, 1.5) {$Y$};
				\node [style=morphism] (15) at (3.5, 0.25) {$f$};
			\end{pgfonlayer}
			\begin{pgfonlayer}{edgelayer}
				\draw (4) to (5.center);
				\draw (3) to (1.center);
				\draw (12) to (13.center);
				\draw (11) to (9.center);
			\end{pgfonlayer}
		\end{tikzpicture}
	}%
\end{equation}
This property follows from point liftings in a causal Markov category.
\begin{proposition}\label{prop:precise_supports}
	If a causal Markov category $\cC$ has point liftings, then it also has precise supports.
\end{proposition}
\begin{proof}
	First, note that we can use \cref{lem:ac_mon_comp} to marginalize $Y$ in Relation \eqref{eq:precise_support_1} to get $p \gg x$.
	Furthermore, Relation \eqref{eq:precise_support_1} also implies
	\begin{equation}
		{%
			\tikzstyle{every picture}=[tikzfig]%
			\begin{tikzpicture}
				\begin{pgfonlayer}{nodelayer}
					\node [style=none] (0) at (0, 0) {$\gg$};
					\node [style=bn] (1) at (-3.25, -0.25) {};
					\node [style=morphism] (2) at (-2.25, 1) {$f$};
					\node [style=none] (3) at (-2.25, 1) {};
					\node [style=none] (4) at (-4.25, 1) {};
					\node [style=none] (5) at (-2.25, 2) {};
					\node [style=none] (6) at (-2.25, 2.5) {$Y$};
					\node [style=none] (7) at (-4.25, 2) {};
					\node [style=none] (8) at (-4.25, 2.5) {$X$};
					\node [style=state] (10) at (2.25, -0.5) {$x$};
					\node [style=none] (11) at (2.25, 1) {};
					\node [style=none] (12) at (2.25, 1.5) {$X$};
					\node [style=state] (13) at (4, -0.5) {$y$};
					\node [style=none] (14) at (4, 1) {};
					\node [style=none] (15) at (4, 1.5) {$Y$};
					\node [style=none] (16) at (-3.25, -1.5) {$X$};
					\node [style=none] (17) at (-3.25, -1) {};
				\end{pgfonlayer}
				\begin{pgfonlayer}{edgelayer}
					\draw [in=-90, out=15] (1) to (3.center);
					\draw [in=-90, out=165] (1) to (4.center);
					\draw (4.center) to (7.center);
					\draw (2) to (5.center);
					\draw (10) to (11.center);
					\draw (13) to (14.center);
					\draw (17.center) to (1);
				\end{pgfonlayer}
			\end{tikzpicture}
		}%
	\end{equation}
	by \cref{lem:factor_ac}.
	Using point liftings, we get that there exists a deterministic $a \colon I \to X$ such that we have 
	\begin{equation}
		{%
			\tikzstyle{every picture}=[tikzfig]%
			\begin{tikzpicture}
				\begin{pgfonlayer}{nodelayer}
					\node [style=none] (0) at (0, 0) {$\gg$};
					\node [style=morphism] (2) at (-2.5, 0.5) {$f$};
					\node [style=none] (3) at (-2.5, 0.5) {};
					\node [style=none] (4) at (-4.5, 0.5) {};
					\node [style=none] (5) at (-2.5, 1.5) {};
					\node [style=none] (6) at (-2.5, 2) {$Y$};
					\node [style=none] (7) at (-4.5, 1.5) {};
					\node [style=none] (8) at (-4.5, 2) {$X$};
					\node [style=state] (10) at (2.5, -0.5) {$x$};
					\node [style=none] (11) at (2.5, 1) {};
					\node [style=none] (12) at (2.5, 1.5) {$X$};
					\node [style=state] (13) at (4.25, -0.5) {$y$};
					\node [style=none] (14) at (4.25, 1) {};
					\node [style=none] (15) at (4.25, 1.5) {$Y$};
					\node [style=state] (16) at (-4.5, -0.75) {$a$};
					\node [style=state] (17) at (-2.5, -0.75) {$a$};
				\end{pgfonlayer}
				\begin{pgfonlayer}{edgelayer}
					\draw (4.center) to (7.center);
					\draw (2) to (5.center);
					\draw (10) to (11.center);
					\draw (13) to (14.center);
					\draw (17) to (3.center);
					\draw (16) to (4.center);
				\end{pgfonlayer}
			\end{tikzpicture}
		}%
	\end{equation}
	By \cref{lem:ac_mon_comp} again, we get $a \gg x$ and $f\comp a \gg y$ by marginalization.
	Note that $a$ is its own support inclusion (\cref{cor:split_mono_supp}), so that $x$ factors across $a$.
	In other words, we have $x = a$, so that Condition \eqref{eq:precise_support_2} follows.

	Conversely, using \cref{lem:ac_mon_comp}, we can post-compose the relation $p \gg x$ appearing in \eqref{eq:precise_support_2} to get the left-most relation in
	\begin{equation}
		{%
			\tikzstyle{every picture}=[tikzfig]%
			\begin{tikzpicture}
				\begin{pgfonlayer}{nodelayer}
					\node [style=none] (0) at (0, 0) {$\gg$};
					\node [style=bn] (1) at (-3, -0.25) {};
					\node [style=morphism] (2) at (-2, 1) {$f$};
					\node [style=none] (3) at (-2, 1) {};
					\node [style=none] (4) at (-4, 1) {};
					\node [style=none] (5) at (-2, 2) {};
					\node [style=none] (6) at (-2, 2.5) {$Y$};
					\node [style=none] (7) at (-4, 2) {};
					\node [style=none] (8) at (-4, 2.5) {$X$};
					\node [style=state] (9) at (-3, -1.25) {$p$};
					\node [style=bn] (10) at (3, -0.25) {};
					\node [style=morphism] (11) at (4, 1) {$f$};
					\node [style=none] (12) at (4, 1) {};
					\node [style=none] (13) at (2, 1) {};
					\node [style=none] (14) at (4, 2) {};
					\node [style=none] (15) at (4, 2.5) {$Y$};
					\node [style=none] (16) at (2, 2) {};
					\node [style=none] (17) at (2, 2.5) {$X$};
					\node [style=state] (18) at (3, -1.25) {$x$};
					\node [style=none] (19) at (5.75, 0) {$=$};
					\node [style=morphism] (21) at (10.25, 0.75) {$f$};
					\node [style=none] (22) at (10.25, 0.75) {};
					\node [style=none] (23) at (8.25, 0.75) {};
					\node [style=none] (24) at (10.25, 1.75) {};
					\node [style=none] (25) at (10.25, 2.25) {$Y$};
					\node [style=none] (26) at (8.25, 1.75) {};
					\node [style=none] (27) at (8.25, 2.25) {$X$};
					\node [style=state] (28) at (8.25, -0.75) {$x$};
					\node [style=state] (29) at (10.25, -0.75) {$x$};
					\node [style=none] (30) at (12.5, 0) {$\gg$};
					\node [style=state] (31) at (14.75, -0.75) {$x$};
					\node [style=none] (32) at (14.75, 0.75) {};
					\node [style=none] (33) at (14.75, 1.25) {$X$};
					\node [style=state] (34) at (16.5, -0.75) {$y$};
					\node [style=none] (35) at (16.5, 0.75) {};
					\node [style=none] (36) at (16.5, 1.25) {$Y$};
				\end{pgfonlayer}
				\begin{pgfonlayer}{edgelayer}
					\draw [in=-90, out=15] (1) to (3.center);
					\draw [in=-90, out=165] (1) to (4.center);
					\draw (4.center) to (7.center);
					\draw (9) to (1);
					\draw (2) to (5.center);
					\draw [in=-90, out=15] (10) to (12.center);
					\draw [in=-90, out=165] (10) to (13.center);
					\draw (13.center) to (16.center);
					\draw (18) to (10);
					\draw (11) to (14.center);
					\draw (23.center) to (26.center);
					\draw (21) to (24.center);
					\draw (29) to (22.center);
					\draw (28) to (23.center);
					\draw (31) to (32.center);
					\draw (34) to (35.center);
				\end{pgfonlayer}
			\end{tikzpicture}
		}%
	\end{equation}
	where the right-most relation follows from the assumed $f\comp x \gg y$ by \cref{prop:ac_tensor}.
	In conclusion, Condition \eqref{eq:precise_support_2} implies Condition \eqref{eq:precise_support_1}.
\end{proof}

\subsection{Categories of statistical models}\label{sec:statistical_models}

As another application of absolute continuity and supports, we now consider categories of statistical models.
Concrete measure-theoretic versions of such categories have been studied since the pioneering work of \v{C}encov~\cite{cencovcategories,cencovstatistical}.
Recent works in this direction include a paper by Jost, L\^e and Tran~\cite{jostmorphisms} and the thesis of Patterson~\cite{patterson2020models}.

Here, we introduce the category of parametric statistical models internal to any causal and representable Markov category. 
Subsequently, in \cref{prop:parametric_vs_unparametric} we use supports to show that in the measure-theoretic setting of standard Borel spaces, a parametric statistical model consisting of a family of distributions $p = (p_a)_{a \in A}$ is isomorphic to the model given by its image in $PX$, provided that this image is measurable.

We assume familiarity with representable Markov categories~\cite{fritz2023representable}; see \cref{sec:observational} for a brief recap.

\begin{definition}
	\label{def:statC}
	Let $\cC$ be a Markov category that is both causal and representable.
	Then its \newterm{category of statistical models} $\stat(\cC)$ is the symmetric monoidal category where:
	\begin{itemize}
		\item Objects are pairs $(X,p)$ where $X \in \cC$ and $p \colon A \to X$ is any morphism in $\cC$.
		\item Morphisms $(X,p) \to (Y,q)$ are $\as{p}$ equivalence classes of morphisms $f \colon X \to Y$ such that ${P(f) \comp p^\sharp \ll q^\sharp}$, or equivalently $(f\comp p)^\sharp \ll q^\sharp$, holds.
		\item Composition and symmetric monoidal structure are inherited from $\cC$.
	\end{itemize}
\end{definition}

The idea is that a pair $(X,p \colon A \to X)$ can be interpreted as a statistical model with $A$ as parameter space~\cite[Definition~14.1]{fritz2019synthetic}.
The absolute continuity condition on $f \colon X \to Y$ intuitively encodes the requirement that $f$ must map distributions on $X$ compatible with the model $p$ to distributions on $Y$ compatible with the model $q$ (\cref{ex:statistical_stoch}).
If $f, g \colon X \to Y$ satisfy $f \ase{p} g$, then they intuitively behave the same on all distributions that are compatible with $p$, and hence we consider them as representatives of the same morphism in $\stat(\cC)$.
The monoidal structure on objects of $\stat(\cC)$ is given by the monoidal product of morphisms in $\cC$, the idea being that if $(X,p \colon A \to X)$ and $(Y,q \colon B \to Y)$ are statistical models, then $(X\otimes Y, p \otimes q)$ is a statistical model with parameter space $A \otimes B$ such that the factors $X$ and $Y$ are modeled as independent.

For the following proof, note that in a causal Markov category, we have
\begin{equation}\label{eq:ac_from_distributions}
	P(f) \comp p^\sharp = (f \comp p)^\sharp \ll q^\sharp  \quad \implies \quad f\comp p \ll q
\end{equation}
by using post-composition with $\samp_Y$ and \cref{lem:ac_mon_comp}.

\begin{proposition}
	\label{prop:stat_category}
	$\stat(\cC)$ is a symmetric monoidal category.
\end{proposition}
The proof bears similarity to the analogous statement for the free regular support completion, \cref{prop:suppcomp}.
\begin{proof}
	We first consider composition.
	If $[f] \colon (X,p) \to (Y,q)$ and $[g] \colon (Y,q) \to (Z,r)$ are morphisms, then also $[g\comp f] \colon (X,p) \to (Z,r)$ is, since
	\begin{equation}
		P(g\comp f) \comp p^\sharp = P(g) \comp P(f) \comp p^\sharp  \ll P(g) \comp q^\sharp \ll r^\sharp,
	\end{equation}
	where the middle step follows by \cref{lem:ac_mon_comp} upon composing the assumed relation ${P(f) \comp p^\sharp \ll q^\sharp}$ with $P(g)$.
	We also need to show that this composition is well-defined with respect to the equivalence classes. 
	So consider a morphism $f' \colon X \to Y$ with $f' \ase{p} f$.
	Then, by \cref{prop:ase_props} \ref{it:ase_post_comp}, we also have $g\comp f \ase{p} g\comp f'$, which amounts to the desired $[g\comp f] = [g\comp f']$.
	Similarly, if $g' \colon Y \to Z$ satisfies $g' \ase{q} g$, then we obtain $g \ase{f\comp p} g'$ by $f\comp p \ll q$, which holds by Implication \eqref{eq:ac_from_distributions}.
	Hence the required $g\comp f \ase{p} g'\comp f$ follows by \cref{lem:as_shift}.

	It is clear that the symmetric monoidal structure morphisms of $\cC$ also belong to $\stat(\cC)$.
	Moreover, \cref{prop:ac_tensor} shows that the tensor product of two morphisms respects the absolute continuity condition: 
	For $[f] \colon (X,p) \to (Y,q)$ and $[f'] \colon (X',p') \to (Y',q')$, we have
	\begin{equation}
		\label{eq:stat_tensor_ac}
		(f \comp p)^\sharp \otimes (f' \comp p')^\sharp \ll q^\sharp \otimes q'^\sharp
	\end{equation}
	by \cref{prop:ac_tensor} and the assumed $(f\comp p)^\sharp \ll q^\sharp$ and $(f'\comp p')^\sharp \ll q'^\sharp$.
	This is an absolute continuity relation between morphisms with codomain $PY \otimes PY'$.
	Now recall that in terms of the symmetric monoidal structure of $P$ given by
	\begin{equation}
		\nabla \: \colon \: PY \otimes PY' \longrightarrow P(Y \otimes Y'),
	\end{equation}
	we have 
	\begin{equation}
		(q \otimes q')^\sharp = \nabla \comp (q^\sharp \otimes q'^\sharp),
	\end{equation}
	and similarly for the left-hand side of \cref{eq:stat_tensor_ac}.
	Therefore composing this absolute continuity relation with $\nabla$ yields the desired
	\begin{equation}
		(f\comp p \otimes f'\comp p')^\sharp \ll (q \otimes q')^\sharp
	\end{equation}
	by \cref{lem:ac_mon_comp}.
	Finally, the well-definedness of the monoidal product with respect to $\ase{p}$-classes is straightforward: 
	For any $f_2$ satisfying $f_2 \ase{p} f$, then we also have $f_2 \otimes f' \ase{p \otimes p'} f \otimes f'$ by \mbox{\cref{prop:ase_props} \ref{it:ase_tensor}}.
\end{proof}

\begin{proposition}
	\label{prop:parametric_vs_unparametric}
	Let $p \colon A \to X$ be such that $p^\sharp \colon A \to PX$ has a support.
	Then there is an isomorphism
	\begin{equation}
		(X,p) \cong (X,\samp_X \comp \suppinc{p^\sharp})
	\end{equation}
	in $\stat(\cC)$.
\end{proposition}

Intuitively, the support of $p^\sharp$ coincides with its image of $PX$, which is the set of distributions compatible with the statistical model $(X,p)$.
Therefore the proposition states that $(X,p)$ is equivalent as a statistical model to this image, as one would expect.

\begin{proof}
	Let us write $q \colon \Supp{p^\sharp} \to X$ as shorthand for the morphism $\samp_X \comp \suppinc{p^\sharp}$ in $\cC$.
	We show that $\id_X$ gives both a morphism $(X,p) \to (X,q)$ and a morphism $(X,q) \to (X,p)$ in $\stat(\cC)$.
	Since $P(\id_X) = \id_{PX}$, we have to show that $p^\sharp \acsim q^\sharp$ holds.
	Indeed,
	\begin{equation}
		q^\sharp = (\samp_X \comp \suppinc{p^\sharp})^\sharp = \suppinc{p^\sharp}
	\end{equation}
	follows because $\suppinc{p^\sharp}$ is deterministic.
	Then the required $p^\sharp \acsim \suppinc{p^\sharp}$ holds by \cref{lem:inc_sim_p}.
\end{proof}

\begin{example}
	\label{ex:statistical_stoch}
	Consider any statistical model $(X,p \colon A \to X)$ in $\borelstoch$, which we can think of as an $A$-indexed family of distributions $(p_a)_{a \in A}$ on $X$.
	The measurable map ${p^\sharp \colon A \to PX}$ is precisely the one that assigns to every parameter value $a \in A$ the associated distribution $p_a \in PX$.

	If the image of $p^\sharp$, namely the set of distributions
	\begin{equation}
		\Set{ p_a  \given  a \in A } \subseteq PX,
	\end{equation}
	is measurable, then its support is given by this image (\cref{cor:support_det_borelstoch}).
	In this case, \cref{prop:parametric_vs_unparametric} applies, and we find that the parametric statistical model $(X,p \colon A \to X)$ is isomorphic to a model given simply by a subset of $PX$, namely the set of distributions compatible with the model.
\end{example}

\subsection{The equalizer principle}
\label{sec:equalizer_principle}

In this subsection, we introduce an axiom for Markov categories that can sometimes be used as a substitute for supports when these do not exist.
This axiom plays an important role in \cref{sec:idempotents}.

\begin{definition}
	\label{def:equalizer_principle}
	A Markov category $\cC$ satisfies the \newterm{equalizer principle} if:
	\begin{enumerate}
		\item Equalizers in $\cC_\det$ exist.
		\item\label{it:equalizer_as} For every equalizer diagram
			\begin{equation}
				\label{eq:equalizer}
				\begin{tikzcd}
					E \ar[r, "\mathsf{eq}"] & X \ar[r, shift right, "g"'] \ar[r, shift left, "f"] & Y
				\end{tikzcd}
			\end{equation}
			in $\cC_\det$, every $p \colon A \to X$ in $\cC$ satisfying $f \ase{p} g$ factors uniquely across $\mathsf{eq}$.
	\end{enumerate}
\end{definition}

Since finite products exist in $\cC_\det$ by virtue of it being cartesian monoidal, the existence of equalizers implies that $\cC_\det$ is even finitely complete.

The relation between the equalizer principle and supports is as follows.

\begin{remark}
	\label{rem:equalizer_principle}
	Property \ref{it:equalizer_as} holds automatically for every $p$ which has a support.
	Because in this case, the assumption $f \ase{p} g$ is equivalent to $f \comp \suppinc{} = g \comp \suppinc{}$ by \cref{lem:supp_asfaithful}, and therefore $\suppinc{}$ factors across $\mathsf{eq}$ by the universal property.
	Furthermore, since we have $p = \suppinc{} \comp \suppfactor{p}$, it follows that $p$ itself factors across $\mathsf{eq}$.

	In particular, if $\cC_\det$ has equalizers and $\cC$ has supports, then the equalizer principle holds.
	For example:
	\begin{enumerate}
		\item $\finstoch$ satisfies the equalizer principle, since its deterministic subcategory is $\finset$ (which is finitely complete) and every morphism has a support by \cref{ex:support_finstoch}.

		\item $\setmulti$ satisfies the equalizer principle, since its deterministic subcategory is $\set$ (which is complete), and every morphism has a support by \cref{ex:support_setmulti}.

		\item $\tychstoch$ satisfies the equalizer principle, since its deterministic subcategory is the category of Tychonoff spaces and continuous maps (which has equalizers since every subspace of a Tychonoff space is Tychonoff) and every morphism has a support by \cref{ex:top_support}.
	\end{enumerate}
	In the other direction, we apply the equalizer principle in the proof of \cref{thm:balanced_split} to \emph{construct} supports for certain morphisms.
	Such relations to supports illustrate why we believe that the equalizer principle can sometimes be used as a substitute for supports.
\end{remark}

\begin{remark}
	Even when the equalizer principle holds, an equalizer in $\cC_\det$ is usually \emph{not} an equalizer in $\cC$.
	In other words, merely having $f p = g p$ instead of $f \ase{p} g$ is not enough to guarantee that $p$ factors across $\mathsf{eq}$.
	For example in $\finstoch$, consider any finite set $Y$ with $|Y| \ge 2$, and let $f$ and $g$ be given by the two deterministic projections $Y \otimes Y \to Y$.
	Then their equalizer in $\finstoch_\det = \finset$ is given by the inclusion of the diagonal.
	Taking $p$ to be the uniform probability measure on $Y \otimes Y$, we have $f p = g p$, but $p$ does not factor across the inclusion of the diagonal.
\end{remark}

We cannot derive the equalizer principle from supports in measure-theoretic probability, since supports do typically not exist there (\cref{thm:borelstoch_supports}).
Nevertheless, we have the following positive result.

\begin{proposition}\label{prop:borelstoch_eq}
	$\borelstoch$ satisfies the equalizer principle.	
\end{proposition}

\begin{proof}
	Equalizers in $\borelstoch_\det = \borelmeas$ clearly exist, since for a given parallel pair as in \cref{eq:equalizer}, we simply consider the subset
	\begin{equation}
		E \coloneqq \Set{ x \in X \given f(x) = g(x) },
	\end{equation}
	which is measurable in $X$ since $f$ and $g$ are measurable maps and the diagonal of their output is measurable.
	Letting $\mathsf{eq} \colon E \hookrightarrow X$ be the inclusion, the universal property of an equalizer is obvious.

	Property \ref{it:equalizer_as} is more interesting.
	Let $p \colon A \to X$ be any Markov kernel satisfying $f \ase{p} g$.
	By Kuratowski's theorem, we know that $Y$ is either discrete and (at most) countable, or isomorphic to $\R$.
	In each of the cases, we can embed $Y$ into $\R$ while preserving the equalizer as well as the equation $f \ase{p} g$.
	Therefore, we can assume $Y = \R$ without loss of generality.

	In short, $f$ and $g$ are assumed to be real-valued measurable functions which are $p(\ph|a)$-almost surely equal for every $a \in A$.
	That is, we have $p(X \setminus E | a) = 0$ for all $a \in A$, and hence $p$ factors across the inclusion $\mathsf{eq} \colon E \hookrightarrow X$.
	The uniqueness of the factorization is obvious.
\end{proof}

The following stronger version of the equalizer principle may also be of interest, although it does not appear in this paper besides the present \cref{sec:equalizer_principle}.
We prove that it is also satisfied for Markov kernels in $\borelstoch$ and illustrate its power by showing how it can be used to obtain the causality axiom.

\begin{definition}
	\label{def:relative_equalizer}
	A Markov category $\cC$ satisfies the \newterm{relative equalizer principle} if:
	\begin{enumerate}
		\item Equalizers in $\cC_\det$ exist.
		\item\label{it:rel_equalizer_as} For every equalizer diagram
			\begin{equation}
				\label{eq:rel_equalizer}
				\begin{tikzcd}
					E \ar[r, "\mathsf{eq}"] & X \ar[r, shift right, "g"'] \ar[r, shift left, "f"] & Y
				\end{tikzcd}
			\end{equation}
			in $\cC_\det$, every $p \colon A \to X$ satisfying $f \ase{p\comp q} g$ for some $q \colon Z \to A$ factors uniquely across $\mathsf{eq}$ up to $\as{q}$ equality.
	\end{enumerate}
\end{definition}

\begin{proposition}
	$\borelstoch$ satisfies the relative equalizer principle.
\end{proposition}
\begin{proof}
	Proceeding as in the previous proof, we have that $f,g \colon X \to \R$ are almost surely equal with respect to the probability measure
	\begin{equation}
		\int_{a \in A} p(\ph|a) \, q(da|z)
	\end{equation}
	for every $z \in Z$, so that we get
	\begin{equation}
		\int_{a \in A} p(X \setminus E | a) \, q(da|z) = 0.
	\end{equation}
	This implies that the integrand vanishes almost surely, meaning that with the notation
	\begin{equation}
	N \coloneqq \Set*[\big]{ a \in A \given p(X \setminus E | a) > 0 }
\end{equation}
we have $q(N|z) = 0$ for all $z \in Z$.
This implies the unique factorization as desired by constructing a Markov kernel $A \to E$ to be given by $p$ on $A \setminus N$ and arbitrarily\footnotemark{} on $N$.
\footnotetext{If $N$ is non-empty, this is possible only if $E$ is non-empty. 
	However, we can assume this without loss of generality.
	If $Z$ is the empty set, the overall claim is trivial.
On the other hand, if $Z$ is non-empty, then $f$ and $g$ are almost surely equal with respect to some probability measure, and thus $E$ is also non-empty.}%
\end{proof}

\begin{proposition}
	\label{prop:equalizer_causal}
	Let $\cC$ be an \as{}-compatibly representable\footnotemark{} Markov category satisfying the relative equalizer principle.
	Then $\cC$ is causal.
\end{proposition}
\footnotetext{See \cref{sec:observational} for the definition.}%
\begin{proof}
	We use the reformulation of the causality axiom from \cref{lem:as_shift}.
	Let $h_1 \ase{g\comp f} h_2$ be given, where $f \colon X \to Y$ and $g \colon Y \to Z$ and $h_1, h_2 \colon Z \to A$ are morphisms of the appropriate type.
	By \as{}-compatible representability, we obtain $h_1^\sharp \ase{g\comp f} h_2^\sharp$, and it is enough to show $h_1^\sharp \comp g \ase{f} h_2^\sharp \comp g$.
	To simplify the notation, we therefore just assume that $h_1$ and $h_2$ are deterministic without loss of generality.

	With $\mathsf{eq}$ denoting the equalizer of $h_1$ and $h_2$ in $\cC_\det$, the relative equalizer principle applied to $h_1 \ase{g\comp f} h_2$ gives us the dashed morphism in
	\begin{equation}
		\begin{tikzcd}[row sep=large]
			E \ar[rr, hookrightarrow, "\mathsf{eq}"] & & Z \ar[r, shift right, "h_2"'] \ar[r, shift left, "h_1"] & A \\
			X \ar[rr, "f"'] & & Y \ar[u, "g"] \ar[ull, dashed, "\ell" description]
		\end{tikzcd}
	\end{equation}
	where the triangle commutes up to $\ase{f}$. Therefore we have
	\begin{equation}
		h_1 \comp g \ase{f} h_1 \comp \mathsf{eq} \comp \ell = h_2 \comp \mathsf{eq} \comp \ell \ase{f} h_2 \comp g,
	\end{equation}
	which amounts to the desired $h_1 \comp g \ase{f} h_2 \comp g$.
\end{proof}

\subsection{Split supports}
\label{sec:split_supports}

In this subsection, we study a stronger notion of support, which we call ``split support''.
In contrast to the mapping-in universal property of ``plain'' supports given in \cref{def:support}, this is a mapping-out universal property, extending the original one \cite[Definition 13.20]{fritz2019synthetic}.
This universal property is simpler in the sense that no compatibility with the monoidal structure is required, since it can in fact be derived (\cref{prop:splitsupp_tensor}).

\begin{example}[Split support for a stochastic map]\label{ex:split_sup_finstoch}
	Let us illustrate the basic idea by considering \cref{ex:first} again.
	For another set $Y=\{x,y,z\}$, let functions $f,g \colon X\to Y$ be given by
	
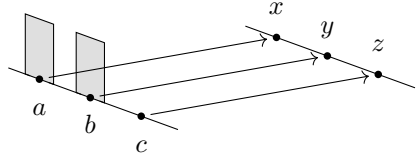
\[
		\begin{tikzpicture}[baseline={(current bounding box.center)},
			x={(-20:0.8cm)},y={(90:0.8cm)},z={(10:0.8cm)}]
			\node (a) at (0.55,0,-2) {};
			\draw [fill=fillcolor!50] (a) -- ++(0.25,0,0) -- ++(0,1,0) -- ++(-0.5,0,0) -- ++(0,-1,0) -- (a);
			\node [circle,inner sep=1pt,fill=black,label=below:$a\strut$] at (a) {};
			\node (b) at (1.45,0,-2) {};
			\draw [fill=fillcolor!50] (b) -- ++(0.25,0,0) -- ++(0,1,0) -- ++(-0.5,0,0) -- ++(0,-1,0) -- (b);
			\node [circle,inner sep=1pt,fill=black,label=below:$b\strut$] at (b) {};
			\node (c) at (2.35,0,-2) {};
			\node [circle,inner sep=1pt,fill=black,label=below:$c\strut$] at (c) {};
			\draw (0,0,-2) -- (3,0,-2) ;

			\node (x) at (0.55,0,2) {};
			\node [circle,inner sep=1pt,fill=black,label=above:$x\strut$] at (x) {};
			\node (y) at (1.45,0,2) {};
			\node [circle,inner sep=1pt,fill=black,label=above:$y\strut$] at (y) {};
			\node (z) at (2.35,0,2) {};
			\node [circle,inner sep=1pt,fill=black,label=above:$z\strut$] at (z) {};
			\draw (0,0,2) -- (3,0,2) ;

			\draw [->] (a) to (x) ;
			\draw [->] (b) to (y) ;
			\draw [->] (c) to (z) ;
		\end{tikzpicture}
		\qquad\qquad
		f(a) = x , \qquad f(b) = y , \qquad f(c) = z ;
	\]
	
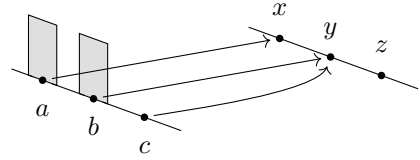
\[
		\begin{tikzpicture}[baseline={(current bounding box.center)},
			x={(-20:0.8cm)},y={(90:0.8cm)},z={(10:0.8cm)}]
			\node (a) at (0.55,0,-2) {};
			\draw [fill=fillcolor!50] (a) -- ++(0.25,0,0) -- ++(0,1,0) -- ++(-0.5,0,0) -- ++(0,-1,0) -- (a);
			\node [circle,inner sep=1pt,fill=black,label=below:$a\strut$] at (a) {};
			\node (b) at (1.45,0,-2) {};
			\draw [fill=fillcolor!50] (b) -- ++(0.25,0,0) -- ++(0,1,0) -- ++(-0.5,0,0) -- ++(0,-1,0) -- (b);
			\node [circle,inner sep=1pt,fill=black,label=below:$b\strut$] at (b) {};
			\node (c) at (2.35,0,-2) {};
			\node [circle,inner sep=1pt,fill=black,label=below:$c\strut$] at (c) {};
			\draw (0,0,-2) -- (3,0,-2) ;

			\node (x) at (0.55,0,2) {};
			\node [circle,inner sep=1pt,fill=black,label=above:$x\strut$] at (x) {};
			\node (y) at (1.45,0,2) {};
			\node [circle,inner sep=1pt,fill=black,label=above:$y\strut$] at (y) {};
			\node (z) at (2.35,0,2) {};
			\node [circle,inner sep=1pt,fill=black,label=above:$z\strut$] at (z) {};
			\draw (0,0,2) -- (3,0,2) ;

			\draw [->] (a) to (x) ;
			\draw [->] (b) to (y) ;
			\draw [->] (c) to[out=10,in=-110,looseness=0.4] (y) ;
		\end{tikzpicture}
		\qquad\qquad
		g(a) = x , \qquad g(b) = y , \qquad g(c) = y .
	\]
	Note that $f$ and $g$ are equal $p$-almost surely.
	The intuitive reason behind this is that restricting them to the support $\Supp{} = \{a,b\}$ of $p$ gives the \emph{same} function of type $\Supp{} \to Y$.
	This is true also for general supports as given by \cref{def:support} (see \cref{lem:supp_asfaithful}).
	However, notice an additional property that is \emph{not captured} by our definition of ``plain'' supports in Markov categories.
	Namely, every morphism $h \colon \Supp{} \to Y$ arises from a restriction as above.
	One can always extend such $h$ to a morphism $X \to Y$ (necessarily \as{$p$} uniquely), so that restricting to $\Supp{}$ produces $h$ itself again.

	Of course, to be able to extend every $h$, it is sufficient for the identity map $\Supp{} \to \Supp{}$ to be extendable, or equivalently for the support inclusion $\suppinc{} \colon \Supp{} \to X$ to have a left inverse ${\suppproj{} \colon X \to \Supp{}}$.
	In our example, we can take it to be any map satisfying 
	\begin{equation}
		\suppproj{}(a) = a , \qquad \suppproj{}(b) = b .
	\end{equation}
	The value of $\suppproj{}(c)$ is arbitrary.
\end{example}

We can express an analogous property in any Markov category.

\begin{definition}\label{def:split_support}
	Let $p \colon A \to X$ be a morphism in a Markov category $\cC$. 
	Then a \newterm{split support} for $p$ is a pre-support such that the pre-support inclusion $\suppinc{} \colon \Supp{} \to X$ has a left inverse ${\suppproj{} \colon X \to \Supp{}}$.
\end{definition}

We also call such a left inverse a \newterm{support projection} for $p$.
We assume a fixed choice of support projection $\suppproj{}$ throughout, while noting that the particular choice plays no role due to the following result.

\begin{lemma}\label{lem:supp_proj}
	The support inclusion and projection of a split support satisfy
	\begin{equation}
		\label{eq:supp_proj_id}
		\suppproj{} \comp \suppinc{} = \id_{\Supp{}}, \qquad 
		\suppinc{} \comp \suppproj{} \ase{p} \id_X.
	\end{equation}
	Moreover, $\suppproj{}$ is \as{$p$} deterministic and uniquely determined by the first equation up to \as{$p$} equality.
\end{lemma}

\begin{proof}
	The first equation is by definition, while the second is a consequence of \cref{lem:supp_asfaithful} and $\suppinc{} \comp \suppproj{} \comp \suppinc{} = \suppinc{}$.
	The uniqueness up to \as{$p$} equality is now immediate.
	Finally, we compute 
	\begin{equation}\label{eq:sp_det}
		{%
			\tikzstyle{every picture}=[tikzfig]%
			\begin{tikzpicture}
				\begin{pgfonlayer}{nodelayer}
					\node [style=none] (15) at (6, 0) {$\ase{p}$};
					\node [style=bn] (27) at (3, 1.25) {};
					\node [style=none] (29) at (4, 3.75) {};
					\node [style=none] (30) at (2, 3.75) {};
					\node [style=none] (32) at (3, -2.5) {};
					\node [style=none] (33) at (3, -3) {$X$};
					\node [style=morphism] (39) at (3, 0) {$\suppinc{}$};
					\node [style=none] (40) at (3, -1.5) {};
					\node [style=morphism] (41) at (3, -1.5) {$\suppproj{}$};
					\node [style=bn] (42) at (9, -0.25) {};
					\node [style=none] (44) at (10, 3.75) {};
					\node [style=none] (45) at (8, 3.75) {};
					\node [style=none] (46) at (10, 2.5) {};
					\node [style=none] (47) at (9, -2.5) {};
					\node [style=none] (48) at (9, -3) {$X$};
					\node [style=none] (51) at (8, 2.5) {};
					\node [style=morphism] (52) at (8, 2.5) {$\suppproj{}$};
					\node [style=morphism] (53) at (10, 2.5) {$\suppproj{}$};
					\node [style=bn] (73) at (-3, -0.25) {};
					\node [style=morphism] (74) at (-4, 1.25) {$\suppinc{}$};
					\node [style=none] (75) at (-2, 3.75) {};
					\node [style=none] (76) at (-4, 3.75) {};
					\node [style=none] (78) at (-3, -2.5) {};
					\node [style=none] (79) at (-3, -3) {$X$};
					\node [style=none] (80) at (-2, 4.25) {$S$};
					\node [style=none] (81) at (-4, 4.25) {$S$};
					\node [style=morphism] (87) at (-3, -1.5) {$\suppproj{}$};
					\node [style=none] (88) at (0, 0) {$=$};
					\node [style=none] (90) at (-4, 1) {};
					\node [style=none] (91) at (-2, 1) {};
					\node [style=morphism] (92) at (-2, 1.25) {$\suppinc{}$};
					\node [style=none] (93) at (2, 2.5) {};
					\node [style=none] (94) at (4, 2.5) {};
					\node [style=bn] (96) at (-9, -0.25) {};
					\node [style=none] (97) at (-8, 3.75) {};
					\node [style=none] (98) at (-10, 3.75) {};
					\node [style=none] (99) at (-9, -2.5) {};
					\node [style=none] (100) at (-9, -3) {$X$};
					\node [style=morphism] (103) at (-9, -1.5) {$\suppproj{}$};
					\node [style=none] (104) at (-10, 1) {};
					\node [style=none] (105) at (-8, 1) {};
					\node [style=morphism] (106) at (-4, 2.75) {$\suppproj{}$};
					\node [style=morphism] (107) at (-2, 2.75) {$\suppproj{}$};
					\node [style=morphism] (108) at (2, 2.75) {$\suppproj{}$};
					\node [style=morphism] (109) at (4, 2.75) {$\suppproj{}$};
					\node [style=none] (110) at (-6, 0) {$=$};
					\node [style=none] (111) at (8, 1) {};
					\node [style=none] (112) at (10, 1) {};
					\node [style=none] (113) at (-8, 4.25) {$S$};
					\node [style=none] (114) at (-10, 4.25) {$S$};
					\node [style=none] (115) at (10, 4.25) {$S$};
					\node [style=none] (116) at (8, 4.25) {$S$};
					\node [style=none] (117) at (4, 4.25) {$S$};
					\node [style=none] (118) at (2, 4.25) {$S$};
				\end{pgfonlayer}
				\begin{pgfonlayer}{edgelayer}
					\draw (41) to (39);
					\draw (39) to (27);
					\draw (32.center) to (41);
					\draw (46.center) to (53);
					\draw (47.center) to (42);
					\draw (78.center) to (87);
					\draw [in=-90, out=165] (73) to (90.center);
					\draw (87) to (73);
					\draw [in=-90, out=15] (73) to (91.center);
					\draw (90.center) to (74);
					\draw (91.center) to (92);
					\draw (94.center) to (29.center);
					\draw (93.center) to (30.center);
					\draw [in=-90, out=165] (27) to (93.center);
					\draw [in=-90, out=15] (27) to (94.center);
					\draw (99.center) to (103);
					\draw [in=-90, out=165] (96) to (104.center);
					\draw (103) to (96);
					\draw [in=-90, out=15] (96) to (105.center);
					\draw (106) to (76.center);
					\draw (107) to (75.center);
					\draw (92) to (107);
					\draw [in=270, out=90] (74) to (106);
					\draw (105.center) to (97.center);
					\draw (104.center) to (98.center);
					\draw [in=-90, out=165] (42) to (111.center);
					\draw [in=-90, out=15] (42) to (112.center);
					\draw (112.center) to (53);
					\draw (53) to (44.center);
					\draw (111.center) to (52);
					\draw (52) to (45.center);
				\end{pgfonlayer}
			\end{tikzpicture}
		}%
	\end{equation}
	to conclude that $\suppproj{}$ is \as{$p$} deterministic.
\end{proof}

While we define split supports in terms of pre-supports, they are actually stable under tensor products, as the next lemma shows.
This also justifies calling $\suppinc{}$ a support inclusion.
\begin{lemma}\label{lem:split_supp_is_supp}
	Every split support is a support.
\end{lemma}	
\begin{proof}
	We can use \cref{lem:supp_crit}, where the only non-trivial criterion to show is \cref{it:suppinc_fact}.
	To this end, consider a morphism $f \colon Z \to W \otimes X$ satisfying $\id_W \otimes p \gg f$.
	Applying \cref{lem:as_eq_new} to the second equality in \eqref{eq:supp_proj_id} gives
	\begin{equation}\label{eq:split_supp_is_supp_1}
		{%
			\tikzstyle{every picture}=[tikzfig]%
			\begin{tikzpicture}
				\begin{pgfonlayer}{nodelayer}
					\node [style=none] (15) at (0, 0) {$\ase{\id_W \otimes p}$};
					\node [style=none] (32) at (-3, -1.75) {};
					\node [style=none] (33) at (-3, -2.25) {$X$};
					\node [style=morphism] (39) at (-3, 0.75) {$\suppinc{}$};
					\node [style=none] (40) at (-3, -0.75) {};
					\node [style=morphism] (41) at (-3, -0.75) {$\suppproj{}$};
					\node [style=none] (44) at (4, 1.75) {};
					\node [style=none] (47) at (4, -1.75) {};
					\node [style=none] (48) at (4, -2.25) {$X$};
					\node [style=none] (115) at (4, 2.25) {$X$};
					\node [style=none] (119) at (-3, 1.75) {};
					\node [style=none] (120) at (-3, 2.25) {$X$};
					\node [style=none] (121) at (2.5, 1.75) {};
					\node [style=none] (122) at (2.5, -1.75) {};
					\node [style=none] (123) at (2.5, -2.25) {$W$};
					\node [style=none] (124) at (2.5, 2.25) {$W$};
					\node [style=none] (125) at (-5, 1.75) {};
					\node [style=none] (126) at (-5, -1.75) {};
					\node [style=none] (127) at (-5, -2.25) {$W$};
					\node [style=none] (128) at (-5, 2.25) {$W$};
				\end{pgfonlayer}
				\begin{pgfonlayer}{edgelayer}
					\draw (41) to (39);
					\draw (32.center) to (41);
					\draw (39) to (119.center);
					\draw (47.center) to (44.center);
					\draw (122.center) to (121.center);
					\draw (126.center) to (125.center);
				\end{pgfonlayer}
			\end{tikzpicture}
		}%
	\end{equation}
	so that we get 
	\begin{equation}\label{eq:split_supp_is_supp_2}
		{%
			\tikzstyle{every picture}=[tikzfig]%
			\begin{tikzpicture}
				\begin{pgfonlayer}{nodelayer}
					\node [style=none] (15) at (0, 0) {$\ase{f}$};
					\node [style=none] (32) at (-3, -1.75) {};
					\node [style=none] (33) at (-3, -2.25) {$X$};
					\node [style=morphism] (39) at (-3, 0.75) {$\suppinc{}$};
					\node [style=none] (40) at (-3, -0.75) {};
					\node [style=morphism] (41) at (-3, -0.75) {$\suppproj{}$};
					\node [style=none] (44) at (4, 1.75) {};
					\node [style=none] (47) at (4, -1.75) {};
					\node [style=none] (48) at (4, -2.25) {$X$};
					\node [style=none] (115) at (4, 2.25) {$X$};
					\node [style=none] (119) at (-3, 1.75) {};
					\node [style=none] (120) at (-3, 2.25) {$X$};
					\node [style=none] (121) at (2.5, 1.75) {};
					\node [style=none] (122) at (2.5, -1.75) {};
					\node [style=none] (123) at (2.5, -2.25) {$W$};
					\node [style=none] (124) at (2.5, 2.25) {$W$};
					\node [style=none] (125) at (-5, 1.75) {};
					\node [style=none] (126) at (-5, -1.75) {};
					\node [style=none] (127) at (-5, -2.25) {$W$};
					\node [style=none] (128) at (-5, 2.25) {$W$};
				\end{pgfonlayer}
				\begin{pgfonlayer}{edgelayer}
					\draw (41) to (39);
					\draw (32.center) to (41);
					\draw (39) to (119.center);
					\draw (47.center) to (44.center);
					\draw (122.center) to (121.center);
					\draw (126.center) to (125.center);
				\end{pgfonlayer}
			\end{tikzpicture}
		}%
	\end{equation}
	by the assumed absolute continuity.
	In particular, this implies
	\begin{equation}\label{eq:split_supp_is_supp_3}
		{%
			\tikzstyle{every picture}=[tikzfig]%
			\begin{tikzpicture}
				\begin{pgfonlayer}{nodelayer}
					\node [style=none] (15) at (0, 0) {$=$};
					\node [style=none] (32) at (-2.25, -1.25) {};
					\node [style=none] (33) at (-3, -3) {$Z$};
					\node [style=morphism] (39) at (-2.25, 1.5) {$\suppinc{}$};
					\node [style=none] (40) at (-2.25, 0) {};
					\node [style=morphism] (41) at (-2.25, 0) {$\suppproj{}$};
					\node [style=none] (44) at (3.75, 2.5) {};
					\node [style=none] (47) at (3.75, 0) {};
					\node [style=none] (115) at (3.75, 3) {$X$};
					\node [style=none] (119) at (-2.25, 2.5) {};
					\node [style=none] (120) at (-2.25, 3) {$X$};
					\node [style=none] (121) at (2.25, 2.5) {};
					\node [style=none] (122) at (2.25, 0) {};
					\node [style=none] (124) at (2.25, 3) {$W$};
					\node [style=none] (125) at (-3.75, 2.5) {};
					\node [style=none] (126) at (-3.75, -1.25) {};
					\node [style=none] (128) at (-3.75, 3) {$W$};
					\node [style=morphism] (129) at (-3, -1.5) {$\quad \; f \quad \; $};
					\node [style=none] (130) at (-3, -2.5) {};
					\node [style=none] (131) at (3, -3) {$Z$};
					\node [style=morphism] (132) at (3, 0) {$\quad \; f \quad \; $};
					\node [style=none] (133) at (3, -2.5) {};
				\end{pgfonlayer}
				\begin{pgfonlayer}{edgelayer}
					\draw (41) to (39);
					\draw (32.center) to (41);
					\draw (39) to (119.center);
					\draw (47.center) to (44.center);
					\draw (122.center) to (121.center);
					\draw (126.center) to (125.center);
					\draw (130.center) to (129);
					\draw (133.center) to (132);
				\end{pgfonlayer}
			\end{tikzpicture}
		}%
	\end{equation}
	so that $f$ factorizes across $\id_W \otimes \suppinc{}$ as required.
\end{proof}

Let us turn to some examples of split supports.
A nontrivial example of a support that is not split will be given later in \cref{prop:support_not_split}.

\begin{example}\label{ex:non-empty_supp_split_fin}
	The support of any morphism $p \colon A \to X$ in $\finstoch$ is a split support, provided that $A$ is non-empty.
	The reason is that its support inclusion $\suppinc{} \colon \Supp{} \to X$ is a split monomorphism with a deterministic left inverse as illustrated in \cref{ex:split_sup_finstoch}.
\end{example}

\begin{example}\label{ex:non-empty_supp_split_borel}
	Similar statements apply to $\borelstoch$.
	In fact, every deterministic monomorphism with non-empty domain in $\borelstoch$ is also a split monomorphism.
	By a classical result, such a map is a measurable isomorphism onto its image \cite[Corollary 15.2]{kechris}, and therefore we can obtain a (deterministic) left inverse by mapping the complement of the image to any element of its domain.
	It follows that any Markov kernel $p \colon A \to X$ in $\borelstoch$ with a support also has a split support (unless $A$ is empty).
\end{example}

The following result is an analogue of \cref{lem:supp_asfaithful} for split supports, and it can be helpful for showing that a given morphism has a split support.

\begin{lemma}[Concrete characterization of split supports]\label{lem:supp_factor}
	To give a split support for a morphism $p \colon A \to X$ is to give an object $\Supp{}$ and morphism $\suppinc{} \colon \Supp{} \to X$ such that the following conditions hold:
	\begin{enumerate}
		\item \label{it:suppinc_det} $\suppinc{}$ is a deterministic split monomorphism.
		\item \label{it:supp_asfaithful_split} for all $f, g \colon W \otimes X \to Y$ with arbitrary $W$ and $Y$, we have
			\begin{equation}\label{eq:supp_asfaithful_split}
				{%
					\tikzstyle{every picture}=[tikzfig]%
					\begin{tikzpicture}
						\begin{pgfonlayer}{nodelayer}
							\node [style=none] (202) at (0, 0) {$\iff$};
							\node [style=none] (251) at (-10, 0.75) {};
							\node [style=none] (252) at (-8.5, 1) {};
							\node [style=none] (253) at (-10.5, 2.25) {};
							\node [style=none] (254) at (-10.5, 2.75) {$Y$};
							\node [style=bn] (255) at (-9.25, 0) {};
							\node [style=none] (256) at (-9.25, -2) {};
							\node [style=none] (257) at (-9.25, -2.5) {$A$};
							\node [style=none] (258) at (-7, 0) {$=$};
							\node [style=none] (259) at (-8.5, 2.75) {$X$};
							\node [style=morphism] (260) at (-10.5, 1.25) {$\,\;\; f \,\;\;$};
							\node [style=none] (261) at (-8.5, 2.25) {};
							\node [style=morphism] (262) at (-9.25, -1) {$p$};
							\node [style=none] (263) at (-11, -2) {};
							\node [style=none] (264) at (-11, -2.5) {$W$};
							\node [style=none] (265) at (-11, 1.25) {};
							\node [style=none] (266) at (-11, -0.5) {};
							\node [style=none] (267) at (-4.5, 0.75) {};
							\node [style=none] (268) at (-3, 1) {};
							\node [style=none] (269) at (-5, 2.25) {};
							\node [style=none] (270) at (-5, 2.75) {$Y$};
							\node [style=bn] (271) at (-3.75, 0) {};
							\node [style=none] (272) at (-3.75, -2) {};
							\node [style=none] (273) at (-3.75, -2.5) {$A$};
							\node [style=none] (274) at (-3, 2.75) {$X$};
							\node [style=morphism] (275) at (-5, 1.25) {$\,\;\; g \,\;\;$};
							\node [style=none] (276) at (-3, 2.25) {};
							\node [style=morphism] (277) at (-3.75, -1) {$p$};
							\node [style=none] (278) at (-5.5, -2) {};
							\node [style=none] (279) at (-5.5, -2.5) {$W$};
							\node [style=none] (280) at (-5.5, 1.25) {};
							\node [style=none] (281) at (-5.5, -0.5) {};
							\node [style=none] (284) at (-10, 1.25) {};
							\node [style=none] (285) at (-4.5, 1.25) {};
							\node [style=none] (288) at (3.5, 2.25) {};
							\node [style=none] (289) at (3.5, 2.75) {$Y$};
							\node [style=none] (290) at (4, -2.5) {$\Supp{}$};
							\node [style=morphism] (292) at (3.5, 1) {$\,\;\; f \,\;\;$};
							\node [style=none] (294) at (3, -2) {};
							\node [style=none] (295) at (3, -2.5) {$W$};
							\node [style=none] (296) at (3, 1) {};
							\node [style=none] (298) at (9, 2.25) {};
							\node [style=none] (299) at (9, 2.75) {$Y$};
							\node [style=none] (300) at (9.5, -2.5) {$\Supp{}$};
							\node [style=morphism] (301) at (9, 1) {$\,\;\; g \,\;\;$};
							\node [style=none] (302) at (8.5, -2) {};
							\node [style=none] (303) at (8.5, -2.5) {$W$};
							\node [style=none] (304) at (8.5, 1) {};
							\node [style=none] (305) at (4, 1) {};
							\node [style=none] (306) at (9.5, 1) {};
							\node [style=morphism] (307) at (4, -0.75) {$\suppinc{}$};
							\node [style=morphism] (308) at (9.5, -0.75) {$\suppinc{}$};
							\node [style=none] (309) at (4, -2) {};
							\node [style=none] (310) at (9.5, -2) {};
							\node [style=none] (311) at (6.25, 0) {$=$};
							\node [style=none] (312) at (-11.75, -2.5) {};
						\end{pgfonlayer}
						\begin{pgfonlayer}{edgelayer}
							\draw [in=-90, out=165] (255) to (251.center);
							\draw [in=-90, out=15] (255) to (252.center);
							\draw (260) to (253.center);
							\draw (252.center) to (261.center);
							\draw (256.center) to (262);
							\draw [style=protected] (263.center) to (266.center);
							\draw [style=protected, in=-90, out=90, looseness=1.25] (266.center) to (265.center);
							\draw [in=270, out=90] (262) to (255);
							\draw [in=-90, out=165] (271) to (267.center);
							\draw [in=-90, out=15] (271) to (268.center);
							\draw (275) to (269.center);
							\draw (268.center) to (276.center);
							\draw (272.center) to (277);
							\draw [style=protected] (278.center) to (281.center);
							\draw [style=protected, in=-90, out=90, looseness=1.25] (281.center) to (280.center);
							\draw (277) to (271);
							\draw (251.center) to (284.center);
							\draw (267.center) to (285.center);
							\draw (292) to (288.center);
							\draw (301) to (298.center);
							\draw (309.center) to (307);
							\draw (307) to (305.center);
							\draw (294.center) to (296.center);
							\draw (302.center) to (304.center);
							\draw (308) to (306.center);
							\draw (310.center) to (308);
						\end{pgfonlayer}
					\end{tikzpicture}
				}%
			\end{equation}
	\end{enumerate}
\end{lemma}

\begin{proof}
	Given a split support as in \cref{def:split_support}, the first property holds by definition, and the second is precisely \cref{lem:supp_asfaithful}.

	Conversely, suppose that we have a morphism $\suppinc{} \colon \Supp{} \to X$ satisfying the two properties, and choose any left inverse $\suppproj{}$.
	Then it is enough to apply \cref{lem:supp_crit} to show that $\suppinc{}$ is a support inclusion.
	The existence of the splitting shows that every $\id_W \otimes \suppinc{}$ is a deterministic monomorphism too.
	The condition $\suppinc{} \ll p$ is straightforward to check directly from~\eqref{eq:supp_asfaithful_split} and an application of \cref{prop:ase_props} \ref{it:ase_copy} to $\suppinc{}$.
	It remains to be shown that if $f \colon Z \to W \otimes X$ satisfies $f \ll \id_W \otimes p$, then it factors across $\id_W \otimes \suppinc{}$.
	To this end, combining
	\[
		(\id_W \otimes \suppinc{} \comp \suppproj{}) \comp (\id_W \otimes \suppinc{}) = \id_W \otimes \suppinc{}
	\]
	and \eqref{eq:supp_asfaithful_split} with $f \ll \id_W \otimes p$ gives
	\begin{equation}
		(\id_W \otimes \suppinc{} \comp \suppproj{}) \ase{f} \id_{W \otimes X},
	\end{equation}
	from which the explicit factorization $f = (\id_W \otimes \suppinc{}) \comp (\id_W \otimes \suppproj{}) \comp f$ follows by marginalization.
\end{proof}

\begin{example}
	\label{ex:suppinc_det_ownsupport}
	Every deterministic split monomorphism $\suppinc{}$ is its own split support inclusion, since \cref{eq:supp_asfaithful_split} holds with $p = \suppinc{}$ by \cref{prop:ase_props} \ref{it:ase_deterministic}.
	This strengthens \cref{cor:split_mono_supp}.

	For example, every deterministic state is trivially a split monomorphism by the terminality of the unit object $I$.
	Hence, deterministic states are their own split support inclusions.
\end{example}

\begin{remark}\label{rem:supp_separable}
	In a locally state-separable Markov category, a deterministic split monomorphism $\suppinc{} \colon \Supp{} \to X$ is a split support for $p$ already if it satisfies condition \ref{it:supp_asfaithful_split} from \cref{lem:supp_factor} without an additional input $W$:
	For all $f, g \colon X \to Y$, we have
	\begin{equation}\label{eq:supp_is_completely}
		{%
			\tikzstyle{every picture}=[tikzfig]%
			\begin{tikzpicture}
				\begin{pgfonlayer}{nodelayer}
					\node [style=none] (13) at (-11, 1) {};
					\node [style=none] (14) at (-9, 1) {};
					\node [style=none] (19) at (-11, 2.25) {};
					\node [style=none] (20) at (-11, 2.75) {$Y$};
					\node [style=bn] (21) at (-10, -0.25) {};
					\node [style=none] (23) at (-10, -2.25) {};
					\node [style=none] (24) at (-10, -2.75) {$A$};
					\node [style=none] (25) at (-7, 0) {$=$};
					\node [style=none] (30) at (-9, 2.75) {$A$};
					\node [style=morphism] (33) at (-11, 1.25) {$f$};
					\node [style=none] (34) at (-9, 2.25) {};
					\node [style=morphism] (35) at (-10, -1.25) {$p$};
					\node [style=none] (202) at (0, 0) {$\iff$};
					\node [style=none] (236) at (-5, 1) {};
					\node [style=none] (237) at (-3, 1) {};
					\node [style=none] (238) at (-5, 2.25) {};
					\node [style=none] (239) at (-5, 2.75) {$Y$};
					\node [style=bn] (240) at (-4, -0.25) {};
					\node [style=none] (241) at (-4, -2.25) {};
					\node [style=none] (242) at (-4, -2.75) {$A$};
					\node [style=none] (243) at (-3, 2.75) {$A$};
					\node [style=morphism] (244) at (-5, 1.25) {$g$};
					\node [style=none] (245) at (-3, 2.25) {};
					\node [style=morphism] (246) at (-4, -1.25) {$p$};
					\node [style=none] (247) at (4, 1) {};
					\node [style=none] (249) at (4, 2.25) {};
					\node [style=none] (250) at (4, 2.75) {$Y$};
					\node [style=none] (252) at (4, -2.25) {};
					\node [style=none] (253) at (4, -2.75) {$\Supp{}$};
					\node [style=none] (254) at (6, 0) {$=$};
					\node [style=morphism] (256) at (4, 1.25) {$f$};
					\node [style=morphism] (258) at (4, -1) {$\suppinc{}$};
					\node [style=none] (259) at (8, 1) {};
					\node [style=none] (261) at (8, 2.25) {};
					\node [style=none] (262) at (8, 2.75) {$Y$};
					\node [style=none] (264) at (8, -2.25) {};
					\node [style=none] (265) at (8, -2.75) {$\Supp{}$};
					\node [style=morphism] (267) at (8, 1.25) {$g$};
					\node [style=morphism] (269) at (8, -1) {$\suppinc{}$};
				\end{pgfonlayer}
				\begin{pgfonlayer}{edgelayer}
					\draw [in=-90, out=165] (21) to (13.center);
					\draw [in=-90, out=15] (21) to (14.center);
					\draw (33) to (19.center);
					\draw (14.center) to (34.center);
					\draw (23.center) to (35);
					\draw (35) to (21);
					\draw [in=-90, out=165] (240) to (236.center);
					\draw [in=-90, out=15] (240) to (237.center);
					\draw (244) to (238.center);
					\draw (237.center) to (245.center);
					\draw (241.center) to (246);
					\draw (246) to (240);
					\draw (256) to (249.center);
					\draw (252.center) to (258);
					\draw (267) to (261.center);
					\draw (264.center) to (269);
					\draw (258) to (247.center);
					\draw (269) to (259.center);
				\end{pgfonlayer}
			\end{tikzpicture}
		}%
	\end{equation}
	This follows by the same argument as the proof of \cref{prop:ac_is_completely}.
\end{remark}

\cref{lem:supp_factor} also lets us prove the announced mapping-out universal property of split supports.
For any morphism $p \colon A \to X$ and arbitrary objects $W$ and $Y$, let us write
\begin{equation}\label{eq:old_supp_def}
	\cC(W \otimes X, Y)_p \,\coloneqq\, \newfaktor{\cC(W \otimes X,Y)}{\ase{p}}
\end{equation}
for the hom-set modulo \as{$p$} equality.

\begin{proposition}
	\label{prop:split_support_universal}
	Let $p \colon A \to X$ be any morphism in a Markov category $\cC$.
	Then a morphism $\suppinc{} \colon \Supp{} \to X$ is a split support inclusion for $p$ if and only if for all objects $Y$ and $W$ in $\cC$, the map
	\begin{equation}\label{eq:split_support_universal}
		\begin{tikzcd}[column sep=1.4pc,row sep=1pt]
			\cC(W \otimes X, Y)_p \ar{r} & \cC(W \otimes \Supp{}, Y) \\
			f \ar[mapsto]{r} & f \comp (\id_W \otimes \suppinc{})
		\end{tikzcd}
	\end{equation}
	is a bijection, and it restricts to a bijection between equivalence classes of $\as{p}$ deterministic morphisms on the left and deterministic morphisms on the right.
	When this is the case, then the inverse map is given by $g \longmapsto g \comp (\id_W \otimes \suppproj{})$ for any support projection $\suppproj{}$.
\end{proposition}

\begin{proof}
	Suppose that $\suppinc{}$ is a split support inclusion with support projection $\suppproj{}$.
	Then the fact that the two composition maps $f \mapsto f \comp (\id_W \otimes \suppinc{})$ and $g \mapsto g \comp (\id_W \otimes \suppproj{})$ are well-defined and mutually inverse bijections follows from \cref{lem:supp_factor} and \cref{eq:supp_proj_id}.
	That this restricts to a bijection at the level of ($\as{p}$) deterministic morphisms follows from the fact that $\suppinc{}$ is deterministic and $\suppproj{}$ is $\as{p}$ deterministic (\cref{lem:supp_proj}).

	Conversely, suppose that the composition map is a bijection restricting to a bijection at the ($\as{p}$) deterministic level.
	Then define $\suppinc{}$ as the counterpart on the right of the identity on the left, which we know to be deterministic.
	Furthermore, we have an $\suppproj{} \colon X \to \Supp{}$ which is the counterpart on the left of the identity on the right, and Yoneda lemma shows that for $W = I$, the map from right to left is given by $g \mapsto g \comp \suppproj{}$, and in particular $\suppproj{}$ is a left inverse for $\suppinc{}$.
	It is now clear that both properties of \cref{lem:supp_factor} are satisfied, and hence $\suppinc{}$ is a split support inclusion.
\end{proof}

\begin{remark}
	In particular, taking $W = I$ in \cref{prop:split_support_universal} produces a bijection
	\[
		\cC(X, Y)_p \cong \cC(\Supp{}, Y),
	\]
	which is natural in $Y$.
	This is the promised mapping-out universal property, which also was our original definition of support in \cite[Definition 13.20]{fritz2019synthetic}.
	If we consider the characterization of \cref{prop:split_support_universal} as the definition of split supports, then the present notion strengthens the earlier one in two ways:
	\begin{itemize}
		\item As in our new definition of absolute continuity (\cref{def:ac}), an extra input $W$ is allowed.
		\item We now have an additional correspondence at the level of deterministic morphisms.
	\end{itemize}
	The latter point is analogous to the definition of Kolmogorov products \cite[Definition 4.1]{fritzrischel2019zeroone}, where both general morphisms and deterministic morphisms into a Kolmogorov product are fixed via its universal property.
\end{remark}

As yet another possible characterization of split supports, we have the following elegant statement.

\begin{corollary}\label{cor:split_supp_mono}
	A morphism $p$ has a split support if and only if there is a deterministic split monomorphism $\suppinc{}$ satisfying $p \acsim \suppinc{}$, in which case $\suppinc{}$ provides the split support.
\end{corollary}

\begin{proof}
	The ``if'' direction follows by \cref{prop:abs_bicont_same_supp} \ref{it:abs_bicont} and the fact that every deterministic split monomorphism is its own split support inclusion (\cref{ex:suppinc_det_ownsupport}).
	The ``only if'' direction is by \cref{lem:inc_sim_p}.
\end{proof}

Split supports enjoy a slightly stronger compatibility with the monoidal product than supports.
Namely, in the following statement, we do not assume that $p \otimes q$ has a support, unlike in \cref{thm:supp_multiplicativity}.

\begin{proposition}[Split supports of monoidal products]\label{prop:splitsupp_tensor}
	If morphisms
	\[
		p \colon A \to X, \qquad q \colon B \to Y
	\]
	have split supports, then also $p \otimes q$ has a split support with support inclusion given by $\suppinc{p} \otimes \suppinc{q}$.
\end{proposition}

\begin{proof}
	We apply \cref{cor:split_supp_mono}. 
	Indeed deterministic split monomorphisms are stable under the monoidal product:
	One can take $\suppproj{p} \otimes \suppproj{q}$ as a left inverse for $\suppinc{p} \otimes \suppinc{q}$.
	Furthermore, $\suppinc{p} \otimes \suppinc{q} \acsim p \otimes q$ follows from $\suppinc{p} \acsim p$ and $\suppinc{q} \acsim q$ by \cref{prop:ac_tensor}.
\end{proof}

We also have an analogue of \cref{prop:support_pull} on how split supports transport along split monomorphisms.

\begin{proposition}\label{prop:split_support_pull}
	Consider an arbitrary morphism $p \colon A \to T$ and a deterministic split monomorphism $\iota \colon T \to X$ with left inverse $\pi \colon X \to T$.
	Then we have the following:
	\begin{enumerate}
		\item If $p$ has a split support, then so does $\iota\comp p$, and its support inclusion is 
			\begin{equation}
				\suppinc{\iota\comp p} = \iota \comp \suppinc{p}.
			\end{equation}
		\item If $\iota\comp p$ has a split support, then so does $p$, and its support inclusion is 
			\begin{equation}
				\suppinc{p} = \pi \comp \suppinc{\iota\comp p}.
			\end{equation}
	\end{enumerate}
\end{proposition}

\begin{proof}
	By \cref{prop:support_pull}, it is enough to establish the existence of a support projection in each case.
	\begin{enumerate}
		\item $\suppinc{\iota\comp p}$ is a split monomorphism as a composite of two split monomorphisms, since
			\begin{equation}
				\suppproj{\iota\comp p} \coloneqq \suppproj{p} \comp \pi
			\end{equation}
			serves as a support projection for $\iota\comp p$.
		\item The support projection
			\begin{equation}
				\suppproj{p} \coloneqq \suppproj{\iota\comp p} \comp \iota
			\end{equation}
			is indeed a left inverse of $\suppinc{p}$, since $\iota\comp \pi \comp \suppinc{\iota\comp p} = \suppinc{\iota \comp p}$ was noted in the proof of \cref{prop:support_pull}.
			\qedhere
	\end{enumerate}
\end{proof}

Finally, we note that not all supports are split, even in the very well-behaved case of supports in $\chausstoch$, which always exist (\cref{ex:top_support}).
The relevant idea that an embedding $\beta\N \hookrightarrow [0,1]^\R$ cannot have a continuous stochastic retraction was suggested to us by Tommaso Russo, who also noted the stronger statement that there is no Banach space embedding of $C(\beta\N)$ into $C([0,1]^\R)$ due to known results in renorming theory~\cite{deville1993renormings}.

\begin{proposition}
	\label{prop:support_not_split}
	There is a probability measure $\nu \colon I \to [0,1]^\R$ in $\chausstoch$ such that $\nu$ has no split support.
\end{proposition}

Surprisingly, the $\nu$ that we construct is even atomic (in the measure-theoretic sense of having atoms).

\begin{proof}
	Let $\beta\N$ be the Stone-\v{C}ech compactification of a countably infinite set.
	Then the hom-set
	\begin{equation}
		\chaus \bigl( \beta\N,[0,1] \bigr) \cong \topcat \bigl( \N,[0,1] \bigr)
	\end{equation}
	has the cardinality of $\R$.
	Using this together with a standard double dualization argument gives us a homeomorphic embedding $\beta\N \hookrightarrow [0,1]^{\R}$.
	We identify $\beta\N$ with the image of this embedding and let $\nu$ be any atomic probability measure whose set of atoms is given by $\N \subseteq \beta\N \subseteq [0,1]^{\R}$.

	Then since $\N \subseteq \beta\N$ is dense and $\beta\N \subseteq [0,1]^{\R}$ is closed, the topological support of $\nu$ is $\beta\N$.
	Therefore, by \cref{ex:top_support}, the support inclusion in our sense is $\suppinc{} \colon \beta\N \hookrightarrow [0,1]^\R$.
	Under Gelfand duality, this inclusion corresponds to an algebra homomorphism ${C(\suppinc{}) \colon C([0,1]^\R) \to C(\beta\N)}$ given by the restriction of functions to the subspace.
	Now, if a support projection existed, then by probabilistic Gelfand duality~\cite{furber_jacobs_gelfand}, the homomorphism $C(\suppinc{})$ would have a positive unital linear section $C(\beta\N) \to C([0,1]^\R)$, assigning to every continuous function on $\beta\N$ a continuous extension to $[0,1]^\R$ of the same norm.
	By a result of Pe{\l}czy\'nski \cite[Corollary 8.14]{pelczynski1968extensions}, such a section does not exist.\footnote{Pe{\l}czy\'nski's result applies to non-metrizable compact spaces which are extremally disconnected, and it is well-known that $\beta\N$ has these properties. It follows that $\beta\N$ is not an \emph{almost Dugundji space}, and in particular not a \emph{Dugundji space}, and this implies our claim by his \cite[Proposition 6.2]{pelczynski1968extensions}.}
\end{proof}

Let us finish with a teaser for the next section.

\begin{remark}
	\label{rem:split_supp_idemp}
	For every split support, the composite $\suppinc{} \comp \suppproj{}$ is an idempotent.
	As a consequence of \cref{lem:supp_factor}, we also have
	\begin{equation}\label{eq:support_copy_proof}
		{%
			\tikzstyle{every picture}=[tikzfig]%
			\begin{tikzpicture}
				\begin{pgfonlayer}{nodelayer}
					\node [style=none] (0) at (7, 0) {$\ase{p}$};
					\node [style=none] (1) at (0, 0) {$\ase{p}$};
					\node [style=bn] (2) at (-3.25, -0.25) {};
					\node [style=morphism] (3) at (-4.25, 1.25) {$\suppinc{}$};
					\node [style=none] (4) at (-2.25, 2.25) {};
					\node [style=none] (5) at (-4.25, 2.25) {};
					\node [style=none] (6) at (-2.25, 1) {};
					\node [style=none] (7) at (-3.25, -2.25) {};
					\node [style=none] (8) at (-3.25, -2.75) {$X$};
					\node [style=none] (9) at (-4.25, 2.75) {$X$};
					\node [style=none] (10) at (-2.25, 2.75) {$\Supp{}$};
					\node [style=none] (11) at (-4.25, 1) {};
					\node [style=morphism] (12) at (-3.25, -1.25) {$\suppproj{}$};
					\node [style=bn] (13) at (3.5, -1.5) {};
					\node [style=morphism] (14) at (2.5, 1.25) {$\suppinc{}$};
					\node [style=none] (15) at (4.5, 2.25) {};
					\node [style=none] (16) at (2.5, 2.25) {};
					\node [style=none] (17) at (4.5, -0.25) {};
					\node [style=none] (18) at (3.5, -2.25) {};
					\node [style=none] (19) at (3.5, -2.75) {$X$};
					\node [style=none] (20) at (4.5, 2.75) {$\Supp{}$};
					\node [style=none] (21) at (2.5, 2.75) {$X$};
					\node [style=none] (22) at (2.5, -0.25) {};
					\node [style=morphism] (23) at (2.5, -0.25) {$\suppproj{}$};
					\node [style=morphism] (24) at (4.5, -0.25) {$\suppproj{}$};
					\node [style=bn] (25) at (10.5, -1) {};
					\node [style=none] (26) at (11.5, 2.25) {};
					\node [style=none] (27) at (9.5, 2.25) {};
					\node [style=none] (28) at (11.5, 0.25) {};
					\node [style=none] (29) at (10.5, -2.25) {};
					\node [style=none] (30) at (10.5, -2.75) {$X$};
					\node [style=none] (31) at (11.5, 2.75) {$\Supp{}$};
					\node [style=none] (32) at (9.5, 2.75) {$X$};
					\node [style=none] (33) at (9.5, 0.25) {};
					\node [style=morphism] (34) at (11.5, 0.75) {$\suppproj{}$};
				\end{pgfonlayer}
				\begin{pgfonlayer}{edgelayer}
					\draw [in=-90, out=165] (2) to (11.center);
					\draw [in=-90, out=15] (2) to (6.center);
					\draw (6.center) to (4.center);
					\draw (3) to (5.center);
					\draw (7.center) to (12);
					\draw (12) to (2);
					\draw [in=-90, out=165] (13) to (22.center);
					\draw [in=-90, out=15] (13) to (17.center);
					\draw (14) to (16.center);
					\draw (23) to (14);
					\draw (17.center) to (24);
					\draw (24) to (15.center);
					\draw (18.center) to (13);
					\draw [in=-90, out=165] (25) to (33.center);
					\draw [in=-90, out=15] (25) to (28.center);
					\draw (28.center) to (34);
					\draw (34) to (26.center);
					\draw (29.center) to (25);
					\draw (33.center) to (27.center);
				\end{pgfonlayer}
			\end{tikzpicture}
		}%
	\end{equation}
	This is an instance of the \emph{relative positivity} axiom \cite[Section 2.5]{fritz2022dilations} which automatically holds in this case.
	Since we have $\suppinc{} \acsim p$, we can also write this as
	\begin{equation}\label{eq:support_copy}
		{%
			\tikzstyle{every picture}=[tikzfig]%
			\begin{tikzpicture}
				\begin{pgfonlayer}{nodelayer}
					\node [style=none] (15) at (0, 0) {$\ase{\suppinc{}}$};
					\node [style=bn] (16) at (-3.25, -0.25) {};
					\node [style=morphism] (17) at (-4.25, 1.25) {$\suppinc{}$};
					\node [style=none] (18) at (-2.25, 2.25) {};
					\node [style=none] (19) at (-4.25, 2.25) {};
					\node [style=none] (20) at (-2.25, 1) {};
					\node [style=none] (21) at (-3.25, -2.25) {};
					\node [style=none] (22) at (-3.25, -2.75) {$X$};
					\node [style=none] (23) at (-4.25, 2.75) {$X$};
					\node [style=none] (24) at (-2.25, 2.75) {$\Supp{}$};
					\node [style=none] (25) at (-4.25, 1) {};
					\node [style=morphism] (26) at (-3.25, -1.25) {$\suppproj{}$};
					\node [style=bn] (39) at (3.5, -1) {};
					\node [style=none] (41) at (4.5, 2.25) {};
					\node [style=none] (42) at (2.5, 2.25) {};
					\node [style=none] (43) at (4.5, 0.25) {};
					\node [style=none] (44) at (3.5, -2.25) {};
					\node [style=none] (45) at (3.5, -2.75) {$X$};
					\node [style=none] (46) at (4.5, 2.75) {$\Supp{}$};
					\node [style=none] (47) at (2.5, 2.75) {$X$};
					\node [style=none] (48) at (2.5, 0.25) {};
					\node [style=morphism] (50) at (4.5, 0.75) {$\suppproj{}$};
				\end{pgfonlayer}
				\begin{pgfonlayer}{edgelayer}
					\draw [in=-90, out=165] (16) to (25.center);
					\draw [in=-90, out=15] (16) to (20.center);
					\draw (20.center) to (18.center);
					\draw (17) to (19.center);
					\draw (21.center) to (26);
					\draw (26) to (16);
					\draw [in=-90, out=165] (39) to (48.center);
					\draw [in=-90, out=15] (39) to (43.center);
					\draw (43.center) to (50);
					\draw (50) to (41.center);
					\draw (44.center) to (39);
					\draw (48.center) to (42.center);
				\end{pgfonlayer}
			\end{tikzpicture}
		}%
	\end{equation}
	In the parlance of \cref{def:idempotents}, this implies that the idempotent $\suppinc{} \comp \suppproj{}$ is of a special type, namely a \emph{static} idempotent.
	We study these in more detail in \cref{sec:proj_idempotent_support}.
\end{remark}

\section{Idempotents and their splitting in Markov categories}
\label{sec:idempotents}

In applications of probability, one often deals with certain operations{\,\textemdash\,}such as averaging over a group action or applying a conditional expectation{\,\textemdash\,}which are idempotent:
Applying such an operation twice is the same as applying it only once.

The main goal of this section, which turns out to be closely related to supports, is to prove that all idempotents split in Markov categories satisfying suitable basic axioms.
We achieve this in \cref{thm:balanced_split}, and obtain as a straightforward consequence that idempotents in $\borelstoch$ split.
As we explain with \cref{thm:idemp_borel}, this strengthens an existing theorem of Blackwell on idempotent Markov kernels.

Being able to state and prove these statements requires a good deal of preparation.
First of all, \cref{sec:projective_etc} introduces various ways in which an idempotent can interact with a Markov category structure, resulting in the definition of \emph{static}, \emph{strong} and \emph{balanced} idempotents.
We provide various examples, such as Reynolds operators, as well as reformulations of the definitions involving notions like detailed balance.
In the auxiliary \cref{sec:proj_idempotent_support}, we study splittings of static idempotents and show that the splitting of a static idempotent is the same as a split support.
This is then put to use in \cref{sec:route_split}, which contains our main results.
Our construction of splittings of the more general balanced idempotents in the proof of \cref{thm:balanced_split} is reduced to the splitting of static idempotents from the previous subsection.

The final \cref{sec:blackwell_envelope} offers a general formal construction for the splitting of balanced idempotents which we call the \emph{Blackwell envelope}.
It is analogous to the Karoubi envelope from ordinary category theory, which embeds every category in a larger category in which all idempotents split.

\subsection{Static, strong and balanced idempotents}
\label{sec:projective_etc}

In a generic category, an \newterm{idempotent} is any endomorphism $e \colon X \to X$ satisfying $e^2 = e$, where $e^2$ stands for the sequential composite of two instances of $e$.
As one would expect in a more specific setting like that of Markov categories, relevant notions of idempotents satisfy additional compatibility conditions with the extra structure that is present.

Here, we present three types of idempotents, namely \emph{static}, \emph{strong} and \emph{balanced}. 
These have the following significance:
\begin{itemize}
	\item Static idempotents are closely related to supports, as investigated in \cref{sec:proj_idempotent_support}. 
		Moreover, they play a crucial role in \cref{thm:balanced_split} on splitting of idempotents, which is the main result of this section. Indeed, a preliminary result (\cref{lem:proj_balanced_split}) reduces the proof of the splitting of all idempotents to those that are static.
	\item Strong idempotents have connections with conditional expectations, as we point out in \cref{ex:reynolds_operator}, but they appear only in our present general theory.
	\item The weakest compatibility condition is that of balanced idempotents.
		As shown in \cref{fig:idempotents_venn}, these contain both static and strong idempotents as special cases.
		In a positive Markov category, also all split idempotents are balanced (\cref{thm:split_props}).
		Moreover, assuming balancedness (\cref{def:balanced_cat}), \emph{every} idempotent is balanced (\cref{thm:balanced_idempotent}).
\end{itemize}

\begin{definition}\label{def:idempotents}
	An idempotent $e \colon X \to X$ in a Markov category $\cC$ is:
	\begin{itemize}
		\item \newterm{static} if it satisfies
			\begin{equation}\label{eq:balanced_as_det}
				{%
					\tikzstyle{every picture}=[tikzfig]%
					\begin{tikzpicture}
						\begin{pgfonlayer}{nodelayer}
							\node [style=bn] (0) at (-3, 0) {};
							\node [style=bn] (1) at (3, 0) {};
							\node [style=morphism] (2) at (-3, -1) {$e$};
							\node [style=none] (3) at (2, 1.25) {};
							\node [style=none] (4) at (4, 1.25) {};
							\node [style=morphism] (5) at (-4, 1.25) {$e$};
							\node [style=none] (6) at (-4, 2.25) {};
							\node [style=none] (7) at (-2, 2.25) {};
							\node [style=none] (8) at (-2, 1.25) {};
							\node [style=none] (9) at (-3, -2) {};
							\node [style=none] (10) at (3, -2) {};
							\node [style=none] (11) at (2, 2.25) {};
							\node [style=none] (12) at (4, 2.25) {};
							\node [style=none] (13) at (0, 0) {$=$};
							\node [style=morphism] (14) at (3, -1) {$e$};
							\node [style=none] (15) at (-4, 1.25) {};
						\end{pgfonlayer}
						\begin{pgfonlayer}{edgelayer}
							\draw (11.center) to (3.center);
							\draw [in=165, out=-90] (3.center) to (1);
							\draw (12.center) to (4.center);
							\draw [in=15, out=-90] (4.center) to (1);
							\draw (1) to (10.center);
							\draw (9.center) to (0);
							\draw [in=-90, out=15] (0) to (8.center);
							\draw (8.center) to (7.center);
							\draw (6.center) to (5);
							\draw [in=-90, out=165] (0) to (15.center);
						\end{pgfonlayer}
					\end{tikzpicture}
				}%
			\end{equation}
		\item \newterm{strong} if it satisfies
			\begin{equation}\label{eq:strong_idempotent}
				{%
					\tikzstyle{every picture}=[tikzfig]%
					\begin{tikzpicture}
						\begin{pgfonlayer}{nodelayer}
							\node [style=bn] (0) at (-3, 0) {};
							\node [style=bn] (1) at (3, -0.75) {};
							\node [style=morphism] (2) at (-3, -1) {$e$};
							\node [style=morphism] (3) at (2, 1) {$e$};
							\node [style=morphism] (4) at (4, 1) {$e$};
							\node [style=morphism] (5) at (-4, 1.25) {$e$};
							\node [style=none] (6) at (-4, 2.25) {};
							\node [style=none] (7) at (-2, 2.25) {};
							\node [style=none] (8) at (-2, 1.25) {};
							\node [style=none] (9) at (-3, -2) {};
							\node [style=none] (10) at (3, -2) {};
							\node [style=none] (11) at (2, 2.25) {};
							\node [style=none] (12) at (4, 2.25) {};
							\node [style=none] (13) at (0, 0) {$=$};
							\node [style=none] (14) at (-4, 1.25) {};
							\node [style=none] (15) at (2, 0.5) {};
							\node [style=none] (16) at (4, 0.5) {};
						\end{pgfonlayer}
						\begin{pgfonlayer}{edgelayer}
							\draw (11.center) to (3);
							\draw (12.center) to (4);
							\draw (1) to (10.center);
							\draw (9.center) to (0);
							\draw [in=-90, out=15] (0) to (8.center);
							\draw (8.center) to (7.center);
							\draw (6.center) to (5);
							\draw [in=-90, out=165] (0) to (14.center);
							\draw (16.center) to (4);
							\draw (15.center) to (3);
							\draw [in=-90, out=165] (1) to (15.center);
							\draw [in=-90, out=15] (1) to (16.center);
						\end{pgfonlayer}
					\end{tikzpicture}
				}%
			\end{equation}
		\item \newterm{balanced} if it satisfies
			\begin{equation}\label{eq:balancednew}
				{%
					\tikzstyle{every picture}=[tikzfig]%
					\begin{tikzpicture}
						\begin{pgfonlayer}{nodelayer}
							\node [style=bn] (0) at (-3, 0) {};
							\node [style=bn] (1) at (3, 0) {};
							\node [style=morphism] (2) at (-3, -1) {$e$};
							\node [style=morphism] (5) at (-4, 1.25) {$e$};
							\node [style=none] (6) at (-4, 2.25) {};
							\node [style=none] (7) at (-2, 2.25) {};
							\node [style=none] (8) at (-2, 1.25) {};
							\node [style=none] (9) at (-3, -2) {};
							\node [style=none] (10) at (3, -2) {};
							\node [style=none] (11) at (2, 2.25) {};
							\node [style=none] (13) at (0, 0) {$=$};
							\node [style=morphism] (14) at (3, -1) {$e$};
							\node [style=none] (15) at (-4, 1.25) {};
							\node [style=morphism] (16) at (2, 1.25) {$e$};
							\node [style=none] (17) at (4, 2.25) {};
							\node [style=none] (18) at (4, 1.25) {};
							\node [style=none] (19) at (2, 1.25) {};
							\node [style=morphism] (20) at (4, 1.25) {$e$};
						\end{pgfonlayer}
						\begin{pgfonlayer}{edgelayer}
							\draw (1) to (10.center);
							\draw (9.center) to (0);
							\draw [in=-90, out=15] (0) to (8.center);
							\draw (8.center) to (7.center);
							\draw [in=165, out=-90] (15.center) to (0);
							\draw (6.center) to (15.center);
							\draw (17.center) to (18.center);
							\draw [in=165, out=-90] (19.center) to (1);
							\draw [in=270, out=15] (1) to (18.center);
							\draw (16) to (11.center);
						\end{pgfonlayer}
					\end{tikzpicture}
				}%
			\end{equation} 
	\end{itemize}
\end{definition}

To understand these equations intuitively, note that the left-hand side is the same for all, and it amounts to considering the Markov chain with transition kernel $e$ for two time steps and recording the intermediate value.
Each of these equations implies that $e$ is an idempotent, by marginalizing either on the right (static, strong) or left (balanced).

The reasons behind our choice of names for these idempotents are as follows:
\begin{itemize}
	\item The defining \cref{eq:balanced_as_det} of a static idempotent can also be written as $e \ase{e} \id_X$. This is static in the sense that the Markov chain already stabilizes after one step.

	\item A balanced idempotent satisfies a version of \emph{detailed balance}.
		This is discussed in the upcoming \cref{rem:detailed_balance} and reflected in the more symmetrical characterization from \cref{prop:balancednew}~\ref{it:detailed_balance}. 

	\item The definition of strong idempotents arises as a strengthening of the characterization of balanced idempotents via \cref{prop:balancednew}~\ref{it:as_strong}. 
\end{itemize}
These interpretations are also going to be bolstered by \cref{thm:split_props}, which classifies split idempotents by their types.
For now, let us discuss the relation between the three types of idempotents just introduced and give some examples.

\begin{figure}
	\begin{center}\rm
		{%
			\tikzstyle{every picture}=[tikzfig]%
			\begin{tikzpicture}
				\begin{pgfonlayer}{nodelayer}
					\node [style=none] (0) at (-8.75, 5.75) {};
					\node [style=none] (1) at (-9.75, 4.75) {};
					\node [style=none] (2) at (-9.75, -2.25) {};
					\node [style=none] (3) at (-8.75, -3.25) {};
					\node [style=none] (4) at (8.75, -3.25) {};
					\node [style=none] (5) at (9.75, -2.25) {};
					\node [style=none] (6) at (9.75, 4.75) {};
					\node [style=none] (7) at (8.75, 5.75) {};
					\node [style=none] (8) at (8.75, 5.75) {};
					\node [style=yellow] (9) at (-6.5, 5) {All idempotents};
					\node [style=none] (10) at (-7.75, 4.25) {};
					\node [style=none] (11) at (-8.5, 3.5) {};
					\node [style=none] (12) at (-8.5, -2) {};
					\node [style=none] (13) at (-7.75, -2.75) {};
					\node [style=none] (14) at (8, -2.75) {};
					\node [style=none] (15) at (8.75, -2) {};
					\node [style=none] (16) at (8.75, 3.5) {};
					\node [style=none] (17) at (8, 4.25) {};
					\node [style=orange] (18) at (0, 3.5) {Balanced};
					\node [style=red] (19) at (-4.75, 1.75) {Strong};
					\node [style=red] (20) at (4.75, 1.75) {Static};
					\node [style=none] (21) at (0, -0.75) {Deterministic};
					\node [style=none] (22) at (-6.25, 2.75) {};
					\node [style=none] (24) at (-7, 2) {};
					\node [style=none] (25) at (-7, -1) {};
					\node [style=none] (26) at (-6.25, -1.75) {};
					\node [style=none] (27) at (2.5, -1.75) {};
					\node [style=none] (28) at (3.25, -1) {};
					\node [style=none] (29) at (3.25, -0.5) {};
					\node [style=none] (31) at (-4.25, 2.75) {};
					\node [style=none] (32) at (6.25, 2.75) {};
					\node [style=none] (33) at (7, 2) {};
					\node [style=none] (34) at (7, -1) {};
					\node [style=none] (35) at (6.25, -1.75) {};
					\node [style=none] (36) at (-2.5, -1.75) {};
					\node [style=none] (37) at (-3.25, -1) {};
					\node [style=none] (38) at (-3.25, -0.5) {};
					\node [style=none] (40) at (4.25, 2.75) {};
					\node [style=none] (41) at (0, 1.75) {};
					\node [style=none] (42) at (2.5, -1.75) {};
					\node [style=none] (43) at (3.25, -1) {};
					\node [style=none] (44) at (3.25, -0.5) {};
					\node [style=none] (45) at (-2.5, -1.75) {};
					\node [style=none] (46) at (-3.25, -1) {};
					\node [style=none] (47) at (-3.25, -0.5) {};
					\node [style=none] (48) at (0, 1.75) {};
				\end{pgfonlayer}
				\begin{pgfonlayer}{edgelayer}
					\draw [style=yellow fill] (7.center)
						to (0.center)
						to [in=90, out=-180] (1.center)
						to (2.center)
						to [in=-180, out=-90] (3.center)
						to [in=180, out=0] (4.center)
						to [in=-90, out=0] (5.center)
						to (6.center)
						to [in=0, out=90] (8.center);
					\draw [style=orange fill] (16.center)
						to [in=0, out=90] (17.center)
						to (10.center)
						to [in=90, out=-180] (11.center)
						to (12.center)
						to [in=-180, out=-90] (13.center)
						to [in=180, out=0] (14.center)
						to [in=-90, out=0] (15.center)
						to cycle;
					\draw [style=red fill] (26.center)
						to (36.center)
						to [in=-90, out=180] (37.center)
						to (38.center)
						to [in=-165, out=90, looseness=0.50] (41.center)
						to [in=0, out=165, looseness=0.75] (31.center)
						to (22.center)
						to [in=90, out=-180] (24.center)
						to (25.center)
						to [in=-180, out=-90] cycle;
					\draw [style=red fill] (32.center)
						to [in=90, out=0] (33.center)
						to (34.center)
						to [in=0, out=-90] (35.center)
						to (27.center)
						to [in=-90, out=0] (28.center)
						to (29.center)
						to [in=-15, out=90, looseness=0.50] (41.center)
						to [in=-180, out=15, looseness=0.75] (40.center)
						to cycle;
					\draw [style=red fill] (36.center) to (27.center);
					\draw [style=dark red fill] (43.center)
						to (44.center)
						to [in=-15, out=90, looseness=0.50] (48.center)
						to [in=90, out=-165, looseness=0.50] (47.center)
						to (46.center)
						to [in=180, out=-90] (45.center)
						to (42.center)
						to [in=-90, out=0] cycle;
				\end{pgfonlayer}
			\end{tikzpicture}
		}%
	\end{center}
	\caption{Venn diagram of types of idempotents in a Markov category.}
	\label{fig:idempotents_venn}
\end{figure}
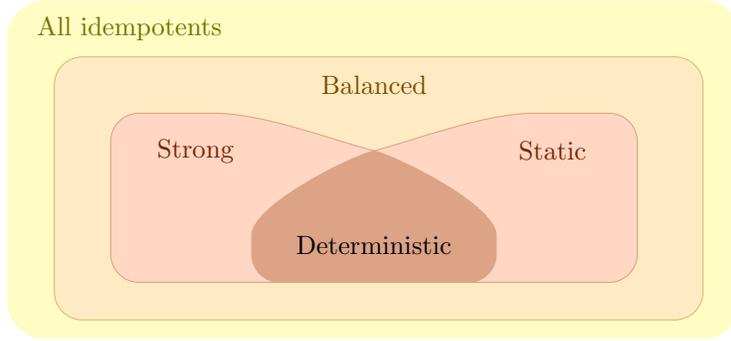

\begin{remark}\label{rem:idempotents}
	We can infer from the definition that both static and strong idempotents are balanced.
	Indeed, in the static case, we can compose \cref{eq:balanced_as_det} with $e$ on the right output to get the mirror image of \cref{eq:balancednew}, which is an equivalent equation by the commutativity of copy.
	In the strong case, we simply pre-compose \cref{eq:strong_idempotent} with $e$ and use $e^2 = e$ to obtain \cref{eq:balancednew}.

	Furthermore, any idempotent that is both static and strong is deterministic, since that is an obvious consequence of chaining \cref{eq:balanced_as_det,eq:strong_idempotent}.
	The Venn diagram of \cref{fig:idempotents_venn} depicts all of these implications.
\end{remark}		

\begin{example}
	Let us consider \emph{group averaging} as a general motivating example in $\finstoch$.
	To this end, let $G$ be a finite group acting on a finite set $X$.
	Then the \emph{$G$-averaging} is the stochastic matrix $X \to X$ with entries
	\begin{equation}
		e(y|x) = \frac{1}{|G|} \sum_{g \in G} \delta_{g \comp x, y},
	\end{equation}
	where $\delta_{a,b}$ stands for $\id(a|b)$.
	Intuitively, this amounts to choosing a uniformly random element $g \in G$ and applying it to $x$, or equivalently choosing a uniformly random element in the orbit $G \comp x$.
	A direct calculation shows that this $e$ is a strong idempotent:
	Looking at the probability for input $x$ and output $(y,z)$, both the left-hand side and the right-hand side of \cref{eq:strong_idempotent} take the value $0$ if these three points are not all in the same orbit, and the value $\linefaktor{|G_x|^2\,}{\,|G|^2} = \linefaktor{1\,}{\,|G \comp x|^2}$ if they are, where $G_x = \Set{g \in G \given g \comp x = x }$ denotes the stabilizer of $x$.
\end{example}

\begin{example}\label{ex:idempotents_finstoch}
	We show, via examples, that in $\finstoch$ all three types of idempotents actually differ.
	More precisely, all regions of the Venn diagram of \cref{fig:idempotents_venn}, besides the outer ``general idempotents'' one, are inhabited by some stochastic matrices. 
	We also give a \newterm{splitting} for each of the examples, i.e.\ a pair $\pi$, $\iota$ of morphisms such that both $\pi\comp \iota = \id$ and $\iota \comp \pi = e$ hold. 
	Computations are omitted for the sake of brevity.
	\begin{itemize}
		\item A strong but not static idempotent is 
			\begin{equation}
				e = \begin{pmatrix}
					1/2 & 1/2 \\
					1/2 & 1/2
				\end{pmatrix}, \qquad 
				\iota = \begin{pmatrix} 1/2 \\1/2\end{pmatrix}, \qquad 
				\pi = \begin{pmatrix} 1 & 1  \end{pmatrix} .
			\end{equation}
			This amounts to a stochastic map which ignores its input and outputs a fair coin flip instead.
		\item A static but not strong idempotent is
			\begin{equation}
				e = \begin{pmatrix}
					1&0&1/2\\
					0&1&1/2\\
					0&0&0
				\end{pmatrix}, \qquad 
				\iota = \begin{pmatrix} 1 &0\\0&1\\0&0\end{pmatrix}, \qquad 
				\pi = \begin{pmatrix} 1 & 0 & 1/2\\ 0&1&1/2 \end{pmatrix}  .
			\end{equation}
			This stochastic map acts as the identity on the first two states and sends the third state to a fair coin flip on the first two.
		\item A balanced idempotent that is neither static nor strong is 
			\begin{equation}
				e = \begin{pmatrix}
					1/2 & 1/2 & 0 & 1/4 \\
					1/2 & 1/2 & 0 & 1/4 \\
					0 & 0 & 1 & 1/2 \\
					0 & 0 & 0 & 0
				\end{pmatrix}, \qquad 
				\iota = \begin{pmatrix} 1/2 & 0 \\ 1/2 & 0 \\ 0 & 1 \\ 0 & 0 \end{pmatrix}, \qquad 
				\pi = \begin{pmatrix} 1 & 1 & 0 & 1/2 \\ 0 & 0 & 1 & 1/2 \end{pmatrix} .
			\end{equation}
			The structure of this map is easier to understand based on \cref{ex:splitting} presented further on.
	\end{itemize}
	Note that $\pi$ is deterministic for the strong idempotent and $\iota$ is deterministic for the static idempotent, while neither of these statements is true for the balanced idempotent. 
	These are general facts, see \cref{thm:split_props}.
\end{example}

\begin{example}
	\label{ex:setmulti_not_balanced}
	In order to populate the remaining region in \cref{fig:idempotents_venn}, let us present two examples of idempotents in $\setmulti$ which are not balanced:
	\begin{itemize}
		\item On a two-element set $X = \{0,1\}$, consider the multivalued function $e \colon X \to X$ given by $e(0) = \{0,1\}$ and $e(1) = \{1\}$.
			Then $e$ is an idempotent, but it is not balanced, as one can see by direct calculation: 
			The right-hand side of \cref{eq:balancednew} relates $0 \in X$ to $(0,1) \in X \times X$, but the left-hand side does not.

		\item Consider $X = \mathbb{R}$ and the strict ordering relation $>$ on real numbers.
			Define $e$ to be the morphism mapping each number to the upset it generates:
			\begin{equation}
				e(x) = \Set{ y \given y > x }.
			\end{equation}
			This is another idempotent that is not balanced: 
			For any ordered triple $z > y > x$, the right-hand side of \cref{eq:balancednew} relates $x$ to $(y,z)$, but the left-hand side does not.
			This example works with any nontrivial dense\footnotemark{} partially ordered set in place of $(\R, >)$.
			\footnotetext{Density means that for every ordered pair $z > x$ there is an element $y$ satisfying $z > y > x$ \cite[Definition 10.5]{harzheim2005ordered}. 
				This is relevant in order to get $e^2 = e$. 
			We need the partial order to be nontrivial only in that there should be \emph{some} nontrivial order relation $y > x$.}%
	\end{itemize}
\end{example}

\begin{example}
	\label{ex:finstochpm_not_balanced}
	Also $\finstoch_\pm$ has idempotents which are not balanced, such as
	\begin{equation}
		e = \begin{pmatrix}
			1 & 0 & 0 \\
			1 & 0 & 0 \\
			-1 & 1 & 1
		\end{pmatrix}.
	\end{equation}
	Indeed upon indexing the rows and columns as $X = \{a,b,c\}$, the right-hand side of \cref{eq:balancednew} takes $a$ to $(a,b)$ with probability $1$, while this transition has probability $0$ on the left-hand side.
\end{example}

\begin{example}\label{ex:reynolds_operator}
	It is instructive to spell out the defining equations of the three types of idempotents in the Markov category $\cring$ from \cref{ex:old_vs_new_ac}.
	This means that we consider a commutative ring $R$ and an additive map $e \colon R \to R$ satisfying $e(1) = 1$.
	Then we have:
	\begin{itemize}
		\item $e$ is a static idempotent if and only if
			\begin{equation}
				e \bigl( e(a) b \bigr) = e(a b) 
			\end{equation}
			holds for all $a, b \in R$.

		\item $e$ is a strong idempotent if and only if
			\begin{equation}\label{eq:strong_cring}
				e \bigl( e(a) b \bigr) = e(a) e(b)
			\end{equation}
			holds for all $a, b \in R$.
			This is precisely the defining equation of an \emph{averaging operator} (also known as a \emph{Reynolds operator})\footnotemark{} on $R$ \cite{billik1967idempotent,mumford1994geometric}.
			\footnotetext{Reynolds operators are frequently defined by a more general equation which does not imply they are idempotent, but which is equivalent to \cref{eq:strong_cring} in the idempotent case \cite{billik1967idempotent}.}%
			Moreover, \cref{eq:strong_cring}, together with linearity and continuity conditions, characterizes \emph{conditional expectations} on the algebra of bounded measurable functions on a probability space \cite{moy1954expectation}.
		\item $e$ is a balanced idempotent if and only if
			\begin{equation}
				e \bigl(e(a) b \bigr) = e \bigl( e(a) e(b) \bigr) 
			\end{equation}
			holds for all $a, b \in R$.
	\end{itemize}
\end{example}

More general examples of idempotents arising in categorical probability can be found after \cref{thm:split_props}.
We conclude this subsection by giving alternative descriptions of balanced and static idempotents. 

\begin{proposition}[Alternative description of balanced idempotents]\label{prop:balancednew}
	For an idempotent $e$, the following are equivalent:
	\begin{enumerate}
		\item\label{it:balancednew} $e$ is balanced.

		\item\label{it:detailed_balance} We have
			\begin{equation}\label{eq:detailed_balance}
				{%
					\tikzstyle{every picture}=[tikzfig]%
					\begin{tikzpicture}
						\begin{pgfonlayer}{nodelayer}
							\node [style=bn] (0) at (-3, 0) {};
							\node [style=morphism] (2) at (-3, -1) {$e$};
							\node [style=none] (6) at (-4, 2.25) {};
							\node [style=none] (7) at (-2, 2.25) {};
							\node [style=none] (8) at (-2, 1.25) {};
							\node [style=none] (9) at (-3, -2) {};
							\node [style=none] (13) at (0, 0) {$=$};
							\node [style=morphism] (15) at (-4, 1.25) {$e$};
							\node [style=none] (16) at (-4, 1.25) {};
							\node [style=bn] (17) at (3, 0) {};
							\node [style=morphism] (18) at (3, -1) {$e$};
							\node [style=none] (19) at (4, 2.25) {};
							\node [style=none] (20) at (2, 2.25) {};
							\node [style=none] (21) at (2, 1.25) {};
							\node [style=none] (22) at (3, -2) {};
							\node [style=morphism] (23) at (4, 1.25) {$e$};
							\node [style=none] (24) at (4, 1.25) {};
						\end{pgfonlayer}
						\begin{pgfonlayer}{edgelayer}
							\draw (9.center) to (0);
							\draw [in=-90, out=15] (0) to (8.center);
							\draw (8.center) to (7.center);
							\draw [in=-90, out=165] (0) to (16.center);
							\draw (16.center) to (6.center);
							\draw (22.center) to (17);
							\draw [in=-90, out=165] (17) to (21.center);
							\draw (21.center) to (20.center);
							\draw [in=-90, out=15] (17) to (24.center);
							\draw (24.center) to (19.center);
						\end{pgfonlayer}
					\end{tikzpicture}
				}%
			\end{equation}

		\item\label{it:as_strong} \cref{eq:strong_idempotent} holds ${e}$-almost surely, i.e.\ we have
			\begin{equation}\label{eq:as_balanced}
				{%
					\tikzstyle{every picture}=[tikzfig]%
					\begin{tikzpicture}
						\begin{pgfonlayer}{nodelayer}
							\node [style=bn] (0) at (-3.25, 0) {};
							\node [style=bn] (1) at (3.25, -0.75) {};
							\node [style=morphism] (2) at (-3.25, -1) {$e$};
							\node [style=morphism] (3) at (2.25, 1) {$e$};
							\node [style=morphism] (4) at (4.25, 1) {$e$};
							\node [style=morphism] (5) at (-4.25, 1.25) {$e$};
							\node [style=none] (6) at (-4.25, 2.25) {};
							\node [style=none] (7) at (-2.25, 2.25) {};
							\node [style=none] (8) at (-2.25, 1.25) {};
							\node [style=none] (9) at (-3.25, -2) {};
							\node [style=none] (10) at (3.25, -2) {};
							\node [style=none] (11) at (2.25, 2.25) {};
							\node [style=none] (12) at (4.25, 2.25) {};
							\node [style=none] (13) at (0, 0) {$\ase{e}$};
							\node [style=none] (14) at (-4.25, 1.25) {};
							\node [style=none] (15) at (2.25, 0.5) {};
							\node [style=none] (16) at (4.25, 0.5) {};
						\end{pgfonlayer}
						\begin{pgfonlayer}{edgelayer}
							\draw (11.center) to (3);
							\draw (12.center) to (4);
							\draw (1) to (10.center);
							\draw (9.center) to (0);
							\draw [in=-90, out=15] (0) to (8.center);
							\draw (8.center) to (7.center);
							\draw (6.center) to (5);
							\draw [in=-90, out=165] (0) to (14.center);
							\draw (16.center) to (4);
							\draw (15.center) to (3);
							\draw [in=-90, out=165] (1) to (15.center);
							\draw [in=-90, out=15] (1) to (16.center);
						\end{pgfonlayer}
					\end{tikzpicture}
				}%
			\end{equation}

		\item\label{it:self_adjoint} For every morphism $p \colon A \to X$ that is invariant, i.e.\ satisfies $e \comp p = p$, we have 
			\begin{equation}\label{eq:self_adjoint}
				{%
					\tikzstyle{every picture}=[tikzfig]%
					\begin{tikzpicture}
						\begin{pgfonlayer}{nodelayer}
							\node [style=none] (13) at (0, 0) {$=$};
							\node [style=bn] (14) at (-3, 0) {};
							\node [style=morphism] (15) at (-3, -1) {$p$};
							\node [style=none] (16) at (-4, 2.25) {};
							\node [style=none] (17) at (-2, 2.25) {};
							\node [style=none] (18) at (-2, 1.25) {};
							\node [style=none] (19) at (-3, -2) {};
							\node [style=morphism] (20) at (-4, 1.25) {$e$};
							\node [style=none] (21) at (-4, 1.25) {};
							\node [style=none] (22) at (-3, -2.5) {$A$};
							\node [style=none] (23) at (-4, 2.75) {$X$};
							\node [style=none] (24) at (-2, 2.75) {$X$};
							\node [style=bn] (25) at (3, 0) {};
							\node [style=morphism] (26) at (3, -1) {$p$};
							\node [style=none] (27) at (4, 2.25) {};
							\node [style=none] (28) at (2, 2.25) {};
							\node [style=none] (29) at (2, 1.25) {};
							\node [style=none] (30) at (3, -2) {};
							\node [style=morphism] (31) at (4, 1.25) {$e$};
							\node [style=none] (32) at (4, 1.25) {};
							\node [style=none] (33) at (3, -2.5) {$A$};
							\node [style=none] (34) at (4, 2.75) {$X$};
							\node [style=none] (35) at (2, 2.75) {$X$};
						\end{pgfonlayer}
						\begin{pgfonlayer}{edgelayer}
							\draw (19.center) to (14);
							\draw [in=-90, out=15] (14) to (18.center);
							\draw (18.center) to (17.center);
							\draw [in=-90, out=165] (14) to (21.center);
							\draw (21.center) to (16.center);
							\draw (30.center) to (25);
							\draw [in=-90, out=165] (25) to (29.center);
							\draw (29.center) to (28.center);
							\draw [in=-90, out=15] (25) to (32.center);
							\draw (32.center) to (27.center);
						\end{pgfonlayer}
					\end{tikzpicture}
				}%
			\end{equation}
	\end{enumerate}
\end{proposition}
\begin{proof} \hfill
	\begin{itemize}
		\item[\ref{it:balancednew} $\Rightarrow$ \ref{it:detailed_balance}:]
			As noted in \cref{rem:idempotents}, applying \cref{eq:balancednew} twice yields the required 
			\begin{equation}\label{eq:balanced_newversion0}
				{%
					\tikzstyle{every picture}=[tikzfig]%
					\begin{tikzpicture}
						\begin{pgfonlayer}{nodelayer}
							\node [style=none] (13) at (-3, 0) {$=$};
							\node [style=none] (21) at (3, 0) {$=$};
							\node [style=bn] (23) at (0, 0) {};
							\node [style=none] (24) at (0, -2) {};
							\node [style=none] (25) at (-1, 2) {};
							\node [style=morphism] (26) at (0, -1) {$e$};
							\node [style=morphism] (27) at (-1, 1) {$e$};
							\node [style=none] (28) at (1, 2) {};
							\node [style=none] (29) at (1, 1) {};
							\node [style=none] (30) at (-1, 1) {};
							\node [style=morphism] (31) at (1, 1) {$e$};
							\node [style=bn] (32) at (-6, 0) {};
							\node [style=none] (33) at (-6, -2) {};
							\node [style=none] (34) at (-7, 2) {};
							\node [style=morphism] (35) at (-6, -1) {$e$};
							\node [style=morphism] (36) at (-7, 1) {$e$};
							\node [style=none] (37) at (-5, 2) {};
							\node [style=none] (38) at (-5, 1) {};
							\node [style=none] (39) at (-7, 1) {};
							\node [style=bn] (40) at (6, 0) {};
							\node [style=none] (41) at (6, -2) {};
							\node [style=none] (42) at (5, 2) {};
							\node [style=morphism] (43) at (6, -1) {$e$};
							\node [style=none] (45) at (7, 2) {};
							\node [style=none] (46) at (7, 1) {};
							\node [style=none] (47) at (5, 1) {};
							\node [style=morphism] (48) at (7, 1) {$e$};
						\end{pgfonlayer}
						\begin{pgfonlayer}{edgelayer}
							\draw (23) to (24.center);
							\draw (28.center) to (29.center);
							\draw (25.center) to (30.center);
							\draw [in=165, out=-90] (30.center) to (23);
							\draw [in=270, out=15] (23) to (29.center);
							\draw (32) to (33.center);
							\draw (37.center) to (38.center);
							\draw (34.center) to (39.center);
							\draw [in=165, out=-90] (39.center) to (32);
							\draw [in=270, out=15] (32) to (38.center);
							\draw (40) to (41.center);
							\draw (45.center) to (46.center);
							\draw (42.center) to (47.center);
							\draw [in=165, out=-90] (47.center) to (40);
							\draw [in=270, out=15] (40) to (46.center);
						\end{pgfonlayer}
					\end{tikzpicture}
				}%
			\end{equation}  

		\item[\ref{it:detailed_balance} $\Rightarrow$ \ref{it:balancednew}:]
			This follows by applying $e$ to the left output in \cref{eq:detailed_balance}.

		\item[\ref{it:balancednew} $\Rightarrow$ \ref{it:as_strong}:]
			Using the associativity of copying as well as \cref{eq:balancednew,eq:balanced_newversion0}, we prove \cref{eq:as_balanced} via the following equalities:
			\begin{equation}
				{%
					\tikzstyle{every picture}=[tikzfig]%
					\begin{tikzpicture}
						\begin{pgfonlayer}{nodelayer}
							\node [style=bn] (0) at (-8, 5) {};
							\node [style=morphism] (2) at (-8, 4) {$e$};
							\node [style=morphism] (5) at (-8.75, 6) {$e$};
							\node [style=none] (6) at (-8.75, 7) {};
							\node [style=none] (7) at (-7.25, 7) {};
							\node [style=none] (8) at (-7.25, 6) {};
							\node [style=bn] (25) at (-7, 3) {};
							\node [style=none] (29) at (-8, 4) {};
							\node [style=morphism] (31) at (-7, 2) {$e$};
							\node [style=none] (32) at (-6, 4) {};
							\node [style=none] (33) at (-7, 1) {};
							\node [style=none] (89) at (-4, 4) {$=$};
							\node [style=none] (104) at (2.5, 4) {$=$};
							\node [style=none] (172) at (8.75, 4) {$=$};
							\node [style=none] (173) at (-6, 7) {};
							\node [style=none] (174) at (-8.75, 6) {};
							\node [style=bn] (175) at (-1.5, 5) {};
							\node [style=morphism] (176) at (-1.5, 4) {$e$};
							\node [style=morphism] (177) at (-2.25, 6) {$e$};
							\node [style=none] (178) at (-2.25, 7) {};
							\node [style=none] (179) at (-0.75, 7) {};
							\node [style=none] (180) at (-0.75, 6) {};
							\node [style=bn] (181) at (-0.5, 3) {};
							\node [style=none] (182) at (-1.5, 4) {};
							\node [style=morphism] (183) at (-0.5, 2) {$e$};
							\node [style=none] (184) at (0.5, 4) {};
							\node [style=none] (185) at (-0.5, 1) {};
							\node [style=none] (186) at (0.5, 7) {};
							\node [style=none] (187) at (-2.25, 6) {};
							\node [style=morphism] (188) at (-0.75, 6) {$e$};
							\node [style=bn] (189) at (4.75, 5) {};
							\node [style=morphism] (190) at (6.75, 6) {$e$};
							\node [style=morphism] (191) at (4, 6) {$e$};
							\node [style=none] (192) at (4, 7) {};
							\node [style=none] (193) at (5.5, 7) {};
							\node [style=none] (194) at (5.5, 6) {};
							\node [style=bn] (195) at (5.75, 3) {};
							\node [style=none] (196) at (4.75, 4) {};
							\node [style=morphism] (197) at (5.75, 2) {$e$};
							\node [style=none] (198) at (6.75, 4) {};
							\node [style=none] (199) at (5.75, 1) {};
							\node [style=none] (200) at (6.75, 7) {};
							\node [style=none] (201) at (4, 6) {};
							\node [style=morphism] (202) at (5.5, 6) {$e$};
							\node [style=none] (203) at (15, 4) {$=$};
							\node [style=none] (204) at (-4, -4) {$=$};
							\node [style=none] (205) at (2.5, -4) {$=$};
							\node [style=bn] (206) at (12.75, 5) {};
							\node [style=morphism] (207) at (10.75, 6) {$e$};
							\node [style=morphism] (208) at (13.5, 6) {$e$};
							\node [style=none] (209) at (13.5, 7) {};
							\node [style=none] (210) at (12, 7) {};
							\node [style=none] (211) at (12, 6) {};
							\node [style=bn] (212) at (11.75, 3) {};
							\node [style=none] (213) at (12.75, 4) {};
							\node [style=morphism] (214) at (11.75, 2) {$e$};
							\node [style=none] (215) at (10.75, 4) {};
							\node [style=none] (216) at (11.75, 1) {};
							\node [style=none] (217) at (10.75, 7) {};
							\node [style=none] (218) at (13.5, 6) {};
							\node [style=morphism] (219) at (12, 6) {$e$};
							\node [style=bn] (220) at (-6.25, -3) {};
							\node [style=morphism] (221) at (-8.25, -2) {$e$};
							\node [style=morphism] (222) at (-5.5, -2) {$e$};
							\node [style=none] (223) at (-5.5, -1) {};
							\node [style=none] (224) at (-7, -1) {};
							\node [style=none] (225) at (-7, -2) {};
							\node [style=bn] (226) at (-7.25, -5) {};
							\node [style=none] (227) at (-6.25, -4) {};
							\node [style=morphism] (228) at (-7.25, -6) {$e$};
							\node [style=none] (229) at (-8.25, -4) {};
							\node [style=none] (230) at (-7.25, -7) {};
							\node [style=none] (231) at (-8.25, -1) {};
							\node [style=none] (232) at (-5.5, -2) {};
							\node [style=morphism] (233) at (-7, -2) {$e$};
							\node [style=morphism] (234) at (-6.25, -4) {$e$};
							\node [style=bn] (235) at (0, -3) {};
							\node [style=morphism] (236) at (-2, -2) {$e$};
							\node [style=none] (237) at (0.75, -1) {};
							\node [style=none] (238) at (-0.75, -1) {};
							\node [style=none] (239) at (-0.75, -2) {};
							\node [style=bn] (240) at (-1, -5) {};
							\node [style=none] (241) at (0, -4) {};
							\node [style=morphism] (242) at (-1, -6) {$e$};
							\node [style=none] (243) at (-2, -4) {};
							\node [style=none] (244) at (-1, -7) {};
							\node [style=none] (245) at (-2, -1) {};
							\node [style=none] (246) at (0.75, -2) {};
							\node [style=morphism] (247) at (-0.75, -2) {$e$};
							\node [style=morphism] (248) at (0, -4) {$e$};
							\node [style=none] (249) at (8.75, -4) {$=$};
							\node [style=bn] (251) at (6.5, -3) {};
							\node [style=morphism] (252) at (4.5, -2) {$e$};
							\node [style=none] (253) at (7.25, -1) {};
							\node [style=none] (254) at (5.75, -1) {};
							\node [style=none] (255) at (5.75, -2) {};
							\node [style=bn] (256) at (5.5, -5) {};
							\node [style=none] (257) at (6.5, -4) {};
							\node [style=morphism] (258) at (5.5, -6) {$e$};
							\node [style=none] (259) at (4.5, -4) {};
							\node [style=none] (260) at (5.5, -7) {};
							\node [style=none] (261) at (4.5, -1) {};
							\node [style=none] (262) at (7.25, -2) {};
							\node [style=morphism] (263) at (5.75, -2) {$e$};
							\node [style=bn] (264) at (11.5, -3) {};
							\node [style=morphism] (265) at (10.75, -2) {$e$};
							\node [style=none] (266) at (10.75, -1) {};
							\node [style=none] (267) at (12.25, -1) {};
							\node [style=none] (268) at (12.25, -2) {};
							\node [style=bn] (269) at (12.5, -5) {};
							\node [style=none] (270) at (11.5, -4) {};
							\node [style=morphism] (271) at (12.5, -6) {$e$};
							\node [style=none] (272) at (13.5, -4) {};
							\node [style=none] (273) at (12.5, -7) {};
							\node [style=none] (274) at (13.5, -1) {};
							\node [style=none] (275) at (10.75, -2) {};
							\node [style=morphism] (276) at (12.25, -2) {$e$};
						\end{pgfonlayer}
						\begin{pgfonlayer}{edgelayer}
							\draw [in=-90, out=15] (0) to (8.center);
							\draw (8.center) to (7.center);
							\draw (6.center) to (5);
							\draw (0) to (29.center);
							\draw [in=165, out=-90] (29.center) to (25);
							\draw [in=-90, out=15] (25) to (32.center);
							\draw [in=-90, out=165] (0) to (174.center);
							\draw (32.center) to (173.center);
							\draw (33.center) to (31);
							\draw (31) to (25);
							\draw [in=-90, out=15] (175) to (180.center);
							\draw (180.center) to (179.center);
							\draw (178.center) to (177);
							\draw (175) to (182.center);
							\draw [in=165, out=-90] (182.center) to (181);
							\draw [in=-90, out=15] (181) to (184.center);
							\draw [in=-90, out=165] (175) to (187.center);
							\draw (184.center) to (186.center);
							\draw (185.center) to (183);
							\draw (183) to (181);
							\draw [in=-90, out=15] (189) to (194.center);
							\draw (194.center) to (193.center);
							\draw (192.center) to (191);
							\draw (189) to (196.center);
							\draw [in=165, out=-90] (196.center) to (195);
							\draw [in=-90, out=15] (195) to (198.center);
							\draw [in=-90, out=165] (189) to (201.center);
							\draw (198.center) to (200.center);
							\draw (199.center) to (197);
							\draw (197) to (195);
							\draw [in=-90, out=165] (206) to (211.center);
							\draw (211.center) to (210.center);
							\draw (209.center) to (208);
							\draw (206) to (213.center);
							\draw [in=15, out=-90] (213.center) to (212);
							\draw [in=-90, out=165] (212) to (215.center);
							\draw [in=-90, out=15] (206) to (218.center);
							\draw (215.center) to (217.center);
							\draw (216.center) to (214);
							\draw (214) to (212);
							\draw [in=-90, out=165] (220) to (225.center);
							\draw (225.center) to (224.center);
							\draw (223.center) to (222);
							\draw (220) to (227.center);
							\draw [in=15, out=-90] (227.center) to (226);
							\draw [in=-90, out=165] (226) to (229.center);
							\draw [in=-90, out=15] (220) to (232.center);
							\draw (229.center) to (231.center);
							\draw (230.center) to (228);
							\draw (228) to (226);
							\draw [in=-90, out=165] (235) to (239.center);
							\draw (239.center) to (238.center);
							\draw (235) to (241.center);
							\draw [in=15, out=-90] (241.center) to (240);
							\draw [in=-90, out=165] (240) to (243.center);
							\draw [in=-90, out=15] (235) to (246.center);
							\draw (243.center) to (245.center);
							\draw (244.center) to (242);
							\draw (242) to (240);
							\draw (246.center) to (237.center);
							\draw [in=-90, out=165] (251) to (255.center);
							\draw (255.center) to (254.center);
							\draw (251) to (257.center);
							\draw [in=15, out=-90] (257.center) to (256);
							\draw [in=-90, out=165] (256) to (259.center);
							\draw [in=-90, out=15] (251) to (262.center);
							\draw (259.center) to (261.center);
							\draw (260.center) to (258);
							\draw (258) to (256);
							\draw (262.center) to (253.center);
							\draw [in=-90, out=15] (264) to (268.center);
							\draw (268.center) to (267.center);
							\draw (264) to (270.center);
							\draw [in=165, out=-90] (270.center) to (269);
							\draw [in=-90, out=15] (269) to (272.center);
							\draw [in=-90, out=165] (264) to (275.center);
							\draw (272.center) to (274.center);
							\draw (273.center) to (271);
							\draw (271) to (269);
							\draw (275.center) to (266.center);
						\end{pgfonlayer}
					\end{tikzpicture}
				}%
			\end{equation}
			This is the desired statement.

		\item[\ref{it:as_strong} $\Rightarrow$ \ref{it:balancednew}:]
			Whenever $e$ satisfies \cref{eq:as_balanced}, \cref{eq:balancednew} defining balanced idempotents follows simply by marginalizing the extra output of that equation and using $e^2 =e$.

		\item[\ref{it:detailed_balance} $\Rightarrow$ \ref{it:self_adjoint}:]
			We can use the invariance of $p$ followed by \cref{eq:detailed_balance} and the invariance of $p$ again:
			\begin{equation}\label{eq:self_adjoint_proof}
				{%
					\tikzstyle{every picture}=[tikzfig]%
					\begin{tikzpicture}
						\begin{pgfonlayer}{nodelayer}
							\node [style=none] (13) at (0, 0) {$=$};
							\node [style=none] (22) at (-6, 0) {$=$};
							\node [style=none] (29) at (6, 0) {$=$};
							\node [style=bn] (33) at (-9, 0) {};
							\node [style=morphism] (34) at (-9, -1) {$p$};
							\node [style=none] (35) at (-10, 2.25) {};
							\node [style=none] (36) at (-8, 2.25) {};
							\node [style=none] (37) at (-8, 1.25) {};
							\node [style=none] (38) at (-9, -2) {};
							\node [style=morphism] (39) at (-10, 1.25) {$e$};
							\node [style=none] (40) at (-10, 1.25) {};
							\node [style=none] (41) at (-9, -2.5) {$A$};
							\node [style=none] (42) at (-10, 2.75) {$X$};
							\node [style=none] (43) at (-8, 2.75) {$X$};
							\node [style=bn] (44) at (9, 0) {};
							\node [style=morphism] (45) at (9, -1) {$p$};
							\node [style=none] (46) at (10, 2.25) {};
							\node [style=none] (47) at (8, 2.25) {};
							\node [style=none] (48) at (8, 1.25) {};
							\node [style=none] (49) at (9, -2) {};
							\node [style=morphism] (50) at (10, 1.25) {$e$};
							\node [style=none] (51) at (10, 1.25) {};
							\node [style=none] (52) at (9, -2.5) {$A$};
							\node [style=none] (53) at (10, 2.75) {$X$};
							\node [style=none] (54) at (8, 2.75) {$X$};
							\node [style=bn] (55) at (-3, 0.5) {};
							\node [style=morphism] (56) at (-3, -1.75) {$p$};
							\node [style=none] (57) at (-4, 2.75) {};
							\node [style=none] (58) at (-2, 2.75) {};
							\node [style=none] (59) at (-2, 1.75) {};
							\node [style=none] (60) at (-3, -2.75) {};
							\node [style=morphism] (61) at (-4, 1.75) {$e$};
							\node [style=none] (62) at (-4, 1.75) {};
							\node [style=none] (63) at (-3, -3.25) {$A$};
							\node [style=none] (64) at (-4, 3.25) {$X$};
							\node [style=none] (65) at (-2, 3.25) {$X$};
							\node [style=morphism] (66) at (-3, -0.5) {$e$};
							\node [style=bn] (67) at (3, 0.5) {};
							\node [style=morphism] (68) at (3, -1.75) {$p$};
							\node [style=none] (69) at (4, 2.75) {};
							\node [style=none] (70) at (2, 2.75) {};
							\node [style=none] (71) at (2, 1.75) {};
							\node [style=none] (72) at (3, -2.75) {};
							\node [style=morphism] (73) at (4, 1.75) {$e$};
							\node [style=none] (74) at (4, 1.75) {};
							\node [style=none] (75) at (3, -3.25) {$A$};
							\node [style=none] (76) at (4, 3.25) {$X$};
							\node [style=none] (77) at (2, 3.25) {$X$};
							\node [style=morphism] (78) at (3, -0.5) {$e$};
						\end{pgfonlayer}
						\begin{pgfonlayer}{edgelayer}
							\draw (38.center) to (33);
							\draw [in=-90, out=15] (33) to (37.center);
							\draw (37.center) to (36.center);
							\draw [in=-90, out=165] (33) to (40.center);
							\draw (40.center) to (35.center);
							\draw (49.center) to (44);
							\draw [in=-90, out=165] (44) to (48.center);
							\draw (48.center) to (47.center);
							\draw [in=-90, out=15] (44) to (51.center);
							\draw (51.center) to (46.center);
							\draw (60.center) to (55);
							\draw [in=-90, out=15] (55) to (59.center);
							\draw (59.center) to (58.center);
							\draw [in=-90, out=165] (55) to (62.center);
							\draw (62.center) to (57.center);
							\draw (72.center) to (67);
							\draw [in=-90, out=165] (67) to (71.center);
							\draw (71.center) to (70.center);
							\draw [in=-90, out=15] (67) to (74.center);
							\draw (74.center) to (69.center);
						\end{pgfonlayer}
					\end{tikzpicture}
				}%
			\end{equation}

		\item[\ref{it:self_adjoint} $\Rightarrow$ \ref{it:detailed_balance}:]
			Since $e$ is assumed to be an idempotent, it is itself invariant, so that we can set $p=e$ in \cref{eq:self_adjoint} to get \cref{eq:detailed_balance}. 
			\qedhere
	\end{itemize}
\end{proof}

\begin{remark}
	\label{rem:detailed_balance}
	In point-separable categories such as $\stoch$, condition \ref{it:self_adjoint} can be equivalently checked just on invariant states, in which case it looks like
	\begin{equation}\label{eq:self_adjoint_states}
		{%
			\tikzstyle{every picture}=[tikzfig]%
			\begin{tikzpicture}
				\begin{pgfonlayer}{nodelayer}
					\node [style=bn] (0) at (-3, -0.5) {};
					\node [style=none] (6) at (-4, 1.75) {};
					\node [style=none] (7) at (-2, 1.75) {};
					\node [style=none] (8) at (-2, 0.75) {};
					\node [style=none] (9) at (-3, -1.5) {};
					\node [style=none] (13) at (0, 0) {$=$};
					\node [style=state] (14) at (-3, -1.5) {$p$};
					\node [style=morphism] (16) at (-4, 0.75) {$e$};
					\node [style=none] (17) at (-4, 0.75) {};
					\node [style=bn] (18) at (3, -0.5) {};
					\node [style=none] (19) at (4, 1.75) {};
					\node [style=none] (20) at (2, 1.75) {};
					\node [style=none] (21) at (2, 0.75) {};
					\node [style=none] (22) at (3, -1.5) {};
					\node [style=state] (23) at (3, -1.5) {$p$};
					\node [style=morphism] (24) at (4, 0.75) {$e$};
					\node [style=none] (25) at (4, 0.75) {};
				\end{pgfonlayer}
				\begin{pgfonlayer}{edgelayer}
					\draw (9.center) to (0);
					\draw [in=-90, out=15] (0) to (8.center);
					\draw (8.center) to (7.center);
					\draw (17.center) to (6.center);
					\draw [in=-90, out=165] (0) to (17.center);
					\draw (22.center) to (18);
					\draw [in=-90, out=165] (18) to (21.center);
					\draw (21.center) to (20.center);
					\draw (25.center) to (19.center);
					\draw [in=-90, out=15] (18) to (25.center);
				\end{pgfonlayer}
			\end{tikzpicture}
		}%
	\end{equation}
	which says that $e$ satisfies \emph{detailed balance} with respect to $p$ in the language of Markov processes \cite[Section 4.10]{laidler1987kinetics}.
	This is also equivalent to saying that the morphism $e$ is its own Bayesian inverse with respect to $p$.
	By \cref{prop:balancednew}, balanced idempotents in $\stoch$ are thus those idempotents satisfying \emph{detailed balance} for every invariant state, i.e.\ ones for which the dynamics is \emph{microscopically reversible}. 

	Alternatively, \cref{eq:self_adjoint_states} can also be interpreted as saying that $e$ is \emph{self-adjoint}. 
	Indeed, if $p$ is an invariant measure for a balanced idempotent Markov kernel $e \colon X\to X$, then the operator ${L^2(X,p)\to L^2(X,p)}$ induced by pre-composition with $e$ is idempotent and self-adjoint, making it an orthogonal projection on the Hilbert space $L^2(X,p)$.
\end{remark}

We have the following weaker formulation of the definition of static idempotent.

\begin{lemma}[Alternative description of static idempotents]\label{lem:balanced_as_det}
	Let $e$ be an idempotent. Then $e$ is static if and only if it is $\as{e}$ deterministic.
\end{lemma}
\begin{proof}
	Let us consider an idempotent $e$ which is $\as{e}$ deterministic. 
	We can then use this assumption together with $e^2 = e$ to get
	\begin{equation}
		{%
			\tikzstyle{every picture}=[tikzfig]%
			\begin{tikzpicture}
				\begin{pgfonlayer}{nodelayer}
					\node [style=bn] (0) at (-8.75, 0.5) {};
					\node [style=none] (6) at (-9.5, 2.75) {};
					\node [style=none] (7) at (-8, 2.75) {};
					\node [style=none] (8) at (-8, 1.5) {};
					\node [style=none] (9) at (-7.75, -0.75) {};
					\node [style=bn] (25) at (-7.75, -0.75) {};
					\node [style=none] (29) at (-8.75, 0.5) {};
					\node [style=bn] (30) at (-7.75, -0.75) {};
					\node [style=morphism] (31) at (-7.75, -1.75) {$e$};
					\node [style=bn] (32) at (-6.75, 1.75) {};
					\node [style=none] (33) at (-7.75, -2.75) {};
					\node [style=bn] (75) at (-2.25, 1.5) {};
					\node [style=morphism] (76) at (-2.25, 0.5) {$e$};
					\node [style=none] (77) at (-3, 2.5) {};
					\node [style=none] (78) at (-3, 2.5) {};
					\node [style=none] (79) at (-1.5, 2.5) {};
					\node [style=none] (80) at (-1.5, 2.5) {};
					\node [style=none] (81) at (-1.25, -0.75) {};
					\node [style=bn] (82) at (-1.25, -0.75) {};
					\node [style=none] (83) at (-2.25, 0.5) {};
					\node [style=bn] (84) at (-1.25, -0.75) {};
					\node [style=morphism] (85) at (-1.25, -1.75) {$e$};
					\node [style=bn] (86) at (-0.25, 1.75) {};
					\node [style=none] (87) at (-1.25, -2.75) {};
					\node [style=none] (89) at (-4.75, 0) {$=$};
					\node [style=none] (104) at (1.75, 0) {$=$};
					\node [style=morphism] (173) at (-8, 1.75) {$e$};
					\node [style=bn] (174) at (-14.5, 0) {};
					\node [style=morphism] (175) at (-14.5, -1.25) {$e$};
					\node [style=none] (177) at (-15.5, 2.75) {};
					\node [style=none] (178) at (-13.5, 2.75) {};
					\node [style=none] (180) at (-14.5, -2.75) {};
					\node [style=none] (181) at (-11.5, 0) {$=$};
					\node [style=bn] (182) at (5, 0) {};
					\node [style=morphism] (183) at (5, -1.25) {$e$};
					\node [style=none] (184) at (4, 1.25) {};
					\node [style=none] (185) at (4, 2.75) {};
					\node [style=none] (186) at (6, 2.75) {};
					\node [style=none] (187) at (6, 1.25) {};
					\node [style=none] (188) at (5, -2.75) {};
					\node [style=none] (189) at (-6.75, 0.5) {};
					\node [style=none] (190) at (-0.25, 0.5) {};
					\node [style=none] (191) at (-15.5, 1.25) {};
					\node [style=none] (192) at (-13.5, 1.25) {};
					\node [style=none] (193) at (-9.5, 1.5) {};
					\node [style=morphism] (194) at (-15.5, 1.75) {$e$};
					\node [style=morphism] (195) at (-13.5, 1.75) {$e$};
					\node [style=morphism] (196) at (-9.5, 1.75) {$e$};
					\node [style=none] (197) at (-3, 2.75) {};
					\node [style=none] (198) at (-1.5, 2.75) {};
				\end{pgfonlayer}
				\begin{pgfonlayer}{edgelayer}
					\draw [in=-90, out=15] (0) to (8.center);
					\draw (8.center) to (7.center);
					\draw (0) to (29.center);
					\draw [in=165, out=-90] (29.center) to (25);
					\draw (33.center) to (30);
					\draw [in=-90, out=15] (75) to (80.center);
					\draw (80.center) to (79.center);
					\draw (78.center) to (77.center);
					\draw [in=165, out=-90] (77.center) to (75);
					\draw (75) to (83.center);
					\draw [in=165, out=-90] (83.center) to (82);
					\draw (87.center) to (84);
					\draw (180.center) to (174);
					\draw (188.center) to (182);
					\draw [in=-90, out=15] (182) to (187.center);
					\draw (187.center) to (186.center);
					\draw (185.center) to (184.center);
					\draw [in=165, out=-90] (184.center) to (182);
					\draw (32) to (189.center);
					\draw (86) to (190.center);
					\draw [in=15, out=-90] (189.center) to (30);
					\draw [in=15, out=-90] (190.center) to (84);
					\draw (177.center) to (191.center);
					\draw (178.center) to (192.center);
					\draw (6.center) to (193.center);
					\draw [in=165, out=-90] (191.center) to (174);
					\draw [in=270, out=15] (174) to (192.center);
					\draw [in=165, out=-90] (193.center) to (29.center);
					\draw (78.center) to (197.center);
					\draw (80.center) to (198.center);
				\end{pgfonlayer}
			\end{tikzpicture}
		}%
	\end{equation}
	The desired \cref{eq:balanced_as_det} can then be derived as follows.
	\begin{equation}\label{eq:balanced_as_det2}
		{%
			\tikzstyle{every picture}=[tikzfig]%
			\begin{tikzpicture}
				\begin{pgfonlayer}{nodelayer}
					\node [style=bn] (0) at (-6, 0) {};
					\node [style=bn] (1) at (0, 0) {};
					\node [style=morphism] (2) at (-6, -1) {$e$};
					\node [style=none] (5) at (-7, 1) {};
					\node [style=none] (6) at (-7, 3) {};
					\node [style=none] (7) at (-5, 3) {};
					\node [style=none] (8) at (-5, 1) {};
					\node [style=none] (9) at (-6, -2) {};
					\node [style=none] (10) at (0, -2) {};
					\node [style=none] (11) at (-1, 3) {};
					\node [style=none] (12) at (1, 3) {};
					\node [style=none] (13) at (-3, 0) {$=$};
					\node [style=morphism] (14) at (0, -1) {$e$};
					\node [style=none] (21) at (3, 0) {$=$};
					\node [style=morphism] (24) at (-7, 2) {$e$};
					\node [style=morphism] (25) at (-1, 2.25) {$e$};
					\node [style=bn] (26) at (6, 0) {};
					\node [style=morphism] (27) at (5, 1.5) {$e$};
					\node [style=morphism] (28) at (7, 1.5) {$e$};
					\node [style=none] (29) at (6, -2) {};
					\node [style=none] (30) at (5, 3) {};
					\node [style=none] (31) at (7, 3) {};
					\node [style=morphism] (32) at (6, -1) {$e$};
					\node [style=bn] (33) at (12, 0) {};
					\node [style=morphism] (34) at (12, -1) {$e$};
					\node [style=none] (35) at (11, 1.5) {};
					\node [style=none] (36) at (11, 3) {};
					\node [style=none] (37) at (13, 3) {};
					\node [style=none] (38) at (13, 1.5) {};
					\node [style=none] (39) at (12, -2) {};
					\node [style=none] (41) at (9, 0) {$=$};
					\node [style=none] (42) at (-1, 1) {};
					\node [style=none] (43) at (1, 1) {};
					\node [style=none] (44) at (-1, 1) {};
					\node [style=morphism] (45) at (-1, 1) {$e$};
					\node [style=morphism] (46) at (1, 1) {$e$};
				\end{pgfonlayer}
				\begin{pgfonlayer}{edgelayer}
					\draw (1) to (10.center);
					\draw (9.center) to (0);
					\draw [in=-90, out=15] (0) to (8.center);
					\draw (8.center) to (7.center);
					\draw (6.center) to (5.center);
					\draw [in=165, out=-90] (5.center) to (0);
					\draw (30.center) to (27);
					\draw [in=165, out=-90] (27) to (26);
					\draw (31.center) to (28);
					\draw [in=15, out=-90] (28) to (26);
					\draw (26) to (29.center);
					\draw (39.center) to (33);
					\draw [in=-90, out=15] (33) to (38.center);
					\draw (38.center) to (37.center);
					\draw (36.center) to (35.center);
					\draw [in=165, out=-90] (35.center) to (33);
					\draw (44.center) to (11.center);
					\draw (12.center) to (43.center);
					\draw [in=165, out=-90] (44.center) to (1);
					\draw [in=15, out=-90] (43.center) to (1);
				\end{pgfonlayer}
			\end{tikzpicture}
		}%
	\end{equation}

	Conversely, a static idempotent $e$ is balanced, as noted in \cref{rem:idempotents}. 
	Therefore, we have
	\begin{equation}
		{%
			\tikzstyle{every picture}=[tikzfig]%
			\begin{tikzpicture}
				\begin{pgfonlayer}{nodelayer}
					\node [style=bn] (0) at (-6, 0) {};
					\node [style=bn] (1) at (6.5, 0) {};
					\node [style=morphism] (2) at (-6, -1) {$e$};
					\node [style=none] (5) at (-7, 1) {};
					\node [style=none] (6) at (-7, 2) {};
					\node [style=none] (7) at (-5, 2) {};
					\node [style=none] (8) at (-5, 1) {};
					\node [style=none] (9) at (-6, -2) {};
					\node [style=none] (10) at (6.5, -2) {};
					\node [style=none] (11) at (5.5, 2) {};
					\node [style=none] (12) at (7.5, 2) {};
					\node [style=none] (13) at (3.25, 0) {$\ase{e}$};
					\node [style=bn] (15) at (0, 0) {};
					\node [style=none] (16) at (-1, 1) {};
					\node [style=none] (17) at (1, 1) {};
					\node [style=none] (18) at (0, -2) {};
					\node [style=none] (19) at (-1, 2) {};
					\node [style=none] (20) at (1, 2) {};
					\node [style=none] (21) at (-3, 0) {$=$};
					\node [style=morphism] (22) at (0, -1) {$e$};
					\node [style=morphism] (23) at (-1, 1) {$e$};
					\node [style=none] (24) at (5.5, 1) {};
					\node [style=none] (25) at (7.5, 1) {};
					\node [style=none] (26) at (7.5, 1) {};
					\node [style=morphism] (27) at (7.5, 1) {$e$};
					\node [style=morphism] (28) at (5.5, 1) {$e$};
				\end{pgfonlayer}
				\begin{pgfonlayer}{edgelayer}
					\draw (1) to (10.center);
					\draw (9.center) to (0);
					\draw [in=-90, out=15] (0) to (8.center);
					\draw (8.center) to (7.center);
					\draw (6.center) to (5.center);
					\draw [in=165, out=-90] (5.center) to (0);
					\draw (19.center) to (16.center);
					\draw [in=165, out=-90] (16.center) to (15);
					\draw (20.center) to (17.center);
					\draw [in=15, out=-90] (17.center) to (15);
					\draw (15) to (18.center);
					\draw (11.center) to (24.center);
					\draw (12.center) to (26.center);
					\draw [in=165, out=-90] (24.center) to (1);
					\draw [in=15, out=-90] (26.center) to (1);
				\end{pgfonlayer}
			\end{tikzpicture}
		}%
	\end{equation}
	by \cref{prop:balancednew}, which says exactly that $e$ is $\as{e}$ deterministic.	
\end{proof}

\subsection{Sufficient condition for idempotents to be balanced}\label{sec:balanced_Markov}

A forthcoming result shows that every split idempotent in a positive Markov category is necessarily balanced (\cref{thm:split_props}). 
Therefore, when studying whether all idempotents in a given Markov category split, it is natural to ask first whether all idempotents are balanced.
To this end, we introduce a new information flow axiom, which is a version of the \emph{multiplication lemma} of Parzygnat~\cite[Lemma~8.10]{parzygnat2020inverses}.

\begin{definition}
	\label{def:balanced_cat}
	A Markov category $\cC$ is \newterm{balanced} if the implication
	\begin{equation}
		\label{eq:balanced_cat}
		{%
			\tikzstyle{every picture}=[tikzfig]%
			\begin{tikzpicture}
				\begin{pgfonlayer}{nodelayer}
					\node [style=bn] (0) at (-11, -0.75) {};
					\node [style=morphism] (1) at (-11, -1.75) {$f$};
					\node [style=morphism] (2) at (-12, 0.5) {$g$};
					\node [style=morphism] (3) at (-10, 0.5) {$g$};
					\node [style=morphism] (4) at (-12, 2) {$h$};
					\node [style=morphism] (5) at (-10, 2) {$h$};
					\node [style=none] (6) at (-12, 2.75) {};
					\node [style=none] (7) at (-12, 3.25) {$Y$};
					\node [style=none] (8) at (-10, 3.25) {$Y$};
					\node [style=none] (9) at (-10, 2.75) {};
					\node [style=none] (10) at (-11, -2.75) {};
					\node [style=none] (11) at (-11, -3.25) {$A$};
					\node [style=none] (12) at (-8, 0) {$=$};
					\node [style=bn] (13) at (-5, 0.75) {};
					\node [style=morphism] (14) at (-5, -1.75) {$f$};
					\node [style=morphism] (17) at (-6, 2) {$h$};
					\node [style=morphism] (18) at (-4, 2) {$h$};
					\node [style=none] (19) at (-6, 2.75) {};
					\node [style=none] (20) at (-6, 3.25) {$Y$};
					\node [style=none] (21) at (-4, 3.25) {$Y$};
					\node [style=none] (22) at (-4, 2.75) {};
					\node [style=none] (23) at (-5, -2.75) {};
					\node [style=none] (24) at (-5, -3.25) {$A$};
					\node [style=morphism] (25) at (-5, -0.25) {$g$};
					\node [style=none] (26) at (0, 0) {$\implies$};
					\node [style=bn] (27) at (5, -0.75) {};
					\node [style=morphism] (28) at (5, -1.75) {$f$};
					\node [style=morphism] (29) at (4, 0.5) {$g$};
					\node [style=morphism] (30) at (6, 0.5) {$g$};
					\node [style=morphism] (31) at (4, 2) {$h$};
					\node [style=none] (33) at (4, 2.75) {};
					\node [style=none] (34) at (4, 3.25) {$Y$};
					\node [style=none] (35) at (6, 3.25) {$X$};
					\node [style=none] (36) at (6, 2.75) {};
					\node [style=none] (37) at (5, -2.75) {};
					\node [style=none] (38) at (5, -3.25) {$A$};
					\node [style=none] (39) at (8, 0) {$=$};
					\node [style=bn] (40) at (11, 0.75) {};
					\node [style=morphism] (41) at (11, -1.75) {$f$};
					\node [style=morphism] (42) at (10, 2) {$h$};
					\node [style=none] (44) at (10, 2.75) {};
					\node [style=none] (45) at (10, 3.25) {$Y$};
					\node [style=none] (46) at (12, 3.25) {$X$};
					\node [style=none] (47) at (12, 2.75) {};
					\node [style=none] (48) at (11, -2.75) {};
					\node [style=none] (49) at (11, -3.25) {$A$};
					\node [style=morphism] (50) at (11, -0.25) {$g$};
					\node [style=none] (51) at (12, 1.75) {};
				\end{pgfonlayer}
				\begin{pgfonlayer}{edgelayer}
					\draw (10.center) to (0);
					\draw [in=-90, out=165] (0) to (2);
					\draw (2) to (6.center);
					\draw [in=-90, out=15] (0) to (3);
					\draw (3) to (9.center);
					\draw (23.center) to (13);
					\draw [in=-90, out=165] (13) to (17);
					\draw [in=-90, out=15] (13) to (18);
					\draw (17) to (19.center);
					\draw (22.center) to (18);
					\draw (37.center) to (27);
					\draw [in=-90, out=165] (27) to (29);
					\draw (29) to (33.center);
					\draw [in=-90, out=15] (27) to (30);
					\draw (30) to (36.center);
					\draw (48.center) to (40);
					\draw [in=-90, out=165] (40) to (42);
					\draw (42) to (44.center);
					\draw (51.center) to (47.center);
					\draw [in=-90, out=15] (40) to (51.center);
				\end{pgfonlayer}
			\end{tikzpicture}
		}%
	\end{equation}
	holds in $\cC$.
\end{definition}
We now show that, as claimed, this property ensures that all idempotents are balanced, a result that motivates its name.
\begin{theorem}
	\label{thm:balanced_idempotent}
	If a Markov category $\cC$ is balanced, then every idempotent in $\cC$ is balanced.
\end{theorem}
\begin{proof} 
	Let $e$ be any idempotent in $\cC$.
	In \cref{eq:balanced_cat}, take $f = g = h = e$.
	Then the antecedent of Implication \eqref{eq:balanced_cat} holds by the assumed $e^2 = e$.
	Its consequent tells us that $e$ is a balanced idempotent.
\end{proof}

\begin{proposition}\label{prop:stoch_balanced_cat}
	The Markov category $\stoch$ is balanced.
\end{proposition}

\begin{proof}
	For $f \colon A \to B$, $g \colon B \to X$, and $h \colon X \to Y$, suppose that the antecedent of Implication \eqref{eq:balanced_cat} holds.
	Let us assume $A = I$, meaning that $f$ is a state.
	By point-separability of $\stoch$ (\cref{def:separability}), this assumption comes without loss of generality.

	Fixing $U \in \Sigma_Y$, let us write $H(x) \coloneqq h(U|x)$ and consider the expectation value
	\begin{equation}
		E(b) \coloneqq \int_{x \in X} H(x) \, g(dx|b)
	\end{equation}
	for every $b \in B$.
	Then, by the strict convexity of $x\mapsto x^2$, we have
	\begin{equation}
		\int_{x \in X} H(x)^2 \, g(dx|b) \ge E(b)^2
	\end{equation}
	with equality if and only if $H(x) = E(b)$ holds for $g(\ph|b)$-almost all $x$.
	Evaluating the antecedent of Implication \eqref{eq:balanced_cat} on the measurable set $U \times U$ gives
	\begin{equation}
		\int_{b \in B} \int_{x_1 \in X} \int_{x_2 \in X} H(x_1) \, H(x_2) \, g(dx_1|b) \, g(dx_2|b) \, f(db) \\
		= \int_{b \in B} \int_{x \in X} H(x)^2 \, g(dx|b) \, f(db),
	\end{equation}
	or equivalently
	\begin{equation}
		\int_{b \in B} \left( \int_{x \in X} H(x)^2 \, g(dx|b) - E(b)^2 \right) \, f(db) = 0.
	\end{equation}
	Together with the established non-negativity of the integrand, it follows that this integrand must vanish almost surely.
	By the above tightness condition, we hence conclude $H(x) = E(b)$ for $g(\ph|b)$-almost all $x$ and $f$-almost all $b$.

	Equipped with this intermediate result, let us now turn towards proving the claim.
	Since it is enough to prove the equality of two kernels with values in a product space on products of measurable sets~\cite[Lemma~4.2]{fritz2019synthetic}, it suffices to show that for all $W \in \Sigma_X$,
	\begin{equation}
		\int_{b \in B} \int_{x \in X} H(x) \, g(dx|b) \, g(W|b) \, f(db)
		= \int_{b \in B} \int_{x \in W} H(x) \, g(dx|b) \, f(db),
	\end{equation}
	or equivalently
	\begin{equation}\label{eq:CS_proof}
		\int_{b \in B} E(b) \, g(W|b) \, f(db)
		= \int_{b \in B} \int_{x \in W} H(x) \, g(dx|b) \, f(db).
	\end{equation}
	Since $H(x)$ equals $E(b)$ for $g(\ph|b)$-almost all $x$ and $f$-almost all $b$, we can evaluate the right-hand side to 
	\begin{equation}
		\int_{b \in B} E(b) \int_{x \in W} g(dx|b) \, f(db),
	\end{equation}
	and this proves \cref{eq:CS_proof}.
\end{proof}

\begin{corollary}\label{cor:balanced_cat} $\stoch, \finstoch, \borelstoch, \topstoch, \tychstoch$ and $\chausstoch$ are balanced.
	Moreover, all idempotents in these categories are balanced.
\end{corollary}
\begin{proof}
	As all of these Markov categories are Markov subcategories of $\stoch$, they are balanced by \cref{prop:stoch_balanced_cat}.
	Using \cref{thm:balanced_idempotent}, we conclude that every idempotent in $\stoch$ is balanced.
\end{proof}

An intriguing feature of balancedness is that this property of probabilistic uncertainty is not shared by possibilistic uncertainty, as we recall in the following.
\begin{example}[Non-balanced categories]
	\label{ex:not_balanced_cat}
	In $\setmulti$ and $\finsetmulti$, letting $e$ be one of the idempotents from \cref{ex:setmulti_not_balanced} violates balancedness with $f = g = h = e$.
	The same holds for the idempotent $e$ from \cref{ex:finstochpm_not_balanced} in $\finstoch_\pm$.
\end{example}

\subsection{Split idempotents in Markov categories}
\label{sec:split}

The goal of this and the following subsection is mainly to study \newterm{splitting} of idempotents in Markov categories and to find criteria which guarantee that all (balanced) idempotents split.
Recall first that a splitting of a generic morphism $e$ in a category is a factorization $e = \iota\comp \pi$ such that $\pi\comp \iota$ is the identity morphism~\cite[Section~6.5]{borceux1994handbook}.
Clearly, if $e$ has such a splitting, then it is an idempotent, and the splitting is unique up to unique isomorphism.
In a Markov category, the types of idempotents from \cref{def:idempotents} are reflected in properties of a splitting as follows.

\begin{theorem}\label{thm:split_props}
	In a positive Markov category $\cC$, let $e \colon X \to X$ be an idempotent that splits as $e = \iota\comp \pi$ for $\pi \colon X \to T$ and $\iota \colon T \to X$.
	\begin{enumerate}
		\item\label{it:split_static} $e$ is static if and only if $\iota$ is deterministic.\footnote{The ``if'' part holds even without assuming that $\cC$ is positive (as our proof shows). \label{fn:without_positive}}
		\item\label{it:split_strong} $e$ is strong if and only if $\pi$ is deterministic.
		\item\label{it:split_balanced} $e$ is balanced (without further assumptions).
	\end{enumerate}
	Moreover, $\pi$ is necessarily $\as{\iota}$ deterministic.
\end{theorem}

\begin{proof}
	Since $\pi\comp \iota$ is the identity and thus deterministic, the positivity assumption on $\cC$ gives 
	\begin{equation}
		\label{eq:split_positivity}
		{%
			\tikzstyle{every picture}=[tikzfig]%
			\begin{tikzpicture}
				\begin{pgfonlayer}{nodelayer}
					\node [style=bn] (0) at (-3, 0) {};
					\node [style=morphism] (1) at (-3, -1) {$\iota$};
					\node [style=none] (2) at (-4, 1) {};
					\node [style=none] (3) at (-4, 2.25) {};
					\node [style=none] (4) at (-2, 2.25) {};
					\node [style=none] (5) at (-2, 1) {};
					\node [style=none] (6) at (-3, -2) {};
					\node [style=morphism] (9) at (-4, 1.25) {$\pi$};
					\node [style=none] (10) at (0, 0) {$=$};
					\node [style=bn] (11) at (3, 0) {};
					\node [style=none] (13) at (2, 1) {};
					\node [style=none] (14) at (2, 2.25) {};
					\node [style=none] (15) at (4, 2.25) {};
					\node [style=none] (16) at (4, 1) {};
					\node [style=none] (17) at (3, -2) {};
					\node [style=morphism] (18) at (4, 1.25) {$\iota$};
					\node [style=none] (19) at (-4, 2.75) {$T$};
					\node [style=none] (20) at (-2, 2.75) {$X$};
					\node [style=none] (21) at (-3, -2.5) {$T$};
					\node [style=none] (22) at (2, 2.75) {$T$};
					\node [style=none] (23) at (4, 2.75) {$X$};
					\node [style=none] (24) at (3, -2.5) {$T$};
				\end{pgfonlayer}
				\begin{pgfonlayer}{edgelayer}
					\draw (6.center) to (0);
					\draw [in=-90, out=15] (0) to (5.center);
					\draw (5.center) to (4.center);
					\draw (3.center) to (2.center);
					\draw [in=165, out=-90] (2.center) to (0);
					\draw (17.center) to (11);
					\draw [in=-90, out=15] (11) to (16.center);
					\draw (16.center) to (15.center);
					\draw (14.center) to (13.center);
					\draw [in=165, out=-90] (13.center) to (11);
				\end{pgfonlayer}
			\end{tikzpicture}
		}%
	\end{equation}
	\begin{enumerate}
		\item Using the assumption that $\iota$ is deterministic twice, we get
			\begin{equation}
				{%
					\tikzstyle{every picture}=[tikzfig]%
					\begin{tikzpicture}
						\begin{pgfonlayer}{nodelayer}
							\node [style=bn] (16) at (-6, 0) {};
							\node [style=morphism] (17) at (-6, -1) {$\iota$};
							\node [style=none] (18) at (-7, 1) {};
							\node [style=none] (19) at (-7, 3.5) {};
							\node [style=none] (20) at (-5, 3.5) {};
							\node [style=none] (21) at (-5, 1) {};
							\node [style=none] (22) at (-6, -3.25) {};
							\node [style=morphism] (24) at (-6, -2.25) {$\pi$};
							\node [style=morphism] (25) at (-7, 2.5) {$\iota$};
							\node [style=morphism] (26) at (-7, 1.25) {$\pi$};
							\node [style=none] (27) at (-3, 0) {$=$};
							\node [style=bn] (28) at (0, -1.25) {};
							\node [style=morphism] (29) at (1, 0) {$\iota$};
							\node [style=none] (30) at (-1, -0.25) {};
							\node [style=none] (31) at (-1, 3.5) {};
							\node [style=none] (32) at (1, 3.5) {};
							\node [style=none] (33) at (1, -0.25) {};
							\node [style=none] (34) at (0, -3.25) {};
							\node [style=morphism] (35) at (0, -2.25) {$\pi$};
							\node [style=morphism] (36) at (-1, 2.5) {$\iota$};
							\node [style=bn] (38) at (6, -0.25) {};
							\node [style=none] (40) at (5, 1) {};
							\node [style=none] (41) at (5, 3.5) {};
							\node [style=none] (42) at (7, 3.5) {};
							\node [style=none] (43) at (7, 1) {};
							\node [style=none] (44) at (6, -3.25) {};
							\node [style=morphism] (45) at (6, -2) {$\pi$};
							\node [style=morphism] (47) at (5, 1.5) {$\iota$};
							\node [style=none] (50) at (9, 0) {$=$};
							\node [style=bn] (51) at (12, 0.5) {};
							\node [style=none] (53) at (11, 1.75) {};
							\node [style=none] (54) at (11, 3.5) {};
							\node [style=none] (55) at (13, 3.5) {};
							\node [style=none] (56) at (13, 1.75) {};
							\node [style=none] (57) at (12, -3.25) {};
							\node [style=none] (66) at (3, 0) {$=$};
							\node [style=morphism] (67) at (-1, 1.25) {$\pi$};
							\node [style=morphism] (68) at (-1, 0) {$\iota$};
							\node [style=morphism] (69) at (7, 1.5) {$\iota$};
							\node [style=morphism] (70) at (12, -0.75) {$\iota$};
							\node [style=morphism] (71) at (12, -2) {$\pi$};
						\end{pgfonlayer}
						\begin{pgfonlayer}{edgelayer}
							\draw (22.center) to (16);
							\draw [in=-90, out=15] (16) to (21.center);
							\draw (21.center) to (20.center);
							\draw (19.center) to (18.center);
							\draw [in=165, out=-90] (18.center) to (16);
							\draw (34.center) to (28);
							\draw [in=-90, out=15] (28) to (33.center);
							\draw (33.center) to (32.center);
							\draw (31.center) to (30.center);
							\draw [in=165, out=-90] (30.center) to (28);
							\draw (44.center) to (38);
							\draw [in=-90, out=15] (38) to (43.center);
							\draw (43.center) to (42.center);
							\draw (41.center) to (40.center);
							\draw [in=165, out=-90] (40.center) to (38);
							\draw (57.center) to (51);
							\draw [in=-90, out=15] (51) to (56.center);
							\draw (56.center) to (55.center);
							\draw (54.center) to (53.center);
							\draw [in=165, out=-90] (53.center) to (51);
						\end{pgfonlayer}
					\end{tikzpicture}
				}%
			\end{equation}
			and thus obtain that $e$ is static.

			Conversely, using \cref{eq:split_positivity} and the assumption that $e$ is static yields
			\begin{equation}
				{%
					\tikzstyle{every picture}=[tikzfig]%
					\begin{tikzpicture}
						\begin{pgfonlayer}{nodelayer}
							\node [style=bn] (0) at (3, 0) {};
							\node [style=morphism] (1) at (3, -1) {$\iota$};
							\node [style=none] (2) at (2, 1) {};
							\node [style=none] (3) at (2, 3.25) {};
							\node [style=none] (4) at (4, 3.25) {};
							\node [style=none] (5) at (4, 1) {};
							\node [style=none] (6) at (3, -3) {};
							\node [style=morphism] (7) at (3, -2.25) {$\pi$};
							\node [style=morphism] (8) at (2, 2.5) {$\iota$};
							\node [style=morphism] (9) at (2, 1.25) {$\pi$};
							\node [style=none] (10) at (6, 0) {$=$};
							\node [style=bn] (11) at (9, 0.75) {};
							\node [style=morphism] (12) at (9, -0.5) {$\iota$};
							\node [style=none] (13) at (8, 2) {};
							\node [style=none] (14) at (8, 3.25) {};
							\node [style=none] (15) at (10, 3.25) {};
							\node [style=none] (16) at (10, 2) {};
							\node [style=none] (17) at (9, -3) {};
							\node [style=morphism] (18) at (9, -2) {$\pi$};
							\node [style=bn] (19) at (-3, 0) {};
							\node [style=none] (21) at (-4, 1.25) {};
							\node [style=none] (22) at (-4, 3.25) {};
							\node [style=none] (23) at (-2, 3.25) {};
							\node [style=none] (24) at (-2, 1.25) {};
							\node [style=none] (25) at (-3, -3) {};
							\node [style=morphism] (26) at (-3, -2) {$\pi$};
							\node [style=morphism] (27) at (-4, 1.75) {$\iota$};
							\node [style=none] (29) at (0, 0) {$=$};
							\node [style=morphism] (30) at (-2, 1.75) {$\iota$};
						\end{pgfonlayer}
						\begin{pgfonlayer}{edgelayer}
							\draw (6.center) to (0);
							\draw [in=-90, out=15] (0) to (5.center);
							\draw (5.center) to (4.center);
							\draw (3.center) to (2.center);
							\draw [in=165, out=-90] (2.center) to (0);
							\draw (17.center) to (11);
							\draw [in=-90, out=15] (11) to (16.center);
							\draw (16.center) to (15.center);
							\draw (14.center) to (13.center);
							\draw [in=165, out=-90] (13.center) to (11);
							\draw (25.center) to (19);
							\draw [in=-90, out=15] (19) to (24.center);
							\draw (24.center) to (23.center);
							\draw (22.center) to (21.center);
							\draw [in=165, out=-90] (21.center) to (19);
						\end{pgfonlayer}
					\end{tikzpicture}
				}%
			\end{equation}	
			and the claim follows by canceling the epimorphism $\pi$ at the bottom.  

		\item If $\pi$ is deterministic, we calculate
			\begin{equation}\label{eq:split_strong}
				{%
					\tikzstyle{every picture}=[tikzfig]%
					\begin{tikzpicture}
						\begin{pgfonlayer}{nodelayer}
							\node [style=bn] (9) at (-6, 0) {};
							\node [style=morphism] (10) at (-6, -1) {$\iota$};
							\node [style=none] (11) at (-7, 1) {};
							\node [style=none] (12) at (-7, 3.5) {};
							\node [style=none] (13) at (-5, 3.5) {};
							\node [style=none] (14) at (-5, 1) {};
							\node [style=none] (15) at (-6, -3.25) {};
							\node [style=morphism] (16) at (-6, -2.25) {$\pi$};
							\node [style=morphism] (17) at (-7, 2.5) {$\iota$};
							\node [style=morphism] (18) at (-7, 1.25) {$\pi$};
							\node [style=none] (19) at (-3, 0) {$=$};
							\node [style=bn] (29) at (0, 0) {};
							\node [style=none] (30) at (-1, 1) {};
							\node [style=none] (31) at (-1, 3.5) {};
							\node [style=none] (32) at (1, 3.5) {};
							\node [style=none] (33) at (1, 1) {};
							\node [style=none] (34) at (0, -3.25) {};
							\node [style=morphism] (35) at (0, -1.75) {$\pi$};
							\node [style=morphism] (36) at (-1, 1.5) {$\iota$};
							\node [style=none] (37) at (3, 0) {$=$};
							\node [style=bn] (38) at (6, -1) {};
							\node [style=none] (39) at (5, 0) {};
							\node [style=none] (40) at (5, 3.5) {};
							\node [style=none] (41) at (7, 3.5) {};
							\node [style=none] (42) at (7, 0) {};
							\node [style=none] (43) at (6, -3.25) {};
							\node [style=morphism] (48) at (1, 1.5) {$\iota$};
							\node [style=morphism] (49) at (5, 0.5) {$\pi$};
							\node [style=morphism] (50) at (7, 0.5) {$\pi$};
							\node [style=morphism] (51) at (7, 2) {$\iota$};
							\node [style=morphism] (52) at (5, 2) {$\iota$};
						\end{pgfonlayer}
						\begin{pgfonlayer}{edgelayer}
							\draw (15.center) to (9);
							\draw [in=-90, out=15] (9) to (14.center);
							\draw (14.center) to (13.center);
							\draw (12.center) to (11.center);
							\draw [in=165, out=-90] (11.center) to (9);
							\draw (34.center) to (29);
							\draw [in=-90, out=15] (29) to (33.center);
							\draw (33.center) to (32.center);
							\draw (31.center) to (30.center);
							\draw [in=165, out=-90] (30.center) to (29);
							\draw (43.center) to (38);
							\draw [in=-90, out=15] (38) to (42.center);
							\draw (42.center) to (41.center);
							\draw (40.center) to (39.center);
							\draw [in=165, out=-90] (39.center) to (38);
						\end{pgfonlayer}
					\end{tikzpicture}
				}%
			\end{equation}
			where the first step is by \cref{eq:split_positivity} and the second is by assumption.
			Thus, $e$ is a strong idempotent.

			Conversely, if $e$ is strong, then we similarly get
			\begin{equation}
				{%
					\tikzstyle{every picture}=[tikzfig]%
					\begin{tikzpicture}
						\begin{pgfonlayer}{nodelayer}
							\node [style=none] (7) at (6, 0) {$=$};
							\node [style=bn] (9) at (3, 0) {};
							\node [style=morphism] (10) at (3, -1) {$\iota$};
							\node [style=none] (11) at (2, 1) {};
							\node [style=none] (12) at (2, 3.5) {};
							\node [style=none] (13) at (4, 3.5) {};
							\node [style=none] (14) at (4, 1) {};
							\node [style=none] (15) at (3, -3.25) {};
							\node [style=morphism] (16) at (3, -2.25) {$\pi$};
							\node [style=morphism] (17) at (2, 2.5) {$\iota$};
							\node [style=morphism] (18) at (2, 1.25) {$\pi$};
							\node [style=none] (19) at (0, 0) {$=$};
							\node [style=bn] (20) at (-3, 0) {};
							\node [style=none] (21) at (-4, 1.25) {};
							\node [style=none] (22) at (-4, 3.5) {};
							\node [style=none] (23) at (-2, 3.5) {};
							\node [style=none] (24) at (-2, 1.25) {};
							\node [style=none] (25) at (-3, -3.25) {};
							\node [style=morphism] (26) at (-3, -1.75) {$\pi$};
							\node [style=morphism] (27) at (-4, 1.75) {$\iota$};
							\node [style=bn] (29) at (9, -1.25) {};
							\node [style=none] (30) at (8, 0) {};
							\node [style=none] (31) at (8, 3.5) {};
							\node [style=none] (32) at (10, 3.5) {};
							\node [style=none] (33) at (10, 0) {};
							\node [style=none] (34) at (9, -3.25) {};
							\node [style=morphism] (35) at (-2, 1.75) {$\iota$};
							\node [style=morphism] (36) at (8, 0.5) {$\pi$};
							\node [style=morphism] (37) at (10, 0.5) {$\pi$};
							\node [style=morphism] (38) at (10, 2) {$\iota$};
							\node [style=morphism] (39) at (8, 2) {$\iota$};
						\end{pgfonlayer}
						\begin{pgfonlayer}{edgelayer}
							\draw (15.center) to (9);
							\draw [in=-90, out=15] (9) to (14.center);
							\draw (14.center) to (13.center);
							\draw (12.center) to (11.center);
							\draw [in=165, out=-90] (11.center) to (9);
							\draw (25.center) to (20);
							\draw [in=-90, out=15] (20) to (24.center);
							\draw (24.center) to (23.center);
							\draw (22.center) to (21.center);
							\draw [in=165, out=-90] (21.center) to (20);
							\draw (34.center) to (29);
							\draw [in=-90, out=15] (29) to (33.center);
							\draw (33.center) to (32.center);
							\draw (31.center) to (30.center);
							\draw [in=165, out=-90] (30.center) to (29);
						\end{pgfonlayer}
					\end{tikzpicture}
				}%
			\end{equation}
			and the claim follows by post-composing with $\pi$ on both outputs, which cancels both the monomorphisms $\iota$ at the top.

		\item We apply \cref{eq:split_positivity} twice to obtain
			\begin{equation}
				{%
					\tikzstyle{every picture}=[tikzfig]%
					\begin{tikzpicture}
						\begin{pgfonlayer}{nodelayer}
							\node [style=bn] (16) at (-6, 0) {};
							\node [style=morphism] (17) at (-6, -1) {$\iota$};
							\node [style=none] (18) at (-7, 1) {};
							\node [style=none] (19) at (-7, 3.5) {};
							\node [style=none] (20) at (-5, 3.5) {};
							\node [style=none] (21) at (-5, 1) {};
							\node [style=none] (22) at (-6, -3.25) {};
							\node [style=morphism] (24) at (-6, -2.25) {$\pi$};
							\node [style=morphism] (25) at (-7, 2.5) {$\iota$};
							\node [style=morphism] (26) at (-7, 1.25) {$\pi$};
							\node [style=none] (27) at (-3, 0) {$=$};
							\node [style=bn] (28) at (0, 0) {};
							\node [style=morphism] (29) at (1, 1.75) {$\iota$};
							\node [style=none] (30) at (-1, 1.25) {};
							\node [style=none] (31) at (-1, 3.5) {};
							\node [style=none] (32) at (1, 3.5) {};
							\node [style=none] (33) at (1, 1.25) {};
							\node [style=none] (34) at (0, -3.25) {};
							\node [style=morphism] (35) at (0, -1.75) {$\pi$};
							\node [style=morphism] (36) at (-1, 1.75) {$\iota$};
							\node [style=none] (37) at (3, 0) {$=$};
							\node [style=bn] (38) at (6, 0) {};
							\node [style=morphism] (39) at (7, 2.5) {$\iota$};
							\node [style=none] (40) at (5, 1) {};
							\node [style=none] (41) at (5, 3.5) {};
							\node [style=none] (42) at (7, 3.5) {};
							\node [style=none] (43) at (7, 1) {};
							\node [style=none] (44) at (6, -3.25) {};
							\node [style=morphism] (45) at (6, -2.25) {$\pi$};
							\node [style=morphism] (47) at (6, -1) {$\iota$};
							\node [style=morphism] (49) at (7, 1.25) {$\pi$};
						\end{pgfonlayer}
						\begin{pgfonlayer}{edgelayer}
							\draw (22.center) to (16);
							\draw [in=-90, out=15] (16) to (21.center);
							\draw (21.center) to (20.center);
							\draw (19.center) to (18.center);
							\draw [in=165, out=-90] (18.center) to (16);
							\draw (34.center) to (28);
							\draw [in=-90, out=15] (28) to (33.center);
							\draw (33.center) to (32.center);
							\draw (31.center) to (30.center);
							\draw [in=165, out=-90] (30.center) to (28);
							\draw (44.center) to (38);
							\draw [in=-90, out=15] (38) to (43.center);
							\draw (43.center) to (42.center);
							\draw (41.center) to (40.center);
							\draw [in=165, out=-90] (40.center) to (38);
						\end{pgfonlayer}
					\end{tikzpicture}
				}%
			\end{equation}
			so that the claim follows by \cref{prop:balancednew}.
	\end{enumerate}
	In order to show that $\pi$ is $\as{\iota}$ deterministic, we use that $e$ is balanced together with \cref{prop:balancednew} to obtain
	\begin{equation}\label{eq:positivebalanced3}
		{%
			\tikzstyle{every picture}=[tikzfig]%
			\begin{tikzpicture}
				\begin{pgfonlayer}{nodelayer}
					\node [style=morphism] (2) at (8.25, -1.5) {$e$};
					\node [style=morphism] (5) at (7.25, 2) {$e$};
					\node [style=none] (6) at (7.25, 3.5) {};
					\node [style=none] (7) at (9.25, 3.5) {};
					\node [style=none] (29) at (8.25, -3) {};
					\node [style=bn] (34) at (2.75, 0) {};
					\node [style=morphism] (35) at (2.75, -1) {$\iota$};
					\node [style=none] (37) at (1.75, 3.5) {};
					\node [style=none] (38) at (3.75, 3.5) {};
					\node [style=none] (39) at (3.75, 1.25) {};
					\node [style=none] (42) at (2.75, -3) {};
					\node [style=morphism] (49) at (2.75, -2.25) {$\pi$};
					\node [style=none] (50) at (5.5, 0) {$=$};
					\node [style=bn] (51) at (-2.75, 0) {};
					\node [style=none] (53) at (-3.75, 1.25) {};
					\node [style=none] (54) at (-3.75, 3.5) {};
					\node [style=none] (55) at (-1.75, 3.5) {};
					\node [style=none] (56) at (-1.75, 1.25) {};
					\node [style=none] (59) at (-2.75, -3) {};
					\node [style=morphism] (66) at (-2.75, -1.5) {$\pi$};
					\node [style=none] (67) at (0, 0) {$=$};
					\node [style=morphism] (71) at (1.75, 2.5) {$\iota$};
					\node [style=morphism] (73) at (-1.75, 1.75) {$\iota$};
					\node [style=morphism] (74) at (-3.75, 1.75) {$\iota$};
					\node [style=none] (75) at (11.5, 0) {$\ase{e}$};
					\node [style=morphism] (78) at (13.75, 1.75) {$e$};
					\node [style=none] (79) at (13.75, 3.5) {};
					\node [style=none] (80) at (15.75, 3.5) {};
					\node [style=none] (82) at (14.75, -3) {};
					\node [style=morphism] (83) at (15.75, 1.75) {$e$};
					\node [style=none] (87) at (19.25, 3.5) {};
					\node [style=none] (88) at (21.25, 3.5) {};
					\node [style=none] (90) at (20.25, -3) {};
					\node [style=none] (92) at (17.5, 0) {$=$};
					\node [style=morphism] (95) at (19.25, 2.5) {$\iota$};
					\node [style=morphism] (96) at (21.25, 2.5) {$\iota$};
					\node [style=morphism] (97) at (21.25, 1) {$\pi$};
					\node [style=none] (98) at (1.75, 1.25) {};
					\node [style=morphism] (99) at (19.25, 1) {$\pi$};
					\node [style=morphism] (100) at (1.75, 1.25) {$\pi$};
					\node [style=bn] (101) at (8.25, 0) {};
					\node [style=none] (102) at (7.25, 1.25) {};
					\node [style=none] (103) at (9.25, 1.25) {};
					\node [style=bn] (104) at (14.75, -0.75) {};
					\node [style=none] (105) at (13.75, 0.5) {};
					\node [style=none] (106) at (15.75, 0.5) {};
					\node [style=bn] (107) at (20.25, -0.75) {};
					\node [style=none] (108) at (19.25, 0.5) {};
					\node [style=none] (109) at (21.25, 0.5) {};
				\end{pgfonlayer}
				\begin{pgfonlayer}{edgelayer}
					\draw (6.center) to (5);
					\draw [in=-90, out=15] (34) to (39.center);
					\draw (39.center) to (38.center);
					\draw (34) to (42.center);
					\draw [in=-90, out=15] (51) to (56.center);
					\draw (56.center) to (55.center);
					\draw (54.center) to (53.center);
					\draw [in=165, out=-90] (53.center) to (51);
					\draw (51) to (59.center);
					\draw (79.center) to (78);
					\draw (37.center) to (98.center);
					\draw [in=165, out=-90] (98.center) to (34);
					\draw [in=-90, out=15] (101) to (103.center);
					\draw [in=165, out=-90] (102.center) to (101);
					\draw [in=-90, out=15] (104) to (106.center);
					\draw [in=165, out=-90] (105.center) to (104);
					\draw (101) to (29.center);
					\draw (104) to (82.center);
					\draw (102.center) to (5);
					\draw (103.center) to (7.center);
					\draw (105.center) to (78);
					\draw [in=-90, out=15] (107) to (109.center);
					\draw [in=165, out=-90] (108.center) to (107);
					\draw (107) to (90.center);
					\draw (109.center) to (88.center);
					\draw (108.center) to (87.center);
					\draw (106.center) to (80.center);
				\end{pgfonlayer}
			\end{tikzpicture}
		}%
	\end{equation}
	Post-composing with $\pi \otimes \pi$ cancels the two $\iota$ morphisms on top, and therefore $\pi$ is $\as{e}$ deterministic. 
	Since we have $\iota = e\comp \iota$ by assumption and $e\comp \iota \ll e$ by \cref{lem:factor_ac}, we conclude that $\pi$ is also $\as{\iota}$ deterministic.
\end{proof}

\begin{example}
	\label{ex:setmulti_not_split}
	Since the idempotents of \cref{ex:setmulti_not_balanced} are not balanced although $\setmulti$ is positive, it follows that these idempotents do not split.
	Note that the first one of these already lives in $\finsetmulti$.
\end{example}

\begin{example}
	In any Markov category, a morphism of the form
	\begin{equation}
		{%
			\tikzstyle{every picture}=[tikzfig]%
			\begin{tikzpicture}
				\begin{pgfonlayer}{nodelayer}
					\node [style=bn] (0) at (-0.5, 0.25) {};
					\node [style=none] (1) at (-1.5, 1.5) {};
					\node [style=none] (2) at (-1.5, 2.75) {};
					\node [style=none] (3) at (0.5, 2.75) {};
					\node [style=none] (4) at (0.5, 1.5) {};
					\node [style=none] (5) at (-0.5, -1.5) {};
					\node [style=morphism] (6) at (0.5, 1.75) {$f$};
					\node [style=none] (7) at (-1.5, 3.25) {$X$};
					\node [style=none] (8) at (0.5, 3.25) {$Y$};
					\node [style=none] (9) at (-0.5, -2) {$X$};
					\node [style=none] (10) at (0.75, -2) {$Y$};
					\node [style=none] (11) at (0.75, -1.5) {};
					\node [style=bn] (12) at (0.75, -0.75) {};
					\node [style=none] (13) at (-3, -0.25) {};
					\node [style=none] (14) at (2, -0.25) {};
					\node [style=none] (15) at (2, -0.25) {};
				\end{pgfonlayer}
				\begin{pgfonlayer}{edgelayer}
					\draw [in=-90, out=15] (0) to (4.center);
					\draw (4.center) to (3.center);
					\draw (2.center) to (1.center);
					\draw [in=165, out=-90] (1.center) to (0);
					\draw [style=protected] (11.center) to (12);
					\draw [style=dashed box] (15.center) to (13.center);
					\draw [style=protected] (5.center) to (0);
				\end{pgfonlayer}
			\end{tikzpicture}
		}%
	\end{equation}
	is a strong idempotent by direct calculations. Under positivity, this idempotent is immediately strong because it is defined in terms of a splitting across $X$ (dashed line) with a deterministic projection given by $\id_X \otimes \discard_Y$.
\end{example}

\begin{example}\label{ex:resamp}
	Let $\cC$ be a representable positive Markov category. 
	Then the composite morphism ${\delta \comp \samp \colon PX \to PX}$ is a static idempotent for every object $X$ because $\delta$ is deterministic and $\samp$ is its left inverse. 
\end{example}

\begin{example}\label{ex:supp_static}
	Let $\cC$ be a positive\footnotemark{} Markov category and $p$ a morphism with a split support.
	\footnotetext{The attentive reader may notice that the positivity assumption is redundant here by \cref{fn:without_positive}.}%
	Then the composition of its support projection and injection
	\begin{equation}
		e \coloneqq \suppinc{p} \comp \suppproj{p}
	\end{equation}
	is a static idempotent, since $\suppproj{p} \comp \suppinc{p} = \id_{\Supp{}}$ and $\suppinc{p}$ is deterministic.
\end{example}

\subsection{Supports and splittings of static idempotents}
\label{sec:proj_idempotent_support}

We now focus further on the case of static idempotents and show that for them, splitting is closely related to supports.
In the following subsection, we will put these results to use in order to construct splittings of more general idempotents under suitable additional assumptions.

\begin{proposition}[Support and splitting of static idempotents]
	\label{prop:proj_idempotent_support}
	Let $e$ be a static idempotent in a positive Markov category. 
	Then the following are equivalent:
	\begin{enumerate}
		\item\label{it:proj_splits} $e$ is a split idempotent.
		\item\label{it:proj_split_support} $e$ has a split support.
		\item\label{it:proj_support} $e$ has a support.
	\end{enumerate}
	Moreover, this equivalence is such that the split support gives the splitting via $e = \suppinc{e}\comp \suppproj{e}$.
\end{proposition}

\begin{proof}
	We first show that \ref{it:proj_splits} implies \ref{it:proj_split_support}.
	Suppose that $e$ splits as $e=\iota\comp \pi$ for $\pi \colon X \to T$ and $\iota \colon T \to X$ satisfying $\pi\comp \iota = \id_T$. 
	We know from \cref{thm:split_props} that $\iota$ is a deterministic split monomorphism.
	Furthermore, \cref{lem:factor_ac} gives $e \gg e \comp \iota = \iota$ and $\iota \gg \iota \comp \pi = e$, so $e$ has a split support by \cref{cor:split_supp_mono}.

	The implication from \ref{it:proj_split_support} to \ref{it:proj_support} is precisely the statement of \cref{lem:split_supp_is_supp}.

	Finally, assume that the static idempotent $e$ has a support with inclusion $\suppinc{} \colon \Supp{} \to X$ and let us prove that $e$ splits.
	Since $e \ase{e} \id$ holds by definition of static idempotents, we obtain $e\comp\suppinc{} = \suppinc{}$ by \cref{lem:supp_asfaithful}.
	The support factorization of \cref{eq:supp_inc_factor} reads as $e = \suppinc{}\comp \suppfactor{e}$, and therefore also
	\begin{equation}
		\suppinc{}\comp \suppfactor{e} \comp \suppinc{} = e\comp \suppinc{} = \suppinc{}
	\end{equation}
	holds.
	Since $\suppinc{}$ is a monomorphism, we get $\suppfactor{e} \comp \suppinc{}= \id$, so that $e = \suppinc{} \comp \suppfactor{e}$ is indeed a split idempotent.
\end{proof}

\begin{example}
	If a morphism $p$ has a split support, then $\suppinc{p} \comp \suppproj{p}$ is a static idempotent by \cref{ex:supp_static}.
	\cref{prop:proj_idempotent_support} now allows us to conclude that this idempotent has the same split support as $p$ itself.
\end{example}

\begin{corollary}\label{cor:char_suppbalanced}
	Let $\cC$ be a positive Markov category and let $e \colon X \to X$ be an arbitrary morphism. 
	Then the following are equivalent:
	\begin{enumerate}
		\item $e$ is a static idempotent that splits.
		\item We have $e = \suppinc{p}\comp \suppproj{p}$ for some morphism $p$ with a split support.
	\end{enumerate}	
\end{corollary}
\begin{proof}
	If $e$ is a static idempotent that splits, then by \cref{prop:proj_idempotent_support} we have $e = \suppinc{e} \comp \suppproj{e}$.
	The other direction is \cref{ex:supp_static}.
\end{proof}

We can apply this to obtain another characterization of split supports similar to \cref{lem:supp_factor}, but where the second condition only requires one direction of implication.

\begin{theorem}[Split supports via static idempotents]\label{thm:proj_idemp_supp}
	Let $\cC$ be a positive Markov category and $p \colon A \to X$ a morphism in $\cC$. 
	Then the following are equivalent: 
	\begin{enumerate}
		\item\label{it:p_split_supp} $p$ admits a split support.
		\item\label{it:p_ase_crit} There exists a split static idempotent $e \colon X \to X$ satisfying $e \comp p = p$, and such that for all $f, g \colon W \otimes X \to Y$ with arbitrary $W$ and $Y$, we have
			\begin{equation}\label{eq:proj_idemp_supp}
				{%
					\tikzstyle{every picture}=[tikzfig]%
					\begin{tikzpicture}
						\begin{pgfonlayer}{nodelayer}
							\node [style=none] (0) at (0, 0) {$\Longrightarrow$};
							\node [style=none] (1) at (-10, 0.75) {};
							\node [style=none] (2) at (-8.5, 1) {};
							\node [style=none] (3) at (-10.5, 2.25) {};
							\node [style=none] (4) at (-10.5, 2.75) {$Y$};
							\node [style=bn] (5) at (-9.25, 0) {};
							\node [style=none] (6) at (-9.25, -2) {};
							\node [style=none] (7) at (-9.25, -2.5) {$A$};
							\node [style=none] (8) at (-7, 0) {$=$};
							\node [style=none] (9) at (-8.5, 2.75) {$X$};
							\node [style=morphism] (10) at (-10.5, 1.25) {$\,\;\; f \,\;\;$};
							\node [style=none] (11) at (-8.5, 2.25) {};
							\node [style=morphism] (12) at (-9.25, -1) {$p$};
							\node [style=none] (13) at (-11, -2) {};
							\node [style=none] (14) at (-11, -2.5) {$W$};
							\node [style=none] (15) at (-11, 1.25) {};
							\node [style=none] (16) at (-11, -0.5) {};
							\node [style=none] (17) at (-4.5, 0.75) {};
							\node [style=none] (18) at (-3, 1) {};
							\node [style=none] (19) at (-5, 2.25) {};
							\node [style=none] (20) at (-5, 2.75) {$Y$};
							\node [style=bn] (21) at (-3.75, 0) {};
							\node [style=none] (22) at (-3.75, -2) {};
							\node [style=none] (23) at (-3.75, -2.5) {$A$};
							\node [style=none] (24) at (-3, 2.75) {$X$};
							\node [style=morphism] (25) at (-5, 1.25) {$\,\;\; g \,\;\;$};
							\node [style=none] (26) at (-3, 2.25) {};
							\node [style=morphism] (27) at (-3.75, -1) {$p$};
							\node [style=none] (28) at (-5.5, -2) {};
							\node [style=none] (29) at (-5.5, -2.5) {$W$};
							\node [style=none] (30) at (-5.5, 1.25) {};
							\node [style=none] (31) at (-5.5, -0.5) {};
							\node [style=none] (32) at (-10, 1.25) {};
							\node [style=none] (33) at (-4.5, 1.25) {};
							\node [style=none] (34) at (3.5, 2.25) {};
							\node [style=none] (35) at (3.5, 2.75) {$Y$};
							\node [style=none] (36) at (4, -2.5) {$X$};
							\node [style=morphism] (37) at (3.5, 1) {$\,\;\; f \,\;\;$};
							\node [style=none] (38) at (3, -2) {};
							\node [style=none] (39) at (3, -2.5) {$W$};
							\node [style=none] (40) at (3, 1) {};
							\node [style=none] (41) at (9, 2.25) {};
							\node [style=none] (42) at (9, 2.75) {$Y$};
							\node [style=none] (43) at (9.5, -2.5) {$X$};
							\node [style=morphism] (44) at (9, 1) {$\,\;\; g \,\;\;$};
							\node [style=none] (45) at (8.5, -2) {};
							\node [style=none] (46) at (8.5, -2.5) {$W$};
							\node [style=none] (47) at (8.5, 1) {};
							\node [style=none] (48) at (4, 1) {};
							\node [style=none] (49) at (9.5, 1) {};
							\node [style=morphism] (50) at (4, -0.75) {$e$};
							\node [style=morphism] (51) at (9.5, -0.75) {$e$};
							\node [style=none] (52) at (4, -2) {};
							\node [style=none] (53) at (9.5, -2) {};
							\node [style=none] (54) at (6.25, 0) {$=$};
						\end{pgfonlayer}
						\begin{pgfonlayer}{edgelayer}
							\draw [in=-90, out=165] (5) to (1.center);
							\draw [in=-90, out=15] (5) to (2.center);
							\draw (10) to (3.center);
							\draw (2.center) to (11.center);
							\draw (6.center) to (12);
							\draw [style=protected] (13.center) to (16.center);
							\draw [style=protected, in=-90, out=90, looseness=1.25] (16.center) to (15.center);
							\draw [in=270, out=90] (12) to (5);
							\draw [in=-90, out=165] (21) to (17.center);
							\draw [in=-90, out=15] (21) to (18.center);
							\draw (25) to (19.center);
							\draw (18.center) to (26.center);
							\draw (22.center) to (27);
							\draw [style=protected] (28.center) to (31.center);
							\draw [style=protected, in=-90, out=90, looseness=1.25] (31.center) to (30.center);
							\draw (27) to (21);
							\draw (1.center) to (32.center);
							\draw (17.center) to (33.center);
							\draw (37) to (34.center);
							\draw (44) to (41.center);
							\draw (52.center) to (50);
							\draw (50) to (48.center);
							\draw (38.center) to (40.center);
							\draw (45.center) to (47.center);
							\draw (51) to (49.center);
							\draw (53.center) to (51);
						\end{pgfonlayer}
					\end{tikzpicture}
				}%
			\end{equation}
	\end{enumerate}
\end{theorem}

\begin{proof}
	Assuming \ref{it:p_split_supp}, take $e \coloneqq \suppinc{p} \comp \suppproj{p}$, which is a split static idempotent as in \cref{ex:supp_static}. 
	Then $e \comp p = p$ holds by the second equality in \eqref{eq:supp_proj_id} and Implication \eqref{eq:proj_idemp_supp} is a direct consequence of \cref{lem:supp_factor}. 

	Conversely, suppose that \ref{it:p_ase_crit} holds with a splitting given by $e = \iota \comp \pi$. 
	By \cref{lem:epi_ac_max}, we have the absolute bicontinuity relation $e = \iota \comp \pi \acsim \iota$.
	So if we can show that its left-hand side $e$ is absolutely bicontinuous to $p$, then $p$ has a split support by \cref{cor:split_supp_mono}, since $\iota$ is indeed a deterministic split monomorphism (see \cref{thm:split_props}).

	To this end, note that $e \comp p \ll e$ holds by \cref{lem:factor_ac}, so that we get $p \ll e$ by the assumed $e \comp p = p$.
	To prove absolute continuity in the other direction, suppose that we have $f \ase{p} g$ and let us show that necessarily also $f \ase{e} g$ holds.
	Using Implication \eqref{eq:proj_idemp_supp} and the fact that $\id_W \otimes \pi$ is an epimorphism, we get
	\begin{equation}\label{eq:proj_idemp_supp_2}
		{%
			\tikzstyle{every picture}=[tikzfig]%
			\begin{tikzpicture}
				\begin{pgfonlayer}{nodelayer}
					\node [style=none] (34) at (-2.75, 1.75) {};
					\node [style=none] (35) at (-2.75, 2.25) {$Y$};
					\node [style=none] (36) at (-2.25, -2.25) {$T$};
					\node [style=morphism] (37) at (-2.75, 0.75) {$\,\;\; f \,\;\;$};
					\node [style=none] (38) at (-3.25, -1.75) {};
					\node [style=none] (39) at (-3.25, -2.25) {$W$};
					\node [style=none] (40) at (-3.25, 0.75) {};
					\node [style=none] (41) at (2.75, 1.75) {};
					\node [style=none] (42) at (2.75, 2.25) {$Y$};
					\node [style=none] (43) at (3.25, -2.25) {$T$};
					\node [style=morphism] (44) at (2.75, 0.75) {$\,\;\; g \,\;\;$};
					\node [style=none] (45) at (2.25, -1.75) {};
					\node [style=none] (46) at (2.25, -2.25) {$W$};
					\node [style=none] (47) at (2.25, 0.75) {};
					\node [style=none] (48) at (-2.25, 0.75) {};
					\node [style=none] (49) at (3.25, 0.75) {};
					\node [style=morphism] (50) at (-2.25, -0.75) {$\iota$};
					\node [style=morphism] (51) at (3.25, -0.75) {$\iota$};
					\node [style=none] (52) at (-2.25, -1.75) {};
					\node [style=none] (53) at (3.25, -1.75) {};
					\node [style=none] (54) at (0, 0) {$=$};
				\end{pgfonlayer}
				\begin{pgfonlayer}{edgelayer}
					\draw (37) to (34.center);
					\draw (44) to (41.center);
					\draw (52.center) to (50);
					\draw (50) to (48.center);
					\draw (38.center) to (40.center);
					\draw (45.center) to (47.center);
					\draw (51) to (49.center);
					\draw (53.center) to (51);
				\end{pgfonlayer}
			\end{tikzpicture}
		}%
	\end{equation}
	By the second part of \cref{lem:ase_detsplitmono}, \cref{eq:proj_idemp_supp_2} is equivalent to $f \ase{\iota} g$, which directly implies the desired $f \ase{e} g$ by pre-composition with $\pi$.
	In conclusion, we thus have both $p \ll e$ and $e \ll p$, so that $p$ has a split support by the above argument.
\end{proof}

\subsection{General conditions for the splitting of idempotents}
\label{sec:route_split}

We now turn to the question of when \emph{general} idempotents split in a Markov category, and we will derive sufficient conditions in order for this to be the case.
The resulting \cref{thm:balanced_split} is the main result of this section.
It implies in particular that all idempotents in $\borelstoch$ split (\cref{cor:borelstoch_idempotents_split}).
As far as we know, this is an entirely new result which has not been proven by measure-theoretic methods before, and we note in \cref{thm:idemp_borel} that it strengthens an existing result of Blackwell on idempotent Markov kernels.

To this end, we assume familiarity with the \emph{equalizer principle} from \cref{sec:equalizer_principle} as well as with the notion of \emph{observationally representable} Markov category from \cref{sec:observational}.
Roughly speaking, the latter postulates that one can distinguish between morphisms by iterated sampling.

\begin{lemma}\label{lem:balanced_sharp}
	Let $e$ be a balanced idempotent in an observationally representable Markov category $\cC$. 
	Then we have $e^{\sharp} \comp e \ase{e} e^{\sharp}$.
\end{lemma}
\begin{proof}
	By the definition of observationally representable Markov categories and \cref{prop:reflect_aseq}, it suffices to show the desired equation for $n$ independent samples of both sides, i.e.\ post-composed with $\samp^{(n)}$ defined in \eqref{eq:samp_N} for every natural number $n$.
	Since $e^{\sharp}$ is deterministic and satisfies $\samp \comp e^\sharp = e$, we only need to show that
	\begin{equation}
		{%
			\tikzstyle{every picture}=[tikzfig]%
			\begin{tikzpicture}
				\begin{pgfonlayer}{nodelayer}
					\node [style=bn] (0) at (-4, 0) {};
					\node [style=morphism] (1) at (-4, -1) {$e$};
					\node [style=none] (4) at (-4, -2) {};
					\node [style=none] (5) at (-5.5, 2.25) {};
					\node [style=none] (6) at (-2.5, 2.25) {};
					\node [style=none] (7) at (-4, 1.25) {$\cdots$};
					\node [style=none] (8) at (0, 0) {$\ase{e}$};
					\node [style=bn] (9) at (4, 0) {};
					\node [style=none] (13) at (4, -2) {};
					\node [style=none] (14) at (2.5, 2.25) {};
					\node [style=none] (15) at (5.5, 2.25) {};
					\node [style=none] (16) at (4, 1.25) {$\cdots$};
					\node [style=none] (17) at (2.5, 1.25) {};
					\node [style=none] (18) at (5.5, 1.25) {};
					\node [style=none] (19) at (-5.5, 1.25) {};
					\node [style=none] (20) at (-2.5, 1.25) {};
					\node [style=morphism] (21) at (-5.5, 1.25) {$e$};
					\node [style=morphism] (22) at (-2.5, 1.25) {$e$};
					\node [style=morphism] (23) at (2.5, 1.25) {$e$};
					\node [style=morphism] (24) at (5.5, 1.25) {$e$};
					\node [style=none] (25) at (-4, 3.25) {$n$ outputs};
					\node [style=none] (26) at (-6, 2.5) {};
					\node [style=none] (27) at (-2, 2.5) {};
					\node [style=none] (28) at (4, 3.25) {$n$ outputs};
					\node [style=none] (29) at (2, 2.5) {};
					\node [style=none] (30) at (6, 2.5) {};
				\end{pgfonlayer}
				\begin{pgfonlayer}{edgelayer}
					\draw (0) to (4.center);
					\draw (9) to (13.center);
					\draw (5.center) to (19.center);
					\draw (6.center) to (20.center);
					\draw (14.center) to (17.center);
					\draw (15.center) to (18.center);
					\draw [in=165, out=-90] (19.center) to (0);
					\draw [in=-90, out=15] (0) to (20.center);
					\draw [in=165, out=-90] (17.center) to (9);
					\draw [in=15, out=-90] (18.center) to (9);
					\draw [style=curly brace] (26.center) to (27.center);
					\draw [style=curly brace] (29.center) to (30.center);
				\end{pgfonlayer}
			\end{tikzpicture}
		}%
	\end{equation}
	holds for any number of outputs $n$.
	This can be proven by induction.
	The case $n=1$ just expresses that $e$ is idempotent. 
	For the induction step, we compute 
	\begin{equation*}
		{%
			\tikzstyle{every picture}=[tikzfig]%
			\begin{tikzpicture}
				\begin{pgfonlayer}{nodelayer}
					\node [style=bn] (10) at (-4.25, -0.5) {};
					\node [style=morphism] (12) at (-5.75, 2) {$e$};
					\node [style=morphism] (13) at (-3.5, 2) {$e$};
					\node [style=none] (14) at (-4.25, -2.5) {};
					\node [style=none] (15) at (-5.75, 3) {};
					\node [style=none] (16) at (-3.5, 3) {};
					\node [style=none] (17) at (-4.5, 2) {$\cdots$};
					\node [style=morphism] (18) at (-2, 2) {$e$};
					\node [style=none] (19) at (-2, 3) {};
					\node [style=bn] (20) at (-2.75, 1) {};
					\node [style=bn] (21) at (3.25, -0.5) {};
					\node [style=morphism] (23) at (1.75, 1) {$e$};
					\node [style=morphism] (24) at (4, 3) {$e$};
					\node [style=none] (25) at (3.25, -2.5) {};
					\node [style=none] (26) at (1.75, 4) {};
					\node [style=none] (27) at (4, 4) {};
					\node [style=none] (28) at (3.25, 1) {$\cdots$};
					\node [style=none] (30) at (5.5, 3) {};
					\node [style=bn] (31) at (4.75, 2) {};
					\node [style=morphism] (32) at (4.75, 1) {$e$};
					\node [style=none] (34) at (0, 0) {$=$};
					\node [style=none] (35) at (-4, 4) {$n+1$ outputs};
					\node [style=none] (36) at (-6.25, 3.25) {};
					\node [style=none] (37) at (-1.5, 3.25) {};
					\node [style=morphism] (51) at (-4.25, -1.5) {$e$};
					\node [style=none] (52) at (-5.75, 1) {};
					\node [style=morphism] (53) at (3.25, -1.5) {$e$};
					\node [style=none] (54) at (1.75, 1) {};
					\node [style=none] (55) at (4.75, 1) {};
					\node [style=bn] (56) at (10.5, -0.5) {};
					\node [style=morphism] (57) at (9, 1) {$e$};
					\node [style=morphism] (58) at (11.25, 3) {$e$};
					\node [style=none] (59) at (10.5, -2.5) {};
					\node [style=none] (60) at (9, 4) {};
					\node [style=none] (61) at (11.25, 4) {};
					\node [style=none] (62) at (10.5, 1) {$\cdots$};
					\node [style=none] (63) at (12.75, 3) {};
					\node [style=bn] (64) at (12, 2) {};
					\node [style=morphism] (65) at (12, 1) {$e$};
					\node [style=none] (66) at (7, 0) {$\ase{e}$};
					\node [style=none] (68) at (9, 1) {};
					\node [style=none] (69) at (12, 1) {};
					\node [style=bn] (70) at (18, -1.5) {};
					\node [style=morphism] (71) at (16.5, 1.5) {$e$};
					\node [style=morphism] (72) at (18.75, 1.5) {$e$};
					\node [style=none] (73) at (18, -2.5) {};
					\node [style=none] (74) at (16.5, 3) {};
					\node [style=none] (75) at (18.75, 3) {};
					\node [style=none] (76) at (17.75, 1.5) {$\cdots$};
					\node [style=morphism] (77) at (20.25, 1.5) {$e$};
					\node [style=none] (78) at (20.25, 3) {};
					\node [style=bn] (79) at (19.5, 0) {};
					\node [style=none] (80) at (18.25, 4) {$n+1$ outputs};
					\node [style=none] (81) at (16, 3.25) {};
					\node [style=none] (82) at (20.75, 3.25) {};
					\node [style=none] (84) at (16.5, 0) {};
					\node [style=none] (85) at (14.25, 0) {$\ase{e}$};
					\node [style=none] (86) at (5.5, 4) {};
					\node [style=none] (87) at (12.75, 4) {};
					\node [style=none] (88) at (18.75, 1) {};
					\node [style=none] (89) at (20.25, 1) {};
				\end{pgfonlayer}
				\begin{pgfonlayer}{edgelayer}
					\draw (16.center) to (13);
					\draw (12) to (15.center);
					\draw (10) to (14.center);
					\draw (18) to (19.center);
					\draw [in=-90, out=15] (10) to (20);
					\draw [in=165, out=-90] (13) to (20);
					\draw [in=-90, out=15] (20) to (18);
					\draw (27.center) to (24);
					\draw (23) to (26.center);
					\draw (21) to (25.center);
					\draw [in=165, out=-90] (24) to (31);
					\draw [in=-90, out=15] (31) to (30.center);
					\draw (32) to (31);
					\draw [style=curly brace] (36.center) to (37.center);
					\draw [in=-90, out=165] (10) to (52.center);
					\draw [in=-90, out=165] (21) to (54.center);
					\draw [in=-90, out=15] (21) to (55.center);
					\draw (61.center) to (58);
					\draw (57) to (60.center);
					\draw (56) to (59.center);
					\draw [in=165, out=-90] (58) to (64);
					\draw [in=-90, out=15] (64) to (63.center);
					\draw (65) to (64);
					\draw [in=-90, out=165] (56) to (68.center);
					\draw [in=-90, out=15] (56) to (69.center);
					\draw (75.center) to (72);
					\draw (71) to (74.center);
					\draw (70) to (73.center);
					\draw (77) to (78.center);
					\draw [in=-90, out=15] (70) to (79);
					\draw [style=curly brace] (81.center) to (82.center);
					\draw [in=-90, out=165] (70) to (84.center);
					\draw (84.center) to (71);
					\draw (52.center) to (12);
					\draw (30.center) to (86.center);
					\draw (63.center) to (87.center);
					\draw [in=-90, out=165] (79) to (88.center);
					\draw (88.center) to (72);
					\draw [in=-90, out=15] (79) to (89.center);
					\draw (89.center) to (77);
				\end{pgfonlayer}
			\end{tikzpicture}
		}%
	\end{equation*} 
	where the second equation uses the induction assumption and the other steps follow from the characterization of balanced idempotents in \cref{prop:balancednew}~\ref{it:as_strong}.
\end{proof}

Splittings of balanced idempotents can be constructed as in the following lemma, which plays a crucial role in the proof of the forthcoming \cref{thm:balanced_split}.

\begin{lemma}
	\label{lem:proj_balanced_split}
	Let $\cC$ be a positive and observationally representable Markov category. 
	If every static idempotent in $\cC$ splits, then also every balanced idempotent in $\cC$ splits.
\end{lemma}
\begin{proof}
	Consider an arbitrary balanced idempotent $e \colon X \to X$. 
	Applying the distribution functor to it gives $Pe \colon PX \to PX$, which is also an idempotent by the functoriality of $P$ and idempotency of $e$:
	\begin{equation}
		(Pe) \comp (Pe) = P(e \comp e) = Pe.
	\end{equation}
	In particular, since $Pe$ is a deterministic morphism by definition, it is both a static and a strong idempotent.
	By assumption, we thus have a splitting $Pe = \iota \comp \pi$ with $\pi \comp \iota = \id_{T}$.
	Since $P e$ is static and strong, $\iota \colon T \to PX$ and $\pi \colon PX \to T$ are both deterministic by \cref{thm:split_props}.

	Let us now prove the following auxiliary equations:
	\begin{align}
		\label{eq:bal_split_1} 	\samp \comp \iota \comp \pi \comp \delta &= e, \\
		\label{eq:bal_split_2}  \samp \comp \iota &= e \comp \samp \comp \iota, \\
		\label{eq:bal_split_3} 	\pi \comp \delta &= \pi \comp e^\sharp .
	\end{align}
	The first one can be verified as
	\begin{equation}
		\samp \comp \iota \comp \pi \comp \delta = \samp \comp (Pe) \comp \delta = e \comp \samp \comp \delta = e,
	\end{equation}
	where the first equation is by definition of $\iota$ and $\pi$, the second by naturality of $\samp$ and the third by $\samp \comp \delta = \id_X$.
	\cref{eq:bal_split_2} follows via
	\begin{equation}
		\samp \comp \iota = \samp \comp (Pe) \comp \iota = e \comp \samp \comp \iota,
	\end{equation}
	where we first use $\iota = (Pe) \comp \iota$ and then the naturality of $\samp$.
	Lastly, \cref{eq:bal_split_3} can be derived similarly via
	\begin{equation}
		\pi \comp \delta = \pi \comp (P e) \comp \delta = \pi \comp e^\sharp,
	\end{equation}
	where we use $\pi = \pi \comp (P e)$ and then $(P e) \comp \delta = e^\sharp$.

	We now employ these facts to show that the composite morphism
	\begin{equation}
		\begin{tikzcd}
			{T} & PX & X & PX & {T}
			\arrow["\iota", from=1-1, to=1-2]
			\arrow["\samp", from=1-2, to=1-3]
			\arrow["\delta", from=1-3, to=1-4]
			\arrow["\pi", from=1-4, to=1-5]
		\end{tikzcd}
	\end{equation}
	is also a static idempotent.
	To this end, we compute
	\begin{equation}\label{eq:balanced_splitting0}
		{%
			\tikzstyle{every picture}=[tikzfig]%
			\begin{tikzpicture}
				\begin{pgfonlayer}{nodelayer}
					\node [style=bn] (0) at (-3, 0) {};
					\node [style=none] (6) at (-4, 6) {};
					\node [style=none] (7) at (-2, 6) {};
					\node [style=none] (9) at (-3, -5.5) {};
					\node [style=morphism] (42) at (-4, 5) {$\pi$};
					\node [style=morphism] (43) at (-4, 3.75) {$\delta$};
					\node [style=morphism] (44) at (-4, 2.5) {$\samp$};
					\node [style=morphism] (50) at (-3, -4.75) {$\iota$};
					\node [style=morphism] (51) at (-3, -1) {$\pi$};
					\node [style=morphism] (52) at (-3, -2.25) {$\delta$};
					\node [style=morphism] (53) at (-3, -3.5) {$\samp$};
					\node [style=none] (70) at (0, 0) {$=$};
					\node [style=bn] (89) at (3, 0) {};
					\node [style=none] (90) at (2, 1.25) {};
					\node [style=none] (91) at (2, 6) {};
					\node [style=none] (92) at (4, 6) {};
					\node [style=none] (93) at (4, 1.25) {};
					\node [style=none] (94) at (3, -5.5) {};
					\node [style=morphism] (95) at (2, 5) {$\pi$};
					\node [style=morphism] (96) at (2, 3.75) {$\delta$};
					\node [style=morphism] (98) at (4, 3) {$\pi$};
					\node [style=morphism] (99) at (4, 1.5) {$\delta$};
					\node [style=morphism] (101) at (3, -4.5) {$\iota$};
					\node [style=morphism] (104) at (3, -3) {$\samp$};
					\node [style=morphism] (105) at (2, 2) {$e$};
					\node [style=none] (107) at (6, 0) {$=$};
					\node [style=bn] (108) at (9, 0) {};
					\node [style=none] (109) at (8, 1.25) {};
					\node [style=none] (110) at (8, 6) {};
					\node [style=none] (111) at (10, 6) {};
					\node [style=none] (112) at (10, 1.25) {};
					\node [style=none] (113) at (9, -5.5) {};
					\node [style=morphism] (114) at (8, 5) {$\pi$};
					\node [style=morphism] (115) at (8, 3.75) {$e^\sharp$};
					\node [style=morphism] (116) at (10, 3) {$\pi$};
					\node [style=morphism] (117) at (10, 1.5) {$\delta$};
					\node [style=morphism] (118) at (9, -4.5) {$\iota$};
					\node [style=morphism] (119) at (9, -3) {$\samp$};
					\node [style=morphism] (120) at (9, -1.5) {$e$};
					\node [style=morphism] (121) at (8, 2) {$e$};
					\node [style=none] (122) at (9, -5.5) {};
					\node [style=none] (123) at (-2, 1.25) {};
					\node [style=none] (124) at (-4, 1.25) {};
					\node [style=morphism] (125) at (-4, 1.25) {$\iota$};
					\node [style=none] (127) at (-4, 1.25) {};
				\end{pgfonlayer}
				\begin{pgfonlayer}{edgelayer}
					\draw (9.center) to (0);
					\draw (94.center) to (89);
					\draw [in=-90, out=15] (89) to (93.center);
					\draw (93.center) to (92.center);
					\draw (91.center) to (90.center);
					\draw [in=165, out=-90] (90.center) to (89);
					\draw (113.center) to (108);
					\draw [in=-90, out=15] (108) to (112.center);
					\draw (112.center) to (111.center);
					\draw (110.center) to (109.center);
					\draw [in=165, out=-90] (109.center) to (108);
					\draw (6.center) to (124.center);
					\draw (7.center) to (123.center);
					\draw [in=-90, out=15] (0) to (123.center);
					\draw [in=-90, out=165] (0) to (127.center);
				\end{pgfonlayer}
			\end{tikzpicture}
		}%
	\end{equation}
	where the first equality uses that both $\pi$ and $\delta$ are deterministic as well as \cref{eq:bal_split_1}, and the second equality uses \cref{eq:bal_split_2,eq:bal_split_3}.
	We can further reduce the right-hand side to
	\begin{equation}
		{%
			\tikzstyle{every picture}=[tikzfig]%
			\begin{tikzpicture}
				\begin{pgfonlayer}{nodelayer}
					\node [style=none] (125) at (0, 0) {$=$};
					\node [style=bn] (170) at (-3, 0.5) {};
					\node [style=none] (171) at (-4, 1.5) {};
					\node [style=none] (172) at (-4, 4.25) {};
					\node [style=none] (173) at (-2, 4.25) {};
					\node [style=none] (174) at (-2, 1.5) {};
					\node [style=none] (175) at (-3, -4) {};
					\node [style=morphism] (176) at (-4, 3) {$\pi$};
					\node [style=morphism] (178) at (-2, 3) {$\pi$};
					\node [style=morphism] (179) at (-2, 1.75) {$\delta$};
					\node [style=morphism] (181) at (-3, -1.75) {$\samp$};
					\node [style=morphism] (183) at (-3, -0.5) {$e$};
					\node [style=morphism] (185) at (-4, 1.75) {$e^\sharp$};
					\node [style=none] (186) at (-3, -4) {};
					\node [style=morphism] (187) at (-3, -3) {$\iota$};
					\node [style=none] (209) at (6, 0) {$=$};
					\node [style=bn] (225) at (9, 1.75) {};
					\node [style=none] (226) at (8, 3.25) {};
					\node [style=none] (229) at (10, 3.25) {};
					\node [style=none] (230) at (9, -4) {};
					\node [style=morphism] (231) at (9, 0.75) {$\pi$};
					\node [style=morphism] (232) at (9, -0.5) {$\delta$};
					\node [style=morphism] (235) at (9, -3) {$\iota$};
					\node [style=morphism] (236) at (9, -1.75) {$\samp$};
					\node [style=none] (239) at (9, -4) {};
					\node [style=none] (243) at (-6, 0) {$=$};
					\node [style=bn] (244) at (3, 0.5) {};
					\node [style=none] (245) at (2, 1.5) {};
					\node [style=none] (246) at (2, 4.25) {};
					\node [style=none] (247) at (4, 4.25) {};
					\node [style=none] (248) at (4, 1.5) {};
					\node [style=none] (249) at (3, -4) {};
					\node [style=morphism] (250) at (2, 3) {$\pi$};
					\node [style=morphism] (251) at (4, 3) {$\pi$};
					\node [style=morphism] (252) at (4, 1.75) {$\delta$};
					\node [style=morphism] (253) at (3, -1.75) {$\samp$};
					\node [style=morphism] (255) at (2, 1.75) {$\delta$};
					\node [style=none] (256) at (3, -4) {};
					\node [style=morphism] (257) at (3, -3) {$\iota$};
					\node [style=none] (258) at (8, 4) {};
					\node [style=none] (259) at (10, 4) {};
				\end{pgfonlayer}
				\begin{pgfonlayer}{edgelayer}
					\draw (175.center) to (170);
					\draw [in=-90, out=15] (170) to (174.center);
					\draw (174.center) to (173.center);
					\draw (172.center) to (171.center);
					\draw [in=165, out=-90] (171.center) to (170);
					\draw (230.center) to (225);
					\draw [in=-90, out=15] (225) to (229.center);
					\draw [in=165, out=-90] (226.center) to (225);
					\draw (249.center) to (244);
					\draw [in=-90, out=15] (244) to (248.center);
					\draw (248.center) to (247.center);
					\draw (246.center) to (245.center);
					\draw [in=165, out=-90] (245.center) to (244);
					\draw (226.center) to (258.center);
					\draw (229.center) to (259.center);
				\end{pgfonlayer}
			\end{tikzpicture}
		}%
	\end{equation}
	where the first step is by \cref{lem:balanced_sharp}, the second by \cref{eq:bal_split_2,eq:bal_split_3} again, and the third by $\pi$ and $\delta$ being deterministic.
	This proves that the composite morphism is indeed a static idempotent.

	Therefore, by assumption, we have another splitting consisting of morphisms $\kappa$ and $\rho$ with
	\begin{equation}\label{eq:bal_split_4}
		\pi\comp \delta\comp \samp\comp \iota = \kappa \comp \rho, \qquad \rho\comp \kappa = \id.
	\end{equation}
	Using \cref{eq:bal_split_1}, we construct a splitting of the original idempotent $e$ as
	\begin{equation}
		e = e \comp e = \samp\comp \iota\comp \pi\comp \delta\comp \samp\comp \iota \comp \pi\comp \delta = (\samp \comp \iota\comp \kappa)\comp (\rho\comp \pi\comp \delta).
	\end{equation}
	This is indeed a splitting since we have
	\begin{equation}
		(\rho\comp \pi\comp\delta)\comp (\samp\comp \iota\comp \kappa) = \rho \comp (\pi\comp\delta\comp \samp\comp \iota)\comp \kappa = \rho\comp \kappa\comp \rho\comp \kappa = \id,
	\end{equation}
	which concludes the proof.
\end{proof}

\Cref{lem:proj_balanced_split} reduces the problem of splitting balanced idempotents to the problem of splitting static idempotents.
Fortunately, the equalizer principle from \cref{sec:equalizer_principle} is an additional condition which guarantees that all static idempotents split.
We therefore obtain the following main result on idempotent splitting in Markov categories.

\begin{theorem}
	\label{thm:balanced_split}
	Let $\cC$ be a Markov category such that
	\begin{enumerate}
		\item $\cC$ is positive,
		\item $\cC$ is observationally representable, and
		\item $\cC$ satisfies the equalizer principle (\cref{def:equalizer_principle}).
	\end{enumerate}
	Then an idempotent in $\cC$ splits if and only if it is balanced. 
	In particular, if $\cC$ is balanced as well, then all idempotents split.
\end{theorem}
\begin{proof}
	Every split idempotent in a positive Markov category is balanced by \cref{thm:split_props}, so we focus on showing that all balanced idempotents split.

	By \cref{lem:proj_balanced_split}, it is enough to show that every static idempotent splits.
	To this end, consider a static idempotent $e \colon X \to X$ and the equalizer in $\cC_\det$ given by
	\begin{equation}\label{eq:fork_incl}
		\begin{tikzcd}
			T \ar[r, "\iota"] & X \ar[r, shift right, "\delta"'] \ar[r, shift left, "e^\sharp"] & PX
		\end{tikzcd}
	\end{equation}
	As our notation already suggests, the deterministic morphism $\iota$ will eventually be the monomorphism part of a splitting of $e$.

	Since $e$ is a static idempotent, we have $e \ase{e} \id$ by definition.
	Because $\cC$ is \as{}-compatibly representable by \cref{prop:obs_asrep} and $\id^\sharp = \delta$, we can infer
	\begin{equation}\label{eq:balanced_split_0}
		e^\sharp \ase{e} \delta.
	\end{equation}
	Using the equalizer principle, the $\as{}$ equality of the parallel pair from diagram \eqref{eq:fork_incl} means that we have a factorization $e = \iota \comp \pi$ for some $\pi \colon X \to T$.

	To show that this provides the requisite splitting, we still need to prove $\pi \comp \iota = \id_T$. 
	First note that we have
	\begin{equation}\label{eq:balanced_split_1}
		\iota \comp \pi \comp \iota = e\comp \iota = \samp\comp e^{\sharp} \comp \iota = \samp \comp \delta \comp \iota = \iota.
	\end{equation}
	In order to conclude the proof, it is enough to argue that $\iota$ is a monomorphism in $\cC$.\footnote{It is clear that $\iota$ is a monomorphism in $\cC_\det$ by virtue of being an equalizer, but this does not directly imply that it is a monomorphism in $\cC$ as well.}
	For this purpose, consider two arbitrary morphisms $f$ and $g$ satisfying $\iota \comp f = \iota \comp  g$. 
	By \cref{eq:balanced_split_1,lem:factor_ac}, we have $\iota \comp f = e \comp \iota \comp f \ll e$.
	Combining this absolute continuity relation with \cref{eq:balanced_split_0} yields $e^{\sharp} \ase{\iota\comp f} \delta$. 
	The equalizer principle thus says that the factorization of $\iota \comp  f$ across $\iota$ is unique.
	Therefore $\iota \comp f = \iota \comp g$ implies $f = g$. 
	Consequently, $\iota$ is a monomorphism and thus \cref{eq:balanced_split_1} implies the desired $\pi \comp \iota = \id_T$.

	The last sentence of the theorem statement is now an immediate consequence of \cref{thm:balanced_idempotent}.
\end{proof}

Our main application of \cref{thm:balanced_split} is to $\borelstoch$, which we consider shortly. 
Before that, it is instructive to note an example application in which \emph{not} all idempotents split due to the failure of balancedness.
\begin{corollary}
	Every balanced idempotent in $\finsetmulti$ splits, but not all idempotents split.
\end{corollary}
\begin{proof}
	$\finsetmulti$ satisfies the assumptions of \cref{thm:balanced_split}:
	\begin{enumerate}
		\item It has conditionals~\cite[Proposition~16]{fritz2022dseparation}, and this implies positivity~\cite[Lemma~11.24]{fritz2019synthetic}.
		\item It is observationally representable by \cref{ex:finsetmulti_obs}.
		\item It satisfies the equalizer principle by \cref{rem:equalizer_principle}.
	\end{enumerate}
	Hence the split idempotents in $\finsetmulti$ are exactly the balanced idempotents.
	However, by \cref{ex:setmulti_not_split}, there are non-balanced idempotents in $\finsetmulti$, so not all idempotents split by \cref{thm:split_props}.
\end{proof}

We also already know that $\borelstoch$ satisfies all of the relevant assumptions, so here is our previously announced key result.

\begin{corollary}
	\label{cor:borelstoch_idempotents_split}
	Every idempotent in $\borelstoch$ is balanced and splits.
\end{corollary}

\begin{proof}
	$\borelstoch$ is balanced (\cref{prop:stoch_balanced_cat}) and also satisfies the other hypotheses of \cref{thm:balanced_split}, since:
	\begin{enumerate}
		\item It is positive (\cite[Example~11.25]{fritz2019synthetic}).
		\item It is observationally representable (\cref{ex:borelastoch_obs}).
		\item It satisfies the equalizer principle (\cref{prop:borelstoch_eq}).
			\qedhere
	\end{enumerate}
\end{proof}

In order to put this result into context, let us note that it strengthens the following theorem due to Blackwell.

\begin{theorem}[{\cite[Theorem 7]{blackwell1942idempotent}}]
	\label{thm:idemp_borel}
	Let $e \colon X\to X$ be an idempotent Markov kernel on a standard Borel space $X$.
	Then there exists a partition 
	\begin{equation}
		\label{eq:blackwell_partition}
		X = N \sqcup \bigsqcup_{t \in T} X_t
	\end{equation}
	of $X$ into measurable sets\footnotemark{} such that:
	\begin{enumerate}
		\item\label{it:blackwell1} $e(X_t|x) = 1$ for all $x \in X_t$ and all $t \in T$.
		\item\label{it:blackwell2} The probability measure $e(\ph|x)$ is independent of $x$ within each $X_t$.
		\item\label{it:blackwell3} $e(N|x) = 0$ for all $x \in X$.
	\end{enumerate}
\end{theorem}
\footnotetext{Note that ``partition'' here only means decomposition into a disjoint union of measurable sets, not decomposition into a coproduct of measurable spaces. 
For example for $e = \id_{[0,1]}$, this partition (unique in this case) consists of all the singleton sets in $[0,1]$.}%

The following reduction of this result to \cref{cor:borelstoch_idempotents_split} amounts to translating the existence of a splitting into measure-theoretic language.

\begin{proof}
	By \cref{cor:borelstoch_idempotents_split}, we know that $e$ can be split as $e = \iota\comp \pi$ for some $\pi \colon X \to T$ and $\iota \colon T \to X$ with $\pi\comp \iota = \id$.
	Consider the set of all $x \in X$ for which $\pi(\ph|x)$ is a Dirac delta measure, or equivalently for which
	\begin{equation}
		\pi(U|x)^2 = \pi(U|x)
	\end{equation}
	holds for all $U \in \Sigma_T$.
	Since $\Sigma_T$ is countably generated by $T$ being standard Borel, this set is itself measurable, so that we can define $N$ to be its complement.
	Then because $\pi$ is $\as{\iota}$ deterministic by \cref{thm:split_props}, we have $\iota(N|t) = 0$ for all $t \in T$.

	If we now define
	\begin{equation}
		\label{eq:Xlambda}
	X_t \coloneqq \Set*[\big]{x \in X \setminus N  \given  \pi(\ph|x) = \delta_t }
\end{equation}
for each $t \in T$, then it is straightforward to check that the required properties hold.
Indeed, for \ref{it:blackwell1}, we have for every $x \in X_t$,
\begin{equation}
	e(X_t|x) = \int_{t' \in T} \iota(X_t|t') \, \pi(dt'|x) = \int_{t' \in T} \iota(X_t|t') \, \delta_t(dt') = \iota(X_t|t) = 1,
\end{equation}
where the final step is due to $\pi\comp \iota = \id_T$.
Property \ref{it:blackwell2} holds simply by the definition of $X_t$.
Finally, \ref{it:blackwell3} is immediate from $\iota(N|t) = 0$ for all $t$.
\end{proof}

\begin{remark}
	It is worth emphasizing that Blackwell's \cref{thm:idemp_borel} does not immediately imply that every idempotent splits, since $T$ is merely an index set rather than a standard Borel measurable space itself.
	It is in this sense that our \cref{cor:borelstoch_idempotents_split} is stronger than Blackwell's theorem.
\end{remark}

In the discrete case, however, Blackwell's theorem can be interpreted as an idempotent splitting:

\begin{example}[Blackwell's decomposition for stochastic maps]
	\label{ex:splitting}
	To gain better intuition for the decomposition of \cref{eq:blackwell_partition}, as well as for the type classification of idempotents in general, it may help to illustrate it for $\finstoch$ using Markov chain terminology.
	In these terms, the set $T$ can be understood as the set of \emph{recurrent communication classes} of the transition kernel $e$.
	These classes are the non-empty sets $C \subseteq X$ satisfying $e(C|x) = 1$ for all $x \in C$, and such that $C$ is minimal with respect to this property.
	The recurrent communication classes are mutually disjoint, and the complement of their union corresponds to $N$ in \cref{eq:blackwell_partition}, which is exactly the set of transient states.

	Then the projection $\pi$ of a splitting $e = \iota \comp \pi$ acts on a state $x$ as follows:
	\begin{itemize}
		\item If $x \not\in N${\,\textemdash\,}meaning that $x$ is recurrent{\,\textemdash\,}then $\pi$ maps $x$ deterministically to the class containing $x$.
		\item If $x \in N${\,\textemdash\,}meaning that $x$ is transient{\,\textemdash\,}then $\pi(\ph|x)$ is arbitrary.
	\end{itemize}
	Given a class $C$ as input, the inclusion $\iota$ samples an element from $C$ according to a full-support distribution which only depends on $C$ itself.
	In particular, the restriction of $e$ to $C$ is a stochastic matrix of rank one and full support.

	These observations elucidate the first two items of \cref{thm:split_props} yet again differently:
	\begin{itemize}
		\item $e$ is static if and only if $\iota$ is deterministic, which means that all recurrent communication classes are singletons.
		\item $e$ is strong if and only if $\pi$ is deterministic, which means that all states are recurrent (or equivalently $N = \emptyset$).
	\end{itemize}			
\end{example}

Once again, however, in the general non-discrete case our result is stronger than Blackwell's. 

\begin{remark}
	Given that we have proven idempotent splitting for $\borelstoch$, the reader may wonder whether idempotents split in $\stoch$ in general.
	Here is what we know about this at present.

	First of all, it is unlikely that $\stoch$ is representable: $\stoch_\det$ does \emph{not} correspond to $\meas$ since e.g.~the deterministic states $I \to X$ on a measurable space $X$ are exactly the $\{0,1\}$-valued probability measures, and these do not coincide with the points of $X$ in general.
	However, $\stoch$ is still the Kleisli category of the Giry monad on $\meas$.
	This is a weaker form of representability which has been studied in \cite{moss2022probability}, where it was also shown that a suitable notion of observationality still applies in this case~\cite[Theorem 9.1]{moss2022probability}.
	The question is thus whether the proof of \cref{thm:balanced_split} still goes through in such weaker setting.

	For \cref{lem:proj_balanced_split}, this does indeed seem to be the case. 
	However, we do not know whether a suitable formulation of the equalizer principle still holds, and therefore we do not currently know whether the proof of \cref{thm:balanced_split} can be adapted accordingly.
\end{remark}

In order to show how our \Cref{thm:balanced_split} can also be applied fruitfully to other Markov categories, let us derive a similar conclusion for another Markov category of interest.

\begin{corollary}
	\label{cor:chausstoch_idempotents_split}
	Every idempotent in $\chausstoch$ is balanced and splits.
\end{corollary}

\begin{proof}
	This follows by an application of \Cref{thm:balanced_split}, since: 
	\begin{enumerate}
		\item $\chausstoch$ is positive, since even $\tychstoch$ is positive, of which $\chausstoch$ is a full Markov subcategory.
			The positivity of $\tychstoch$ follows because it is causal (\Cref{ex:tychstoch_supp}) and causality implies positivity~\cite[Theorem~2.24]{fritz2022dilations}.
		\item $\chausstoch$ is observationally representable (\Cref{prop:chausstoch_obs_rep}).
		\item It satisfies the equalizer principle since even $\tychstoch$ does (\Cref{rem:equalizer_principle}), of which $\chausstoch$ is a full Markov subcategory.
	\end{enumerate}
	Finally, $\chausstoch$ is also balanced (\Cref{cor:balanced_cat}).
\end{proof}

It is worth noting that \Cref{cor:chausstoch_idempotents_split} is essentially known: Hamana proved in~\cite[Theorem~2.3]{hamana1979envelopes} that every idempotent and Schwarz-positive unital map between unital C*-algebras splits across another unital C*-algebra.
In the special case where the C*-algebra on which the idempotent is defined is commutative, this recovers~\Cref{cor:chausstoch_idempotents_split} thanks to probabilistic Gelfand duality~\cite{furber_jacobs_gelfand}.

\begin{question}
	Do idempotents still split in $\tychstoch$?
\end{question}

A quick look at the previous proof shows that the only missing ingredient here is to prove the observational representability of $\tychstoch$.

\subsection{The Blackwell envelope}
\label{sec:blackwell_envelope}

When not all idempotents split in a generic category $\cC$, it is natural to consider the \emph{Karoubi envelope} $\cC^{\mathcal{K}}$, which is the new category obtained by formally splitting all idempotents. 
The Karoubi envelope of a symmetric monoidal category can also be made symmetric monoidal in a canonical way.
However, if $\cC$ is a Markov category, there does \emph{not} seem to be any obvious way to turn $\cC^{\mathcal{K}}$ into a Markov category as well.
This works only if one restricts idempotents used to construct $\cC^{\mathcal{K}}$ to the balanced idempotents.
In this way, we indeed get a new Markov category which we call the \emph{Blackwell envelope} of $\cC$.

Let us recall the definition of the Karoubi envelope first, also known as \emph{idempotent completion}.

\begin{definition}[{e.g.\ \cite[Section 6.5]{borceux1994handbook}}]
	For any category $\cC$, its \newterm{Karoubi envelope} $\cC^{\mathcal{K}}$ is the category where:
	\begin{itemize}
		\item Objects are pairs $(X,e_X)$ with $X \in \cC$ and $e_X \colon X \to X$ an idempotent.
		\item Morphisms $f \colon (X,e_X)\to (Y,e_Y)$ are morphisms $X \to Y$ in $\cC$ satisfying $f \comp e_X = f = e_Y \comp f$. 
	\end{itemize}
\end{definition}

In order to move closer to Markov categories, let us consider symmetric monoidal structure next.

\begin{proposition}[{e.g.\ \cite[Proposition 6.5.9]{borceux1994handbook}}]
	\label{prop:karoubi_symm_monoidal}
	Let $\cC$ be a symmetric monoidal category. 
	Then its Karoubi envelope $\cC^{\mathcal{K}}$ inherits a canonical symmetric monoidal structure.
\end{proposition}
\begin{proof}[Sketch of proof]
	Every functor between categories $F\colon \cC \to \cD$ has a natural extension to the Karoubi envelopes $F^{\mathcal{K}} \colon \cC^{\mathcal{K}} \to \cD^{\mathcal{K}}$ given by
	\begin{equation}
		F^{\mathcal{K}}(X, e_X) \coloneqq \bigl( F(X), F(e_X) \bigr), \qquad F^{\mathcal{K}}(f) \coloneqq F(f).
	\end{equation}
	Likewise, any natural transformation $\eta \colon F \Rightarrow G$ induces a natural transformation ${\eta^{\mathcal{K}} \colon F^{\mathcal{K}} \Rightarrow G^{\mathcal{K}}}$ with components\footnotemark{}
	\footnotetext{Note that, since $\eta$ is natural, we have $\eta_{(X,e)}^{\mathcal{K}}= \eta_X \comp F(e)=G(e) \comp \eta_X$.}%
	\begin{equation}
		\eta_{(X, e_X)}^{\mathcal{K}} \coloneqq G(e_X)\comp \eta_X\comp F(e_X).
	\end{equation}
	Therefore, via the obvious definition $(X, e_X) \otimes (Y, e_Y) \coloneqq (X \otimes Y, e_X \otimes e_Y)$, we can choose the expected associators, unitors and swap morphisms, and get the desired result.
\end{proof}

Interestingly, there does not seem to be a canonical way to make the Karoubi envelope of a Markov category into a Markov category. 
However, the following modified construction works.
\begin{definition}
	If $\cC$ is a Markov category, its \newterm{Blackwell envelope} $\cC^{\mathcal{B}}$ is the full subcategory of $\cC^{\mathcal{K}}$ whose objects are given by pairs $(X, e_X)$ where $e_X$ is a balanced idempotent.
\end{definition}

We name this construction after Blackwell in honor of his \cref{thm:idemp_borel}, which was arguably the first result in the direction of idempotent splitting in measure-theoretic probability.

\begin{proposition}\label{prop:blackwellenv_markov}
	Let $\cC$ be a Markov category. 
	Then its Blackwell envelope $\cC^{\mathcal{B}}$ inherits a canonical Markov category structure.
\end{proposition}
\begin{proof}
	It is straightforward to see that if $e_X$ and $e_Y$ are balanced idempotents, then so is $e_X \otimes e_Y$.
	Also $\id_I$ is trivially a balanced idempotent.
	Through \cref{prop:karoubi_symm_monoidal}, $\cC^{\mathcal{B}}$ therefore becomes a symmetric monoidal category.
	It is obvious that it is again semicartesian.

	Define now the copy morphism on an object $(X, e_X)$ to be given by
	\begin{equation}\label{eq:blackwell_env_def}
		{%
			\tikzstyle{every picture}=[tikzfig]%
			\begin{tikzpicture}
				\begin{pgfonlayer}{nodelayer}
					\node [style=bn] (20) at (3, 0) {};
					\node [style=none] (29) at (3, -1.75) {};
					\node [style=none] (30) at (2, 2) {};
					\node [style=none] (31) at (4, 2) {};
					\node [style=none] (32) at (0, 0) {$\coloneqq$};
					\node [style=morphism] (33) at (3, -1) {$e_X$};
					\node [style=none] (42) at (2, 1) {};
					\node [style=none] (43) at (4, 1) {};
					\node [style=morphism] (48) at (2, 1.25) {$e_X$};
					\node [style=morphism] (49) at (4, 1.25) {$e_X$};
					\node [style=none] (51) at (-2.5, 0) {$\cop_{(X,e_X)}$};
					\node [style=none] (52) at (3, -2.25) {$X$};
					\node [style=none] (53) at (2, 2.5) {$X$};
					\node [style=none] (54) at (4, 2.5) {$X$};
				\end{pgfonlayer}
				\begin{pgfonlayer}{edgelayer}
					\draw (20) to (29.center);
					\draw (30.center) to (42.center);
					\draw (31.center) to (43.center);
					\draw [in=165, out=-90] (42.center) to (20);
					\draw [in=15, out=-90] (43.center) to (20);
				\end{pgfonlayer}
			\end{tikzpicture}
		}%
	\end{equation}
	Then all relevant properties are trivial except for the associativity of copy, which follows from the assumption that $e_X$ is balanced. 
\end{proof}

In particular, if every idempotent in $\cC$ is balanced, then the Karoubi envelope coincides with the Blackwell envelope and is thus a Markov category.
For example, this happens by \cref{thm:balanced_idempotent} as soon as $\cC$ is balanced, which we already noted to be the case in standard probabilistic settings.

\begin{remark}
	It is difficult to argue that $\cC^{\mathcal{K}}$ \emph{cannot be} a Markov category in general, since there could be many different ways of defining copy morphisms in general, even if one additionally requires the canonical inclusion $\cC \to \cC^{\mathcal{K}}$ to be a Markov functor.
	However, what we can say is that there is no way to achieve this for $\cC = \setmulti$ if one also wants $\cC^{\mathcal{K}}$ to be a \emph{positive} Markov category.

	The reason is that if there was a positive Markov category structure on $\cC^{\mathcal{K}}$, then by \cref{thm:split_props} every idempotent in $\cC^{\mathcal{K}}$ would be balanced.
	But with the inclusion $\cC \to \cC^{\mathcal{K}}$ being assumed a Markov functor, this would imply that every idempotent in $\cC$ is balanced.
	But this is not the case for $\cC = \setmulti$ by \cref{ex:setmulti_not_balanced}.
\end{remark}

Finally, we study the relationship between the Blackwell envelope and the causality axiom.

\begin{lemma}
	\label{lem:blackwell_envelope_ase}
	Let $\cC$ be a Markov category, and consider morphisms
	\[
		p \colon (X,e_X) \to (Y,e_Y), \qquad f, g\colon (Y,e_Y) \to (Z,e_Z)
	\]
	in its Blackwell envelope $\cC^{\mathcal{B}}$.
	Then $f \ase{p} g$ holds if and only if the same is true for the corresponding morphisms in $\cC$. 
\end{lemma}
\begin{proof}
	In $\cC$, we have
	\begin{equation}
		{%
			\tikzstyle{every picture}=[tikzfig]%
			\begin{tikzpicture}
				\begin{pgfonlayer}{nodelayer}
					\node [style=bn] (0) at (3, 0) {};
					\node [style=bn] (1) at (9, 0) {};
					\node [style=morphism] (2) at (3, -1) {$e_Y$};
					\node [style=none] (6) at (2, 3.75) {};
					\node [style=none] (7) at (4, 3.75) {};
					\node [style=none] (8) at (4, 1) {};
					\node [style=none] (9) at (3, -3.5) {};
					\node [style=none] (10) at (9, -3.5) {};
					\node [style=none] (11) at (8, 3.75) {};
					\node [style=none] (12) at (10, 3.75) {};
					\node [style=none] (13) at (6, 0) {$=$};
					\node [style=morphism] (14) at (9, -1) {$e_Y$};
					\node [style=morphism] (15) at (3, -2.5) {$p$};
					\node [style=morphism] (16) at (2, 2.75) {$f$};
					\node [style=morphism] (17) at (9, -2.5) {$p$};
					\node [style=morphism] (18) at (8, 2.75) {$f$};
					\node [style=none] (39) at (2, 1) {};
					\node [style=none] (40) at (8, 1) {};
					\node [style=none] (41) at (10, 1) {};
					\node [style=morphism] (45) at (2, 1.25) {$e_Y$};
					\node [style=morphism] (46) at (8, 1.25) {$e_Y$};
					\node [style=morphism] (47) at (10, 1.25) {$e_Y$};
					\node [style=none] (51) at (3, -4) {$X$};
					\node [style=none] (52) at (2, 4.25) {$Z$};
					\node [style=none] (53) at (4, 4.25) {$Y$};
					\node [style=none] (54) at (9, -4) {$X$};
					\node [style=none] (55) at (8, 4.25) {$Z$};
					\node [style=none] (56) at (10, 4.25) {$Y$};
					\node [style=bn] (57) at (-3, 0) {};
					\node [style=none] (59) at (-4, 3.75) {};
					\node [style=none] (60) at (-2, 3.75) {};
					\node [style=none] (61) at (-2, 1) {};
					\node [style=none] (62) at (-3, -3.5) {};
					\node [style=none] (63) at (0, 0) {$=$};
					\node [style=morphism] (64) at (-3, -2) {$p$};
					\node [style=morphism] (65) at (-4, 2.25) {$f$};
					\node [style=none] (66) at (-4, 1) {};
					\node [style=none] (68) at (-3, -4) {$X$};
					\node [style=none] (69) at (-4, 4.25) {$Z$};
					\node [style=none] (70) at (-2, 4.25) {$Y$};
				\end{pgfonlayer}
				\begin{pgfonlayer}{edgelayer}
					\draw (1) to (10.center);
					\draw (9.center) to (0);
					\draw [in=-90, out=15] (0) to (8.center);
					\draw (8.center) to (7.center);
					\draw (6.center) to (39.center);
					\draw (11.center) to (40.center);
					\draw [in=165, out=-90] (39.center) to (0);
					\draw [in=165, out=-90] (40.center) to (1);
					\draw (12.center) to (41.center);
					\draw [in=15, out=-90] (41.center) to (1);
					\draw (62.center) to (57);
					\draw [in=-90, out=15] (57) to (61.center);
					\draw (61.center) to (60.center);
					\draw (59.center) to (66.center);
					\draw [in=165, out=-90] (66.center) to (57);
				\end{pgfonlayer}
			\end{tikzpicture}
		}%
	\end{equation}
	because $e_Y$ is a balanced idempotent by assumption.
	The analogous equations hold for $g$ in place of $f$.
	Recalling the definition of the copy morphisms in $\cC^{\mathcal{B}}$ as given in \cref{eq:blackwell_env_def}, the result follows.
\end{proof}

\begin{proposition}\label{prop:blackwell_causal}
	A Markov category $\cC$ is causal if and only if its Blackwell envelope $\cC^{\mathcal{B}}$ is causal.
\end{proposition}

\begin{proof}
	Since $\cC$ is a Markov subcategory of $\cC^{\mathcal{B}}$, it is obvious that whenever $\cC^{\mathcal{B}}$ is causal, then the same is true for $\cC$. 
	To prove the converse, we use the characterization of causality from \cref{lem:as_shift}. 

	Consider morphisms 
	\begin{equation}
		\begin{split}
			f &\colon (A, e_A) \to (X, e_X), \\
			g &\colon (X, e_X) \to (Y, e_Y), \\
			h_1, h_2 &\colon (Y, e_Y) \to (Z, e_Z),
		\end{split}
	\end{equation}
	all in the Markov category $\cC^{\mathcal{B}}$. 
	By \cref{lem:blackwell_envelope_ase}, if $h_1 \ase{g\comp f} h_2$ holds, then the same is true in $\cC$, so that we obtain $h_1 \comp g \ase{f} h_2\comp g$ in $\cC$ by \cref{lem:as_shift}. 
	Another application of \cref{lem:blackwell_envelope_ase} ensures that this holds also in $\cC^{\mathcal{B}}$.
	In conclusion, $\cC^{\mathcal{B}}$ is thus causal by \cref{lem:as_shift}.
\end{proof}

\appendix
\section{Background on Markov categories}
\label{appendix}

\subsection{Some basics}
\label{sec:basics}

\begin{definition}
	Consider a symmetric monoidal category $\cC$ in which every object $X \in \cC$ comes equipped with a \newterm{copy morphism} $\cop_X \colon X \to X\otimes X$ and a \newterm{discard morphism} $\discard_X \colon X \to I$, where $I$ is the monoidal unit, written in string diagram form as
	\begin{equation}
		{%
			\tikzstyle{every picture}=[tikzfig]%
			\begin{tikzpicture}
				\begin{pgfonlayer}{nodelayer}
					\node [style=bn] (0) at (-4, 0) {};
					\node [style=none] (1) at (-5, 1.25) {};
					\node [style=none] (2) at (-4, -1) {};
					\node [style=none] (3) at (-3, 1.25) {};
					\node [style=none] (4) at (-5, 1.75) {$X$};
					\node [style=none] (6) at (-3, 1.75) {$X$};
					\node [style=none] (7) at (-4, -1.5) {$X$};
					\node [style=none] (8) at (-9.5, 0) {$\cop_X$};
					\node [style=bn] (9) at (7, 0.5) {};
					\node [style=none] (10) at (7, -1) {};
					\node [style=none] (11) at (7, -1.5) {$X$};
					\node [style=none] (12) at (3, 0) {$\discard_X$};
					\node [style=none] (13) at (-7, 0) {$=$};
					\node [style=none] (14) at (5, 0) {$=$};
				\end{pgfonlayer}
				\begin{pgfonlayer}{edgelayer}
					\draw [in=165, out=-90] (1.center) to (0);
					\draw [in=270, out=15] (0) to (3.center);
					\draw (0) to (2.center);
					\draw (9) to (10.center);
				\end{pgfonlayer}
			\end{tikzpicture}
		}%
	\end{equation}
	Then $\cC$ is a \newterm{Markov category} if:
	\begin{enumerate}
		\item The triple $(X, \cop_X , \discard_X)$ is a commutative comonoid for every object $X$.
		\item Every morphism $f \colon X \to Y$ satisfies $\discard_Y\comp f = \discard_X$.
		\item The copying of $X \otimes Y$ is obtained by the copying of $X$ and $Y$ in a natural way:
			\begin{equation}
				{%
					\tikzstyle{every picture}=[tikzfig]%
					\begin{tikzpicture}
						\begin{pgfonlayer}{nodelayer}
							\node [style=none] (0) at (-5, 2.5) {$X\otimes Y$};
							\node [style=none] (1) at (3, -2) {};
							\node [style=bn] (2) at (3, -1) {};
							\node [style=none] (3) at (4, 0) {};
							\node [style=none] (4) at (4, 0) {};
							\node [style=none] (5) at (2, 0) {};
							\node [style=none] (6) at (7, -2) {};
							\node [style=bn] (7) at (7, -1) {};
							\node [style=none] (8) at (8, 0) {};
							\node [style=none] (9) at (8, 0) {};
							\node [style=none] (10) at (6, 0) {};
							\node [style=none] (11) at (2, 2) {};
							\node [style=none] (12) at (4, 2) {};
							\node [style=none] (13) at (6, 2) {};
							\node [style=none] (14) at (8, 2) {};
							\node [style=none] (15) at (0, 0) {$=$};
							\node [style=bn] (16) at (-3.5, -0.5) {};
							\node [style=none] (17) at (-3.5, -2) {};
							\node [style=none] (18) at (-5, 1) {};
							\node [style=none] (19) at (-2, 1) {};
							\node [style=none] (20) at (-2, 2.5) {$X\otimes Y$};
							\node [style=none] (21) at (-3.5, -2.5) {$X\otimes Y$};
							\node [style=none] (22) at (-5, 2) {};
							\node [style=none] (23) at (-2, 2) {};
							\node [style=none] (24) at (2, 2.5) {$X$};
							\node [style=none] (25) at (4, 2.5) {$Y$};
							\node [style=none] (26) at (6, 2.5) {$X$};
							\node [style=none] (27) at (8, 2.5) {$Y$};
							\node [style=none] (28) at (3, -2.5) {$X$};
							\node [style=none] (29) at (7, -2.5) {$Y$};
						\end{pgfonlayer}
						\begin{pgfonlayer}{edgelayer}
							\draw [style=none] (1.center) to (2);
							\draw [style=none, bend left=45] (2) to (5.center);
							\draw [style=none, bend right=45] (2) to (4.center);
							\draw [style=none] (6.center) to (7);
							\draw [style=none, bend left=45] (7) to (10.center);
							\draw [style=none, bend right=45] (7) to (9.center);
							\draw [style=none] (11.center) to (5.center);
							\draw [style=none] (8.center) to (14.center);
							\draw [style=none, in=-90, out=90] (3.center) to (13.center);
							\draw (17.center) to (16);
							\draw [bend right=315] (16) to (18.center);
							\draw [bend right=45] (16) to (19.center);
							\draw (23.center) to (19.center);
							\draw (18.center) to (22.center);
							\draw [style=protected, in=90, out=-90] (12.center) to (10.center);
						\end{pgfonlayer}
					\end{tikzpicture}
				}%
			\end{equation}
	\end{enumerate}
\end{definition}

It follows that the monoidal unit $I$ is terminal, meaning that $\cC$ is a \newterm{semicartesian} monoidal category \cite[Remark 2.3]{fritz2019synthetic}.

\begin{definition}
	Let $\cC$ be a Markov category. 
	A \newterm{Markov subcategory} $\cD \subseteq \cC$ is a symmetric monoidal subcategory that is a Markov category with respect to the same copy and discard morphisms as $\cC$.
\end{definition}

\begin{example}\label{ex:examples}
	Here we list some important examples of Markov categories.
	More details can be found in \cite{fritz2019synthetic}.
	\begin{enumerate}
		\item $\finstoch$, the category with finite sets as objects and stochastic matrices as morphisms.
			Its variant $\finstoch_\pm$ is defined the same way, but allows negative entries in the matrices.

		\item $\stoch$, the category of all measurable spaces with measurable Markov kernels.
		\item $\borelstoch$, the category of standard Borel measurable spaces with measurable Markov kernels.

		\item $\setmulti$, the category of sets and multivalued functions, where a morphism $X \to Y$ is a function from $X$ to the power set $2^Y$ such that each $f(x)$ is non-empty.
			We write $\finsetmulti$ for the full subcategory whose objects are finite sets.

		\item\label{it:monad_examples} For a cartesian category $\cD$ and an affine symmetric monoidal monad $P$ on $\cD$, the associated Kleisli category $\Kl(P)$ is a Markov category. 

			For instance, the Giry monad on the category of measurable spaces $\meas$ gives rise to $\stoch$, the Giry monad on the category of standard Borel spaces $\borelmeas$ to $\borelstoch$, and the non-empty power set monad on $\set$ produces $\setmulti$. 
			See \cite[Section 3]{fritz2023representable} for more details.
	\end{enumerate}
\end{example} 

\begin{example}\label{ex:continuous_kernels}
	Expanding on \ref{it:monad_examples}, another example of an affine symmetric monoidal monad on a cartesian category is the probability monad on $\cat{Top}$.
	It is given by taking $\tau$-smooth Borel measures on topological spaces \cite[Section 4]{fritz2019support}. 
	We denote the resulting Kleisli category by $\cat{TopStoch}$, which is equivalently the category of topological spaces and \emph{continuous} Markov kernels, where the continuity is with respect to the \emph{$A$-topology} on measures.
	In particular, a Markov kernel $p \colon A \to X$ is continuous if and only if the map $a \mapsto p(U|a)$ is lower semi-continuous for every open $U \subseteq X$.

	The resulting monad can be restricted to Tychonoff (meaning $T_{3\nicefrac{1}{2}}$) spaces, which include all metric spaces and compact Hausdorff spaces.
	In this case, the topology on each space of probability measures can be equivalently taken to be the weak topology induced by the integration maps against continuous functions \cite{banakh}. 
	We call the resulting Kleisli category $\tychstoch$.

	If we restrict the monad further to compact Hausdorff spaces, we obtain the so-called \emph{Radon monad} \cite{swirszcz}.
	Its Kleisli category, which we denote by $\chausstoch$, can equivalently be seen as the opposite of the category of commutative unital C*-algebras and positive unital maps, a generalization of the usual Gelfand duality to the stochastic case (see \cite{furber_jacobs_gelfand} and \cite{parzygnat_gelfand} for details).
\end{example}

Another central definition is the one of deterministic morphism.

\begin{definition}
	A morphism $f \colon X\to Y$ is \newterm{deterministic} if it satisfies
	\begin{equation}
		{%
			\tikzstyle{every picture}=[tikzfig]%
			\begin{tikzpicture}
				\begin{pgfonlayer}{nodelayer}
					\node [style=bn] (0) at (-3, 0.5) {};
					\node [style=bn] (1) at (3, -0.5) {};
					\node [style=morphism] (2) at (-3, -0.75) {$f$};
					\node [style=morphism] (3) at (2, 1) {$f$};
					\node [style=morphism] (4) at (4, 1) {$f$};
					\node [style=none] (5) at (-4, 1.5) {};
					\node [style=none] (6) at (-4, 2) {};
					\node [style=none] (7) at (-2, 2) {};
					\node [style=none] (8) at (-2, 1.5) {};
					\node [style=none] (9) at (-3, -2) {};
					\node [style=none] (10) at (3, -2) {};
					\node [style=none] (11) at (2, 2) {};
					\node [style=none] (12) at (4, 2) {};
					\node [style=none] (13) at (0, 0) {$=$};
					\node [style=none] (16) at (-3, -2.75) {$X$};
					\node [style=none] (17) at (3, -2.75) {$X$};
					\node [style=none] (18) at (-4, 2.75) {$Y$};
					\node [style=none] (19) at (-2, 2.75) {$Y$};
					\node [style=none] (20) at (2, 2.75) {$Y$};
					\node [style=none] (21) at (4, 2.75) {$Y$};
				\end{pgfonlayer}
				\begin{pgfonlayer}{edgelayer}
					\draw (11.center) to (3);
					\draw [in=165, out=-90] (3) to (1);
					\draw (12.center) to (4);
					\draw [in=15, out=-90] (4) to (1);
					\draw (1) to (10.center);
					\draw (9.center) to (0);
					\draw [in=-90, out=15] (0) to (8.center);
					\draw (8.center) to (7.center);
					\draw (6.center) to (5.center);
					\draw [in=165, out=-90] (5.center) to (0);
				\end{pgfonlayer}
			\end{tikzpicture}
		}%
	\end{equation}
\end{definition}
The deterministic morphisms form a cartesian monoidal subcategory $\cC_\det \subseteq \cC$ \cite[Remark 10.13]{fritz2019synthetic}.

\begin{example}
	The deterministic morphisms in $\tychstoch$ are precisely the ones coming from continuous functions.
	They are continuous Markov kernels taking values in $\{0,1\}$-valued $\tau$-smooth probability measures, and these correspond to the points by sobriety (see \cite[Appendix~A]{moss2022ergodic} and \cite[Section~6]{moss2022probability} for additional context).
\end{example}

Throughout this work, we use several properties of Markov categories that have been studied extensively in other works and so we do not go into much detail here.
Specifically, these are \newterm{positivity} \cite[Definition 11.22]{fritz2019synthetic}, \newterm{representability} \cite[Definition 3.10]{fritz2023representable} and \newterm{causality} \cite[Definition 11.31]{fritz2019synthetic}.
The causality axiom states that
\begin{equation}\label{eq:causality}
	{%
		\tikzstyle{every picture}=[tikzfig]%
		\begin{tikzpicture}
			\begin{pgfonlayer}{nodelayer}
				\node [style=none] (251) at (-10.5, -0.25) {};
				\node [style=none] (253) at (-10.5, 1.5) {};
				\node [style=none] (256) at (-10.5, -1.5) {};
				\node [style=none] (258) at (-8, 0) {$\ase{gf}$};
				\node [style=morphism] (260) at (-10.5, 0) {$ h_1 $};
				\node [style=none] (321) at (-10.5, -2) {$Y$};
				\node [style=none] (323) at (-10.5, 2) {$Z$};
				\node [style=none] (324) at (-5.5, -0.25) {};
				\node [style=none] (325) at (-5.5, 1.5) {};
				\node [style=none] (326) at (-5.5, -1.5) {};
				\node [style=morphism] (327) at (-5.5, 0) {$ h_2 $};
				\node [style=none] (330) at (-5.5, -2) {$Y$};
				\node [style=none] (331) at (0, 0) {$\implies$};
				\node [style=none] (332) at (7.75, 0) {$\ase{f}$};
				\node [style=none] (333) at (3.75, 1) {};
				\node [style=none] (334) at (5.75, 1) {};
				\node [style=none] (335) at (3.75, 2.25) {};
				\node [style=none] (336) at (3.75, 2.75) {$Z$};
				\node [style=bn] (337) at (4.75, -0.25) {};
				\node [style=morphism] (338) at (3.75, 1.25) {$ h_1 $};
				\node [style=none] (339) at (5.75, 2.25) {};
				\node [style=morphism] (340) at (4.75, -1.25) {$g$};
				\node [style=none] (341) at (4.75, -2.25) {};
				\node [style=none] (342) at (4.75, -2.75) {$X$};
				\node [style=none] (343) at (9.75, 1) {};
				\node [style=none] (344) at (11.75, 1) {};
				\node [style=none] (345) at (9.75, 2.25) {};
				\node [style=bn] (346) at (10.75, -0.25) {};
				\node [style=morphism] (347) at (9.75, 1.25) {$ h_2 $};
				\node [style=none] (348) at (11.75, 2.25) {};
				\node [style=morphism] (349) at (10.75, -1.25) {$g$};
				\node [style=none] (350) at (10.75, -2.25) {};
				\node [style=none] (351) at (10.75, -2.75) {$X$};
				\node [style=none] (352) at (5.75, 2.75) {$Y$};
				\node [style=none] (353) at (11.75, 2.75) {$Y$};
				\node [style=none] (354) at (9.75, 2.75) {$Z$};
				\node [style=none] (355) at (-5.5, 2) {$Z$};
			\end{pgfonlayer}
			\begin{pgfonlayer}{edgelayer}
				\draw (260) to (253.center);
				\draw (327) to (325.center);
				\draw (256.center) to (251.center);
				\draw (326.center) to (324.center);
				\draw [in=-90, out=165] (337) to (333.center);
				\draw [in=-90, out=15] (337) to (334.center);
				\draw (338) to (335.center);
				\draw (334.center) to (339.center);
				\draw (340) to (337);
				\draw (341.center) to (340);
				\draw [in=-90, out=165] (346) to (343.center);
				\draw [in=-90, out=15] (346) to (344.center);
				\draw (347) to (345.center);
				\draw (344.center) to (348.center);
				\draw (349) to (346);
				\draw (350.center) to (349);
			\end{pgfonlayer}
		\end{tikzpicture}
	}%
\end{equation}
holds for any morphisms $f \colon A \to X$, $g \colon X \to Y$ and $h_1, h_2 \colon Y \to Z$, where we use $\as{}$ equality (\cref{def:as_eq}). 
If $h_1, h_2 \colon W \otimes Y \to Z$ with an extra input $W$ instead, then in a causal Markov category, we also have
\begin{equation}\label{eq:causality2}
	{%
		\tikzstyle{every picture}=[tikzfig]%
		\begin{tikzpicture}
			\begin{pgfonlayer}{nodelayer}
				\node [style=none] (251) at (-10.5, -0.25) {};
				\node [style=none] (256) at (-10.5, -1.5) {};
				\node [style=none] (258) at (-8, 0) {$\ase{gf}$};
				\node [style=none] (321) at (-10.5, -2) {$Y$};
				\node [style=none] (331) at (0, 0) {$\implies$};
				\node [style=none] (332) at (8.25, 0) {$\ase{f}$};
				\node [style=none] (333) at (4.25, 1) {};
				\node [style=none] (334) at (6.25, 1) {};
				\node [style=none] (335) at (3.75, 2.25) {};
				\node [style=none] (336) at (3.75, 2.75) {$Z$};
				\node [style=bn] (337) at (5.25, -0.25) {};
				\node [style=none] (339) at (6.25, 2.25) {};
				\node [style=morphism] (340) at (5.25, -1.25) {$g$};
				\node [style=none] (341) at (5.25, -2.25) {};
				\node [style=none] (342) at (5.25, -2.75) {$X$};
				\node [style=none] (352) at (6.25, 2.75) {$Y$};
				\node [style=none] (356) at (-11.5, -0.25) {};
				\node [style=none] (357) at (-11.5, -1.5) {};
				\node [style=none] (358) at (-11.5, -2) {$W$};
				\node [style=none] (359) at (-11, 1.5) {};
				\node [style=morphism] (360) at (-11, 0) {$\;\; h_1 \;\;$};
				\node [style=none] (361) at (-11, 2) {$Z$};
				\node [style=none] (362) at (-5, -0.25) {};
				\node [style=none] (363) at (-5, -1.5) {};
				\node [style=none] (364) at (-5, -2) {$Y$};
				\node [style=none] (365) at (-6, -0.25) {};
				\node [style=none] (366) at (-6, -1.5) {};
				\node [style=none] (367) at (-6, -2) {$W$};
				\node [style=none] (368) at (-5.5, 1.5) {};
				\node [style=morphism] (369) at (-5.5, 0) {$\;\; h_2 \;\;$};
				\node [style=none] (370) at (-5.5, 2) {$Z$};
				\node [style=morphism] (371) at (3.75, 1.25) {$\;\; h_1 \;\;$};
				\node [style=none] (372) at (3.25, 1) {};
				\node [style=none] (373) at (3.25, -2.25) {};
				\node [style=none] (374) at (3.25, -2.75) {$W$};
				\node [style=none] (375) at (11.25, 1) {};
				\node [style=none] (376) at (13.25, 1) {};
				\node [style=none] (377) at (10.75, 2.25) {};
				\node [style=none] (378) at (10.75, 2.75) {$Z$};
				\node [style=bn] (379) at (12.25, -0.25) {};
				\node [style=none] (380) at (13.25, 2.25) {};
				\node [style=morphism] (381) at (12.25, -1.25) {$g$};
				\node [style=none] (382) at (12.25, -2.25) {};
				\node [style=none] (383) at (12.25, -2.75) {$X$};
				\node [style=none] (384) at (13.25, 2.75) {$Y$};
				\node [style=morphism] (385) at (10.75, 1.25) {$\;\; h_2 \;\;$};
				\node [style=none] (386) at (10.25, 1) {};
				\node [style=none] (387) at (10.25, -2.25) {};
				\node [style=none] (388) at (10.25, -2.75) {$W$};
			\end{pgfonlayer}
			\begin{pgfonlayer}{edgelayer}
				\draw (256.center) to (251.center);
				\draw [in=-90, out=165] (337) to (333.center);
				\draw [in=-90, out=15] (337) to (334.center);
				\draw (334.center) to (339.center);
				\draw (340) to (337);
				\draw (341.center) to (340);
				\draw (357.center) to (356.center);
				\draw (360) to (359.center);
				\draw (363.center) to (362.center);
				\draw (366.center) to (365.center);
				\draw (369) to (368.center);
				\draw (371) to (335.center);
				\draw (373.center) to (372.center);
				\draw [in=-90, out=165] (379) to (375.center);
				\draw [in=-90, out=15] (379) to (376.center);
				\draw (376.center) to (380.center);
				\draw (381) to (379);
				\draw (382.center) to (381);
				\draw (385) to (377.center);
				\draw (387.center) to (386.center);
			\end{pgfonlayer}
		\end{tikzpicture}
	}%
\end{equation}
Indeed by \cref{lem:as_eq_new}, this is an instance of causality with $\id_W \otimes f$ and $\id_W\otimes g$ in place of $f$ and $g$.
By expanding the $\as{}$ equalities, this implication can be expressed as 
\begin{equation}\label{eq:causal_extrawire}
	{%
		\tikzstyle{every picture}=[tikzfig]%
		\begin{tikzpicture}
			\begin{pgfonlayer}{nodelayer}
				\node [style=none] (17) at (-7, 0) {$=$};
				\node [style=none] (35) at (-10.25, 1.75) {};
				\node [style=none] (36) at (-8.75, 1.75) {};
				\node [style=none] (37) at (-10.75, 3) {};
				\node [style=none] (38) at (-8.75, 3.5) {$Y$};
				\node [style=bn] (39) at (-9.5, 0.5) {};
				\node [style=none] (40) at (-9.5, -3.5) {};
				\node [style=none] (41) at (-9.5, -4) {$A$};
				\node [style=morphism] (43) at (-10.75, 2) {$\;\; h_1 \;\;$};
				\node [style=none] (44) at (-8.75, 3) {};
				\node [style=morphism] (45) at (-9.5, -2) {$f$};
				\node [style=none] (46) at (-11.25, -3.5) {};
				\node [style=none] (47) at (-11.25, -4) {$W$};
				\node [style=none] (48) at (-11.25, 1.75) {};
				\node [style=none] (49) at (-11.25, 0) {};
				\node [style=morphism] (51) at (-9.5, -0.5) {$g$};
				\node [style=none] (52) at (-4.25, 1.75) {};
				\node [style=none] (53) at (-2.75, 1.75) {};
				\node [style=none] (54) at (-4.75, 3) {};
				\node [style=none] (55) at (-2.75, 3.5) {$Y$};
				\node [style=bn] (56) at (-3.5, 0.5) {};
				\node [style=none] (57) at (-3.5, -3.5) {};
				\node [style=none] (58) at (-3.5, -4) {$A$};
				\node [style=morphism] (60) at (-4.75, 2) {$\;\; h_2 \;\;$};
				\node [style=none] (61) at (-2.75, 3) {};
				\node [style=morphism] (62) at (-3.5, -2) {$f$};
				\node [style=none] (63) at (-5.25, -3.5) {};
				\node [style=none] (64) at (-5.25, -4) {$W$};
				\node [style=none] (65) at (-5.25, 1.75) {};
				\node [style=none] (66) at (-5.25, 0) {};
				\node [style=morphism] (68) at (-3.5, -0.5) {$g$};
				\node [style=none] (69) at (8, 0) {$=$};
				\node [style=none] (70) at (3.75, 2) {};
				\node [style=none] (71) at (5.25, 2) {};
				\node [style=none] (72) at (3.25, 3.25) {};
				\node [style=none] (73) at (5.25, 3.75) {$Y$};
				\node [style=bn] (74) at (4.5, 0.75) {};
				\node [style=none] (75) at (5.25, -3.5) {};
				\node [style=none] (76) at (5.25, -4) {$A$};
				\node [style=none] (77) at (6.25, 3.75) {$X$};
				\node [style=morphism] (78) at (3.25, 2.25) {$\;\; h_1 \;\;$};
				\node [style=none] (79) at (5.25, 3.25) {};
				\node [style=morphism] (80) at (5.25, -2.5) {$f$};
				\node [style=none] (81) at (2.75, -0.25) {};
				\node [style=none] (82) at (2.75, -4) {$W$};
				\node [style=none] (83) at (2.75, 1.75) {};
				\node [style=none] (84) at (2.75, 0.25) {};
				\node [style=none] (85) at (2.75, 2) {};
				\node [style=morphism] (86) at (4.5, -0.25) {$g$};
				\node [style=bn] (87) at (5.25, -1.5) {};
				\node [style=none] (88) at (6.25, 0) {};
				\node [style=none] (89) at (6.25, 3.25) {};
				\node [style=none] (91) at (2.75, -3.5) {};
				\node [style=none] (92) at (2.75, -1.75) {};
				\node [style=none] (93) at (4.5, -0.5) {};
				\node [style=none] (118) at (0, 0) {$\implies$};
				\node [style=none] (119) at (10.75, 2) {};
				\node [style=none] (120) at (12.25, 2) {};
				\node [style=none] (121) at (10.25, 3.25) {};
				\node [style=none] (122) at (12.25, 3.75) {$Y$};
				\node [style=bn] (123) at (11.5, 0.75) {};
				\node [style=none] (124) at (12.25, -3.5) {};
				\node [style=none] (125) at (12.25, -4) {$A$};
				\node [style=none] (126) at (13.25, 3.75) {$X$};
				\node [style=morphism] (127) at (10.25, 2.25) {$\;\; h_2 \;\;$};
				\node [style=none] (128) at (12.25, 3.25) {};
				\node [style=morphism] (129) at (12.25, -2.5) {$f$};
				\node [style=none] (130) at (9.75, -0.25) {};
				\node [style=none] (131) at (9.75, -4) {$W$};
				\node [style=none] (132) at (9.75, 1.75) {};
				\node [style=none] (133) at (9.75, 0.25) {};
				\node [style=none] (134) at (9.75, 2) {};
				\node [style=morphism] (135) at (11.5, -0.25) {$g$};
				\node [style=bn] (136) at (12.25, -1.5) {};
				\node [style=none] (137) at (13.25, 0) {};
				\node [style=none] (138) at (13.25, 3.25) {};
				\node [style=none] (140) at (9.75, -3.5) {};
				\node [style=none] (141) at (9.75, -1.75) {};
				\node [style=none] (142) at (11.5, -0.5) {};
				\node [style=none] (143) at (-10.75, 3.5) {$Z$};
				\node [style=none] (144) at (-4.75, 3.5) {$Z$};
				\node [style=none] (145) at (3.25, 3.75) {$Z$};
				\node [style=none] (146) at (10.25, 3.75) {$Z$};
			\end{pgfonlayer}
			\begin{pgfonlayer}{edgelayer}
				\draw [in=-90, out=165] (39) to (35.center);
				\draw [in=-90, out=15] (39) to (36.center);
				\draw (43) to (37.center);
				\draw (36.center) to (44.center);
				\draw (40.center) to (45);
				\draw [style=protected] (46.center) to (49.center);
				\draw [style=protected, in=-90, out=90, looseness=1.25] (49.center) to (48.center);
				\draw (51) to (39);
				\draw (45) to (51);
				\draw [in=-90, out=165] (56) to (52.center);
				\draw [in=-90, out=15] (56) to (53.center);
				\draw (60) to (54.center);
				\draw (53.center) to (61.center);
				\draw (57.center) to (62);
				\draw [style=protected] (63.center) to (66.center);
				\draw [style=protected, in=-90, out=90, looseness=1.25] (66.center) to (65.center);
				\draw (68) to (56);
				\draw (62) to (68);
				\draw [in=-90, out=165] (74) to (70.center);
				\draw [in=-90, out=15] (74) to (71.center);
				\draw (78) to (72.center);
				\draw (71.center) to (79.center);
				\draw (75.center) to (80);
				\draw [style=protected] (81.center) to (84.center);
				\draw [style=protected, in=-90, out=90] (84.center) to (83.center);
				\draw [style=protected] (83.center) to (85.center);
				\draw (86) to (74);
				\draw (80) to (87);
				\draw [style=protected] (88.center) to (89.center);
				\draw [style=protected, in=-90, out=15] (87) to (88.center);
				\draw [style=protected, in=-90, out=90, looseness=0.75] (92.center) to (81.center);
				\draw [style=protected] (91.center) to (92.center);
				\draw [style=protected, in=-90, out=165] (87) to (93.center);
				\draw [in=-90, out=165] (123) to (119.center);
				\draw [in=-90, out=15] (123) to (120.center);
				\draw (127) to (121.center);
				\draw (120.center) to (128.center);
				\draw (124.center) to (129);
				\draw [style=protected] (130.center) to (133.center);
				\draw [style=protected, in=-90, out=90] (133.center) to (132.center);
				\draw [style=protected] (132.center) to (134.center);
				\draw (135) to (123);
				\draw (129) to (136);
				\draw [style=protected] (137.center) to (138.center);
				\draw [style=protected, in=-90, out=15] (136) to (137.center);
				\draw [style=protected, in=-90, out=90, looseness=0.75] (141.center) to (130.center);
				\draw [style=protected] (140.center) to (141.center);
				\draw [style=protected, in=-90, out=165] (136) to (142.center);
			\end{pgfonlayer}
		\end{tikzpicture}
	}%
\end{equation}

We also use a terminology for functors that play nicely with the Markov category structure.

\begin{definition}[{\cite[Definition~10.14]{fritz2019synthetic}}]
	\label{def:markov_functor}
	Let $\cC$ and $\cD$ be Markov categories. Then a \newterm{Markov functor} $F \colon \cC \to \cD$ is a strong symmetric monoidal functor making the diagram
	\begin{equation}
		\begin{tikzcd}
						& FX \ar[swap]{dl}{\cop_{FX}} \ar{dr}{F(\cop_X)} \\
			FX \otimes_{\cD} FX \ar{rr}{\cong} && F(X \otimes_{\cC} X)
		\end{tikzcd}
	\end{equation}
	commute for every $X \in \cC$, where the horizontal arrow is the coherence isomorphism of $F$.
	A \newterm{strict Markov functor} is a Markov functor that is strict as a monoidal functor.
\end{definition}

Hence a strict Markov functor preserves the copy morphisms on the nose, while a general Markov functor need not.
Note also that the preservation of the discard morphisms is automatic as the monoidal units are terminal.

\subsection{Some auxiliary results on almost sure equality}
\label{sec:preliminaries_aseq}

Here, we prove several basic properties of almost sure equality as defined in \cref{def:as_eq}. 
One should keep in mind that this definition applies more broadly than the one from previous papers \cite[Definition 13.1]{fritz2019synthetic} due to the extra input wire.

\begin{proposition}\label{prop:ase_props}
	With $p$, $f$ and $g$ as in \cref{def:as_eq}, we have the following properties of almost sure equality:
	\begin{enumerate}[label=(\roman*)]
		\item \label{it:ase_post_comp} If $f \ase{p} g$ holds, then also $h \comp f \ase{p} h\comp g$ holds for any $h$, i.e.~we have
			\begin{equation}
				{%
					\tikzstyle{every picture}=[tikzfig]%
					\begin{tikzpicture}
						\begin{pgfonlayer}{nodelayer}
							\node [style=none] (251) at (-10.5, 1) {};
							\node [style=none] (252) at (-9, 1) {};
							\node [style=none] (253) at (-11, 3) {};
							\node [style=bn] (255) at (-9.75, 0) {};
							\node [style=none] (256) at (-9.75, -2.5) {};
							\node [style=none] (258) at (-7.5, 0) {$=$};
							\node [style=morphism] (260) at (-11, 1.5) {$\,\;\; f \,\;\;$};
							\node [style=none] (261) at (-9, 3) {};
							\node [style=morphism] (262) at (-9.75, -1.25) {$p$};
							\node [style=none] (263) at (-11.5, -2.5) {};
							\node [style=none] (265) at (-11.5, 1.5) {};
							\node [style=none] (266) at (-11.5, -1) {};
							\node [style=none] (267) at (-5, 1) {};
							\node [style=none] (268) at (-3.5, 1) {};
							\node [style=none] (269) at (-5.5, 3) {};
							\node [style=bn] (271) at (-4.25, 0) {};
							\node [style=none] (272) at (-4.25, -2.5) {};
							\node [style=morphism] (275) at (-5.5, 1.5) {$\,\;\; g \,\;\;$};
							\node [style=none] (276) at (-3.5, 3) {};
							\node [style=morphism] (277) at (-4.25, -1.25) {$p$};
							\node [style=none] (278) at (-6, -2.5) {};
							\node [style=none] (280) at (-6, 1.5) {};
							\node [style=none] (281) at (-6, -1) {};
							\node [style=none] (284) at (-10.5, 1.5) {};
							\node [style=none] (285) at (-5, 1.5) {};
							\node [style=none] (286) at (4.5, 0.5) {};
							\node [style=none] (287) at (6, 0.5) {};
							\node [style=bn] (289) at (5.25, -0.5) {};
							\node [style=none] (290) at (5.25, -2.5) {};
							\node [style=none] (291) at (7.5, 0) {$=$};
							\node [style=morphism] (292) at (4, 0.75) {$\,\;\; f \,\;\;$};
							\node [style=none] (293) at (6, 3) {};
							\node [style=morphism] (294) at (5.25, -1.5) {$p$};
							\node [style=none] (295) at (3.5, -2.5) {};
							\node [style=none] (296) at (3.5, 0.75) {};
							\node [style=none] (297) at (3.5, -1.25) {};
							\node [style=none] (298) at (10, 0.5) {};
							\node [style=none] (299) at (11.5, 0.5) {};
							\node [style=bn] (301) at (10.75, -0.5) {};
							\node [style=none] (302) at (10.75, -2.5) {};
							\node [style=morphism] (303) at (9.5, 0.75) {$\,\;\; g \;\;\,$};
							\node [style=none] (304) at (11.5, 3) {};
							\node [style=morphism] (305) at (10.75, -1.5) {$p$};
							\node [style=none] (306) at (9, -2.5) {};
							\node [style=none] (307) at (9, 0.75) {};
							\node [style=none] (308) at (9, -1.25) {};
							\node [style=none] (309) at (4.5, 0.75) {};
							\node [style=none] (310) at (10, 0.75) {};
							\node [style=none] (311) at (0, 0) {$\implies$};
							\node [style=morphism] (312) at (4, 2) {$h$};
							\node [style=morphism] (313) at (9.5, 2) {$h$};
							\node [style=none] (314) at (4, 3) {};
							\node [style=none] (315) at (9.5, 3) {};
							\node [style=none] (316) at (-11, 3.5) {$Y$};
							\node [style=none] (317) at (-9, 3.5) {$X$};
							\node [style=none] (318) at (-11.5, -3) {$W$};
							\node [style=none] (319) at (-9.75, -3) {$A$};
							\node [style=none] (320) at (-5.5, 3.5) {$Y$};
							\node [style=none] (321) at (-3.5, 3.5) {$X$};
							\node [style=none] (322) at (-6, -3) {$W$};
							\node [style=none] (323) at (-4.25, -3) {$A$};
							\node [style=none] (324) at (4, 3.5) {$Y$};
							\node [style=none] (325) at (6, 3.5) {$X$};
							\node [style=none] (326) at (3.5, -3) {$W$};
							\node [style=none] (327) at (5.25, -3) {$A$};
							\node [style=none] (328) at (9.5, 3.5) {$Y$};
							\node [style=none] (329) at (11.5, 3.5) {$X$};
							\node [style=none] (330) at (9, -3) {$W$};
							\node [style=none] (331) at (10.75, -3) {$A$};
						\end{pgfonlayer}
						\begin{pgfonlayer}{edgelayer}
							\draw [in=-90, out=165] (255) to (251.center);
							\draw [in=-90, out=15] (255) to (252.center);
							\draw (260) to (253.center);
							\draw (252.center) to (261.center);
							\draw (256.center) to (262);
							\draw [style=protected] (263.center) to (266.center);
							\draw [style=protected, in=-90, out=90, looseness=1.25] (266.center) to (265.center);
							\draw [in=270, out=90] (262) to (255);
							\draw [in=-90, out=165] (271) to (267.center);
							\draw [in=-90, out=15] (271) to (268.center);
							\draw (275) to (269.center);
							\draw (268.center) to (276.center);
							\draw (272.center) to (277);
							\draw [style=protected] (278.center) to (281.center);
							\draw [style=protected, in=-90, out=90, looseness=1.25] (281.center) to (280.center);
							\draw (277) to (271);
							\draw (251.center) to (284.center);
							\draw (267.center) to (285.center);
							\draw [in=-90, out=165] (289) to (286.center);
							\draw [in=-90, out=15] (289) to (287.center);
							\draw (287.center) to (293.center);
							\draw (290.center) to (294);
							\draw [style=protected] (295.center) to (297.center);
							\draw [style=protected, in=-90, out=90, looseness=1.25] (297.center) to (296.center);
							\draw [in=270, out=90] (294) to (289);
							\draw [in=-90, out=165] (301) to (298.center);
							\draw [in=-90, out=15] (301) to (299.center);
							\draw (299.center) to (304.center);
							\draw (302.center) to (305);
							\draw [style=protected] (306.center) to (308.center);
							\draw [style=protected, in=-90, out=90, looseness=1.25] (308.center) to (307.center);
							\draw (305) to (301);
							\draw (286.center) to (309.center);
							\draw (298.center) to (310.center);
							\draw (292) to (312);
							\draw (303) to (313);
							\draw (313) to (315.center);
							\draw (312) to (314.center);
						\end{pgfonlayer}
					\end{tikzpicture}
				}%
			\end{equation}

		\item \label{it:ase_post_comp_2} If $h\comp f \ase{p} k\comp f$ and $f \ase{p} g$ hold, then also $h\comp g \ase{p} k\comp g$ holds.

		\item \label{it:ase_tensor} If $f \ase{p} f'$ and $g \ase{q} g'$ hold, then also $f \otimes g \ase{p \otimes q} f' \otimes g'$ holds.

		\item \label{it:ase_copy} $f \ase{p} g$ holds for $p \colon A \to X$ if and only if 
			\begin{equation}
				{%
					\tikzstyle{every picture}=[tikzfig]%
					\begin{tikzpicture}
						\begin{pgfonlayer}{nodelayer}
							\node [style=none] (251) at (-4, 0.25) {};
							\node [style=none] (252) at (-2, 0.25) {};
							\node [style=none] (253) at (-4.5, 1.5) {};
							\node [style=bn] (255) at (-3, -1) {};
							\node [style=none] (256) at (-3, -2) {};
							\node [style=none] (258) at (0, 0) {$\ase{p}$};
							\node [style=morphism] (260) at (-4.5, 0.5) {$\,\;\; f \,\;\;$};
							\node [style=none] (261) at (-2, 1.5) {};
							\node [style=none] (265) at (-5, 0.25) {};
							\node [style=none] (284) at (-4, 0.5) {};
							\node [style=none] (286) at (-3, 0.25) {};
							\node [style=bn] (287) at (-4, -1) {};
							\node [style=none] (288) at (-4, -2) {};
							\node [style=none] (289) at (-3, 1.5) {};
							\node [style=none] (290) at (-5, 0.5) {};
							\node [style=none] (291) at (3, 0.25) {};
							\node [style=none] (292) at (5, 0.25) {};
							\node [style=none] (293) at (2.5, 1.5) {};
							\node [style=bn] (294) at (4, -1) {};
							\node [style=none] (295) at (4, -2) {};
							\node [style=morphism] (296) at (2.5, 0.5) {$\,\;\; g \,\;\;$};
							\node [style=none] (297) at (5, 1.5) {};
							\node [style=none] (298) at (2, 0.25) {};
							\node [style=none] (299) at (3, 0.5) {};
							\node [style=none] (300) at (4, 0.25) {};
							\node [style=bn] (301) at (3, -1) {};
							\node [style=none] (302) at (3, -2) {};
							\node [style=none] (303) at (4, 1.5) {};
							\node [style=none] (304) at (2, 0.5) {};
							\node [style=none] (305) at (-4, -2.5) {$W$};
							\node [style=none] (306) at (-3, -2.5) {$X$};
							\node [style=none] (307) at (-4.5, 2) {$Y$};
							\node [style=none] (308) at (-3, 2) {$W$};
							\node [style=none] (309) at (-2, 2) {$X$};
							\node [style=none] (310) at (3, -2.5) {$W$};
							\node [style=none] (311) at (4, -2.5) {$X$};
							\node [style=none] (312) at (2.5, 2) {$Y$};
							\node [style=none] (313) at (4, 2) {$W$};
							\node [style=none] (314) at (5, 2) {$X$};
						\end{pgfonlayer}
						\begin{pgfonlayer}{edgelayer}
							\draw [in=-90, out=165] (255) to (251.center);
							\draw [in=-90, out=15] (255) to (252.center);
							\draw (260) to (253.center);
							\draw (252.center) to (261.center);
							\draw (251.center) to (284.center);
							\draw (256.center) to (255);
							\draw (286.center) to (289.center);
							\draw (288.center) to (287);
							\draw [in=-90, out=165] (287) to (265.center);
							\draw (265.center) to (290.center);
							\draw [style=protected, in=-90, out=15] (287) to (286.center);
							\draw [in=-90, out=165] (294) to (291.center);
							\draw [in=-90, out=15] (294) to (292.center);
							\draw (296) to (293.center);
							\draw (292.center) to (297.center);
							\draw (291.center) to (299.center);
							\draw (295.center) to (294);
							\draw (300.center) to (303.center);
							\draw (302.center) to (301);
							\draw [in=-90, out=165] (301) to (298.center);
							\draw (298.center) to (304.center);
							\draw [style=protected, in=-90, out=15] (301) to (300.center);
						\end{pgfonlayer}
					\end{tikzpicture}
				}%
			\end{equation}
			does.

		\item \label{it:ase_deterministic}
			If $p$ is deterministic, then $f \ase{p} g$ is equivalent to
			\begin{equation}
				{%
					\tikzstyle{every picture}=[tikzfig]%
					\begin{tikzpicture}
						\begin{pgfonlayer}{nodelayer}
							\node [style=morphism] (0) at (-2.5, 0.75) {$\;\; f \;\;$};
							\node [style=morphism] (1) at (-2, -0.75) {$p$};
							\node [style=none] (2) at (-2.5, 1.75) {};
							\node [style=none] (3) at (-2.5, 2.25) {$Y$};
							\node [style=none] (4) at (-3, -1.75) {};
							\node [style=none] (5) at (-3, -2.25) {$W$};
							\node [style=none] (6) at (-2, -1.75) {};
							\node [style=none] (7) at (-2, -2.25) {$A$};
							\node [style=none] (8) at (-3, 0.75) {};
							\node [style=none] (9) at (-2, 0.75) {};
							\node [style=none] (10) at (0, 0) {$=$};
							\node [style=morphism] (11) at (2.5, 0.75) {$\;\; g \;\;$};
							\node [style=morphism] (12) at (3, -0.75) {$p$};
							\node [style=none] (13) at (2.5, 1.75) {};
							\node [style=none] (14) at (2.5, 2.25) {$Y$};
							\node [style=none] (15) at (2, -1.75) {};
							\node [style=none] (16) at (2, -2.25) {$W$};
							\node [style=none] (17) at (3, -1.75) {};
							\node [style=none] (18) at (3, -2.25) {$A$};
							\node [style=none] (19) at (2, 0.75) {};
							\node [style=none] (20) at (3, 0.75) {};
						\end{pgfonlayer}
						\begin{pgfonlayer}{edgelayer}
							\draw (8.center) to (4.center);
							\draw (9.center) to (6.center);
							\draw (2.center) to (0);
							\draw (19.center) to (15.center);
							\draw (20.center) to (17.center);
							\draw (13.center) to (11);
						\end{pgfonlayer}
					\end{tikzpicture}
				}%
			\end{equation}
	\end{enumerate}
\end{proposition}

\begin{proof}
	\begin{itemize}
		\item[\ref{it:ase_post_comp}:] This follows directly from the definition.

		\item[\ref{it:ase_post_comp_2}:] This follows by two applications of property \ref{it:ase_post_comp}: $h\comp g \ase{p} h\comp f \ase{p} k\comp f \ase{p} k\comp g$.

		\item[\ref{it:ase_tensor}:] 
			By \cref{lem:as_eq_new}, we can restrict to morphisms $f$ and $g$ with $W$ equal to the unit $I$.
			Then the statement follows via:
			\begin{equation}
				{%
					\tikzstyle{every picture}=[tikzfig]%
					\begin{tikzpicture}
						\begin{pgfonlayer}{nodelayer}
							\node [style=none] (35) at (-3.5, 1.5) {};
							\node [style=none] (36) at (-1.5, 1) {};
							\node [style=none] (37) at (-4.5, 2.75) {};
							\node [style=bn] (39) at (-2.5, -0.25) {};
							\node [style=none] (40) at (-2.5, -2.75) {};
							\node [style=none] (44) at (-1.5, 2.75) {};
							\node [style=morphism] (45) at (-2.5, -1.5) {$q$};
							\node [style=none] (46) at (-3.5, -2.75) {};
							\node [style=none] (48) at (-4.5, 1.5) {};
							\node [style=bn] (52) at (-3.5, -0.25) {};
							\node [style=none] (53) at (-2.5, 1) {};
							\node [style=none] (54) at (-2.5, 2.75) {};
							\node [style=none] (74) at (0, 0) {$=$};
							\node [style=morphism] (102) at (-3.5, -1.5) {$p$};
							\node [style=none] (103) at (-3.5, 2.75) {};
							\node [style=morphism] (104) at (-4.5, 1.5) {$f$};
							\node [style=morphism] (105) at (-3.5, 1.5) {$\vphantom{f}g$};
							\node [style=none] (124) at (-4.5, 1) {};
							\node [style=none] (125) at (-3.5, 1) {};
							\node [style=none] (126) at (4.25, 0.75) {};
							\node [style=none] (127) at (5.75, 0.5) {};
							\node [style=none] (128) at (1.75, 2.75) {};
							\node [style=bn] (129) at (5, -0.75) {};
							\node [style=none] (130) at (5, -2.75) {};
							\node [style=none] (131) at (5.75, 2.75) {};
							\node [style=morphism] (132) at (5, -1.75) {$q$};
							\node [style=none] (133) at (2.5, -2.75) {};
							\node [style=none] (134) at (1.75, 0.75) {};
							\node [style=bn] (135) at (2.5, -0.75) {};
							\node [style=none] (136) at (3.25, 0.5) {};
							\node [style=none] (137) at (3.25, 1.5) {};
							\node [style=morphism] (138) at (2.5, -1.75) {$p$};
							\node [style=none] (139) at (4.25, 1.5) {};
							\node [style=morphism] (140) at (1.75, 0.75) {$f$};
							\node [style=morphism] (141) at (4.25, 0.75) {$\vphantom{f}g$};
							\node [style=none] (142) at (1.75, 0.5) {};
							\node [style=none] (143) at (4.25, 0.5) {};
							\node [style=none] (144) at (3.25, 2.5) {};
							\node [style=none] (145) at (4.25, 2.5) {};
							\node [style=none] (146) at (3.25, 2.75) {};
							\node [style=none] (147) at (4.25, 2.75) {};
							\node [style=none] (148) at (7.25, 0) {$=$};
							\node [style=none] (149) at (11.5, 0.75) {};
							\node [style=none] (150) at (13, 0.5) {};
							\node [style=none] (151) at (9, 2.75) {};
							\node [style=bn] (152) at (12.25, -0.75) {};
							\node [style=none] (153) at (12.25, -2.75) {};
							\node [style=none] (154) at (13, 2.75) {};
							\node [style=morphism] (155) at (12.25, -1.75) {$q$};
							\node [style=none] (156) at (9.75, -2.75) {};
							\node [style=none] (157) at (9, 0.75) {};
							\node [style=bn] (158) at (9.75, -0.75) {};
							\node [style=none] (159) at (10.5, 0.5) {};
							\node [style=none] (160) at (10.5, 1.5) {};
							\node [style=morphism] (161) at (9.75, -1.75) {$p$};
							\node [style=none] (162) at (11.5, 1.5) {};
							\node [style=morphism] (163) at (9, 0.75) {$f'$};
							\node [style=morphism] (164) at (11.5, 0.75) {$\vphantom{f'}g'$};
							\node [style=none] (165) at (9, 0.5) {};
							\node [style=none] (166) at (11.5, 0.5) {};
							\node [style=none] (167) at (10.5, 2.5) {};
							\node [style=none] (168) at (11.5, 2.5) {};
							\node [style=none] (169) at (10.5, 2.75) {};
							\node [style=none] (170) at (11.5, 2.75) {};
							\node [style=none] (172) at (19.25, 1) {};
							\node [style=none] (173) at (16, 2.75) {};
							\node [style=bn] (174) at (18.25, -0.25) {};
							\node [style=none] (175) at (18.25, -2.75) {};
							\node [style=none] (176) at (19.25, 2.75) {};
							\node [style=morphism] (177) at (18.25, -1.5) {$q$};
							\node [style=none] (178) at (17, -2.75) {};
							\node [style=none] (179) at (16, 1.5) {};
							\node [style=bn] (180) at (17, -0.25) {};
							\node [style=none] (181) at (18.25, 1) {};
							\node [style=none] (182) at (18.25, 2.75) {};
							\node [style=morphism] (183) at (17, -1.5) {$p$};
							\node [style=none] (184) at (17.25, 2.75) {};
							\node [style=morphism] (185) at (16, 1.5) {$f'$};
							\node [style=morphism] (186) at (17.25, 1.5) {$\vphantom{f'}g'$};
							\node [style=none] (187) at (16, 1) {};
							\node [style=none] (188) at (17.25, 1) {};
							\node [style=none] (189) at (14.5, 0) {$=$};
							\node [style=none] (190) at (-4.5, 3.25) {$Y$};
							\node [style=none] (191) at (-2.5, 3.25) {$X$};
							\node [style=none] (192) at (-2.5, -3.25) {$B$};
							\node [style=none] (193) at (-3.5, -3.25) {$A$};
							\node [style=none] (194) at (-3.5, 3.25) {$Z$};
							\node [style=none] (195) at (-1.5, 3.25) {V};
							\node [style=none] (196) at (1.75, 3.25) {$Y$};
							\node [style=none] (197) at (4.25, 3.25) {$X$};
							\node [style=none] (198) at (3.75, -3.25) {$B$};
							\node [style=none] (199) at (2.75, -3.25) {$A$};
							\node [style=none] (200) at (3.25, 3.25) {$Z$};
							\node [style=none] (201) at (5.75, 3.25) {V};
							\node [style=none] (202) at (9, 3.25) {$Y$};
							\node [style=none] (203) at (11.5, 3.25) {$X$};
							\node [style=none] (204) at (12.25, -3.25) {$B$};
							\node [style=none] (205) at (9.75, -3.25) {$A$};
							\node [style=none] (206) at (10.5, 3.25) {$Z$};
							\node [style=none] (207) at (13, 3.25) {V};
							\node [style=none] (208) at (16, 3.25) {$Y$};
							\node [style=none] (209) at (18.25, 3.25) {$X$};
							\node [style=none] (210) at (18.25, -3.25) {$B$};
							\node [style=none] (211) at (17, -3.25) {$A$};
							\node [style=none] (212) at (17.25, 3.25) {$Z$};
							\node [style=none] (213) at (19.25, 3.25) {V};
						\end{pgfonlayer}
						\begin{pgfonlayer}{edgelayer}
							\draw [in=-90, out=15] (39) to (36.center);
							\draw (36.center) to (44.center);
							\draw (40.center) to (45);
							\draw (53.center) to (54.center);
							\draw (103.center) to (35.center);
							\draw (37.center) to (48.center);
							\draw (45) to (39);
							\draw (46.center) to (102);
							\draw (102) to (52);
							\draw [style=protected, in=-90, out=165] (39) to (125.center);
							\draw [style=protected, in=-90, out=15] (52) to (53.center);
							\draw [style=protected, in=-90, out=165] (52) to (124.center);
							\draw [in=-90, out=15] (129) to (127.center);
							\draw (127.center) to (131.center);
							\draw (130.center) to (132);
							\draw (136.center) to (137.center);
							\draw (139.center) to (126.center);
							\draw (128.center) to (134.center);
							\draw (132) to (129);
							\draw (133.center) to (138);
							\draw (138) to (135);
							\draw [style=protected, in=-90, out=165] (129) to (143.center);
							\draw [style=protected, in=-90, out=15] (135) to (136.center);
							\draw [style=protected, in=-90, out=165] (135) to (142.center);
							\draw [style=protected, in=-90, out=90, looseness=0.75] (137.center) to (145.center);
							\draw [style=protected, in=-90, out=90, looseness=0.75] (139.center) to (144.center);
							\draw [style=protected] (144.center) to (146.center);
							\draw [style=protected] (145.center) to (147.center);
							\draw [style=protected] (124.center) to (104);
							\draw [style=protected] (125.center) to (105);
							\draw [in=-90, out=15] (152) to (150.center);
							\draw (150.center) to (154.center);
							\draw (153.center) to (155);
							\draw (159.center) to (160.center);
							\draw (162.center) to (149.center);
							\draw (151.center) to (157.center);
							\draw (155) to (152);
							\draw (156.center) to (161);
							\draw (161) to (158);
							\draw [style=protected, in=-90, out=165] (152) to (166.center);
							\draw [style=protected, in=-90, out=15] (158) to (159.center);
							\draw [style=protected, in=-90, out=165] (158) to (165.center);
							\draw [style=protected, in=-90, out=90, looseness=0.75] (160.center) to (168.center);
							\draw [style=protected, in=-90, out=90, looseness=0.75] (162.center) to (167.center);
							\draw [style=protected] (167.center) to (169.center);
							\draw [style=protected] (168.center) to (170.center);
							\draw [in=-90, out=15] (174) to (172.center);
							\draw (172.center) to (176.center);
							\draw (175.center) to (177);
							\draw (181.center) to (182.center);
							\draw (173.center) to (179.center);
							\draw (177) to (174);
							\draw (178.center) to (183);
							\draw (183) to (180);
							\draw [style=protected, in=-90, out=165] (174) to (188.center);
							\draw [style=protected, in=-90, out=15] (180) to (181.center);
							\draw [style=protected, in=-90, out=165] (180) to (187.center);
							\draw [style=protected] (187.center) to (185);
							\draw [style=protected] (188.center) to (186);
							\draw [style=protected] (186) to (184.center);
						\end{pgfonlayer}
					\end{tikzpicture}
				}%
			\end{equation}

		\item[\ref{it:ase_copy}:] Once again by \cref{lem:as_eq_new}, we can restrict to morphisms without extra inputs.
			In such case we can use the manipulations 
			\begin{equation}
				{%
					\tikzstyle{every picture}=[tikzfig]%
					\begin{tikzpicture}
						\begin{pgfonlayer}{nodelayer}
							\node [style=none] (96) at (-3.75, 0.75) {};
							\node [style=none] (97) at (-1.75, 0.75) {};
							\node [style=none] (98) at (-1.75, 3) {};
							\node [style=bn] (100) at (-2.75, -0.5) {};
							\node [style=none] (101) at (-2.75, -2.75) {};
							\node [style=morphism] (103) at (-4.5, 2) {$f$};
							\node [style=morphism] (104) at (-2.75, -1.75) {$p$};
							\node [style=none] (106) at (-1.75, 0.75) {};
							\node [style=none] (108) at (-4.5, 2) {};
							\node [style=none] (109) at (-3, 2) {};
							\node [style=bn] (110) at (-3.75, 0.75) {};
							\node [style=none] (112) at (0, 0) {$=$};
							\node [style=none] (113) at (-4.5, 3) {};
							\node [style=none] (114) at (-3, 3) {};
							\node [style=none] (115) at (3.75, 0.75) {};
							\node [style=none] (116) at (1.75, 0.75) {};
							\node [style=none] (117) at (1.75, 3) {};
							\node [style=bn] (118) at (2.75, -0.5) {};
							\node [style=none] (119) at (2.75, -2.75) {};
							\node [style=morphism] (120) at (1.75, 2) {$f$};
							\node [style=morphism] (121) at (2.75, -1.75) {$p$};
							\node [style=none] (122) at (1.75, 0.75) {};
							\node [style=none] (123) at (4.5, 1.75) {};
							\node [style=none] (124) at (3, 2) {};
							\node [style=bn] (125) at (3.75, 0.75) {};
							\node [style=none] (126) at (4.5, 3) {};
							\node [style=none] (127) at (3, 3) {};
							\node [style=none] (128) at (6, 0) {$=$};
							\node [style=none] (129) at (9.75, 0.75) {};
							\node [style=none] (130) at (7.75, 0.75) {};
							\node [style=none] (131) at (7.75, 3) {};
							\node [style=bn] (132) at (8.75, -0.5) {};
							\node [style=none] (133) at (8.75, -2.75) {};
							\node [style=morphism] (134) at (7.75, 2) {$g$};
							\node [style=morphism] (135) at (8.75, -1.75) {$p$};
							\node [style=none] (136) at (7.75, 0.75) {};
							\node [style=none] (137) at (10.5, 1.75) {};
							\node [style=none] (138) at (9, 2) {};
							\node [style=bn] (139) at (9.75, 0.75) {};
							\node [style=none] (140) at (10.5, 3) {};
							\node [style=none] (141) at (9, 3) {};
							\node [style=none] (142) at (13.5, 0.75) {};
							\node [style=none] (143) at (15.5, 0.75) {};
							\node [style=none] (144) at (15.5, 3) {};
							\node [style=bn] (145) at (14.5, -0.5) {};
							\node [style=none] (146) at (14.5, -2.75) {};
							\node [style=morphism] (147) at (12.75, 2) {$g$};
							\node [style=morphism] (148) at (14.5, -1.75) {$p$};
							\node [style=none] (149) at (15.5, 0.75) {};
							\node [style=none] (150) at (12.75, 2) {};
							\node [style=none] (151) at (14.25, 2) {};
							\node [style=bn] (152) at (13.5, 0.75) {};
							\node [style=none] (153) at (12.75, 3) {};
							\node [style=none] (154) at (14.25, 3) {};
							\node [style=none] (155) at (11.5, 0) {$=$};
							\node [style=none] (156) at (-2.75, -3.25) {$X$};
							\node [style=none] (157) at (2.75, -3.25) {$X$};
							\node [style=none] (158) at (8.75, -3.25) {$X$};
							\node [style=none] (159) at (14.5, -3.25) {$X$};
							\node [style=none] (160) at (-1.75, 3.5) {$X$};
							\node [style=none] (161) at (4.5, 3.5) {$X$};
							\node [style=none] (162) at (10.5, 3.5) {$X$};
							\node [style=none] (163) at (15.5, 3.5) {$X$};
							\node [style=none] (164) at (-3, 3.5) {$X$};
							\node [style=none] (165) at (3, 3.5) {$X$};
							\node [style=none] (166) at (9, 3.5) {$X$};
							\node [style=none] (167) at (14.25, 3.5) {$X$};
							\node [style=none] (168) at (-4.5, 3.5) {$Y$};
							\node [style=none] (169) at (1.75, 3.5) {$Y$};
							\node [style=none] (170) at (7.75, 3.5) {$Y$};
							\node [style=none] (171) at (12.75, 3.5) {$Y$};
						\end{pgfonlayer}
						\begin{pgfonlayer}{edgelayer}
							\draw [in=-90, out=165] (100) to (96.center);
							\draw [in=-90, out=15] (100) to (97.center);
							\draw (101.center) to (104);
							\draw [in=-90, out=165] (110) to (108.center);
							\draw [in=-90, out=15] (110) to (109.center);
							\draw (106.center) to (98.center);
							\draw (113.center) to (108.center);
							\draw (114.center) to (109.center);
							\draw (104) to (100);
							\draw [in=-90, out=15] (118) to (115.center);
							\draw [in=-90, out=165] (118) to (116.center);
							\draw (119.center) to (121);
							\draw [in=-90, out=15] (125) to (123.center);
							\draw [in=-90, out=165] (125) to (124.center);
							\draw (122.center) to (117.center);
							\draw (126.center) to (123.center);
							\draw (127.center) to (124.center);
							\draw (121) to (118);
							\draw [in=-90, out=15] (132) to (129.center);
							\draw [in=-90, out=165] (132) to (130.center);
							\draw (133.center) to (135);
							\draw [in=-90, out=15] (139) to (137.center);
							\draw [in=-90, out=165] (139) to (138.center);
							\draw (136.center) to (131.center);
							\draw (140.center) to (137.center);
							\draw (141.center) to (138.center);
							\draw (135) to (132);
							\draw [in=-90, out=165] (145) to (142.center);
							\draw [in=-90, out=15] (145) to (143.center);
							\draw (146.center) to (148);
							\draw [in=-90, out=165] (152) to (150.center);
							\draw [in=-90, out=15] (152) to (151.center);
							\draw (149.center) to (144.center);
							\draw (153.center) to (150.center);
							\draw (154.center) to (151.center);
							\draw (148) to (145);
						\end{pgfonlayer}
					\end{tikzpicture}
				}%
			\end{equation}
			to prove the claim.

		\item[\ref{it:ase_deterministic}:] This is analogous to \cite[Lemma 15.5]{fritz2019synthetic}.
			\qedhere
	\end{itemize}
\end{proof}

\begin{lemma}\label{lem:as_shift}
	Let $\cC$ be a Markov category. 
	Then $\cC$ is causal if and only if for all $f \colon A \to X$ and $g \colon X \to Y$ and $h_1, h_2 \colon Y \to Z$, we have
	\begin{equation}\label{eq:as_shift}
		{%
			\tikzstyle{every picture}=[tikzfig]%
			\begin{tikzpicture}
				\begin{pgfonlayer}{nodelayer}
					\node [style=none] (17) at (-7.5, 0) {$=$};
					\node [style=none] (35) at (-11.25, 1.75) {};
					\node [style=none] (36) at (-9.25, 1.75) {};
					\node [style=none] (37) at (-11.25, 3) {};
					\node [style=none] (38) at (-11.25, 3.5) {$Z$};
					\node [style=bn] (39) at (-10.25, 0.5) {};
					\node [style=none] (40) at (-10.25, -3) {};
					\node [style=none] (41) at (-10.25, -3.5) {$A$};
					\node [style=none] (42) at (-9.25, 3.5) {$Y$};
					\node [style=morphism] (43) at (-11.25, 2) {$ h_1 $};
					\node [style=none] (44) at (-9.25, 3) {};
					\node [style=morphism] (45) at (-10.25, -2) {$f$};
					\node [style=morphism] (51) at (-10.25, -0.5) {$g$};
					\node [style=none] (69) at (7.75, 0) {$=$};
					\node [style=none] (72) at (3.75, 3) {};
					\node [style=none] (73) at (3.75, 3.5) {$Z$};
					\node [style=none] (75) at (4.75, -3) {};
					\node [style=none] (76) at (4.75, -3.5) {$A$};
					\node [style=morphism] (78) at (3.75, 2) {$ h_1 $};
					\node [style=morphism] (80) at (4.75, -2) {$f$};
					\node [style=morphism] (86) at (3.75, 0.5) {$g$};
					\node [style=bn] (87) at (4.75, -1) {};
					\node [style=none] (88) at (5.75, 0.25) {};
					\node [style=none] (89) at (5.75, 3) {};
					\node [style=none] (90) at (5.75, 3.5) {$X$};
					\node [style=none] (93) at (3.75, 0.25) {};
					\node [style=none] (118) at (0, 0) {$\implies$};
					\node [style=none] (143) at (-5.75, 1.75) {};
					\node [style=none] (144) at (-3.75, 1.75) {};
					\node [style=none] (145) at (-5.75, 3) {};
					\node [style=none] (146) at (-5.75, 3.5) {$Z$};
					\node [style=bn] (147) at (-4.75, 0.5) {};
					\node [style=none] (148) at (-4.75, -3) {};
					\node [style=none] (149) at (-4.75, -3.5) {$A$};
					\node [style=none] (150) at (-3.75, 3.5) {$Y$};
					\node [style=morphism] (151) at (-5.75, 2) {$ h_2 $};
					\node [style=none] (152) at (-3.75, 3) {};
					\node [style=morphism] (153) at (-4.75, -2) {$f$};
					\node [style=morphism] (154) at (-4.75, -0.5) {$g$};
					\node [style=none] (155) at (9.75, 3) {};
					\node [style=none] (156) at (9.75, 3.5) {$Z$};
					\node [style=none] (157) at (10.75, -3) {};
					\node [style=none] (158) at (10.75, -3.5) {$A$};
					\node [style=morphism] (159) at (9.75, 2) {$ h_2 $};
					\node [style=morphism] (160) at (10.75, -2) {$f$};
					\node [style=morphism] (161) at (9.75, 0.5) {$g$};
					\node [style=bn] (162) at (10.75, -1) {};
					\node [style=none] (163) at (11.75, 0.25) {};
					\node [style=none] (164) at (11.75, 3) {};
					\node [style=none] (165) at (11.75, 3.5) {$X$};
					\node [style=none] (166) at (9.75, 0.25) {};
				\end{pgfonlayer}
				\begin{pgfonlayer}{edgelayer}
					\draw [in=-90, out=165] (39) to (35.center);
					\draw [in=-90, out=15] (39) to (36.center);
					\draw (43) to (37.center);
					\draw (36.center) to (44.center);
					\draw (40.center) to (45);
					\draw (51) to (39);
					\draw (45) to (51);
					\draw (78) to (72.center);
					\draw (75.center) to (80);
					\draw (80) to (87);
					\draw [style=protected] (88.center) to (89.center);
					\draw [style=protected, in=-90, out=15] (87) to (88.center);
					\draw [style=protected, in=-90, out=165] (87) to (93.center);
					\draw (86) to (78);
					\draw [in=-90, out=165] (147) to (143.center);
					\draw [in=-90, out=15] (147) to (144.center);
					\draw (151) to (145.center);
					\draw (144.center) to (152.center);
					\draw (148.center) to (153);
					\draw (154) to (147);
					\draw (153) to (154);
					\draw (159) to (155.center);
					\draw (157.center) to (160);
					\draw (160) to (162);
					\draw [style=protected] (163.center) to (164.center);
					\draw [style=protected, in=-90, out=15] (162) to (163.center);
					\draw [style=protected, in=-90, out=165] (162) to (166.center);
					\draw (161) to (159);
				\end{pgfonlayer}
			\end{tikzpicture}
		}%
	\end{equation}
\end{lemma}

\begin{proof}
	If $\cC$ is causal, we conclude by discarding the middle output in the consequent of causality, as on the right of Implication~\eqref{eq:causal_extrawire} with $W = I$.

	Conversely, assume that Implication \eqref{eq:as_shift} holds.
	The antecedent of the defining implication of causality, $h_1 =_{g\comp f} h_2$, is equivalent to
	\begin{equation}
		{%
			\tikzstyle{every picture}=[tikzfig]%
			\begin{tikzpicture}
				\begin{pgfonlayer}{nodelayer}
					\node [style=none] (251) at (-4, 0.25) {};
					\node [style=none] (252) at (-2, 0.25) {};
					\node [style=none] (253) at (-4, 1.5) {};
					\node [style=bn] (255) at (-3, -1) {};
					\node [style=none] (256) at (-3, -2) {};
					\node [style=none] (258) at (0, 0) {$\ase{gf}$};
					\node [style=morphism] (260) at (-4, 0.5) {$ h_1 $};
					\node [style=none] (261) at (-2, 1.5) {};
					\node [style=none] (284) at (-4, 0.5) {};
					\node [style=none] (307) at (-4, 2) {$Z$};
					\node [style=none] (310) at (2, 0.25) {};
					\node [style=none] (311) at (4, 0.25) {};
					\node [style=none] (312) at (2, 1.5) {};
					\node [style=bn] (313) at (3, -1) {};
					\node [style=none] (314) at (3, -2) {};
					\node [style=morphism] (315) at (2, 0.5) {$ h_2 $};
					\node [style=none] (316) at (4, 1.5) {};
					\node [style=none] (321) at (-3, -2.5) {$Y$};
					\node [style=none] (322) at (3, -2.5) {$Y$};
					\node [style=none] (323) at (-2, 2) {$Y$};
					\node [style=none] (324) at (4, 2) {$Y$};
					\node [style=none] (325) at (2, 2) {$Z$};
				\end{pgfonlayer}
				\begin{pgfonlayer}{edgelayer}
					\draw [in=-90, out=165] (255) to (251.center);
					\draw [in=-90, out=15] (255) to (252.center);
					\draw (260) to (253.center);
					\draw (252.center) to (261.center);
					\draw (251.center) to (284.center);
					\draw (256.center) to (255);
					\draw [in=-90, out=165] (313) to (310.center);
					\draw [in=-90, out=15] (313) to (311.center);
					\draw (315) to (312.center);
					\draw (311.center) to (316.center);
					\draw (314.center) to (313);
				\end{pgfonlayer}
			\end{tikzpicture}
		}%
	\end{equation}
	by \cref{prop:ase_props} \ref{it:ase_copy}, so that applying Implication \eqref{eq:as_shift} yields the required 
	\begin{equation}
		{%
			\tikzstyle{every picture}=[tikzfig]%
			\begin{tikzpicture}
				\begin{pgfonlayer}{nodelayer}
					\node [style=none] (69) at (0, 0) {$\ase{f}$};
					\node [style=none] (70) at (-4, 1) {};
					\node [style=none] (71) at (-2, 1) {};
					\node [style=none] (72) at (-4, 2.25) {};
					\node [style=none] (73) at (-4, 2.75) {$Z$};
					\node [style=bn] (74) at (-3, -0.25) {};
					\node [style=morphism] (78) at (-4, 1.25) {$ h_1 $};
					\node [style=none] (79) at (-2, 2.25) {};
					\node [style=morphism] (86) at (-3, -1.25) {$g$};
					\node [style=none] (93) at (-3, -2.25) {};
					\node [style=none] (94) at (-3, -2.75) {$X$};
					\node [style=none] (95) at (2, 1) {};
					\node [style=none] (96) at (4, 1) {};
					\node [style=none] (97) at (2, 2.25) {};
					\node [style=bn] (99) at (3, -0.25) {};
					\node [style=morphism] (101) at (2, 1.25) {$ h_2 $};
					\node [style=none] (102) at (4, 2.25) {};
					\node [style=morphism] (103) at (3, -1.25) {$g$};
					\node [style=none] (104) at (3, -2.25) {};
					\node [style=none] (105) at (3, -2.75) {$X$};
					\node [style=none] (106) at (-2, 2.75) {$Y$};
					\node [style=none] (107) at (4, 2.75) {$Y$};
					\node [style=none] (108) at (2, 2.75) {$Z$};
				\end{pgfonlayer}
				\begin{pgfonlayer}{edgelayer}
					\draw [in=-90, out=165] (74) to (70.center);
					\draw [in=-90, out=15] (74) to (71.center);
					\draw (78) to (72.center);
					\draw (71.center) to (79.center);
					\draw (86) to (74);
					\draw (93.center) to (86);
					\draw [in=-90, out=165] (99) to (95.center);
					\draw [in=-90, out=15] (99) to (96.center);
					\draw (101) to (97.center);
					\draw (96.center) to (102.center);
					\draw (103) to (99);
					\draw (104.center) to (103);
				\end{pgfonlayer}
			\end{tikzpicture}
		}%
	\end{equation}
	which entails causality of $\cC$.
\end{proof}

\begin{lemma}\label{lem:ase_detsplitmono}
	In an arbitrary Markov category, let us consider a morphism $p\colon A \to T$ and a deterministic split monomorphism $\iota\colon T \to X$ with left inverse $\pi$. 
	Then we have
	\begin{equation}\label{eq:p_iff_mp}
		{%
			\tikzstyle{every picture}=[tikzfig]%
			\begin{tikzpicture}
				\begin{pgfonlayer}{nodelayer}
					\node [style=none] (0) at (7, 0) {$\ase{\iota \comp p}$};
					\node [style=none] (1) at (4.5, 0.75) {};
					\node [style=none] (3) at (4, 1.75) {};
					\node [style=none] (4) at (4, 2.25) {$Y$};
					\node [style=none] (6) at (4.5, -1.75) {};
					\node [style=none] (7) at (4.5, -2.25) {$X$};
					\node [style=morphism] (9) at (4, 0.75) {$\;\; f_1 \;\;$};
					\node [style=none] (12) at (3.5, -1.75) {};
					\node [style=none] (13) at (3.5, -2.25) {$W$};
					\node [style=none] (14) at (3.5, 0.75) {};
					\node [style=none] (15) at (3.5, 0.5) {};
					\node [style=morphism] (34) at (4.5, -0.75) {$\pi$};
					\node [style=none] (35) at (10.5, 0.75) {};
					\node [style=none] (36) at (10, 1.75) {};
					\node [style=none] (37) at (10, 2.25) {$Y$};
					\node [style=none] (38) at (10.5, -1.75) {};
					\node [style=none] (39) at (10.5, -2.25) {$X$};
					\node [style=morphism] (40) at (10, 0.75) {$\;\; g_1 \;\;$};
					\node [style=none] (41) at (9.5, -1.75) {};
					\node [style=none] (42) at (9.5, -2.25) {$W$};
					\node [style=none] (43) at (9.5, 0.75) {};
					\node [style=none] (44) at (9.5, 0.5) {};
					\node [style=morphism] (45) at (10.5, -0.75) {$\pi$};
					\node [style=none] (46) at (0, 0) {$\iff$};
					\node [style=none] (47) at (-8.5, 0) {};
					\node [style=none] (48) at (-9, 1.75) {};
					\node [style=none] (49) at (-9, 2.25) {$Y$};
					\node [style=morphism] (50) at (-9, 0) {$\;\; f_1 \;\;$};
					\node [style=none] (51) at (-9.5, -1.75) {};
					\node [style=none] (52) at (-9.5, -2.25) {$W$};
					\node [style=none] (53) at (-9.5, 0) {};
					\node [style=none] (54) at (-9.5, -0.25) {};
					\node [style=none] (55) at (-6.5, 0) {$\ase{p}$};
					\node [style=none] (56) at (-8.5, -1.75) {};
					\node [style=none] (57) at (-8.5, -2.25) {$T$};
					\node [style=none] (58) at (-3.5, 0) {};
					\node [style=none] (59) at (-4, 1.75) {};
					\node [style=none] (60) at (-4, 2.25) {$Y$};
					\node [style=morphism] (61) at (-4, 0) {$\;\; g_1 \;\;$};
					\node [style=none] (62) at (-4.5, -1.75) {};
					\node [style=none] (63) at (-4.5, -2.25) {$W$};
					\node [style=none] (64) at (-4.5, 0) {};
					\node [style=none] (65) at (-4.5, -0.25) {};
					\node [style=none] (66) at (-3.5, -1.75) {};
					\node [style=none] (67) at (-3.5, -2.25) {$T$};
				\end{pgfonlayer}
				\begin{pgfonlayer}{edgelayer}
					\draw (9) to (3.center);
					\draw [style=protected] (12.center) to (15.center);
					\draw [style=protected, in=-90, out=90, looseness=1.25] (15.center) to (14.center);
					\draw (6.center) to (34);
					\draw (34) to (1.center);
					\draw (40) to (36.center);
					\draw [style=protected] (41.center) to (44.center);
					\draw [style=protected, in=-90, out=90, looseness=1.25] (44.center) to (43.center);
					\draw (38.center) to (45);
					\draw (45) to (35.center);
					\draw (50) to (48.center);
					\draw [style=protected] (51.center) to (54.center);
					\draw [style=protected, in=-90, out=90, looseness=1.25] (54.center) to (53.center);
					\draw (56.center) to (47.center);
					\draw (61) to (59.center);
					\draw [style=protected] (62.center) to (65.center);
					\draw [style=protected, in=-90, out=90, looseness=1.25] (65.center) to (64.center);
					\draw (66.center) to (58.center);
				\end{pgfonlayer}
			\end{tikzpicture}
		}%
	\end{equation}
	and
	\begin{equation}\label{eq:mp_iff_p}
		{%
			\tikzstyle{every picture}=[tikzfig]%
			\begin{tikzpicture}
				\begin{pgfonlayer}{nodelayer}
					\node [style=none] (0) at (-7, 0) {$\ase{p}$};
					\node [style=none] (1) at (-9.5, 0.75) {};
					\node [style=none] (3) at (-10, 1.75) {};
					\node [style=none] (4) at (-10, 2.25) {$Y$};
					\node [style=none] (6) at (-9.5, -1.75) {};
					\node [style=none] (7) at (-9.5, -2.25) {$T$};
					\node [style=morphism] (9) at (-10, 0.75) {$\;\; f_2 \;\;$};
					\node [style=none] (12) at (-10.5, -1.75) {};
					\node [style=none] (13) at (-10.5, -2.25) {$W$};
					\node [style=none] (14) at (-10.5, 0.75) {};
					\node [style=none] (15) at (-10.5, 0.5) {};
					\node [style=morphism] (34) at (-9.5, -0.75) {$\iota$};
					\node [style=none] (35) at (-3.5, 0.75) {};
					\node [style=none] (36) at (-4, 1.75) {};
					\node [style=none] (37) at (-4, 2.25) {$Y$};
					\node [style=none] (38) at (-3.5, -1.75) {};
					\node [style=none] (39) at (-3.5, -2.25) {$T$};
					\node [style=morphism] (40) at (-4, 0.75) {$\;\; g_2 \;\;$};
					\node [style=none] (41) at (-4.5, -1.75) {};
					\node [style=none] (42) at (-4.5, -2.25) {$W$};
					\node [style=none] (43) at (-4.5, 0.75) {};
					\node [style=none] (44) at (-4.5, 0.5) {};
					\node [style=morphism] (45) at (-3.5, -0.75) {$\iota$};
					\node [style=none] (46) at (0, 0) {$\iff$};
					\node [style=none] (47) at (4.5, 0) {};
					\node [style=none] (48) at (4, 1.75) {};
					\node [style=none] (49) at (4, 2.25) {$Y$};
					\node [style=morphism] (50) at (4, 0) {$\;\; f_2 \;\;$};
					\node [style=none] (51) at (3.5, -1.75) {};
					\node [style=none] (52) at (3.5, -2.25) {$W$};
					\node [style=none] (53) at (3.5, 0) {};
					\node [style=none] (54) at (3.5, -0.25) {};
					\node [style=none] (55) at (6.5, 0) {$\ase{\iota \comp p}$};
					\node [style=none] (56) at (4.5, -1.75) {};
					\node [style=none] (57) at (4.5, -2.25) {$X$};
					\node [style=none] (58) at (9.5, 0) {};
					\node [style=none] (59) at (9, 1.75) {};
					\node [style=none] (60) at (9, 2.25) {$Y$};
					\node [style=morphism] (61) at (9, 0) {$\;\; g_2 \;\;$};
					\node [style=none] (62) at (8.5, -1.75) {};
					\node [style=none] (63) at (8.5, -2.25) {$W$};
					\node [style=none] (64) at (8.5, 0) {};
					\node [style=none] (65) at (8.5, -0.25) {};
					\node [style=none] (66) at (9.5, -1.75) {};
					\node [style=none] (67) at (9.5, -2.25) {$X$};
				\end{pgfonlayer}
				\begin{pgfonlayer}{edgelayer}
					\draw (9) to (3.center);
					\draw [style=protected] (12.center) to (15.center);
					\draw [style=protected, in=-90, out=90, looseness=1.25] (15.center) to (14.center);
					\draw (6.center) to (34);
					\draw (34) to (1.center);
					\draw (40) to (36.center);
					\draw [style=protected] (41.center) to (44.center);
					\draw [style=protected, in=-90, out=90, looseness=1.25] (44.center) to (43.center);
					\draw (38.center) to (45);
					\draw (45) to (35.center);
					\draw (50) to (48.center);
					\draw [style=protected] (51.center) to (54.center);
					\draw [style=protected, in=-90, out=90, looseness=1.25] (54.center) to (53.center);
					\draw (56.center) to (47.center);
					\draw (61) to (59.center);
					\draw [style=protected] (62.center) to (65.center);
					\draw [style=protected, in=-90, out=90, looseness=1.25] (65.center) to (64.center);
					\draw (66.center) to (58.center);
				\end{pgfonlayer}
			\end{tikzpicture}
		}%
	\end{equation}
	for any choice of morphisms $f_1, g_1 \colon W \otimes T \to Y$ and $f_2 , g_2 \colon W \otimes X \to Y$.
\end{lemma}
\begin{proof}
	We prove the equivalence \eqref{eq:p_iff_mp} by rewriting the right-hand side.
	Specifically, by using the fact that $\iota$ is deterministic and satisfies $\pi\comp \iota = \id$, it is equivalent to
	\begin{equation}
		{%
			\tikzstyle{every picture}=[tikzfig]%
			\begin{tikzpicture}
				\begin{pgfonlayer}{nodelayer}
					\node [style=none] (0) at (0, 0) {$=$};
					\node [style=none] (1) at (-4, 1.25) {};
					\node [style=none] (2) at (-2, 1.25) {};
					\node [style=none] (3) at (-4.5, 2.5) {};
					\node [style=none] (4) at (-4.5, 3) {$Y$};
					\node [style=bn] (5) at (-3, 0) {};
					\node [style=none] (6) at (-3, -2) {};
					\node [style=none] (7) at (-3, -2.5) {$A$};
					\node [style=none] (8) at (-2, 3) {$X$};
					\node [style=morphism] (9) at (-4.5, 1.5) {$\;\; f_1 \;\;$};
					\node [style=none] (10) at (-2, 2.5) {};
					\node [style=morphism] (11) at (-3, -1) {$p$};
					\node [style=none] (12) at (-5, -2) {};
					\node [style=none] (13) at (-5, -2.5) {$W$};
					\node [style=none] (14) at (-5, 1.25) {};
					\node [style=none] (15) at (-5, 1) {};
					\node [style=morphism] (17) at (-2, 1.5) {$\iota$};
					\node [style=none] (18) at (3, 1.25) {};
					\node [style=none] (19) at (5, 1.25) {};
					\node [style=none] (20) at (2.5, 2.5) {};
					\node [style=none] (21) at (2.5, 3) {$Y$};
					\node [style=bn] (22) at (4, 0) {};
					\node [style=none] (23) at (4, -2) {};
					\node [style=none] (24) at (4, -2.5) {$A$};
					\node [style=none] (25) at (5, 3) {$X$};
					\node [style=morphism] (26) at (2.5, 1.5) {$\;\; g_1 \;\;$};
					\node [style=none] (27) at (5, 2.5) {};
					\node [style=morphism] (28) at (4, -1) {$p$};
					\node [style=none] (29) at (2, -2) {};
					\node [style=none] (30) at (2, -2.5) {$W$};
					\node [style=none] (31) at (2, 1.25) {};
					\node [style=none] (32) at (2, 1) {};
					\node [style=morphism] (33) at (5, 1.5) {$\iota$};
				\end{pgfonlayer}
				\begin{pgfonlayer}{edgelayer}
					\draw [in=-90, out=165] (5) to (1.center);
					\draw [in=-90, out=15] (5) to (2.center);
					\draw (9) to (3.center);
					\draw (2.center) to (10.center);
					\draw (6.center) to (11);
					\draw [style=protected] (12.center) to (15.center);
					\draw [style=protected, in=-90, out=90, looseness=1.25] (15.center) to (14.center);
					\draw (11) to (5);
					\draw [in=-90, out=165] (22) to (18.center);
					\draw [in=-90, out=15] (22) to (19.center);
					\draw (26) to (20.center);
					\draw (19.center) to (27.center);
					\draw (23.center) to (28);
					\draw [style=protected] (29.center) to (32.center);
					\draw [style=protected, in=-90, out=90, looseness=1.25] (32.center) to (31.center);
					\draw (28) to (22);
				\end{pgfonlayer}
			\end{tikzpicture}
		}%
	\end{equation}
	This equality follows from $f_1 \ase{p} g_1$ by post-composing with $\iota$ on the right.
	For the converse implication, we post-compose with $\pi$ on the right. 

	Showing the equivalence \eqref{eq:mp_iff_p} works similarly.
	Post-composing the left-hand side with $\iota$ on the right gives the right-hand side since $\iota$ is deterministic.
	The converse implication is obtained by post-composing with $\pi$ on the right.
\end{proof}

Let $f \colon A \to X \otimes Y$ be a morphism in a Markov category $\cC$. 
Then a \emph{conditional of $f$ given $X$} is a morphism $f_{|X} \colon X \otimes A \to Y$ which satisfies \cite[Definition~11.5]{fritz2019synthetic}
\begin{equation}\label{eq:conditional}
	{%
		\tikzstyle{every picture}=[tikzfig]%
		\begin{tikzpicture}
			\begin{pgfonlayer}{nodelayer}
				\node [style=none] (22) at (0, 0) {$=$};
				\node [style=morphism] (54) at (-3, 0) {$\;\; f \;\; $};
				\node [style=none] (55) at (-3, -1.5) {};
				\node [style=none] (56) at (-2.5, 0) {};
				\node [style=none] (57) at (-2.5, 1.5) {};
				\node [style=none] (58) at (-3.5, 1.5) {};
				\node [style=none] (59) at (-3.5, 0) {};
				\node [style=none] (60) at (-3.5, 2) {$X$};
				\node [style=none] (61) at (-2.5, 2) {$Y$};
				\node [style=none] (62) at (-3, -2) {$A$};
				\node [style=morphism] (63) at (3, -0.5) {$\;f\;$};
				\node [style=none] (64) at (3.25, -0.5) {};
				\node [style=bn] (65) at (3.25, 0.5) {};
				\node [style=none] (66) at (2.75, -0.5) {};
				\node [style=bn] (67) at (2.75, 1.25) {};
				\node [style=bn] (68) at (3.75, -1.75) {};
				\node [style=none] (69) at (3.75, -2.75) {};
				\node [style=none] (70) at (4.5, -0.5) {};
				\node [style=none] (71) at (3.5, 2.5) {};
				\node [style=morphism] (72) at (4, 2.5) {$\;f_{|X}\;$};
				\node [style=none] (74) at (4, 3.5) {};
				\node [style=none] (75) at (4, 4) {$Y$};
				\node [style=none] (76) at (2, 2.5) {};
				\node [style=none] (77) at (2, 3.5) {};
				\node [style=none] (78) at (2, 4) {$X$};
				\node [style=none] (79) at (3.75, -3.25) {$A$};
				\node [style=none] (80) at (4.5, 0) {};
				\node [style=none] (81) at (4.5, 2.5) {};
				\node [style=none] (82) at (5.25, 0.75) {};
				\node [style=none] (83) at (2, 0.75) {};
				\node [style=none] (84) at (2, -2.25) {};
				\node [style=none] (85) at (5.25, -2.25) {};
				\node [style=none] (86) at (3, -0.5) {};
			\end{pgfonlayer}
			\begin{pgfonlayer}{edgelayer}
				\draw (55.center) to (54);
				\draw (59.center) to (58.center);
				\draw (56.center) to (57.center);
				\draw (64.center) to (65);
				\draw [in=-90, out=30] (68) to (70.center);
				\draw [in=-90, out=30] (67) to (71.center);
				\draw (72) to (74.center);
				\draw (76.center) to (77.center);
				\draw [in=-90, out=150] (67) to (76.center);
				\draw (70.center) to (80.center);
				\draw [style=dashed box] (84.center)
				to (85.center)
				to (82.center)
				to (83.center)
				to cycle;
				\draw [style=protected] (80.center) to (81.center);
				\draw [style=protected] (66.center) to (67);
				\draw [style=protected, in=-90, out=150] (68) to (86.center);
				\draw [style=protected] (69.center) to (68);
			\end{pgfonlayer}
		\end{tikzpicture}
	}%
\end{equation}
The following observation generalizes \cite[Proposition~13.7]{fritz2019synthetic}.
\begin{lemma}[Almost sure uniqueness of conditionals]\label{lem:cond_unique}
	If $f^{(1)}_{|X}$ and $f^{(2)}_{|X}$ are two morphisms satisfying \cref{eq:conditional}, then we have the almost sure equality 
	\begin{equation}\label{eq:cond_unique}
		f^{(1)}_{|X} \ase{b} f^{(2)}_{|X},
	\end{equation}
	where $b \colon A \to X \otimes A$ denotes the dashed box in \cref{eq:conditional}.
\end{lemma}
Thus we can think of the conditional $f_{|X}$ as being defined uniquely up to $\as{b}$ equality. 
In particular, if $A$ is the monoidal unit $I$, then the lemma says that the conditional distribution $f_{|X}$ is $f_X$-almost surely unique.
\begin{proof}
	Let us use first the definition of $b$, then properties of copying, and finally the definition of conditionals to write
	\begin{equation}\label{eq:cond_unique_proof}
		{%
			\tikzstyle{every picture}=[tikzfig]%
			\begin{tikzpicture}
				\begin{pgfonlayer}{nodelayer}
					\node [style=morphism] (124) at (-3.5, -1.5) {$\;\; b \;\;$};
					\node [style=bn] (125) at (-4, 0) {};
					\node [style=morphism] (126) at (-4.5, 2) {$f^{(1)}_{|X} $};
					\node [style=none] (128) at (-3, 3.5) {};
					\node [style=none] (129) at (-4.5, 3.5) {};
					\node [style=none] (130) at (-3.5, -3.5) {};
					\node [style=none] (131) at (0, 0) {$=$};
					\node [style=none] (132) at (4.25, -4.25) {};
					\node [style=none] (133) at (4.25, -4.75) {$A$};
					\node [style=bn] (134) at (4.25, -3.25) {};
					\node [style=none] (135) at (3.25, -1.5) {};
					\node [style=none] (136) at (5, -2) {};
					\node [style=morphism] (137) at (3.5, -1.5) {$\;\; f \;\;$};
					\node [style=none] (138) at (3, -1.5) {};
					\node [style=none] (139) at (3.75, -1.5) {};
					\node [style=bn] (140) at (3, 0.25) {};
					\node [style=bn] (141) at (3.75, -0.5) {};
					\node [style=none] (143) at (2, 1.5) {};
					\node [style=none] (144) at (4.25, 1.5) {};
					\node [style=none] (145) at (3, 1.5) {};
					\node [style=none] (146) at (5.75, 1.5) {};
					\node [style=none] (147) at (4.25, 3.5) {};
					\node [style=none] (148) at (5.75, 3.5) {};
					\node [style=morphism] (149) at (2.5, 2) {$\; f^{(1)}_{|X} \;$};
					\node [style=none] (150) at (2.5, 3.5) {};
					\node [style=none] (151) at (7.75, 0) {$=$};
					\node [style=morphism] (183) at (18, -0.25) {$\;\; f \;\;$};
					\node [style=bn] (188) at (19, -2.25) {};
					\node [style=none] (189) at (20, -1) {};
					\node [style=none] (190) at (20, 3.5) {};
					\node [style=none] (191) at (19, -4.25) {};
					\node [style=none] (192) at (17.5, -0.25) {};
					\node [style=none] (193) at (18.5, -0.25) {};
					\node [style=none] (194) at (17.5, 1) {};
					\node [style=none] (195) at (18.5, 1) {};
					\node [style=none] (196) at (17.5, 2.75) {};
					\node [style=none] (197) at (18.5, 2.75) {};
					\node [style=none] (198) at (17.5, 3.5) {};
					\node [style=none] (199) at (18.5, 3.5) {};
					\node [style=none] (200) at (17.5, 4) {$Y$};
					\node [style=none] (201) at (18.5, 4) {$X$};
					\node [style=none] (202) at (20, 4) {$A$};
					\node [style=none] (203) at (19, -4.75) {$A$};
					\node [style=none] (204) at (15.25, 0) {$=$};
					\node [style=none] (212) at (2.5, 4) {$Y$};
					\node [style=none] (213) at (4.25, 4) {$X$};
					\node [style=none] (214) at (5.75, 4) {$A$};
					\node [style=none] (215) at (-4.5, 4) {$Y$};
					\node [style=none] (217) at (-3, 4) {$X$};
					\node [style=none] (219) at (-3.5, -4) {$A$};
					\node [style=none] (221) at (3, 2) {};
					\node [style=none] (222) at (2, 2) {};
					\node [style=none] (223) at (-5, 1.25) {};
					\node [style=none] (224) at (-3, 1.25) {};
					\node [style=none] (225) at (3.5, -2) {};
					\node [style=none] (228) at (12.5, -4.25) {};
					\node [style=none] (229) at (12.5, -4.75) {$A$};
					\node [style=bn] (230) at (11.75, -2.5) {};
					\node [style=none] (231) at (10.75, -1.25) {};
					\node [style=none] (232) at (12.5, -1.25) {};
					\node [style=morphism] (233) at (11, -1.25) {$\;\; f \;\;$};
					\node [style=none] (234) at (10.5, -1.25) {};
					\node [style=none] (235) at (11.25, -1.25) {};
					\node [style=bn] (236) at (10.5, 0.25) {};
					\node [style=bn] (237) at (11.25, -0.25) {};
					\node [style=none] (238) at (11, 1.25) {};
					\node [style=none] (239) at (10, 1.25) {};
					\node [style=none] (240) at (12, 0.75) {};
					\node [style=none] (241) at (13.25, 0.75) {};
					\node [style=none] (242) at (10, 2.25) {};
					\node [style=none] (243) at (13.25, 3.5) {};
					\node [style=morphism] (244) at (11.5, 1.5) {$\; f^{(1)}_{|X} \;$};
					\node [style=none] (245) at (11.5, 2.25) {};
					\node [style=none] (246) at (10, 4) {$Y$};
					\node [style=none] (247) at (11.5, 4) {$X$};
					\node [style=none] (248) at (13.25, 4) {$A$};
					\node [style=none] (249) at (12.5, -0.75) {};
					\node [style=none] (250) at (12, 1.5) {};
					\node [style=none] (251) at (11, 1.5) {};
					\node [style=none] (252) at (11, -1.5) {};
					\node [style=bn] (253) at (12.5, -3.5) {};
					\node [style=none] (254) at (13.25, -2.5) {};
					\node [style=bn] (255) at (5, 0.25) {};
					\node [style=none] (256) at (10, 3.5) {};
					\node [style=none] (257) at (11.5, 3.5) {};
					\node [style=none] (258) at (18, -1) {};
					\node [style=bn] (259) at (-3, 0) {};
					\node [style=none] (260) at (-2, 3.5) {};
					\node [style=none] (261) at (-2, 1.25) {};
					\node [style=none] (262) at (-2, 4) {$A$};
					\node [style=none] (263) at (-4, -1.25) {};
					\node [style=none] (264) at (-3, -1.25) {};
					\node [style=none] (265) at (-5, 1.75) {};
					\node [style=none] (266) at (-4, 1.75) {};
					\node [style=none] (267) at (-4, 1.25) {};
				\end{pgfonlayer}
				\begin{pgfonlayer}{edgelayer}
					\draw (130.center) to (124);
					\draw (126) to (129.center);
					\draw [in=-90, out=15] (134) to (136.center);
					\draw (140) to (138.center);
					\draw (141) to (139.center);
					\draw [in=-90, out=165] (140) to (143.center);
					\draw [in=-90, out=15] (140) to (144.center);
					\draw (146.center) to (148.center);
					\draw (144.center) to (147.center);
					\draw (150.center) to (149);
					\draw [in=15, out=-90] (189.center) to (188);
					\draw (190.center) to (189.center);
					\draw (191.center) to (188);
					\draw (193.center) to (195.center);
					\draw [in=-90, out=90] (195.center) to (196.center);
					\draw (192.center) to (194.center);
					\draw (197.center) to (199.center);
					\draw (196.center) to (198.center);
					\draw [style=protected] (145.center) to (221.center);
					\draw [in=-90, out=165] (125) to (223.center);
					\draw [in=-90, out=15] (125) to (224.center);
					\draw (224.center) to (128.center);
					\draw (143.center) to (222.center);
					\draw (225.center) to (137);
					\draw [in=-90, out=165] (134) to (225.center);
					\draw [in=-90, out=15] (230) to (232.center);
					\draw (236) to (234.center);
					\draw (237) to (235.center);
					\draw [in=-90, out=30] (236) to (238.center);
					\draw [in=-90, out=150] (236) to (239.center);
					\draw (241.center) to (243.center);
					\draw (239.center) to (242.center);
					\draw (245.center) to (244);
					\draw (232.center) to (249.center);
					\draw [style=protected, in=-90, out=90] (249.center) to (240.center);
					\draw [style=protected] (240.center) to (250.center);
					\draw (238.center) to (251.center);
					\draw (252.center) to (233);
					\draw [in=-90, out=165] (230) to (252.center);
					\draw [in=-90, out=165] (253) to (230);
					\draw (254.center) to (241.center);
					\draw [in=-90, out=15] (253) to (254.center);
					\draw (228.center) to (253);
					\draw (132.center) to (134);
					\draw (136.center) to (255);
					\draw [in=-90, out=30] (255) to (146.center);
					\draw [style=protected, in=-90, out=165, looseness=0.75] (255) to (145.center);
					\draw [style=protected, in=-90, out=90, looseness=0.75] (242.center) to (257.center);
					\draw [style=protected, in=-90, out=90, looseness=0.75] (245.center) to (256.center);
					\draw [style=protected, in=-90, out=165] (188) to (258.center);
					\draw [style=protected, in=-90, out=90] (194.center) to (197.center);
					\draw [style=protected] (258.center) to (183);
					\draw (261.center) to (260.center);
					\draw [in=15, out=-90] (261.center) to (259);
					\draw (264.center) to (259);
					\draw (263.center) to (125);
					\draw (266.center) to (267.center);
					\draw (223.center) to (265.center);
					\draw [in=165, out=-90] (267.center) to (259);
				\end{pgfonlayer}
			\end{tikzpicture}
		}%
	\end{equation}
	which notably only involves $f$ and not the choice of its conditional $f^{(1)}_{|X}$.
	Since we can do the same for $f^{(2)}_{|X}$, the desired \cref{eq:cond_unique} follows.
\end{proof}

\subsection{Observationally representable Markov categories}
\label{sec:observational}

In a representable Markov category $\cC$ \cite{fritz2023representable}, every object $X$ comes equipped with a sampling morphism $\samp_X \colon PX \to X$, and for each morphism $p \colon A \to X$ there is a deterministic morphism $p^{\sharp}\colon A \to PX$ such that we have $\samp \comp p^{\sharp} = p$. 
These establish a natural bijection
\begin{equation}
	\label{eq:representable}
	\cC(A,X) \cong \cC_\det(A, PX)
\end{equation}
given by $p \mapsto p^{\sharp}$ from left to right and by post-composition with $\samp_X$ (often denoted by $\samp$ further on) from right to left.

In particular, the sampling morphism is a split epi with right inverse given by $(\id_X)^{\sharp}$, which is nothing but $\delta_X \colon X \to PX$.
Moreover, $\samp$ is not monic, unless $\cC$ is cartesian, in which case $P$ is isomorphic to the identity functor.
Nevertheless, in standard probability theory it is the case that any two distinct probability measures can be distinguished by \emph{iterated} sampling.\footnote{For the intuition behind this statement, note that by the strong law of large numbers, the relative frequency of an event in a sequence of samples approaches its actual probability almost surely.}
In a Markov category, collecting $n$ independent samples amounts to composition with the morphism
\begin{equation}\label{eq:samp_N}
	{%
		\tikzstyle{every picture}=[tikzfig]%
		\begin{tikzpicture}
			\begin{pgfonlayer}{nodelayer}
				\node [style=bn] (0) at (4.25, -1) {};
				\node [style=morphism] (2) at (2.5, 0.75) {$\samp$};
				\node [style=morphism] (3) at (6, 0.75) {$\samp$};
				\node [style=none] (4) at (4.25, -2) {};
				\node [style=none] (5) at (2.5, 1.75) {};
				\node [style=none] (6) at (6, 1.75) {};
				\node [style=none] (8) at (4.25, -2.5) {$PX$};
				\node [style=none] (9) at (2.5, 2.25) {$X$};
				\node [style=none] (10) at (6, 2.25) {$X$};
				\node [style=none] (12) at (4.25, 0.75) {$\cdots$};
				\node [style=none] (13) at (4.25, 3.75) {$n$ outputs};
				\node [style=none] (14) at (2.25, 3) {};
				\node [style=none] (15) at (6.25, 3) {};
				\node [style=none] (16) at (6, 0.5) {};
				\node [style=none] (17) at (2.5, 0.5) {};
				\node [style=morphism] (18) at (-3, 0) {$\samp^{(n)}$};
				\node [style=none] (19) at (-3, -2) {};
				\node [style=none] (20) at (-3, 1.75) {};
				\node [style=none] (21) at (-3, -2.5) {$PX$};
				\node [style=none] (22) at (-3, 2.25) {$X^n$};
				\node [style=none] (24) at (0, 0) {$\coloneqq$};
			\end{pgfonlayer}
			\begin{pgfonlayer}{edgelayer}
				\draw (6.center) to (3);
				\draw (2) to (5.center);
				\draw (0) to (4.center);
				\draw [style=curly brace] (14.center) to (15.center);
				\draw [in=-90, out=15] (0) to (16.center);
				\draw [in=-90, out=165] (0) to (17.center);
				\draw (18) to (20.center);
				\draw (19.center) to (18);
			\end{pgfonlayer}
		\end{tikzpicture}
	}%
\end{equation}
where $X^n$ denotes the $n$-fold tensor product of objects identical to $X$.
In this section, we investigate what happens when these morphisms are jointly monic.
This means that if two morphisms $f$, $g$ satisfy $\samp^{(n)} f = \samp^{(n)} g$ for all $n \in \N$, then $f=g$ holds.
Our claim is that this cancellation property extends automatically to almost sure equality classes.
This is a consequence of the following observation.
\begin{lemma}\label{lem:nabla_mono}
	Let $\cC$ be a representable Markov category. Then the multiplication map 
	\begin{equation*}
		\begin{tikzcd}[column sep=17ex]
			\nabla \,\colon\, PX \otimes PY \ar[r,"\delta_{PX \otimes PY}"] & P (PX \otimes PY) \ar[r,"P(\samp_X \otimes \samp_Y)"] & P(X \otimes Y) 
		\end{tikzcd}
	\end{equation*} 
	is a split monomorphism, with left inverse given by 
	\begin{equation*}
		\begin{tikzcd}[column sep=15ex]
			\Delta \,\colon\,  P(X \otimes Y) \ar[r,"\cop_{P(X \otimes Y)}"] & P(X \otimes Y) \otimes P(X \otimes Y) \ar[r,"P(\pi_1)\otimes P(\pi_2)"] & PX \otimes PY,
		\end{tikzcd}
	\end{equation*} 
	where
	\[
		\pi_1 \coloneqq \id_X \otimes \discard_Y , \qquad \pi_2 \coloneqq \discard_X {}\otimes \id_Y 
	\]
	are the two projections.
\end{lemma}
See also \cite[Proposition 3.4]{fritzperrone2018bimonoidal} for more general context.
\begin{proof}
	Since both the morphisms making up $\nabla$ are deterministic, we have
	\begin{equation}
		\Delta \comp \nabla = (f_1 \otimes f_2) \comp \cop_{PX \otimes PY},
	\end{equation}
	where $f_1$ is given by the topmost composite in
	\begin{equation}
		\begin{tikzcd}
				&  &                                                                                               &  & P(X \otimes Y) \arrow[rrd, "P(\pi_1)"] &  &    \\
				&  & P(PX \otimes PY) \arrow[rru, "P(\samp_X \otimes\, \samp_Y)"] \arrow[rrd, "P(\id_{PX} \otimes\, \discard_{PY})"] &  &                                        &  & PX \\
			PX \otimes PY \arrow[rru, "\delta_{PX \otimes PY}"] \arrow[rrd, "\id_{PX} \otimes\, \discard_{PY}"'] &  &                                                                                               &  & PPX \arrow[rru, "P(\samp_X) \," description]          &  &    \\
															     &  & PX \arrow[rru, "\delta_{PX}"] \arrow[rrrruu, "\id_{PX}"', bend right]                         &  &                                        &  &   
		\end{tikzcd}
	\end{equation}
	and similarly for $f_2$.
	The top parallelogram commutes by the definition of $\pi_1$, functoriality of $P$ and terminality of the monoidal unit $I$.
	The left parallelogram commutes by the naturality of $\delta$ against deterministic morphisms~\cite[Proposition 3.14]{fritz2023representable}, while the bottom right triangle by the fact that $(P,P \samp,\delta)$ forms a monad on $\cC_\det$.
	Hence the commutativity of the diagram shows $f_1 = \id_{PX} \otimes \discard_{PY}$.

	Similarly we can show $f_2 = \discard_{PX} \otimes \id_{PY}$, so that we indeed have $\Delta \comp \nabla = \id_{PX \otimes PY}$ as required.
\end{proof}
Note that the left inverse $P(X \otimes Y) \to PX \otimes PY$ acts on Kleisli morphisms out of $I$ as the map which takes a joint state to the product of its marginals:
\begin{equation}
	{%
		\tikzstyle{every picture}=[tikzfig]%
		\begin{tikzpicture}
			\begin{pgfonlayer}{nodelayer}
				\node [style=wide state] (0) at (-2.75, -0.75) {$p$};
				\node [style=none] (1) at (-3.25, 0.75) {};
				\node [style=none] (2) at (-3.25, 1.25) {$X$};
				\node [style=none] (3) at (-2.25, 0.75) {};
				\node [style=none] (4) at (-2.25, 1.25) {$Y$};
				\node [style=none] (5) at (-3.25, -0.5) {};
				\node [style=none] (6) at (-2.25, -0.5) {};
				\node [style=none] (7) at (0, 0) {$\mapsto$};
				\node [style=wide state] (8) at (2.75, -0.75) {$p$};
				\node [style=none] (9) at (2.25, 0.75) {};
				\node [style=none] (10) at (2.25, 1.25) {$X$};
				\node [style=none] (13) at (2.25, -0.5) {};
				\node [style=none] (14) at (3.25, -0.5) {};
				\node [style=wide state] (15) at (4.75, -0.75) {$p$};
				\node [style=none] (18) at (5.25, 0.75) {};
				\node [style=none] (19) at (5.25, 1.25) {$Y$};
				\node [style=none] (20) at (4.25, -0.5) {};
				\node [style=none] (21) at (5.25, -0.5) {};
				\node [style=bn] (22) at (3.25, 0.25) {};
				\node [style=bn] (23) at (4.25, 0.25) {};
			\end{pgfonlayer}
			\begin{pgfonlayer}{edgelayer}
				\draw (5.center) to (1.center);
				\draw (6.center) to (3.center);
				\draw (13.center) to (9.center);
				\draw (21.center) to (18.center);
				\draw (20.center) to (23);
				\draw (14.center) to (22);
			\end{pgfonlayer}
		\end{tikzpicture}
	}%
\end{equation}

\begin{proposition}\label{prop:reflect_aseq}
	Let $\cC$ be a representable Markov category. Then the following are equivalent:
	\begin{enumerate}
		\item\label{it:observational} For every object $X$, the morphisms $(\samp^{(n)}_X)_{n \in \N}$ are jointly monic.
		\item\label{it:observational_monicity} For all objects $X$ and $Y$, the morphisms $(\samp^{(n)}_X)_{n \in \N} \otimes \id_Y$ are jointly monic.
		\item\label{it:relatively_observational} For every object $X$, the morphisms $(\samp^{(n)}_X)_{n \in \N}$ are jointly monic modulo \as{} equality: for any morphisms $p \colon A \to B$ and $f,g \colon B \to PX$, we have
			\begin{equation}\label{eq:reflect_aseq}
				\samp^{(n)}_X \comp f \ase{p} \samp^{(n)}_X\comp g \quad\; \forall \, n \in \N \qquad \implies \qquad f\ase{p} g.
			\end{equation}
	\end{enumerate}
\end{proposition}

\begin{proof}
	The implications from \ref{it:observational_monicity} to \ref{it:relatively_observational} and from \ref{it:relatively_observational} to \ref{it:observational} are immediate.

	We therefore focus on proving that \ref{it:observational} implies \ref{it:observational_monicity}.
	For ease of notation, we omit the subscript $X$ from $\samp_X$.
	We thus assume that the $(\samp^{(n)})_{n \in \N}$ are jointly monic, and consider two parallel morphisms $f,g\colon A \to PX \tensor Y$ satisfying 
	\begin{equation}\label{eq:obs_monicity}
		{%
			\tikzstyle{every picture}=[tikzfig]%
			\begin{tikzpicture}
				\begin{pgfonlayer}{nodelayer}
					\node [style=none] (0) at (-5, 0.25) {};
					\node [style=morphism] (2) at (-6.25, 1.5) {$\samp$};
					\node [style=none] (3) at (-5, 2.25) {$\cdots$};
					\node [style=bn] (4) at (-5, 0.25) {};
					\node [style=morphism] (7) at (-3.75, 1.5) {$\samp$};
					\node [style=none] (8) at (-3.75, 3) {};
					\node [style=none] (10) at (-2, 1.75) {};
					\node [style=none] (11) at (-6.25, 1.5) {};
					\node [style=none] (12) at (-3.75, 1.5) {};
					\node [style=none] (13) at (-5, 0.25) {};
					\node [style=none] (14) at (-5, -1.25) {};
					\node [style=none] (15) at (-3, -1.25) {};
					\node [style=none] (16) at (-4, -2.75) {};
					\node [style=none] (17) at (-4, -2.75) {};
					\node [style=none] (18) at (-3, -0.75) {};
					\node [style=none] (19) at (-4, -3.25) {$A$};
					\node [style=morphism] (20) at (-4, -1.25) {$\;\quad f\quad\;$};
					\node [style=none] (21) at (-6.25, 3) {};
					\node [style=none] (23) at (-2, 3.5) {$Y$};
					\node [style=none] (24) at (-3.75, 3.5) {$X$};
					\node [style=none] (25) at (-2, 3) {};
					\node [style=none] (26) at (3.25, 0.25) {};
					\node [style=morphism] (27) at (2, 1.5) {$\samp$};
					\node [style=none] (28) at (3.25, 2.25) {$\cdots$};
					\node [style=bn] (29) at (3.25, 0.25) {};
					\node [style=morphism] (30) at (4.5, 1.5) {$\samp$};
					\node [style=none] (31) at (4.5, 3) {};
					\node [style=none] (32) at (6.25, 1.75) {};
					\node [style=none] (33) at (2, 1.5) {};
					\node [style=none] (34) at (4.5, 1.5) {};
					\node [style=none] (35) at (3.25, 0.25) {};
					\node [style=none] (36) at (3.25, -1.25) {};
					\node [style=none] (37) at (5.25, -1.25) {};
					\node [style=none] (38) at (4.25, -2.75) {};
					\node [style=none] (39) at (4.25, -2.75) {};
					\node [style=none] (40) at (5.25, -0.75) {};
					\node [style=none] (41) at (4.25, -3.25) {$A$};
					\node [style=morphism] (42) at (4.25, -1.25) {$\;\quad g\quad\;$};
					\node [style=none] (43) at (2, 3) {};
					\node [style=none] (47) at (6.25, 3) {};
					\node [style=none] (49) at (0, 0) {$=$};
					\node [style=none] (50) at (-6.25, 3.5) {$X$};
					\node [style=none] (51) at (6.25, 3.5) {$Y$};
					\node [style=none] (52) at (4.5, 3.5) {$X$};
					\node [style=none] (53) at (2, 3.5) {$X$};
				\end{pgfonlayer}
				\begin{pgfonlayer}{edgelayer}
					\draw [style=protected] (7) to (8.center);
					\draw [style=protected, in=-90, out=165] (4) to (11.center);
					\draw [style=protected, in=-90, out=15] (4) to (12.center);
					\draw (13.center) to (14.center);
					\draw (18.center) to (15.center);
					\draw (20) to (17.center);
					\draw [in=90, out=-90, looseness=1.50] (10.center) to (18.center);
					\draw (21.center) to (11.center);
					\draw (10.center) to (25.center);
					\draw [style=protected] (30) to (31.center);
					\draw [style=protected, in=-90, out=165] (29) to (33.center);
					\draw [style=protected, in=-90, out=15] (29) to (34.center);
					\draw (35.center) to (36.center);
					\draw (40.center) to (37.center);
					\draw (42) to (39.center);
					\draw [in=90, out=-90, looseness=1.50] (32.center) to (40.center);
					\draw (43.center) to (33.center);
					\draw (32.center) to (47.center);
				\end{pgfonlayer}
			\end{tikzpicture}
		}%
	\end{equation}
	We consider the following equalities:
	\begin{equation}\label{eq:reflect_aseq1}
		{%
			\tikzstyle{every picture}=[tikzfig]%
			\begin{tikzpicture}
				\begin{pgfonlayer}{nodelayer}
					\node [style=none] (0) at (-5.5, 0.25) {};
					\node [style=none] (4) at (-5.5, -2.75) {};
					\node [style=none] (36) at (-3.5, -2.75) {};
					\node [style=none] (37) at (-4.5, -4.25) {};
					\node [style=none] (38) at (-4.5, -4.25) {};
					\node [style=none] (42) at (-3.5, 0.25) {};
					\node [style=morphism] (43) at (-4.5, 0.25) {$\quad \nabla \quad $};
					\node [style=bn] (44) at (-4.5, 1.5) {};
					\node [style=morphism] (45) at (-6.25, 3) {$\samp$};
					\node [style=none] (47) at (-6.25, 4.25) {};
					\node [style=none] (49) at (-4.5, 3) {$\cdots$};
					\node [style=none] (51) at (-2.75, 4.25) {};
					\node [style=morphism] (52) at (-2.75, 3) {$\samp$};
					\node [style=none] (71) at (0, 0) {$=$};
					\node [style=none] (189) at (3.5, 0) {};
					\node [style=none] (198) at (5.5, 0) {};
					\node [style=none] (199) at (2, 1.75) {};
					\node [style=none] (200) at (4.5, 1.75) {$\cdots$};
					\node [style=none] (201) at (7, 1.75) {};
					\node [style=bn] (205) at (3.5, 0) {};
					\node [style=bn] (206) at (5.5, 0) {};
					\node [style=none] (207) at (3, 1.75) {};
					\node [style=none] (208) at (6, 1.75) {};
					\node [style=none] (210) at (2.5, 4.25) {};
					\node [style=none] (211) at (6.5, 4.25) {};
					\node [style=morphism] (213) at (2.5, 1.75) {$\;  \nabla \;  $};
					\node [style=morphism] (215) at (6.5, 1.75) {$\;  \nabla \;  $};
					\node [style=morphism] (218) at (2.5, 3.25) {$\samp$};
					\node [style=morphism] (219) at (6.5, 3.25) {$\samp$};
					\node [style=none] (221) at (-4.5, -4.75) {$A$};
					\node [style=none] (223) at (-6.25, 4.75) {$X \otimes Y$};
					\node [style=none] (230) at (-2.75, 2.75) {};
					\node [style=none] (231) at (-6.25, 2.75) {};
					\node [style=none] (232) at (9, 0) {$=$};
					\node [style=none] (233) at (12.75, 0) {};
					\node [style=none] (240) at (14.75, 0) {};
					\node [style=none] (241) at (10.75, 1.75) {};
					\node [style=none] (242) at (14, 1.75) {$\cdots$};
					\node [style=none] (243) at (16.75, 1.75) {};
					\node [style=bn] (244) at (12.75, 0) {};
					\node [style=bn] (245) at (14.75, 0) {};
					\node [style=none] (246) at (12.25, 1.75) {};
					\node [style=none] (247) at (15.25, 1.75) {};
					\node [style=none] (248) at (10.75, 4.25) {};
					\node [style=none] (249) at (12.25, 4.25) {};
					\node [style=morphism] (254) at (10.75, 3.25) {$\samp$};
					\node [style=morphism] (255) at (12.25, 2) {$\samp$};
					\node [style=none] (257) at (10.75, 4.75) {$X$};
					\node [style=none] (258) at (12.25, 4.75) {$Y$};
					\node [style=none] (259) at (15.25, 4.25) {};
					\node [style=none] (260) at (16.75, 4.25) {};
					\node [style=morphism] (261) at (15.25, 3.25) {$\samp$};
					\node [style=morphism] (262) at (16.75, 2) {$\samp$};
					\node [style=none] (265) at (-2.75, 4.75) {$X \otimes Y$};
					\node [style=none] (266) at (2.5, 4.75) {$X \otimes Y$};
					\node [style=none] (267) at (6.5, 4.75) {$X \otimes Y$};
					\node [style=none] (268) at (15.25, 4.75) {$X$};
					\node [style=none] (269) at (16.75, 4.75) {$Y$};
					\node [style=morphism] (272) at (-3.5, -1.25) {$\delta$};
					\node [style=morphism] (273) at (-4.5, -2.75) {$\;\quad f\quad\;$};
					\node [style=none] (274) at (3.5, 0) {};
					\node [style=none] (275) at (3.5, -2.75) {};
					\node [style=none] (276) at (5.5, -2.75) {};
					\node [style=none] (277) at (4.5, -4.25) {};
					\node [style=none] (278) at (4.5, -4.25) {};
					\node [style=none] (279) at (5.5, 0) {};
					\node [style=none] (280) at (4.5, -4.75) {$A$};
					\node [style=morphism] (281) at (5.5, -1.25) {$\delta$};
					\node [style=morphism] (282) at (4.5, -2.75) {$\;\quad f\quad\;$};
					\node [style=none] (287) at (12.75, 0) {};
					\node [style=none] (288) at (12.75, -2.75) {};
					\node [style=none] (289) at (14.75, -2.75) {};
					\node [style=none] (290) at (13.75, -4.25) {};
					\node [style=none] (291) at (13.75, -4.25) {};
					\node [style=none] (292) at (14.75, 0) {};
					\node [style=none] (293) at (13.75, -4.75) {$A$};
					\node [style=morphism] (294) at (14.75, -1.25) {$\delta$};
					\node [style=morphism] (295) at (13.75, -2.75) {$\;\quad f\quad\;$};
				\end{pgfonlayer}
				\begin{pgfonlayer}{edgelayer}
					\draw (0.center) to (4.center);
					\draw (45) to (47.center);
					\draw (51.center) to (52);
					\draw (44) to (43);
					\draw [in=165, out=-90, looseness=0.75] (207.center) to (206);
					\draw [in=165, out=-90] (199.center) to (205);
					\draw [in=-90, out=15] (206) to (201.center);
					\draw (211.center) to (215);
					\draw (210.center) to (213);
					\draw [style=protected, in=-90, out=15, looseness=0.75] (205) to (208.center);
					\draw [style=protected, in=-90, out=15] (44) to (230.center);
					\draw [style=protected, in=-90, out=165] (44) to (231.center);
					\draw [in=165, out=-90, looseness=0.75] (246.center) to (245);
					\draw [in=165, out=-90] (241.center) to (244);
					\draw [in=-90, out=15] (245) to (243.center);
					\draw [style=protected, in=-90, out=15, looseness=0.75] (244) to (247.center);
					\draw (241.center) to (254);
					\draw (254) to (248.center);
					\draw (246.center) to (255);
					\draw (255) to (249.center);
					\draw (247.center) to (261);
					\draw (243.center) to (262);
					\draw (261) to (259.center);
					\draw (262) to (260.center);
					\draw (42.center) to (36.center);
					\draw (273) to (38.center);
					\draw (274.center) to (275.center);
					\draw (279.center) to (276.center);
					\draw (282) to (278.center);
					\draw (287.center) to (288.center);
					\draw (292.center) to (289.center);
					\draw (295) to (291.center);
				\end{pgfonlayer}
			\end{tikzpicture}
		}%
	\end{equation}
	The first step holds because the multiplication map $\nabla$ is deterministic.
	The second step uses an instance of 
	\begin{equation}\label{eq:sharp}
		h = \samp \comp h^\sharp  \qquad \text{where}  \qquad  h^\sharp \coloneqq (Ph) \comp \delta_V,
	\end{equation}
	which holds for any morphism $h \colon V \to W$ in a representable Markov category.
	In particular, we use it for $h$ given by $\samp_X \otimes \samp_Y \colon PX \otimes PY \to X \otimes Y$, since $\nabla$ is equal to ${(\samp_X \otimes \samp_Y)^\sharp}$ by definition.

	Since $\delta_Y$ is deterministic and has $\samp_Y$ as left inverse, the right-hand side is equal to		
	\begin{equation}\label{eq:reflect_aseq3}
		{%
			\tikzstyle{every picture}=[tikzfig]%
			\begin{tikzpicture}
				\begin{pgfonlayer}{nodelayer}
					\node [style=none] (254) at (-1, 0) {};
					\node [style=none] (263) at (2, 2.75) {};
					\node [style=morphism] (264) at (-2.25, 1.25) {$\samp$};
					\node [style=none] (265) at (-1, 2.75) {$\cdots$};
					\node [style=none] (266) at (3, 4) {};
					\node [style=bn] (270) at (-1, 0) {};
					\node [style=bn] (271) at (2, 2.75) {};
					\node [style=none] (272) at (1, 3.5) {};
					\node [style=morphism] (273) at (0.25, 1.25) {$\samp$};
					\node [style=none] (274) at (-2.25, 5.25) {};
					\node [style=none] (275) at (-1.25, 5) {};
					\node [style=none] (276) at (0.25, 3) {};
					\node [style=none] (277) at (3, 5.25) {};
					\node [style=none] (278) at (2, 3.5) {$\cdots$};
					\node [style=none] (280) at (1.75, 4.75) {};
					\node [style=none] (281) at (0.5, 5) {$\cdots$};
					\node [style=none] (282) at (-4, 2.25) {};
					\node [style=none] (283) at (3, 2.25) {};
					\node [style=none] (284) at (3, -2.75) {};
					\node [style=none] (285) at (-4, -2.75) {};
					\node [style=none] (287) at (-1.25, 5.75) {$Y$};
					\node [style=none] (288) at (-2.25, 5.75) {$X$};
					\node [style=none] (289) at (3, 5.75) {$Y$};
					\node [style=none] (290) at (1.75, 5.75) {$X$};
					\node [style=none] (291) at (-1.25, 5.25) {};
					\node [style=none] (292) at (-1.25, 5.25) {};
					\node [style=none] (293) at (1.75, 5.25) {};
					\node [style=none] (295) at (2, 1.5) {};
					\node [style=none] (296) at (-2.25, 1.25) {};
					\node [style=none] (297) at (0.25, 1.25) {};
					\node [style=none] (302) at (-1, 0) {};
					\node [style=none] (303) at (-1, -1.5) {};
					\node [style=none] (304) at (1, -1.5) {};
					\node [style=none] (305) at (0, -3.75) {};
					\node [style=none] (306) at (0, -3.25) {};
					\node [style=none] (307) at (1, -1) {};
					\node [style=none] (308) at (0, -3.75) {$A$};
					\node [style=morphism] (310) at (0, -1.5) {$\;\quad f\quad\;$};
				\end{pgfonlayer}
				\begin{pgfonlayer}{edgelayer}
					\draw [in=165, out=-45] (272.center) to (271);
					\draw [in=-90, out=15] (271) to (266.center);
					\draw [in=135, out=-90] (275.center) to (272.center);
					\draw (277.center) to (266.center);
					\draw [style=dashed box] (283.center)
					to (282.center)
					to (285.center)
					to (284.center)
					to cycle;
					\draw [style=protected] (273) to (276.center);
					\draw [style=protected] (264) to (274.center);
					\draw [style=protected, in=-90, out=90, looseness=0.75] (276.center) to (280.center);
					\draw [style=protected] (275.center) to (292.center);
					\draw [style=protected] (280.center) to (293.center);
					\draw [style=protected] (295.center) to (271);
					\draw [style=protected, in=-90, out=165] (270) to (296.center);
					\draw [style=protected, in=-90, out=15] (270) to (297.center);
					\draw (302.center) to (303.center);
					\draw (307.center) to (304.center);
					\draw (310) to (306.center);
					\draw [in=90, out=-90, looseness=1.50] (295.center) to (307.center);
				\end{pgfonlayer}
			\end{tikzpicture}
		}%
	\end{equation}
	Note that using the assumed \cref{eq:obs_monicity}, we can replace $f$ by $g$ in Diagram \eqref{eq:reflect_aseq3}.
	Following the same steps backwards then shows that the left-hand side of \cref{eq:reflect_aseq1} coincides with the same expression with $g$ instead of $f$.
	By assumption, the morphisms in $(\samp^{(n)})_{n \in \N}$ are jointly monic, and by \cref{lem:nabla_mono}, $\nabla$ is also a monomorphism, so that we get
	\begin{equation}
		{%
			\tikzstyle{every picture}=[tikzfig]%
			\begin{tikzpicture}
				\begin{pgfonlayer}{nodelayer}
					\node [style=none] (263) at (-4, 2.5) {$PX$};
					\node [style=none] (264) at (-2, 2.5) {$PY$};
					\node [style=none] (276) at (4, 2.5) {$PY$};
					\node [style=none] (277) at (2, 2.5) {$PX$};
					\node [style=none] (280) at (0, 0) {$=$};
					\node [style=none] (281) at (-4, 2) {};
					\node [style=none] (282) at (-4, -0.5) {};
					\node [style=none] (283) at (-2, -0.5) {};
					\node [style=none] (284) at (-3, -2) {};
					\node [style=none] (285) at (-3, -2) {};
					\node [style=none] (286) at (-2, 2) {};
					\node [style=none] (287) at (-3, -2.5) {$A$};
					\node [style=morphism] (288) at (-2, 1) {$\delta$};
					\node [style=morphism] (289) at (-3, -0.5) {$\;\quad f\quad\;$};
					\node [style=none] (290) at (2, 2) {};
					\node [style=none] (291) at (2, -0.5) {};
					\node [style=none] (292) at (4, -0.5) {};
					\node [style=none] (293) at (3, -2) {};
					\node [style=none] (294) at (3, -2) {};
					\node [style=none] (295) at (4, 2) {};
					\node [style=none] (296) at (3, -2.5) {$A$};
					\node [style=morphism] (297) at (4, 1) {$\delta$};
					\node [style=morphism] (298) at (3, -0.5) {$\;\quad g\quad\;$};
				\end{pgfonlayer}
				\begin{pgfonlayer}{edgelayer}
					\draw (281.center) to (282.center);
					\draw (286.center) to (283.center);
					\draw (289) to (285.center);
					\draw (290.center) to (291.center);
					\draw (295.center) to (292.center);
					\draw (298) to (294.center);
				\end{pgfonlayer}
			\end{tikzpicture}
		}%
	\end{equation}
	and applying $\samp_Y$ to the right output gives $f = g$.
	This concludes the proof of Property \ref{it:observational_monicity} provided that Property \ref{it:observational} holds.
\end{proof}

\begin{definition}
	\label{def:observational}
	A representable Markov category is \newterm{observationally representable} if it satisfies the equivalent conditions of \cref{prop:reflect_aseq}.
\end{definition}
The terminology follows that of \cite[Definition 6.1]{moss2022probability}, which is based on the idea that repeated sampling amounts to a sequence of observations.
From \cite{fritz2023representable}, recall that we call a representable Markov category \newterm{\as{}-compatibly representable} if the bijections \eqref{eq:representable} respect almost sure equality.
\begin{proposition}\label{prop:obs_asrep}
	Every observationally representable Markov category is \as{}-compatibly representable.
\end{proposition}
\begin{proof}
	Consider $f, g \colon X \to Y$ with $f \ase{p} g$.
	Then we have 
	\begin{equation}\label{eq:obs_asrep}
		{%
			\tikzstyle{every picture}=[tikzfig]%
			\begin{tikzpicture}
				\begin{pgfonlayer}{nodelayer}
					\node [style=morphism] (264) at (-7.5, 1.25) {$f$};
					\node [style=none] (265) at (-6, 1.25) {$\cdots$};
					\node [style=none] (267) at (-9.5, 0) {$=$};
					\node [style=bn] (270) at (-6, -0.75) {};
					\node [style=morphism] (273) at (-4.5, 1.25) {$f$};
					\node [style=none] (274) at (-7.5, 2.5) {};
					\node [style=none] (276) at (-4.5, 2.5) {};
					\node [style=bn] (291) at (-12.75, 0) {};
					\node [style=morphism] (292) at (-14.5, 1.5) {$\samp$};
					\node [style=morphism] (293) at (-11, 1.5) {$\samp$};
					\node [style=none] (295) at (-14.5, 2.5) {};
					\node [style=none] (296) at (-11, 2.5) {};
					\node [style=none] (298) at (-14.5, 3) {$Y$};
					\node [style=none] (299) at (-11, 3) {$Y$};
					\node [style=none] (300) at (-12.75, 1.5) {$\cdots$};
					\node [style=none] (301) at (-11, 1.5) {};
					\node [style=none] (302) at (-14.5, 1.5) {};
					\node [style=morphism] (303) at (-12.75, -1.25) {$f^\sharp$};
					\node [style=none] (304) at (-12.75, -2.5) {};
					\node [style=none] (305) at (-12.75, -3) {$X$};
					\node [style=none] (306) at (-6, -2.5) {};
					\node [style=none] (307) at (-6, -3) {$X$};
					\node [style=none] (308) at (-7.5, 3) {$Y$};
					\node [style=none] (309) at (-4.5, 3) {$Y$};
					\node [style=none] (310) at (-7.5, 0.75) {};
					\node [style=none] (311) at (-4.5, 0.75) {};
					\node [style=none] (312) at (-2.75, 0) {$\ase{p}$};
					\node [style=morphism] (313) at (-0.75, 1.25) {$g$};
					\node [style=none] (314) at (0.75, 1.25) {$\cdots$};
					\node [style=bn] (315) at (0.75, -0.75) {};
					\node [style=morphism] (316) at (2.25, 1.25) {$g$};
					\node [style=none] (317) at (-0.75, 2.5) {};
					\node [style=none] (318) at (2.25, 2.5) {};
					\node [style=none] (319) at (0.75, -2.5) {};
					\node [style=none] (320) at (0.75, -3) {$X$};
					\node [style=none] (321) at (-0.75, 3) {$Y$};
					\node [style=none] (322) at (2.25, 3) {$Y$};
					\node [style=none] (323) at (-0.75, 0.75) {};
					\node [style=none] (324) at (2.25, 0.75) {};
					\node [style=none] (325) at (3.75, 0) {$=$};
					\node [style=bn] (326) at (7.25, 0) {};
					\node [style=morphism] (327) at (5.5, 1.5) {$\samp$};
					\node [style=morphism] (328) at (9, 1.5) {$\samp$};
					\node [style=none] (329) at (5.5, 2.5) {};
					\node [style=none] (330) at (9, 2.5) {};
					\node [style=none] (331) at (5.5, 3) {$Y$};
					\node [style=none] (332) at (9, 3) {$Y$};
					\node [style=none] (333) at (7.25, 1.5) {$\cdots$};
					\node [style=none] (334) at (9, 1.5) {};
					\node [style=none] (335) at (5.5, 1.5) {};
					\node [style=morphism] (336) at (7.25, -1.25) {$g^\sharp$};
					\node [style=none] (337) at (7.25, -2.5) {};
					\node [style=none] (338) at (7.25, -3) {$X$};
				\end{pgfonlayer}
				\begin{pgfonlayer}{edgelayer}
					\draw [in=90, out=-90] (276.center) to (273);
					\draw (274.center) to (264);
					\draw (296.center) to (293);
					\draw (292) to (295.center);
					\draw [in=-90, out=15] (291) to (301.center);
					\draw [in=-90, out=165] (291) to (302.center);
					\draw (304.center) to (303);
					\draw (303) to (291);
					\draw (306.center) to (270);
					\draw [in=-90, out=165] (270) to (310.center);
					\draw [in=-90, out=15] (270) to (311.center);
					\draw (310.center) to (264);
					\draw (311.center) to (273);
					\draw [in=90, out=-90] (318.center) to (316);
					\draw (317.center) to (313);
					\draw (319.center) to (315);
					\draw [in=-90, out=165] (315) to (323.center);
					\draw [in=-90, out=15] (315) to (324.center);
					\draw (323.center) to (313);
					\draw (324.center) to (316);
					\draw (330.center) to (328);
					\draw (327) to (329.center);
					\draw [in=-90, out=15] (326) to (334.center);
					\draw [in=-90, out=165] (326) to (335.center);
					\draw (337.center) to (336);
					\draw (336) to (326);
				\end{pgfonlayer}
			\end{tikzpicture}
		}%
	\end{equation}
	where the middle equality follows directly from the assumed $f \ase{p} g$ via \cite[Lemma 13.5]{fritz2019synthetic}.
	The other two equations hold because $f^\sharp$ is deterministic and satisfies \cref{eq:sharp} by representability.
	By observational representability, we obtain $f^{\sharp} \ase{p} g^{\sharp}$, which concludes the proof of \as{}-compatible representability.
\end{proof}

\begin{example}\label{ex:borelastoch_obs}
	$\borelstoch$ is observationally representable.
	Indeed, \cite[Theorem 9.1]{moss2022probability} shows that (even on all of $\cat{Stoch}$) the family $(\samp^{(n)})_{n \in \N}$ is jointly monic.
\end{example}

\begin{example}
	\label{ex:finsetmulti_obs}
	$\finsetmulti$ is observationally representable.
	Indeed, $\finsetmulti$ is representable since it is the Kleisli category of the non-empty powerset monad on $\finset$, which is its deterministic subcategory.
	The sampling morphism $\samp \colon PX \to X$ is given by the multivalued function mapping every non-empty subset of $X$ to its elements.
	The observationality is now the special case of \cite[Theorem 9.4]{moss2022probability} for finite discrete spaces.
\end{example}

\begin{example}
	$\setmulti$ is not observationally representable.
	The functor $X\mapsto PX$ forms the set of all nonempty subsets of $X$, including infinite ones.
	A $\setmulti$ state $I\to PX$ is then a set of subsets of $X$.
	We now define two distinct states $p,q\colon I\to PX$ which will be the same after composing with $\samp^{(n)}$ for any $n \in \N$.
	To this end, let $p \coloneqq PX$ be the set of \emph{all} nonempty subsets of $X$ itself, and let $q \subseteq PX$ be the set of \emph{all finite} nonempty subsets of $X$.
	Assuming that $X$ is infinite, these two states are trivially distinct.
	Composing any state with $\samp^{(n)}\colon PX\to X^n$ produces the set of sequences $(x_1,\dots,x_n)$ for which there is a set that contains all the $x_i$ and belongs to the state.
	Therefore both $\samp^{(n)} \comp p$ and $\samp^{(n)} \comp q$ are the set of all $n$-element sequences from $X$.
	It follows that the maps $\samp^{(n)}$ cannot distinguish between $p$ and $q$.
	Intuitively, the part where they differ{\,\textemdash\,}namely the infinite subsets of $X${\,\textemdash\,}is not detectable by finitary sampling.
\end{example}

\begin{proposition}
	\label{prop:chausstoch_obs_rep}
	$\cat{CHausStoch}$ is observationally representable.
\end{proposition}

\begin{proof}
	By point separability, it is enough to test the conditions of \Cref{prop:reflect_aseq} on states.
	We will apply \cite[Lemma~6.6]{moss2022probability} (see also the proof of Theorem 9.1 therein for context).
	For $X$ a compact Hausdorff space, $PX$ is the space of Radon probability measures on $X$ equipped with the topology of weak convergence, which is the weakest topology making the integration map
	\[
		\begin{tikzcd}[row sep=0]
			\e_f \: \colon \: PX \ar{r}{\e_f} & {[0,1]} \\
			p \ar[mapsto]{r} & \int f\,dp
		\end{tikzcd}
	\]
	continuous for every continuous map $f \colon X\to[0,1]$.
	In order to prove observationality, by the mentioned~\cite[Lemma~6.6]{moss2022probability} it suffices to show that two probability measures $\mu$ and $\nu$ on $PX$ (i.e.\ \emph{in} $PPX$) are equal whenever they give the same integral on finite products of the integration maps,
	\begin{equation}\label{eq:onproducts}
		\int_{PX} \e_{f_1}\cdots\e_{f_n} d\mu = \int_{PX} \e_{f_1}\cdots\e_{f_n} d\nu,
	\end{equation}
	for $f_1,\dots,f_n \colon X\to[0,1]$ continuous.
	Denote now by $C(PX,\R)$ the algebra of bounded continuous functions $PX\to\R$ and by $\mathcal{S}$ the unital subalgebra generated by the functions $\e_f$ for all continuous $f \colon X\to[0,1]$.
	Given that the integrals \eqref{eq:onproducts} are equal for all continuous $f_i \colon X\to[0,1]$, and that the finite products of integration maps $\e_f$ span $\mathcal{S}$, we can conclude that
	\begin{equation}
		\int_{PX} h\,d\mu = \int_{PX} h\,d\nu 
	\end{equation}
	for all $h \in \mathcal{S}$.

	With $C(PX, \R)$ carrying the usual topology of uniform convergence, we can apply the Stone--Weierstrass theorem (note that $PX$ is compact Hausdorff too) to obtain that $\mathcal{S} \subseteq C(PX, \R)$ is dense.
	Indeed the subalgebra $\mathcal{S}$ separates the points of $PX$ because the $\e_f$ generate the topology of $PX$, which is Hausdorff.

	Finally, let $g \in C(PX, \R)$ and $\eps > 0$ be arbitrary.
	Then by the density of $\mathcal{S}$ there is $h \in \mathcal{S}$ with $\| g - h \|_\infty < \eps$, which implies
	\begin{equation}
		\begin{split}
			\left| \int g\,d\mu - \int g\,d\nu \right| &= \left| \int g\,d\mu - \int h\,d\mu + \int h\,d\nu- \int g\,d\nu \right| \\
								   &\le \int |g(x)-h(x)|\,\mu(dx) + \int |g(x)-h(x)|\,\nu(dx) \\
								   &\le 2\,\e .
		\end{split}
	\end{equation}
	Therefore $\mu$ and $\nu$ assign the same integral to each continuous function $g \colon PX\to\R$, and so $\mu = \nu$ by the Riesz--Markov representation theorem.
\end{proof}

\section{An isomorphism of measure spaces which does not preserve supports}
\label{not_support}

\begin{theorem}\label{thm:not_support}
	Let us denote the Lebesgue measure on the unit interval $[0,1]$ by $\lambda_{[0,1]}$ and the normalized Lebesgue measure on the interval $[0, \linefaktor{1}{2}]$ by $2\lambda_{[0,1/2]}$.
	
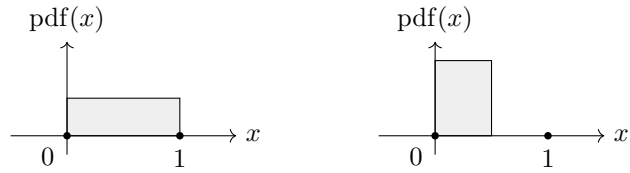
\begin{equation*}
		\begin{tikzpicture}[xscale=1.5,yscale=0.5,baseline]
			\draw [fill=fillcolor!25, domain=0:1, variable=\x]
			(0, 0)
			-- plot ({\x}, {1})
			-- (1, 0)
			-- cycle;
			\draw[->] (-0.5, 0) -- (1.5, 0) node[right] {$x$};
			\draw[->] (0, -0.5) -- (0, 2.5) node[above] {$\mathrm{pdf}(x)$};
			\node [circle,inner sep=1pt,fill=black,label=below left:$0$] at (0,0) {};
			\node [circle,inner sep=1pt,fill=black,label=below:$1$] at (1,0) {};
		\end{tikzpicture}
		\qquad\qquad
		\begin{tikzpicture}[xscale=1.5,yscale=0.5,baseline]
			\draw [fill=fillcolor!25, domain=0:0.5, variable=\x]
			(0, 0)
			-- plot ({\x}, {2})
			-- (0.5, 0)
			-- cycle;
			\draw[->] (-0.5, 0) -- (1.5, 0) node[right] {$x$};
			\draw[->] (0, -0.5) -- (0, 2.5) node[above] {$\mathrm{pdf}(x)$};
			\node [circle,inner sep=1pt,fill=black,label=below left:$0$] at (0,0) {};
			\node [circle,inner sep=1pt,fill=black,label=below:$1$] at (1,0) {};
		\end{tikzpicture}
	\end{equation*}
	Then there exists an isomorphism of measure spaces of type $\bigl( [0,1],\lambda_{[0,1]} \bigr) \to \bigl( [0,1],2\lambda_{[0, 1/2]} \bigr)$.
\end{theorem}

In order to prove the theorem, we make use of the following lemma.

\begin{lemma}\label{lem:intervals}
	There is an isomorphism of measure spaces $[0,1]\to[0,1)$, where both measurable spaces are equipped with the Lebesgue measure.
\end{lemma}

\begin{proof}[Proof of \Cref{lem:intervals}]
	Denote by $K\subseteq[0,1]$ the countable subset
	$$
	\Set{ \frac{1}{n} \given n\in\N_{>0} } = \left\{ 1, \frac{1}{2}, \frac{1}{3}, \frac{1}{4}, \dots \right\} .
	$$
	Consider now the bijection $b \colon K \to K \setminus \{1\}$ given by
	\begin{equation}
		b \left( \frac{1}{n} \right)  \coloneqq  \frac{1}{n+1}.
	\end{equation}
	The map
	\begin{equation}
		x\mapsto
		\begin{cases}
			b(x) & \textrm{if } x\in K ,\\
			x & \textrm{if } x\notin K
		\end{cases}
	\end{equation}
	is now a measure-preserving isomorphism $[0,1]\to[0,1)$, since both measurability and measure preservation follow from the fact that it differs from the identity at only countably many points.
\end{proof}

\begin{proof}[Proof of \Cref{thm:not_support}]
	Let $C\subset[0,1]$ be the Cantor set, which is a closed but uncountable subset of Lebesgue measure zero.
	Since every measurable subset of a standard Borel space is itself standard Borel, and since by Kuratowski's theorem all uncountable standard Borel spaces are isomorphic, there exists a measurable isomorphism $\Phi_0 \colon C\to(\linefaktor{1}{2},1]$. 
	This becomes an isomorphism of measure spaces if we equip both $C$ and $(\linefaktor{1}{2},1]$ with the zero measure, or equivalently, if we consider $C$ as a null set in $([0,1],\lambda_{[0,1]})$ and $(\linefaktor{1}{2},1]$ as a null set in $\bigl( [0,1],2\lambda_{[0, 1/2]} \bigr)$.

	Consider now the complement $[0,1]\setminus C$, which has Lebesgue measure $1$.
	By construction of the Cantor set, its complement is a countable disjoint union of open intervals $\coprod_{n=0}^\infty(a_n,a_n+\ell_n)$, with $0\le a_n<a_n+\ell_n\le 1$, and such that $\sum_{n=0}^\infty\ell_n=1$.
	Now by \Cref{lem:intervals}, there is a measure-preserving isomorphism $\psi_0 \colon (a_0,a_0+\ell_0)\to [0,\ell_0]$, as well as a measure-preserving isomorphism
	\begin{equation*}
		\psi_n  \:\colon\: (a_n,a_n+\ell_n) \longrightarrow \big( \ell_0+\dots+\ell_{n-1}, \ell_0+\dots+\ell_{n} \big]
	\end{equation*}
	for each $n>0$.
	Since we have
	\begin{equation}
		[0,1] = [0,\ell_0] \sqcup \coprod_{n=1}^\infty \big( \ell_0+\dots+\ell_{n-1}, \ell_0+\dots+\ell_{n} \big],
	\end{equation}
	we obtain a measure-preserving isomorphism $\Phi_1 \colon [0,1]\setminus C\to [0,1]$ given by $x\mapsto \psi_n(x)$ for all $x\in (a_n,a_n+\ell_n)$ and $n \in \N$.

	The map
	\begin{equation}
		\Phi_0 \sqcup \Phi_1 \colon [0,1] = ([0,1] \setminus C) \sqcup C \longrightarrow \left[ 0, \linefaktor{1}{2} \right] \sqcup \left( \linefaktor{1}{2}, 1 \right] = [0,1]
	\end{equation}
	given by
	\begin{equation}
		x\mapsto\begin{cases}
			\Phi_0(x) & \textrm{if } x\in C ,\\
			\tfrac{1}{2}\,\Phi_1(x) & \textrm{if } x\notin C
		\end{cases}
	\end{equation}
	is now an isomorphism between measure spaces $\bigl( [0,1],\lambda_{[0,1]} \bigr)$ and $\bigl([0,1],2\lambda_{[0, 1/2]} \bigr)$.
\end{proof}

\newpage
\bibliographystyle{abbrvnat}
\bibliography{markov}

\begin{thebibliography}{44}
\providecommand{\natexlab}[1]{#1}
\providecommand{\url}[1]{\texttt{#1}}
\expandafter\ifx\csname urlstyle\endcsname\relax
  \providecommand{\doi}[1]{doi: #1}\else
  \providecommand{\doi}{doi: \begingroup \urlstyle{rm}\Url}\fi

\bibitem[Banakh(1995)]{banakh}
T.~Banakh.
\newblock Topology of spaces of probability measures. {I}: {The} functors
  {{\(P_{\tau}\)}} and {{\(\hat P\)}}.
\newblock \emph{Mat. Stud.}, 5\penalty0 (1-2):\penalty0 65--87, 1995.
\newblock ISSN 1027-4634.
\newblock \href{https://arxiv.org/abs/1112.6161}{arXiv:1112.6161}.

\bibitem[Billik(1967)]{billik1967idempotent}
M.~Billik.
\newblock Idempotent {R}eynolds operators.
\newblock \emph{J. Math. Anal. Appl.}, 18:\penalty0 135--144, 1967.
\newblock \doi{10.1016/0022-247X(67)90188-6}.

\bibitem[Blackwell(1942)]{blackwell1942idempotent}
D.~Blackwell.
\newblock Idempotent {M}arkoff chains.
\newblock \emph{Ann. of Math. (2)}, 43:\penalty0 560--567, 1942.
\newblock \doi{10.2307/1968811}.

\bibitem[Borceux(1994)]{borceux1994handbook}
F.~Borceux.
\newblock \emph{Handbook of Categorical Algebra I}.
\newblock Cambridge University Press, 1994.
\newblock \doi{10.1017/CBO9780511525858}.

\bibitem[Borceux and Dejean(1986)]{borceux_idempotents}
F.~Borceux and D.~Dejean.
\newblock Cauchy completion in category theory.
\newblock \emph{Cahiers de Topologie et Géométrie Différentielle
  Catégoriques}, 27\penalty0 (2):\penalty0 133--146, 1986.

\bibitem[Braithwaite and Hedges(2022)]{dependent_bayesian_lenses}
D.~Braithwaite and J.~Hedges.
\newblock Dependent {B}ayesian lenses: Categories of bidirectional {M}arkov
  kernels with canonical {B}ayesian inversion, 2022.
\newblock \href{https://arxiv.org/abs/2209.14728}{arXiv:2209.14728}.

\bibitem[Braithwaite et~al.(2023)Braithwaite, Hedges, and
  Smithe]{braithwaite2023inference}
D.~Braithwaite, J.~Hedges, and T.~S.~C. Smithe.
\newblock The compositional structure of {B}ayesian inference, 2023.
\newblock \href{https://arxiv.org/abs/2305.06112}{arXiv:2305.06112}.

\bibitem[Chen et~al.(2025)Chen, Fritz, Gonda, Klingler, and
  Lorenzin]{chen2025aldoushoover}
L.~Chen, T.~Fritz, T.~Gonda, A.~Klingler, and A.~Lorenzin.
\newblock The aldous--hoover theorem in categorical probability.
\newblock \emph{Alg. Stat.}, 16:\penalty0 131--174, 2025.
\newblock \href{https://arxiv.org/abs/2411.12840}{arXiv:2411.12840}.

\bibitem[Cho and Jacobs(2019)]{chojacobs2019strings}
K.~Cho and B.~Jacobs.
\newblock Disintegration and {B}ayesian inversion via string diagrams.
\newblock \emph{Math. Structures Comput. Sci.}, 29:\penalty0 938--971, 2019.
\newblock \doi{10.1017/s0960129518000488}.
\newblock \href{https://arxiv.org/abs/1709.00322}{arXiv:1709.00322}.

\bibitem[Deville et~al.(1993)Deville, Godefroy, and
  Zizler]{deville1993renormings}
R.~Deville, G.~Godefroy, and V.~Zizler.
\newblock \emph{Smoothness and renormings in {B}anach spaces}, volume~64 of
  \emph{Pitman Monogr. Surv. Pure Appl. Math.}
\newblock Longman Publishing Group, 1993.

\bibitem[Di~Lavore et~al.(2022)Di~Lavore, de~Felice, and
  Rom{\'a}n]{di2022monoidal}
E.~Di~Lavore, G.~de~Felice, and M.~Rom{\'a}n.
\newblock Monoidal streams for dataflow programming.
\newblock In \emph{Proceedings of the 37th Annual ACM/IEEE Symposium on Logic
  in Computer Science}, pages 1--14, 2022.
\newblock \doi{10.1145/3531130.3533365}.

\bibitem[Fritz(2020)]{fritz2019synthetic}
T.~Fritz.
\newblock A synthetic approach to {M}arkov kernels, conditional independence
  and theorems on sufficient statistics.
\newblock \emph{Adv. Math.}, 370:\penalty0 107239, 2020.
\newblock \doi{10.1016/j.aim.2020.107239}.
\newblock
  \href{https://arxiv.org/abs/https://arxiv.org/abs/1908.07021}{arXiv:1908.07021}.

\bibitem[Fritz and Klingler(2023)]{fritz2022dseparation}
T.~Fritz and A.~Klingler.
\newblock The $d$-separation criterion in categorical probability.
\newblock \emph{J. Mach. Learn. Res.}, 24\penalty0 (46):\penalty0 1--49, 2023.
\newblock \href{https://arxiv.org/abs/2207.05740}{arXiv:2207.05740}.

\bibitem[Fritz and Perrone(2018)]{fritzperrone2018bimonoidal}
T.~Fritz and P.~Perrone.
\newblock Bimonoidal structure of probability monads.
\newblock In \emph{Proceedings of the {T}hirty-{F}ourth {C}onference on the
  {M}athematical {F}oundations of {P}rogramming {S}emantics ({MFPS} {XXXIV})},
  volume 341 of \emph{Electron. Notes Theor. Comput. Sci.}, pages 121--149.
  Elsevier Sci. B. V., Amsterdam, 2018.
\newblock \doi{10.1016/j.entcs.2018.11.007}.
\newblock \href{https://arxiv.org/abs/1804.03527}{arXiv:1804.03527}.

\bibitem[Fritz and Rischel(2020)]{fritzrischel2019zeroone}
T.~Fritz and E.~F. Rischel.
\newblock Infinite products and zero-one laws in categorical probability.
\newblock \emph{Compositionality}, 2:\penalty0 3, 2020.
\newblock \doi{10.32408/compositionality-2-3}.
\newblock \href{https://arxiv.org/abs/1912.02769}{arXiv:1912.02769}.

\bibitem[Fritz et~al.(2021{\natexlab{a}})Fritz, Gonda, and
  Perrone]{fritz2021definetti}
T.~Fritz, T.~Gonda, and P.~Perrone.
\newblock de {F}inetti's theorem in categorical probability.
\newblock \emph{J. Stoch. Anal.}, 2\penalty0 (4), 2021{\natexlab{a}}.
\newblock \doi{10.31390.josa.2.4.06}.
\newblock \href{https://arxiv.org/abs/2105.02639}{arXiv:2105.02639}.

\bibitem[Fritz et~al.(2021{\natexlab{b}})Fritz, Perrone, and
  Rezagholi]{fritz2019support}
T.~Fritz, P.~Perrone, and S.~Rezagholi.
\newblock Probability, valuations, hyperspace: Three monads on {T}op and the
  support as a morphism.
\newblock \emph{Math. Structures Comput. Sci.}, 31\penalty0 (8):\penalty0
  850--897, 2021{\natexlab{b}}.
\newblock \doi{10.1017/S0960129521000414}.
\newblock \href{https://arxiv.org/abs/1910.03752}{arXiv:1910.03752}.

\bibitem[Fritz et~al.(2022)Fritz, Gonda, Houghton-Larsen, Lorenzin, Perrone,
  and Stein]{fritz2022dilations}
T.~Fritz, T.~Gonda, N.~G. Houghton-Larsen, A.~Lorenzin, P.~Perrone, and
  D.~Stein.
\newblock Dilations and information flow axioms in categorical probability,
  2022.
\newblock \href{https://arxiv.org/abs/2211.02507}{arXiv:2211.02507}.

\bibitem[Fritz et~al.(2023)Fritz, Gonda, Perrone, and
  Rischel]{fritz2023representable}
T.~Fritz, T.~Gonda, P.~Perrone, and E.~F. Rischel.
\newblock Representable {M}arkov categories and comparison of statistical
  experiments in categorical probability.
\newblock \emph{Theoretical Computer Science}, 961:\penalty0 113896, 2023.
\newblock \doi{10.1016/j.tcs.2023.113896}.
\newblock \href{https://arxiv.org/abs/2010.07416}{arXiv:2010.07416}.

\bibitem[Furber and Jacobs(2015)]{furber_jacobs_gelfand}
R.~W.~J. Furber and B.~P.~F. Jacobs.
\newblock From {K}leisli categories to commutative {C}*-algebras: Probabilistic
  {G}elfand duality.
\newblock \emph{Logical Methods in Computer Science}, 11, 2015.
\newblock \doi{10.2168/LMCS-11(2:5)2015}.
\newblock \href{https://arxiv.org/abs/1303.1115}{arXiv:1303.1115}.

\bibitem[Gadducci(1996)]{gadducci1996}
F.~Gadducci.
\newblock \emph{On The Algebraic Approach To Concurrent Term Rewriting}.
\newblock PhD thesis, University of Pisa, 1996.

\bibitem[Gonda et~al.(2023)Gonda, Reinhart, Stengele, and
  Coves]{gonda2023framework}
T.~Gonda, T.~Reinhart, S.~Stengele, and G.~D.~l. Coves.
\newblock A framework for universality in physics, computer science, and
  beyond, 2023.
\newblock \href{https://arxiv.org/abs/2307.06851}{arXiv:2307.06851}.

\bibitem[Hamana(1979)]{hamana1979envelopes}
M.~Hamana.
\newblock Injective envelopes of $c^*$-algebras.
\newblock \emph{J. Math. Soc. Japan}, 31\penalty0 (1):\penalty0 181--197, 1979.

\bibitem[Harzheim(2005)]{harzheim2005ordered}
E.~Harzheim.
\newblock \emph{Ordered sets}, volume~7 of \emph{Advances in Mathematics
  (Springer)}.
\newblock Springer, New York, 2005.

\bibitem[Hu and Tholen(1996)]{hu1996free_regular}
H.~Hu and W.~Tholen.
\newblock A note on free regular and exact completions and their infinitary
  generalizations.
\newblock \emph{Theory Appl. Categ.}, 2\penalty0 (10):\penalty0 113--132, 1996.
\newblock URL \url{https://eudml.org/doc/119140}.

\bibitem[Johnstone(2002)]{johnstone2002elephant}
P.~T. Johnstone.
\newblock \emph{Sketches of an Elephant: A Topos Theory Compendium}, volume~1.
\newblock Oxford University Press, 2002.

\bibitem[Jost et~al.(2021)Jost, L\^e, and Tran]{jostmorphisms}
J.~Jost, H.~V. L\^e, and T.~D. Tran.
\newblock Probabilistic morphisms and {B}ayesian nonparametrics.
\newblock \emph{Eur. Phys. J. Plus}, 136:\penalty0 441, 2021.
\newblock \doi{10.1140/epjp/s13360-021-01427-7}.
\newblock \href{https://arxiv.org/abs/1905.11448}{arXiv:1905.11448}.

\bibitem[Karatzas and Shreve(1991)]{karatzas1991brownian}
I.~Karatzas and S.~E. Shreve.
\newblock \emph{Brownian motion and stochastic calculus.}, volume 113 of
  \emph{Grad. Texts Math.}
\newblock Springer, 2nd ed. edition, 1991.

\bibitem[Kechris(1995)]{kechris}
A.~S. Kechris.
\newblock \emph{Classical descriptive set theory}, volume 156 of \emph{Graduate
  Texts in Mathematics}.
\newblock Springer-Verlag, New York, 1995.

\bibitem[Laidler(1987)]{laidler1987kinetics}
K.~J. Laidler.
\newblock \emph{Chemical Kinetics}.
\newblock Pearson, 3rd edition, 1987.

\bibitem[Moss and Perrone(2022)]{moss2022probability}
S.~Moss and P.~Perrone.
\newblock Probability monads with submonads of deterministic states.
\newblock In \emph{Proceedings of the 37th Annual ACM/IEEE Symposium on Logic
  in Computer Science}, pages 1--13, 2022.
\newblock \doi{10.1145/3531130.3533355}.
\newblock \href{https://arxiv.org/abs/2204.07003}{arXiv:2204.07003}.

\bibitem[Moss and Perrone(2023)]{moss2022ergodic}
S.~Moss and P.~Perrone.
\newblock A category-theoretic proof of the ergodic decomposition theorem.
\newblock \emph{Ergodic Theory Dynam. Systems}, pages 1--27, 2023.
\newblock \doi{10.1017/etds.2023.6}.
\newblock \href{https://arxiv.org/abs/2207.07353}{arXiv:2207.07353}.

\bibitem[Moy(1954)]{moy1954expectation}
S.-T.~C. Moy.
\newblock Characterizations of conditional expectation as a transformation on
  function spaces.
\newblock \emph{Pacific J. Math.}, 4:\penalty0 47--63, 1954.

\bibitem[Mumford et~al.(1994)Mumford, Fogarty, and
  Kirwan]{mumford1994geometric}
D.~Mumford, J.~Fogarty, and F.~Kirwan.
\newblock \emph{Geometric invariant theory}, volume~34 of \emph{Ergebnisse der
  Mathematik und ihrer Grenzgebiete (2) [Results in Mathematics and Related
  Areas (2)]}.
\newblock Springer-Verlag, Berlin, third edition, 1994.

\bibitem[Parzygnat(2017)]{parzygnat_gelfand}
A.~J. Parzygnat.
\newblock Discrete probabilistic and algebraic dynamics: a stochastic
  commutative {G}elfand-{N}aimark theorem, 2017.
\newblock \href{https://arxiv.org/abs/1708.00091}{arXiv:1708.00091}.

\bibitem[Parzygnat(2020)]{parzygnat2020inverses}
A.~J. Parzygnat.
\newblock Inverses, disintegrations, and {B}ayesian inversion in quantum
  {M}arkov categories, 2020.
\newblock \href{https://arxiv.org/abs/2001.08375}{arXiv:2001.08375}.

\bibitem[Patterson(2020)]{patterson2020models}
E.~Patterson.
\newblock \emph{The algebra and machine representation of statistical models}.
\newblock PhD thesis, Stanford University, 2020.
\newblock \href{https://arxiv.org/abs/2006.08945}{arXiv:2006.08945}.

\bibitem[Pe{\l}czy\'{n}ski(1968)]{pelczynski1968extensions}
A.~Pe{\l}czy\'{n}ski.
\newblock Linear extensions, linear averagings, and their applications to
  linear topological classification of spaces of continuous functions.
\newblock \emph{Dissertationes Math. (Rozprawy Mat.)}, 58:\penalty0 92, 1968.
\newblock \href{https://eudml.org/doc/268522}{eudml.org/doc/268522}.

\bibitem[Perrone(2022)]{ours_entropy}
P.~Perrone.
\newblock Markov categories and entropy, 2022.
\newblock \href{https://arxiv.org/abs/2212.11719}{arXiv:2212.11719}.

\bibitem[Stein and Staton(2021)]{stein2021conditioning}
D.~Stein and S.~Staton.
\newblock Compositional semantics for probabilistic programs with exact
  conditioning.
\newblock In \emph{Logic in Computer Science}. IEEE, 2021.
\newblock \doi{10.1109/LICS52264.2021.9470552}.
\newblock \href{https://arxiv.org/abs/2101.11351}{arXiv:2101.11351}.

\bibitem[\v{C}encov(1965)]{cencovcategories}
N.~N. \v{C}encov.
\newblock The categories of mathematical statistics.
\newblock \emph{Dokl. Akad. Nauk SSSR}, 164:\penalty0 511--514, 1965.
\newblock \href{https://www.mathnet.ru/rus/dan31602}{mathnet.ru/rus/dan31602}.

\bibitem[\v{C}encov(1982)]{cencovstatistical}
N.~N. \v{C}encov.
\newblock \emph{Statistical decision rules and optimal inference}, volume~53 of
  \emph{Translations of Mathematical Monographs}.
\newblock American Mathematical Society, Providence, R.I., 1982.
\newblock Translation from the Russian edited by Lev J. Leifman.

\bibitem[Zhang and Sugiyama(2023)]{zhang}
Y.~Zhang and M.~Sugiyama.
\newblock A category-theoretical meta-analysis of definitions of
  disentanglement, 2023.
\newblock \href{https://arxiv.org/abs/2305.06886}{arXiv:2305.06886}.

\bibitem[Świrszcz(1974)]{swirszcz}
T.~Świrszcz.
\newblock Monadic functors and convexity.
\newblock \emph{Bull. Acad. Polon. Sci. Sér. Sci. Math. Astronom. Phys.}, 22,
  1974.

\end{thebibliography}

\end{document}